\documentclass[11pt]{article}

\usepackage[a4paper,height=190mm,width=134mm]{geometry}

\usepackage[utf8]{inputenc}
\usepackage{amsthm}
\usepackage{amssymb}
\usepackage{latexsym}
\usepackage{amsmath}
\usepackage{amsfonts}
\usepackage{mathrsfs}
\usepackage{amstext}
\usepackage{enumitem}
\setitemize{noitemsep,topsep=2pt,parsep=2pt,partopsep=2pt}
\usepackage{mathtools}
\DeclarePairedDelimiter{\ceil}{\lceil}{\rceil}
\DeclarePairedDelimiter{\floor}{\lfloor}{\rfloor}
\mathtoolsset{showonlyrefs}
\usepackage{stmaryrd}
\usepackage{tocloft}
\usepackage{cite}
\usepackage{accents}
\newcommand{\ubar}[1]{\underaccent{\bar}{#1}}
\usepackage[final]{graphicx}
\usepackage{longtable}

\makeatletter
\newsavebox{\@brx}
\newcommand{\llangle}[1][]{\savebox{\@brx}{\(\m@th{#1\langle}\)}%
  \mathopen{\copy\@brx\kern-0.6\wd\@brx\usebox{\@brx}}}
\newcommand{\rrangle}[1][]{\savebox{\@brx}{\(\m@th{#1\rangle}\)}%
  \mathclose{\copy\@brx\kern-0.6\wd\@brx\usebox{\@brx}}}
\newcommand{\VERT}[1][]{\savebox{\@brx}{\(\m@th{#1|}\)}%
 \mathopen{\copy\@brx\kern-0.6\wd\@brx\copy\@brx\kern-0.6\wd\@brx\usebox{\@brx}}}
\newcommand{\VERTT}[1][]{\savebox{\@brx}{\(\m@th{#1|}\)}%
 \mathopen{\copy\@brx\kern-0.6\wd\@brx\copy\@brx\kern-0.6\wd\@brx\copy\@brx\kern-0.6\wd\@brx\usebox{\@brx}}} 
\newcommand{\RECT}[1][]{\savebox{\@brx}{\(\m@th{#1[\hspace{-0.3mm}]}\)}}
\makeatother

\newcommand{\bN}{\mathbb{N}} 
\newcommand{\bE}{\mathbb{E}}

\newcommand{\bZ}{\mathbb{Z}}
\newcommand{\bR}{\mathbb{R}}
\newcommand{\bH}{\mathbb{H}}
\newcommand{\bT}{\mathbb{T}}
\newcommand{\bM}{\mathbb{M}}

\newcommand{\bY}{\mathbb{Y}}
\newcommand{\bS}{\mathbb{S}}
\newcommand{\cO}{\mathcal{O}}
\newcommand{\cD}{\mathcal{D}}
\newcommand{\cC}{\mathcal{C}}
\newcommand{\cS}{\mathcal{S}}
\newcommand{\cP}{\mathcal{P}}
\newcommand{\cX}{\mathcal{X}}
\newcommand{\cY}{\mathcal{Y}}
\newcommand{\cK}{\mathcal{K}}
\newcommand{\cG}{\mathcal{G}}

\newcommand{\fA}{\mathbf{A}}
\newcommand{\fB}{\mathbf{B}}
\newcommand{\fV}{\mathbf{V}}
\newcommand{\fI}{\mathbf{I}}
\newcommand{\fF}{\mathbf{F}}
\newcommand{\fS}{\mathbf{S}}
\newcommand{\fL}{\mathbf{L}}
\newcommand{\fQ}{\mathbf{Q}}
\newcommand{\fP}{\mathbf{P}}
\newcommand{\fR}{\mathbf{R}}
\newcommand{\fT}{\mathbf{T}}
\newcommand{\fU}{\mathbf{U}}
\newcommand{\fX}{\mathbf{X}}
\newcommand{\fY}{\mathbf{Y}}
\newcommand{\fZ}{\mathbf{Z}}
\newcommand{\fM}{\mathbf{M}}
\newcommand{\fN}{\mathbf{N}}
\newcommand{\frI}{\mathfrak{I}}

\newcommand{\frM}{\mathfrak{M}}
\newcommand{\cB}{\mathcal{B}}
\newcommand{\cV}{\mathcal{V}}

\newcommand{\cVtilde}{{\tilde{\mathcal{V}}}}

\newcommand{\sM}{\mathscr{M}}
\newcommand{\sS}{\mathscr{S}}
\newcommand{\sC}{\mathscr{C}}
\newcommand{\sD}{\mathscr{D}}
\newcommand{\ri}{\mathrm{i}}
\newcommand{\rid}{\mathrm{id}}

\newcommand{\rd}{\mathrm{d}}
\newcommand{\rdim}{\mathsf{d}}
\newcommand{\rDim}{\mathsf{D}}

\newcommand{\rc}{\mathrm{c}}
\newcommand{\rb}{\mathrm{b}}

\newcommand{\ry}{\mathrm{y}}

\newcommand{\rt}{\mathrm{t}}
\newcommand{\rD}{\mathrm{D}}
\newcommand{\sIR}{\mu}
\newcommand{\sIRb}{{ \mu}}
\newcommand{\sIRa}{{\tilde\mu}}
\newcommand{\uIR}{\eta}

\newcommand{\uv}{\kappa}
\newcommand{\sUV}{\nu}
\newcommand{\oo}{{\mathtt{g}}}
\newcommand{\ooo}{\mathtt{f}}
\newcommand{\oooo}{\mathtt{h}}
\newcommand{\I}{\mathtt{I}}
\newcommand{\vI}{\mathtt{I}}
\newcommand{\vJ}{\mathtt{J}}
\newcommand{\vK}{\mathtt{K}}
\newcommand{\vL}{\mathtt{L}}
\newcommand{\vM}{\mathtt{M}}
\newcommand{\ro}{\mathtt{l}}
\newcommand{\Tmax}{T}

\newcommand{\rri}{i}
\newcommand{\ra}{a}
\newcommand{\rrm}{m}
\newcommand{\rs}{s}
\newcommand{\rsb}{{ s}}
\newcommand{\rsa}{{\tilde s}}
\newcommand{\rr}{r}
\newcommand{\rn}{n}

\newcommand{\supp}{\mathrm{supp}}

\renewcommand{\vartriangle}{\mathrm{reg}}

\newcommand{\microf}{\mathrm{\ubar{\mathrm{f}}}}
\newcommand{\microF}{\mathrm{\ubar{\mathrm{F}}}}
\newcommand{\microXi}{\mathrm{\ubar{\Xi}}}
\newcommand{\microPhi}{\mathrm{\ubar{\Phi}}}
\newcommand{\microPsi}{\mathrm{\ubar{\Psi}}}

\makeatletter
\newcommand*\botimes{{\mathpalette\botimes@{1.5}}}
\newcommand*\botimes@[2]{\mathbin{\vcenter{\hbox{\scalebox{#2}{\hspace{0.0mm}$\m@th#1\otimes$\hspace{0.0mm}}}}}}
\newcommand*\Cdot{{\mathpalette\Cdot@{.6}}}
\newcommand*\Cdot@[2]{\mathbin{\vcenter{\hbox{\scalebox{#2}{\hspace{0.5mm}$\m@th#1\bullet$\hspace{0.5mm}}}}}}
\makeatother

\newtheorem{thm}{Theorem}
\newtheorem{lem}[thm]{Lemma}

\theoremstyle{remark}
\newtheorem{rem}[thm]{Remark}

\theoremstyle{remark}
\newtheorem{example}[thm]{Example}

\theoremstyle{definition}
\newtheorem{dfn}[thm]{Definition}
\newtheorem{ass}[thm]{Assumption}

\numberwithin{equation}{section}
\numberwithin{thm}{section}

\usepackage{color,hyperref}
\definecolor{darkblue}{rgb}{0.0,0.0,0.5}
\hypersetup{
  linkcolor  = darkblue,
  citecolor  = darkblue,
  urlcolor   = darkblue,
  colorlinks = true,
}

\makeatletter
\AtBeginDocument{
  \hypersetup{
    pdftitle = {\@title},
    pdfauthor = {\@author}
  }
}
\makeatother

\begin{document}
\title{Flow equation approach to singular stochastic PDEs}
\author{Pawe{\l} Duch
\\
Max-Planck Institute for Mathematics in the Sciences\\
Inselstr. 22, 04103 Leipzig, Germany\\
Institut f\"ur Theoretische Physik, Universit\"at Leipzig\\
Br\"uderstr.\ 16, 04103 Leipzig, Germany\\
pawel.duch@epfl.ch
}
\date{\today}

\maketitle

\begin{abstract}
We prove universality of a macroscopic behavior of solutions of a large class of semi-linear parabolic SPDEs on $\bR_+\times\bT$ with fractional Laplacian $(-\Delta)^{\sigma/2}$, additive noise and polynomial non-linearity, where $\bT$ is the \mbox{$\rdim$-dimensional} torus. We consider the weakly non-linear regime and not necessarily Gaussian noises which are stationary, centered, sufficiently regular and satisfy some integrability and mixing conditions. We prove that the macroscopic scaling limit exists and has a universal law characterized by parameters of the relevant perturbations of the linear equation. We develop a new solution theory for singular SPDEs of the above-mentioned form using the Wilsonian renormalization group theory and the Polchinski flow equation. In particular, in the case of $\rdim=4$ and the cubic non-linearity our analysis covers the whole sub-critical regime $\sigma>2$. Our technique avoids completely all the algebraic and combinatorial problems arising in different approaches.

\vspace{1mm}

\noindent {\it \small MSC classification: 60H17, 81T17 }
\end{abstract}

\tableofcontents

\section{Introduction}

The aim of this paper is to investigate the large-scale behavior of solutions of semi-linear parabolic stochastic partial differential equations (SPDEs) of the form
\begin{equation}\label{eq:intro_spde}
 \fQ \microPhi_\sUV(x) = \microF_\sUV[\microPhi_\sUV](x),
 \qquad x=(\mathring x,\bar x)\in\bR_+\times\bR^\rdim,
\end{equation}
with an appropriate initial condition at time $\mathring x=0$, where 
\begin{equation}\label{eq:force_micro}
 \microF_\sUV[\varphi](x):=\microXi_\sUV(x)+
 \sum_{i=1}^{\infty}\sum_{m=0}^{\infty}\sum_{a\in\bN_0^{\rdim m}} 
 \lambda^i\, 
 \microf^{i,m,a}_\sUV~
 \partial^{a_1}\varphi(x)\ldots \partial^{a_m}\varphi(x).
\end{equation}
We make the following assumptions: 
\begin{itemize}\label{itemize:intro}
\item $\rdim\in\bN_+$ is the spatial dimension of the spacetime $\bM=\bR\times\bR^\rdim$.
\item $\fQ$ is a parabolic (pseudo-)differential operator of the form $\partial_{\mathring x}+(-\Delta_{\bar x})^{\sigma/2}$, where $(-\Delta_{\bar x})^{\sigma/2}$ is the fractional Laplacian of order $\sigma\in(0,\rdim)$.
\item $\lambda\in\bR$ is some constant and $\sUV\in(0,1]$ is the scaling parameter.
\item $\microXi_\sUV$ is some centered and stationary random field on $\bM$ that is sufficiently regular and satisfies some mixing and integrability conditions.
\item The parameters $\microf^{i,m,a}_\sUV$ are non-zero only for finitely many $i\in\bN_+$, $m\in\bN_0$ and \mbox{$a=(a_1,\ldots,a_m)$} such that the spatial multi-indices $a_1,\ldots,a_m$ satisfy the condition $|a_1|,\ldots,|a_m|<\sigma$ (the last condition says that the SPDE~\eqref{eq:intro_spde} is semi-linear).
\end{itemize}

The SPDE~\eqref{eq:intro_spde} is called the microscopic equation, its RHS denoted by $\microF_\sUV[\varphi]$ is called the microscopic force, the random field $\microXi_\sUV$ is called the microscopic noise and the parameters $\microf^{i,m,a}_\sUV\in\bR$ are called microscopic force coefficients. The microscopic equation is supposed to describe some self-interacting system influenced by a random environment modeled by the noise~$\microXi_\sUV$.

\begin{dfn}\label{dfn:square_bracket}
For $\sUV\in\bR$ we define $[\sUV]:=|\sUV|^{1/\sigma}$. 
\end{dfn}

We consider the so-called weakly non-linear regime. We allow the non-linear terms in the microscopic equation to depend explicitly on the scaling parameter $\sUV\in(0,1]$ and demand that they vanish in the limit $\sUV\searrow0$. More specifically, we suppose that $|\microf^{i,m,a}_\sUV|\lesssim [\sUV]^{i\dim(\lambda)-\varepsilon}$ uniformly in $\sUV\in(0,1]$ for some $\dim(\lambda)>0$ and all $\varepsilon\in(0,1]$.

In order to avoid technical problems we investigate the microscopic equation on a finite spatial domain. More specifically, we assume that the microscopic noise~$\microXi_\sUV$ is $2\pi/[\sUV]$ periodic in space. In particular, it depends on the scaling parameter $\sUV$. However, the dependence of~$\microXi_\sUV$ on $\sUV$ is assumed to be very weak.

\begin{dfn}\label{dfn:dimensions}
To every quantity $X$ we assign a certain parabolic dimension \mbox{$\dim(X)\in\bR$}. In particular, for $x=(\mathring x,\bar x)=(\mathring x,\bar x^1,\ldots,\bar x^\rdim)\in\bM$ we define $\dim(\mathring x) = -\sigma$, \mbox{$\dim(\bar x^k)=-1$}, $k\in\{1,\ldots,\rdim\}$. The parabolic dimension of the spacetime $\bM$ equals \mbox{$-\rDim:=-\sigma-\rdim$}. We assume that
\begin{equation}
 \dim(\lambda)>0, 
 \quad 
 \dim(\varPhi):=(\rdim-\sigma)/2>0,
 \quad
 \dim(\varXi):=(\rdim+\sigma)/2>0.
\end{equation}
\end{dfn}

\begin{rem}
The constant $\lambda$ plays a role of a bookkeeping parameter that can be set equal to $1$ at the end of the construction. In particular, the dimension of $\lambda$ need not coincide with the physical dimension.
\end{rem}

\begin{dfn}\label{dfn:bar_frI}
The set $\bar\frI$ consists of all triples $(i,m,a)$ such that $i\in\bN_+$, $m\in\bN_0$ and $a=(a_1,\ldots,a_m)$, where $a_1,\ldots,a_m$ are spatial multi-indices such that $|a_1|,\ldots,|a_m|<\sigma$. For $(i,m,a)\in\bar\frI$ we define
\begin{equation}
 \varrho(i,m,a):=-\dim(\varXi)+i\dim(\lambda) + m\dim(\varPhi) + |a|,
 \qquad
 |a|=|a_1|+\ldots+|a_m|.
\end{equation}
The sets $\bar\frI^-$ and $\bar\frI^+$ consist of triples $(i,m,a)\in\bar\frI$ such that $\varrho(i,m,a)\leq0$ and $\varrho(i,m,a)>0$, respectively. The force coefficients $\microf^{i,m,a}_\sUV$ with $(i,m,a)\in\bar\frI^\mp$ are called relevant/irrelevant. Because by assumption $\dim(\lambda)>0$ and $\dim(\varPhi)>0$ there are only finitely many relevant coefficients.
\end{dfn}

Since the noise $\microXi_\sUV$ is assumed to be regular the initial value problem for the microscopic equation is classically well-posed (in some open possibly random time interval). Let $\microPhi_\sUV$ be the solution of the microscopic equation~\eqref{eq:intro_spde}. In the paper we investigate the limit $\sUV\searrow0$ of the stochastic process $\varPhi_\sUV$ related to the process $\microPhi_\sUV$ by the equation
\begin{equation}\label{eq:scaling_Phi}
 \varPhi_\sUV(\mathring x,\bar x) = [\sUV]^{-\dim(\varPhi)}\,\microPhi_\sUV(\mathring x/\sUV,\bar x/[\sUV]).
\end{equation}
Note that the rescaled process~$\varPhi_\sUV$ satisfies the following equation
\begin{equation}\label{eq:intro_spde_macro}
 \fQ \varPhi_\sUV(x) 
 = F_\sUV[\varPhi_\sUV](x),
 \qquad x\in\bR_+\times\bR^\rdim,
\end{equation}
with the initial condition $\varPhi_\sUV(0,\Cdot)=\phi_\sUV$, where
\begin{equation}\label{eq:force_macro}
 F_\sUV[\varphi](x):= \varXi_\sUV(x) +  \sum_{(i,m,a)\in\bar\frI}\,\lambda^i\, 
 f^{i,m,a}_\sUV~\partial^{a_1}\varphi(x)\ldots \partial^{a_m}\varphi(x)
\end{equation} 
and
\begin{gather}
\label{eq:noise_rescaled}
\varXi_\sUV(\mathring x,\bar x) := [\sUV]^{-\dim(\varXi)}\,\microXi_\sUV(\mathring x/\sUV,\bar x/[\sUV]),
\\
\label{eq:f_micro_macro}
f^{i,m,a}_\sUV := [\sUV]^{-\dim(\varXi)+m\dim(\varPhi)+|a|}\, \microf^{i,m,a}_\sUV.
\end{gather}
The SPDE~\eqref{eq:intro_spde_macro} is called the macroscopic equation, its RHS denoted by~$F_\sUV[\varphi]$ is called the (macroscopic) force, the random field $\varXi_\sUV$ is called the (macroscopic) noise and the parameters $f^{i,m,a}_\sUV\in\bR$ are called the (macroscopic) force coefficients. 

Note that the rescaled noise $\varXi_\sUV$ is $2\pi$ periodic in space. We assume that the initial data $\phi_\sUV$ is also periodic and look for periodic solutions of the macroscopic equation. Consequently, the macroscopic equation can be equivalently viewed as an equation on $\bR_+\times\bT$, where $\bT$ is the $\rdim$-dimensional torus of size $2\pi$. 

The microscopic force coefficients vanish in the limit $\sUV\searrow0$. The same is true for the irrelevant macroscopic force coefficients. On the other hand, the relevant macroscopic force coefficients often do not have a limit as $\sUV\searrow0$. 

Under our assumptions about the noise $\microXi_\sUV$ the law of $\varXi_\sUV$ converges to the law of the white noise $\varXi_0$ on $\bH:=\bR\times\bT$ as $\sUV\searrow0$. Moreover, $\lim_{\sUV\searrow0} f^{i,m,a}_\sUV=0$ for all irrelevant force coefficients. This suggests that formally the macroscopic process defined by the scaling limit $\varPhi_0:=\lim_{\sUV\searrow0}\varPhi_\sUV$ should satisfy the following SPDE driven by the white noise
\begin{equation}\label{eq:intro_singular}
 \fQ \varPhi_0(x) = \varXi_0(x) +  \sum_{(i,m,a)\in\bar\frI^-} 
 \lambda^i\, 
 f^{i,m,a}_0~\partial^{a_1}\varPhi_0(x)\ldots \partial^{a_m}\varPhi_0(x).
\end{equation}
Observe that the sum on the RHS of the above equation involves only the relevant force coefficients $f^{i,m,a}_0=\lim_{\sUV\searrow0}f^{i,m,a}_\sUV$. However, the coefficients $f^{i,m,a}_0$ need not be well-defined or finite. The regularity of the solution of the above equation is expected to be slightly worse than $-\dim(\varPhi)=(\sigma-\rdim)/2$, which is negative by assumption. As a result, the products that appear in the above SPDE are typically ill-defined and the equation is generically singular (apart from the spacial cases when it happens to be linear).

\begin{rem}
In order to make sense of a singular SPDE of the form~\eqref{eq:intro_singular} one considers its approximations. Any macroscopic SPDE of the form~\eqref{eq:intro_spde_macro} is a valid approximation. The standard approximation of a singular SPDE of the form~\eqref{eq:intro_singular} is obtained by replacing the white noise $\varXi_0$ with a mollified noise $\varXi_{\sUV;0}$ and allowing the force coefficients $f_{\sUV;0}^{i,m,a}$, $(i,m,a)\in\bar\frI^-$, to depend on the UV cutoff $\sUV\in(0,1]$. It turns out the standard approximation is of the form~\eqref{eq:intro_spde_macro}.
\end{rem}

\begin{rem}
Because of the assumption $\dim(\lambda)>0$ the singular SPDEs obtained from the macroscopic equations~\eqref{eq:intro_spde_macro} in the limit $\sUV\searrow0$ are always sub-critical in the sense of~\cite{hairer2014structures,bruned2021renormalising}.
\end{rem}

The main result of the paper is the proof of the existence and universality of the macroscopic scaling limit $\varPhi_0=\lim_{\sUV\searrow0}\varPhi_\sUV$, where the process $\varPhi_\sUV$ is obtained by the rescaling~\eqref{eq:scaling_Phi} of the solution $\microPhi_\sUV$ of the microscopic equation~\eqref{eq:intro_spde}. We assume the following:
\begin{enumerate}
\item[(1)] The noise is stationary, centered and satisfies certain regularity, integrability and mixing conditions.

\item[(2)] The initial data is sufficiently close to equilibrium (in the sense explained below).

\item[(3)] The relevant microscopic force coefficients have a specific asymptotic behavior for small values of the scaling parameter $\sUV$.
\end{enumerate}

Our assumptions about the noise are quite typical and essentially coincide with those introduced in~\cite{hairer2013kpz}. Recall that in the case of singular SPDEs close to criticality in general it is not possible to start the equation from initial data of the same regularity as the solution~\cite{bruned2021renormalising}. The need for the above-mentioned fine-tuning of the relevant force coefficients is called the renormalization problem. 

The solution manifold of the macroscopic SPDE~\eqref{eq:intro_spde_macro} can be parameterized in terms of the initial data $\phi_\sUV$ and the force coefficients $f^{i,m,a}_\sUV$. In order to solve the renormalization problem we use a different parametrization that is well defined in the limit $\sUV\searrow0$ and involves certain macroscopically measurable quantities like, for example, the effective mass at a spatial scale of order one. 

Let $G\in\sS'(\bM)$ be the fundamental solution for the pseudo-differential operator $\fQ$. We call $G$ the fractional heat kernel. Note that $G$ is supported in $[0,\infty)\times\bR^\rdim$ and is smooth outside the hyperplane $\{0\}\times\bR^\rdim$ (if $\sigma\in2\bN_+$, then $G$ is smooth outside the origin). We define a differentiable family of smooth kernels $G_\sIR$, $\sIR\in[0,\infty)$ in such a way that $G_\sIR$ coincides with $G$ in the time interval $[2\sIR,\infty)$ and its support is contained in the time interval $[\sIR,\infty)$. Note that the characteristic spatial length scale of the kernel $G_\sIR$ is of order $[\sIR]=\sIR^{1/\sigma}$ and $G_0=G$.

\begin{rem}
The fact that for $\sigma\notin 2\bN_+$ the fractional heat kernel is not smooth on the hyperplane $\{0\}\times\bR^\rdim$, which is not a compact set, poses a major difficulty. Our analysis can be significantly simplified in the case $\sigma\in2\bN_+$, that is, when the SPDE under investigation is expressed in terms of local differentiable operators.
\end{rem}

We claim that under the assumption about the initial data stated below the renormalization problem for the initial value formulation of the macroscopic SPDE~\eqref{eq:intro_spde_macro} is equivalent to the renormalization problem for the following equation
\begin{equation}\label{eq:informal_statement_stationary}
 \varPsi_\sUV = (G-G_1)\ast F_\sUV[\varPsi_\sUV]
\end{equation}
posed in the whole space $\bM=\bR\times\bR^\rdim$. The fundamental object of the flow equation approach is the effective force functional
\begin{equation}\label{eq:intro_effective_force}
 C^\infty_\rc(\bM)\ni\varphi\mapsto F_{\sUV,\sIR}[\varphi]\in \sD'(\bM),\qquad \sIR\in[0,1],
\end{equation}
which is a formal power series in $\lambda$ of the form
\begin{equation}\label{eq:idea_ansatz}
 \langle F_{\sUV,\sIR}[\varphi],\psi\rangle = \sum_{i=0}^\infty\sum_{m=0}^\infty \lambda^i\,\langle F_{\sUV,\sIR}^{i,m},\psi\otimes\varphi^{\otimes m}\rangle, 
\end{equation}
where $\psi,\varphi\in C^\infty_\rc(\bM)$ and given $i\in\bN_0$ the coefficients $F_{\sUV,\sIR}^{i,m}\in\sD'(\bM\times\bM^m)$ vanish for all but finitely many $m\in\bN_0$. The effective force coefficients $F_{\sUV,\sIR}^{i,m}$ are constructed recursively with the use of a certain flow equation of the form
\begin{equation}\label{eq:flow_eq_form}
 \partial_\sIR F^{i,m}_{\sUV,\sIR} = \sum_{j=0}^i \sum_{k=0}^m (k+1)\,\fB_\sIR(F^{j,k+1}_{\sUV,\sIR},F^{i-j,m-k}_{\sUV,\sIR}),
\end{equation}
where $\fB_\sIR$ is some bilinear operator. By definition, $F_{\sUV,0}[\varphi]$ coincides with the force $F_\sUV[\varphi]$.  For $\sIR,\uIR\in[0,1]$ such that $\uIR\leq\sIR$ the effective force satisfies the equation
\begin{equation}\label{eq:effective_force_property}
 F_{\sUV,\sIR}[\varphi] = F_{\sUV,\uIR}[\varphi+(G_\uIR-G_\sIR)\ast F_{\sUV,\sIR}[\varphi]].
\end{equation}
This equation with $\uIR=0$ and $\sIR=1$ implies that the formal power series
\begin{equation}\label{eq:intro_stationary_solution}
 \varPsi_\sUV=(G-G_1)\ast F_{\sUV,1}[0] = 
 \sum_{i=0}^\infty \lambda^i \,(G-G_1)\ast F^{i,0}_{\sUV,1}
\end{equation}
is a stationary solution of Eq.~\eqref{eq:informal_statement_stationary}. For a list $a=(a_1,\ldots,a_m)$ of spatial multi-indices we define a polynomial 
\begin{equation}\label{eq:polynomial_relative}
 \cX^{m,a}(x;x_1,\ldots,x_m):=(x-x_1)^{a_1}\ldots(x-x_m)^{a_m}.
\end{equation}
Let $1_\bM\in C^\infty(\bM)$ be the constant function equal to $1$. We shall prove that the formula 
\begin{equation}\label{eq:intro_eff_force_coeff}
 \langle f^{i,m,a}_{\sUV,\sIR},\psi\rangle:=\langle F^{i,m,a}_{\sUV,\sIR},\psi\otimes 1_\bM^{\otimes m}\rangle,
 \qquad
 F^{i,m,a}_{\sUV,\sIR}:=\cX^{m,a} F^{i,m}_{\sUV,\sIR},
\end{equation}
defines the coefficients $f^{i,m,a}_{\sUV,\sIR}\in \sD'(\bM)$. Note that $f^{i,m,a}_{\sUV,0}$ coincides with the force coefficient $f^{i,m,a}_\sUV\in\bR$. For $\sIR>0$ the coefficients $f^{i,m,a}_{\sUV,\sIR}$ are measurable functions of the noise.

\begin{dfn}\label{dfn:expected_value}
The expectation of a random variable $X$ is denoted by $\llangle X\rrangle$. 
\end{dfn}

We are ready to state our assumptions about the force coefficients. We suppose that the irrelevant force coefficients are arbitrary such that the following bounds
\begin{equation}\label{eq:informal_statement_bound_force}
 |f^{i,m,a}_\sUV|\lesssim [\sUV]^{\varrho(i,m,a)-\varepsilon},
 \qquad
 (i,m,a)\in\bar\frI^+,
\end{equation}
hold uniformly in $\sUV\in(0,1]$ and we fix the relevant force coefficients $f^{i,m,a}_\sUV$, $(i,m,a)\in\bar\frI^-$, by the following renormalization conditions
\begin{equation}\label{eq:intro_ren_conditions}
 \lim_{\sUV\searrow0}\,\llangle f^{i,m,a}_{\sUV,1} \rrangle =\mathfrak{f}^{i,m,a}_0,\quad (i,m,a)\in\bar\frI^-,
\end{equation}
where $\mathfrak{f}^{i,m,a}_0\in\bR$ are arbitrarily chosen renormalization parameters. By stationarity $\llangle f^{i,m,a}_{\sUV,\sIR} \rrangle$ is constant over spacetime.

\begin{rem}
Eq.~\eqref{eq:f_micro_macro} implies that the bound~\eqref{eq:informal_statement_bound_force} is equivalent to the bound $|\microf^{i,m,a}_\sUV|\lesssim [\sUV]^{i\dim(\lambda)-\varepsilon}$ expressed in terms of microscopic force coefficients.
\end{rem}

\begin{rem}
In our approach all the relevant force coefficients (including the counterterms) are fixed by the renormalization conditions and for any $\sUV\in(0,1]$ their values can be computed with the use of the flow equation.
\end{rem}

\begin{rem}
It turns out that the bound~\eqref{eq:informal_statement_bound_force} holds automatically also for the relevant force coefficients. 
\end{rem} 

\begin{rem}
It is important that the kernel $G-G_1$ that appears in Eq.~\eqref{eq:informal_statement_stationary} belongs to $L^1(\bM)$. Note that this is not the case for the kernel $G$.
\end{rem}

As we mentioned above, we need to assume the initial data $\phi_\sUV$ is a small perturbation of the initial data leading to a stationary solution. More precisely, we suppose that $\varPhi_\sUV(0,\Cdot)=\phi_\sUV$ is of the following form
\begin{equation}\label{eq:intro_initial}
 \phi_\sUV = \phi^\triangleright_\sUV +\phi_\sUV^\vartriangle,
 \quad
 \phi^\triangleright_\sUV:=\varPhi_\sUV^\triangleright(0,\Cdot),
 \quad
 \varPhi_\sUV^\triangleright := (G-G_1)\ast f^\triangleright_\sUV,
 \quad
 f^\triangleright_\sUV:=\sum_{i=0}^{i_\triangleright}\lambda^i F^{i,0}_{\sUV,1},
\end{equation}
where $F^{i,m}_{\sUV,\sIR}$ are the effective force coefficients defined above. We choose $i_\triangleright\in\bN_0$ in such a way that for any $i>i_\triangleright$ the coefficient of order $\lambda^i$ in the series~\eqref{eq:intro_stationary_solution} can be uniformly bounded in some space with positive regularity. Note that $\varPhi_\sUV^\triangleright$ defined above is obtained by truncating the stationary solution~\eqref{eq:intro_stationary_solution} of Eq.~\eqref{eq:informal_statement_stationary} at order $\lambda^{i_\triangleright}$. We assume that the random functions $\phi_\sUV^\vartriangle$ are such that $\phi_0^\vartriangle=\lim_{\sUV\searrow0}\phi_\sUV^\vartriangle$ converges in a Besov space $\sC^\beta(\bT)$ with the regularity $\beta$ depending on $\dim(\lambda)$ and $\dim(\varPhi)$. If \mbox{$\dim(\lambda)>\dim(\varPhi)$}, then $\beta$ coincides with the expected regularity of the macroscopic scaling limit \mbox{$\varPhi_0=\lim_{\sUV\searrow0}\varPhi_\sUV$}. If $\dim(\lambda)$ is small, then $\beta$ can be only slightly negative.

Now we are ready to give an informal statement of the main result of the paper. For the precise formulation we refer the reader to Theorem~\ref{thm:main}.
\begin{thm}\label{thm:main_informal}
 Let $\varPhi_\sUV$ be the solution of the macroscopic equation~\eqref{eq:intro_spde_macro} with the initial data $\phi_\sUV$. Under the assumptions described above the macroscopic scaling limit $\varPhi_0=\lim_{\sUV\searrow0}\varPhi_\sUV$ exists and its law depends only on the renormalization parameters $\mathfrak{f}^{i,m,a}_0\in\bR$, $(i,m,a)\in\bar\frI^-$, and the law of \mbox{$\phi_0^\vartriangle=\lim_{\sUV\searrow0}\phi_\sUV^\vartriangle$}.
\end{thm}

\begin{rem}
For every $\sUV\in(0,1]$ the law of the solution $\varPhi_\sUV$ of the macroscopic equation depends on the law of the noise $\varXi_\sUV$, the law of the initial data $\phi_\sUV$ and the values of the relevant and irrelevant force coefficients~$f^{i,m,a}_\sUV$.
\end{rem}

\begin{rem}
The fact that the law of the scaling limit $\varPhi_0$ depends only on the renormalization parameters $\mathfrak{f}^{i,m,a}_0\in\bR$, $(i,m,a)\in\bar\frI^-$, and the law of $\phi_0^\vartriangle$ is a universality-type result. For a pedagogical introduction of the notion of the weak universality in the context of SPDEs see~\cite{hairer2018renormalization}.
\end{rem}

\begin{rem}
Let us recall that by assumption $\sigma\in(0,\rdim)$. Hence \mbox{$\dim(\varPhi)>0$}. This together with $\dim(\lambda)>0$ guarantees that the number of the relevant force coefficients as well as the number of renormalization conditions is finite. Consequently, given the initial data $\phi_0^\vartriangle$ the law of the scaling limit $\varPhi_0$ depends only on finitely many renormalization parameters $\mathfrak{f}^{i,m,a}_0\in\bR$, $(i,m,a)\in\bar\frI^-$.
\end{rem}

\begin{rem}
In the case $\dim(\varPhi)\leq 0$ there are infinitely many possible relevant interaction terms including terms which are not polynomials. Consequently, generically no universality result is expected to hold. The exception are physical systems with symmetries (such as the system in Example~\ref{example:kpz} below).
\end{rem}

\begin{rem}
The technique developed in the paper allows to renormalize singular SPDEs with polynomial non-linearities also in the regime $\dim(\varPhi)=0$. Note that in this case we fix the renormalization parameters in such a way that only finitely many of them are non-zero. The generalization to non-polynomial non-linearities and the regime $\dim(\varPhi)<0$ is an open problem.
\end{rem}

\begin{rem}
The existence of the limit $\varPhi_0=\lim_{\sUV\searrow0}\varPhi_\sUV$ is established only on bounded time intervals up to some possibly random finite explosion time. The global in time existence is not expected for some of the equations studied in the paper.
\end{rem}

\begin{rem}
Our result provides a very general approximation theory for singular sub-critical SPDEs of the form~\eqref{eq:intro_singular} with a polynomial non-linearity and driven by the additive space-time white noise. In particular, our method covers the case of non-Gaussian approximations of the white noise. 
\end{rem}

Let us briefly describe the technique we use to establish the above result. As usual, the proof can be divided into the probabilistic and deterministic part. In the probabilistic part we show that the noise $\varXi_\sUV\equiv f^{0,0,0}_{\sUV,\sIR}\in C(\bM)$ as well as the relevant effective force coefficients $f^{i,m,a}_{\sUV,\sIR}\in C(\bM)$, $(i,m,a)\in\bar\frI^-$, introduced above, converge in law as \mbox{$\sUV\searrow0$} to distributional random fields
\begin{equation}
 f^{i,m,a}_{0,\sIR}\in\sD'(\bM), \qquad 
 (i,m,a)\in\bar\frI^-_0:=\{(0,0,0)\}\cup\bar\frI^-,
\end{equation}
for every $\sIR\in(0,1]$ and the law of the random fields $f^{i,m,a}_{0,\sIR}$, $(i,m,a)\in\bar\frI^-_0$ depends only on the choice of the renormalization parameters $\mathfrak f^{i,m,a}_0$, \mbox{$(i,m,a)\in\bar\frI^-$}. In this step we solve the renormalization problem, that is we fix appropriately the relevant force coefficients $f^{i,m,a}_\sUV$, $(i,m,a)\in\bar\frI^-$. Let us remark that the collection of the relevant effective force coefficients $f^{i,m,a}_{\sUV,\sIR}$, $(i,m,a)\in\bar\frI^-_0$ plays an analogous role to the model and the enhanced noise in the approaches to singular SPDEs based on the theory of regularity structures~\cite{hairer2014structures} and the paracontrolled calculus~\cite{gubinelli2015}, respectively. Note also that the number of the relevant effective force coefficients diverges as the SPDE approaches the critical one. Proving the convergence of the relevant effective force coefficients as \mbox{$\sUV\searrow0$} is the main challenge in constructing solutions of singular SPDEs close to criticality. To this end, we study cumulants of the effective force coefficients and observe that they satisfy a certain flow equation that follows from the flow equation~\eqref{eq:flow_eq_form}. Using the flow equation we prove by induction certain uniform bounds for cumulants, which imply convergence of the relevant effective force coefficients as \mbox{$\sUV\searrow0$}. The renormalization problem is solved by imposing appropriate boundary conditions for the flow equation for cumulants. The correct values of the counterterms, that is the relevant force coefficients $f^{i,m,a}_\sUV$, $(i,m,a)\in\bar\frI^-$, are determined automatically by the flow equation and the boundary conditions. We stress that in contrast to other approaches to singular SPDEs our approach does not rely on any diagrammatical representations and, in particular, avoids the algebraic and combinatorial problems related to the so-called divergent sub-diagrams.

In the deterministic part of the proof for every $\sUV\in[0,1]$ we construct \mbox{$\breve\varPhi_\sUV\in\sD'(\bM)$} as a function of (generic realizations of):
\begin{itemize}
 \item relevant effective force coefficients $f^{i,m,a}_{\sUV,\sIR}\in \sD'(\bM)$, $(i,m,a)\in\bar\frI^-_0$, $\sIR\in(0,1]$,
 \item irrelevant force coefficients $f^{i,m,a}_{\sUV}\in\bR$, $(i,m,a)\in\bar\frI^+$ and
 \item the initial condition $\phi_\sUV^\vartriangle\in\sD'(\bR^\rdim)$.
\end{itemize}
We prove that the above-mentioned function satisfies a certain continuity property which together with the result of the probabilistic analysis implies the existence of the limit $\lim_{\sUV\searrow0}\breve\varPhi_\sUV=\breve\varPhi_0$. For $\sUV\in(0,1]$ the distribution $\breve\varPhi_\sUV$ is actually a function and is related to the classical mild solution $\varPhi_\sUV$ of the macroscopic equation~\eqref{eq:intro_spde_macro} by the formula
\begin{equation}
 \varPhi_\sUV 
 =\varPhi_\sUV^\triangleright + \breve\varPhi_\sUV,
\end{equation}
where $\varPhi_\sUV^\triangleright$ was introduced in Eq.~\eqref{eq:intro_initial}. The function $\breve\varPhi_\sUV$ coincides with the mild solution of the equation
\begin{equation}\label{eq:macroscopic_breve}
 \fQ \breve\varPhi_\sUV = F_\sUV[\breve\varPhi_\sUV + \varPhi_\sUV^\triangleright]-\fQ \varPhi_\sUV^\triangleright
\end{equation}
with the initial condition $\breve\varPhi_\sUV(0,\Cdot)=\phi_\sUV^\vartriangle$. The solution of the above equation in a~time interval $[t,t+T]$ is written in the form 
\begin{equation}
 \breve\varPhi_\sUV = \varPhi_\sUV^\vartriangle + \tilde\varPhi_\sUV,
 \qquad
 \tilde\varPhi_\sUV = (G-G_T)\ast \tilde F_{\sUV,T}[0]
 =(G-G_T)\ast\sum_{i=0}^\infty \lambda^i\, \tilde F^{i,0}_{\sUV,T},
\end{equation}
where $\varPhi_\sUV^\vartriangle$ solves the equation $\fQ\varPhi_\sUV^\vartriangle=0$ with an appropriately chosen initial condition at $\mathring x=t$ and $\tilde F^{i,m}_{\sUV,\sIR}\in \sD'(\bM\times\bM^m)$, $i,m\in\bN_0$, are coefficients of some effective force functional $\tilde F_{\sUV,\sIR}[\varphi]$. We construct the coefficients $\tilde F^{i,m}_{\sUV,\sIR}$ recursively using a flow equation of the form~\eqref{eq:flow_eq_form} and show that they can be expressed in terms of the initial data at time $\mathring x=t$, the relevant effective force coefficients $f^{i,m,a}_{\sUV,\sIR}\in\sD'(\bM)$, $(i,m,a)\in\bar\frI^-_0$ as well as the irrelevant force coefficients $f^{i,m,a}_{\sUV}\in\bR$, $(i,m,a)\in\bar\frI^+$. We prove bounds for the coefficients $\tilde F^{i,m}_{\sUV,\sIR}$ that imply the absolute convergence of the series defining $\tilde\varPhi_\sUV$ for sufficiently small $T\in(0,1)$. The maximal solution $\breve\varPhi_\sUV$ of Eq.~\eqref{eq:macroscopic_breve} is constructed by patching together solutions in small time intervals.

\begin{rem}
The effective force $\tilde F_{\sUV,\sIR}[\varphi]$ depends on $t\in[0,\infty)$ and the choice of the initial data at $\mathring x=t$. We stress that $\tilde F_{\sUV,\sIR}[\varphi]$ does not coincide with the stationary effective force $F_{\sUV,\sIR}[\varphi]$ introduced earlier to solve the renormalization problem. However, as we shall see, in a certain region the coefficients of $\tilde F_{\sUV,\sIR}[\varphi]$ can be expressed in terms of the coefficients of $F_{\sUV,\sIR}[\varphi]$. The above-mentioned relation between the two effective force functionals plays a crucial role in the proof.
\end{rem}

The rest of this article is organized as follows. In Sec.~\ref{sec:technique} we give an informal overview of the proof. The precise statement of the main result of the paper is given in Sec.~\ref{sec:result}. In Sec.~\ref{sec:kernels} we introduce kernels that are used to control the regularity of various distributions. We also define the scale decomposition of the fractional heat kernel. Sec.~\ref{sec:topology} contains definitions of function spaces for the effective force coefficients and their cumulants. In Sec~\ref{sec:taylor} we use the Taylor theorem to decompose a distribution into a local part and a certain better behaved remainder. The material presented in Sec.~\ref{sec:kernels}-\ref{sec:taylor} has a preliminary character.  When first reading these sections the reader should concentrate on familiarizing with the definitions of new objects. The lemmas and theorems stated in these sections can be read later when they are referred to. In Sec.~\ref{sec:effective_force} we construct the effective force coefficients with the cutoff $\nu\in(0,1]$ and establish some of their properties. Let us stress that the bounds stated in this section are not uniform in the cutoff. In Sec.~\ref{sec:cumulants_estimates} we prove bounds for the cumulants of the effective force coefficients that are uniform in the cutoff $\nu\in(0,1]$. In Sec.~\ref{sec:probabilistic} we show how to conclude convergence of the relevant local effective force coefficients as $\nu\searrow0$ using the bounds for cumulants. In Sec.~\ref{sec:deterministic} we construct the solution of the SPDE and express it as a function of the initial data, the relevant local effective force coefficients and the irrelevant force coefficients. Theorem~\ref{thm:main_informal} follows by combining the results of Sec.~\ref{sec:probabilistic} and~\ref{sec:deterministic}. Appendix~\ref{sec:app} contains the symbolic index.

\section{Technique of the proof}\label{sec:technique}

In this section we give an informal description of the key steps of the proof of our result. We first give a general overview of the flow equation approach. Then we show how to solve the renormalization problem. Finally, we discuss the initial value problem.

\subsection{Flow equation approach}\label{sec:intro_flow}

Recall that $G$ denotes the fractional heat kernel and we employ the notation $x=(\mathring x,\bar x)\in\bR\times\bR^\rdim=\bM$. The kernel $G(\mathring x,\bar x)$ vanishes for $\mathring x<0$ and is smooth outside the hyper-surface $\mathring x=0$ (outside the point $(\mathring x,\bar x)=0$ in the special case $\sigma\in2\bN_+$). Fix $\chi\in C^\infty(\bR)$ such that $\supp\,\chi\subset(1,\infty)$ and $\chi=1$ on some neighborhood of $[2,\infty)$. The smooth kernel with UV cutoff $\sIR\in(0,\infty)$ is defined by $G_\sIR(\mathring x,\bar x):=\chi(\mathring x/\sIR) G(\mathring x,\bar x)$. Note that $G_\sIR(\mathring x,\bar x) = G(\mathring x,\bar x)$ for $\mathring x\geq 2\sIR$ and $G_\sIR(\mathring x,\bar x)=0$ for $\mathring x\leq \sIR$. We also set $G_0=G$. For any $T\in(0,\infty)$ we have
\begin{equation}
 G=-\int_0^T \dot G_\sIR\,\rd\sIR+G_T,
 \qquad
 \dot G_\sIR:=\partial_\sIR G_\sIR.
\end{equation}
The following bound
\begin{equation}
 \int_\bM |x^a \partial^b \dot G_\sIR(x)|\,\rd x\lesssim [\sIR]^{[a]-[b]}
\end{equation}
holds uniformly in $\sIR\in(0,1]$ for all multi-indices $a,b\in\frM$ such that $|\bar a|<\sigma$, where $[\sIR]=\sIR^{1/\sigma}$ and $[a]=\sigma \mathring a + |\bar a|$ for any $a=(\mathring a,\bar a)\in\frM=\bN_0^{1+\rdim}$ (see Def.~\ref{dfn:st_multi}).

Let $\sUV\in(0,1]$. Recall that the force $F_\sUV[\varphi](x)$, defined by Eq.~\eqref{eq:force_macro}, is a local functional of $\varphi\in C^\infty(\bM)$. It is a sum of the noise $\varXi_\sUV(x)$ and some non-linearity $N_\sUV[\varphi](x)$ which is a polynomial in $\varphi(x)$ and its derivatives of order strictly bounded by $\sigma$. Let $T\in(0,1]$ and consider the following equation
\begin{equation}\label{eq:intro_stationary_mild}
 \varPsi_\sUV = (G-G_T)\ast F_\sUV[\varPsi_\sUV]
\end{equation}
in $\bM$. This equation does not coincide with the mild form of the macroscopic equation~\eqref{eq:intro_spde_macro}. Nevertheless, the UV behaviors of solutions of both equations are similar. The basic object of the flow equation approach is the so-called effective force functional
\begin{equation}
 C^\infty_\rc(\bM)\ni\varphi\mapsto F_{\sUV,\sIR}[\varphi]\in \sD'(\bM)
\end{equation}
defined for $\sIR\in[0,T]$. The effective force is generically a non-local functional of $\varphi\in C^\infty_\rc(\bM)$. Under our assumptions about the noise $\varXi_\sUV$ the effective force $F_{\sUV,\sIR}[\varphi](x)$ is a continuous function of $x\in\bM$ but it can be bounded uniformly in $\sUV\in(0,1]$ only in the space of distributions. By definition, the effective force satisfies the following flow equation
\begin{equation}\label{eq:intro_flow_eq}
 \langle\partial_\sIR F_{\sUV,\sIR}[\varphi],\psi\rangle
 =
 -\langle \rD F_{\sUV,\sIR}[\varphi,\dot G_\sIR\ast F_{\sUV,\sIR}[\varphi]],\psi\rangle
\end{equation}
with the boundary condition $F_{\sUV,0}[\varphi]=F_\sUV[\varphi]$, where $\varphi,\psi\in C^\infty_\rc(\bM)$. The pairing between a distribution $V$ and a test function $\psi$ is denoted by $\langle V,\psi\rangle$ and $\langle\rD V[\varphi,\zeta],\psi\rangle$ is the derivative of the functional $C^\infty_\rc(\bM)\ni\varphi\mapsto \langle V[\varphi],\psi\rangle\in\bR$ along $\zeta\in C^\infty_\rc(\bM)$. Note that the boundary condition involves the force $F_\sUV[\varphi]$, which is a local functional of $\varphi$.

\begin{rem}\label{rem:intro_flow_solution}
Let $f_{\sUV,T}:=F_{\sUV,T}[0]$. For every $\sIR\in[0,T]$ we have
\begin{equation}\label{eq:intro_stationary_relation}
 f_{\sUV,T} = F_{\sUV,\sIR}[(G_\sIR-G_T)\ast f_{\sUV,T}].
\end{equation}
This identity follows from the fact that
\begin{equation}
 g_{\sUV,\sIR}=f_{\sUV,T} - F_{\sUV,\sIR}[(G_\sIR-G_T)\ast f_{\sUV,T}]
\end{equation}
satisfies the following linear ODE
\begin{equation}
 \partial_\sIR g_{\sUV,\sIR}
 =
 -\rD F_{\sUV,\sIR}[(G_\sIR-G_T)\ast f_{\sUV,T},\dot G_\sIR\ast g_{\sUV,\sIR}]
\end{equation}
with the trivial boundary condition $g_{\sUV,T}=0$. Note that Eq.~\eqref{eq:intro_stationary_relation} with $\sIR=0$ implies that
\begin{equation}
 \varPsi_\sUV := (G-G_T)\ast F_{\sUV,T}[0]
\end{equation}
is a solution of Eq.~\eqref{eq:intro_stationary_mild}.
\end{rem}

Thus, in order to solve Eq.~\eqref{eq:intro_stationary_mild} for any $\sUV\in(0,1]$ it is enough to construct the effective force $F_{\sUV,\sIR}[\varphi]$. Similarly, to prove the existence of the limit $\sUV\searrow0$ of the solution it is enough to investigate the limit $\sUV\searrow0$ of the effective force $F_{\sUV,\sIR}[\varphi]$. In what follows, we assume that $T=1$ and define the effective force $F_{\sUV,\sIR}[\varphi]$ for $\sUV\in(0,1]$, $\sIR\in[0,1]$.

The starting point of the construction of the effective force is the following formal ansatz
\begin{equation}\label{eq:intro_ansatz}
 \langle F_{\sUV,\sIR}[\varphi],\psi\rangle
 :=\sum_{i=0}^\infty \sum_{m=0}^\infty \lambda^i\,\langle F^{i,m}_{\sUV,\sIR},\psi\otimes\varphi^{\otimes m}\rangle,
\end{equation}
where $\lambda\in\bR$ is the parameter that appears in Eq.~\eqref{eq:force_macro} defining the force $F_\sUV[\varphi]$. The distributions $F^{i,m}_{\sUV,\sIR}\in\sD'(\bM\times\bM^m)$, $i,m\in\bN_0$, are called the effective force coefficients. By definition~$\langle F^{i,m}_{\sUV,\sIR},\psi\otimes\varphi_1\otimes\ldots\otimes\varphi_m\rangle$ is invariant under permutations of the test functions $\varphi_1,\ldots,\varphi_m\in\ C^\infty_\rc(\bM)$. The force $F_{\sUV,0}[\varphi]=F_{\sUV}[\varphi]$ is defined by Eq.~\eqref{eq:force_macro} in terms of the force coefficients $f^{i,m,a}_\sUV$, $(i,m,a)\in\bar\frI$, satisfying the conditions specified in Assumption~\ref{ass:renormalization_conditions}. Hence, the force coefficients $F^{i,m}_{\sUV,0}=F^{i,m}_{\sUV}\in\sD'(\bM\times\bM^m)$ are supported on the diagonal $\{(x,\ldots,x)\in\bM^{1+m}\,|\,x\in\bM\}$ and only finitely many of them are not identically equal to zero. It also turns out that $F^{i,m}_{\sUV,\sIR}=0$ if $m>i m_\flat$, where $m_\flat\in\bN_+$ is such that all of the force coefficients $F^{i,m}_\sUV$ vanish identically whenever $m>m_\flat$. Moreover, $F^{0,0}_{\sUV,\sIR}$ coincides with the noise $\varXi_\sUV$. The flow equation~\eqref{eq:intro_flow_eq} for $F_{\sUV,\sIR}[\varphi]$ formally implies that the effective force coefficients $F^{i,m}_{\sUV,\sIR}$ satisfy the following flow equation
\begin{multline}\label{eq:intro_flow_eq_i_m}
 \langle \partial_\sIR^{\phantom{i}} F^{i,m}_{\sUV,\sIR}\,,
 \,\psi\otimes\varphi^{\otimes m}\rangle
 \\
 =
 -\sum_{j=1}^i\sum_{k=0}^m
 \,(k+1)
 ~\big\langle F^{j,k+1}_{\sUV,\sIR}\otimes F^{i-j,m-k}_{\sUV,\sIR},
 \psi\otimes \varphi^{\otimes k}
 \otimes \fV\dot G_\sIR
 \otimes \varphi^{\otimes(m-k)}
 \big\rangle
\end{multline}
together with the boundary condition $F^{i,m}_{\sUV,0}=F^{i,m}_{\sUV}$, where $\fV\dot G_\sIR\in C^\infty(\bM\times\bM)$ is defined by $\fV\dot G_\sIR(x,y):=\dot G_\sIR(x-y)$.

The basic idea behind the flow equation approach is a recursive construction of the effective force coefficients $F^{i,m}_{\sUV,\sIR}$. We set $F^{0,0}_{\sUV,\sIR}=\varXi_\sUV$ and $F^{i,m}_{\sUV,\sIR}\equiv 0$ if $m>i m_\flat$. Assuming that all $F^{i,m}_{\sUV,\sIR}$ with $i<i_\circ$, or $i=i_\circ$ and $m>m_\circ$ have been constructed we define $\partial_\sIR F^{i,m}_{\sUV,\sIR}$ with $i=i_\circ$ and $m=m_\circ$ with the use of the flow equation~\eqref{eq:intro_flow_eq_i_m}. Subsequently, $F^{i,m}_{\sUV,\sIR}$ is defined by
\begin{equation}
 F^{i,m}_{\sUV,\sIR} = F^{i,m}_{\sUV} + \int_0^\sIR \partial_\uIR F^{i,m}_{\sUV,\uIR}\,\rd\uIR.
\end{equation}
Using this procedure one can define all the effective force coefficients $F^{i,m}_{\sUV,\sIR}$ for arbitrary $\sUV\in(0,1]$, $\sIR\in[0,1]$. It is also possible to prove the absolute convergence of the series on the RHS of Eq.~\eqref{eq:intro_ansatz} and construct the solution of Eq.~\eqref{eq:intro_stationary_mild} under the assumption that $\lambda$ is sufficiently small and the noise $\varXi_\sUV$ is a periodic in space and time. However, this is not the aim of the present work. Instead of looking for global stationary solutions, we investigate the initial value problem. We give a non-perturbative construction of a solution for arbitrary $\lambda\in\bR$ under the assumptions that the noise $\varXi_\sUV$ is periodic in space. Our analysis does not exclude a finite-time blow-up of the solution.

\begin{rem}\label{rem:discrete_RG}
Let $\sUV\in(0,1]$, $\sIR\in[0,1]$. The effective force coefficients $F^{i,m}_{\sUV,\sIR}$ can be alternatively defined recursively by the following equation
\begin{equation}\label{eq:effective_force_force}
 F_{\sUV,\sIR}[\varphi] = F_{\sUV}[\varphi+(G-G_\sIR)\ast F_{\sUV,\sIR}[\varphi]],
\end{equation}
where as above the effective force functional $F_{\sUV,\sIR}[\varphi]$ coincides with the formal power series~\eqref{eq:intro_ansatz}. More generally, for $\sUV\in(0,1]$, $\uIR,\sIR\in[0,1]$ it holds
\begin{equation}\label{eq:effective_force_property_idea}
 F_{\sUV,\sIR}[\varphi] = F_{\sUV,\uIR}[\varphi+(G_\uIR-G_\sIR)\ast F_{\sUV,\sIR}[\varphi]].
\end{equation}
Note that Eq.~\eqref{eq:effective_force_force} is a special case of Eq.~\eqref{eq:effective_force_property_idea} for $\uIR=0$. To see that Eq.~\eqref{eq:effective_force_property_idea} follows from the flow equation~\eqref{eq:intro_flow_eq_i_m} denote by $H_{\sUV,\uIR,\sIR}[\varphi]$ the difference between the LHS and RHS of Eq.~\eqref{eq:effective_force_property_idea}. Then
\begin{equation}
\partial_\sIR H_{\sUV,\uIR,\sIR}[\varphi] = -\rD H_{\sUV,\uIR,\sIR}[\varphi,\dot G_\sIR\ast F_{\sUV,\sIR}[\varphi]].
\end{equation}
Using the above formula we prove by ascending induction on $i\in\bN_0$ and descending induction on $m\in\bN_0$ that the coefficients $H_{\sUV,\uIR,\sIR}^{i,m}$ of the functional $H_{\sUV,\uIR,\sIR}[\varphi]$ defined by an analog of Eq.~\eqref{eq:intro_ansatz} vanish identically. We use the fact that $H_{\sUV,\uIR,\sIR}^{i,m}=0$ unless $m\leq i m_\flat$. Let us mention that in certain situations one can make sense of Eq.~\eqref{eq:effective_force_property_idea} for $\uIR\leq\sIR$ non-perturbatively.
\end{rem}

\subsection{Renormalization problem}\label{sec:intro_renormalization}

Now let us describe the technique we use to bound the effective force coefficients $F^{i,m}_{\sUV,\sIR}\in\sD'(\bM\times\bM^{m})$ uniformly in $\sUV,\sIR\in(0,1]$ and prove the existence of the limit $\lim_{\sUV\searrow0}F^{i,m}_{\sUV,\sIR}$. This step of the construction is very subtle and requires an appropriate fine-tuning of the coefficients $\microf^{i,m,a}_{\sUV}$, $(i,m,a)\in\bar\frI$, appearing in the expression~\eqref{eq:force_micro} for the microscopic force $\microF_\sUV[\varphi]$ (or equivalently, the coefficients $f^{i,m,a}_{\sUV}$, $(i,m,a)\in\bar\frI$, appearing in the expression~\eqref{eq:force_macro} for the force $F_\sUV[\varphi]$). This is the so-called renormalization problem.

\begin{dfn}\label{dfn:st_multi}
For a spatial multi-index $a\in\bar\frM:=\bN_0^{\rdim}$ we set $|a|:=a^1+\ldots+a^\rdim$. For a spacetime multi-index $a\in\frM:=\bN_0^{1+\rdim}$ we set
\begin{equation}
 |a|:=\mathring a + |\bar a|,
 \qquad
 [a]:=\sigma \mathring a + |\bar a|,
 \qquad
 a=(\mathring a,\bar a)\in\bN_0\times\bN_0^\rdim.
\end{equation}
We identify $\bar\frM$ with a subset of $\frM$ consisting of $a\in\frM$ such that $\mathring a=0$. Let $\sigma_\diamond:=\sigma$ if $\sigma\in2\bN_+$ and $\sigma_\diamond:=\ceil{\sigma}-1$ otherwise. By definition $\frM_\sigma$ consists of $a\in\frM$ such that $|a|\leq\sigma_\diamond$ and $\bar\frM_\sigma$ consists of $a\in\bar\frM$ such that $|a|<\sigma$.
\end{dfn}

\begin{dfn}\label{dfn:relevant_irrelevant}
For for $m\in\bN_+$ and $a=(a_1,\ldots,a_m)\in\frM^m$ we set
\begin{equation}
 |a|:=|a_1|+\ldots+|a_m|,
 \qquad
 [a]:=[a_1] +\ldots+[a_m].
\end{equation}
We also set $\frM^0\equiv\{0\}$. The set $\frI_0$ consists of all triples $(i,m,a)$ such that $i\in\bN_0$, $m\in\bN_0$ and $a=(a_1,\ldots,a_m)\in\frM_\sigma^m$. For $(i,m,a)\in\frI_0$ we define
\begin{equation}
 \varrho(i,m,a) :=
 -\dim(\varXi)
 + m\, \dim(\varPhi)
 + i\, \dim(\lambda)
 + [a] \in \bR.
\end{equation}
Furthermore, we define
\begin{equation}
 \frI^-_0:=\{(i,m,a)\in\frI_0
 \,|\,\varrho(i,m,a)\leq 0\},
 \quad
 \frI^+:=\{(i,m,a)\in\frI_0
 \,|\,\varrho(i,m,a)> 0\}.
\end{equation}
We also define $\frI:=\frI_0\setminus\{0,0,0\}$ and $\frI^-:=\frI_0^-\setminus\{0,0,0\}$. The set $\frI$ is called the set of indices and $\frI^-$ and $\frI^+$ are called the sets of relevant and irrelevant indices, respectively. The objects indexed by indices belonging $\frI^-$ and $\frI^+$ are called relevant and irrelevant, respectively.
\end{dfn}

\begin{rem}
Note that $\varrho(i,m,a)$ introduced above generalizes $\varrho(i,m,a)$ introduced in Def.~\ref{dfn:bar_frI}. Observe also that $\frI^-=\bar\frI^-$ and $\bar\frI^+\subsetneq\frI^+$, where the sets $\bar\frI^\mp$ were introduced in Def.~\ref{dfn:bar_frI}.
\end{rem}

\begin{dfn}\label{dfn:im}
We introduce the following notation:
\begin{itemize}
 \item $i_\diamond$ is the largest integer such that $\varrho(i_\diamond,0,0)\leq0$,
 \item $m_\diamond$ is the largest integer such that $\varrho(1,m_\diamond,0)\leq0$,
 \item $i_\triangleright$ is the largest integer such that $\varrho(i_\triangleright,0,0)+\sigma\leq0$,
 \item $i_\flat$ is the smallest integer such that $i_\flat\geq i_\diamond$ and $f^{i,m,a}_\sUV=0$ for all $i>i_\flat$,
 \item $m_\flat$ is the smallest integer such that $m_\flat\geq m_\diamond$ and $f^{i,m,a}_\sUV= 0$ for all $m>m_\flat$,
 \item $i_\dagger:=i_\flat+i_\triangleright m_\flat$.
\end{itemize}
\end{dfn}
\begin{rem}
We have $i_\diamond,m_\diamond,i_\flat,m_\flat,i_\dagger\in\bN_+$ and $i_\triangleright\in\{1,\ldots,i_\diamond\}$.
\end{rem}
\begin{dfn}\label{dfn:polynomials}
For $m\in\bN_+$ and a list of multi-indices $a=(a_1,\ldots,a_m)\in\frM^m$ the polynomial $\cX^{m,a}\in C^\infty(\bM\times\bM^m)$ is given by \begin{equation}
 \cX^{m,a}(x;x_1,\ldots,x_m):=(x-x_1)^{a_1}\ldots(x-x_m)^{a_m}.
\end{equation}
\end{dfn}
The effective force coefficients $F^{i,m,a}_{\sUV,\sIR}\in\sD'(\bM\times\bM^m)$ and $f^{i,m,a}_{\sUV,\sIR}\in\sD'(\bM)$ are defined by Eq.~\eqref{eq:intro_eff_force_coeff}. We have $F^{0,0}_{\sUV,\sIR}=F^{0,0,0}_{\sUV,\sIR}=f^{0,0,0}_{\sUV,\sIR}=\varXi_\sUV$. We note that $F^{i,m,a}_{\sUV,\sIR}$ satisfies a certain flow equation that follows from the flow equation~\eqref{eq:intro_flow_eq_i_m}. For any sufficiently small $\varepsilon\in(0,1]$ we prove inductively that
\begin{gather}\label{eq:intro_bound_F}
 \|K_{\sUV,\sIR}^{m;\oo}\ast F^{i,m,a}_{\sUV,\sIR}\|_{\cV^m} \lesssim [\sUV\vee\sIR]^{\varrho_\varepsilon(i,m,a)},
 \\\label{eq:intro_bound_F'}
 \|K_{\sUV,\sIR}^{m;\oo}\ast\partial_\sIR F^{i,m,a}_{\sUV,\sIR}\|_{\cV^m} \lesssim[\sIR]^{\varepsilon-\sigma}\, [\sUV\vee\sIR]^{\varrho_\varepsilon(i,m,a)-\varepsilon},
\end{gather}
uniformly in $\sUV,\sIR\in(0,1]$, where
\begin{equation}
 \varrho_\varepsilon(i,m,a) :=
 -\dim(\varXi)-\varepsilon
 + m\, (\dim(\varPhi)+\varepsilon)
 + i\, (\dim(\lambda)-\varepsilon m_\flat)
 + [a]
\end{equation}
and $[a]$ denotes a degree of the list of multi-indices $a$. The above bounds control only the convolution of the effective force coefficients $F^{i,m,a}_{\sUV,\sIR}$ with some regularizing kernel $K_{\sUV,\sIR}^{m;\oo}$ with the characteristic length scale $[\sUV\vee\sIR]$. By assumption, $\dim(\lambda)>0$ and $\dim(\varPhi)>0$. Hence, for sufficiently small $\varepsilon>0$ there are only finitely many $i,m\in\bN_0$ and $a\in\frM^m$ such that $\varrho_\varepsilon(i,m,a)<0$. If $i,m\in\bN_0$ and $a\in\frM^m$ are such that $\varrho_\varepsilon(i,m,a)<0$, then $F^{i,m,a}_{\sUV,\sIR}$ and $f^{i,m,a}_{\sUV,\sIR}$ are called relevant. Otherwise, $F^{i,m,a}_{\sUV,\sIR}$ and $f^{i,m,a}_{\sUV,\sIR}$ are called irrelevant. Let us also mention that given $i,m\in\bN_0$ there always exists $\ro\in\bN_+$ such that $\varrho_\varepsilon(i,m,a)>0$ for all $|a|=\ro$.

Now we outline the proof of the bounds~\eqref{eq:intro_bound_F} and~\eqref{eq:intro_bound_F'}. We follow the recursive construction of the effective force coefficients presented in the previous section. We use the fact that $F^{i,m,a}_{\sUV,\sIR}= 0$ if $m>i m_\flat$ and assume the relevant coefficients $f^{i,m,a}_{\sUV,\sIR}\in\sD'(\bM)$ almost surely satisfy certain bounds of the form
\begin{equation}
 \|K^{\ast\oo}_{\sIR}\ast f^{i,m,a}_{\sUV,\sIR}\|\lesssim [\sUV\vee\sIR]^{\varrho_\varepsilon(i,m,a)}.
\end{equation}
These bounds for the relevant coefficients $f^{i,m,a}_{\sUV,\sIR}$ are proved using probabilistic arguments. We describe this step of the construction later. Assuming that the bound~\eqref{eq:intro_bound_F} holds for all $F^{i,m,a}_{\sUV,\sIR}$ such that $i<i_\circ$, or $i=i_\circ$ and $m>m_\circ$ we prove the bound~\eqref{eq:intro_bound_F'} for $\partial_\sIR F^{i,m,a}_{\sUV,\sIR}$ with $i=i_\circ$ and $m=m_\circ$. To this end, we use a certain flow equation for $F^{i,m,a}_{\sUV,\sIR}$. The bound for irrelevant $F^{i,m,a}_{\sUV,\sIR}$ follows then from the identity
\begin{equation}
 F^{i,m,a}_{\sUV,\sIR} = F^{i,m,a}_{\sUV,0} + \int_0^\sIR \partial_\uIR F^{i,m,a}_{\sUV,\uIR}\,\rd\uIR.
\end{equation}
Here we use the fact that
\begin{equation}
 \int_0^\sIR [\uIR]^{\varepsilon-\sigma}\, [\sUV\vee\uIR]^{\rho-\varepsilon}\,\rd\uIR\lesssim [\sUV]^\rho \vee [\sIR]^\rho,
\end{equation}
for any $\rho\neq 0$ and $\varepsilon\in(0,1]$. Note that we obtain the correct bound for $F^{i,m,a}_{\sUV,\sIR}$ only if $\rho=\varrho_\varepsilon(i,m,a)>0$. To prove the bound in the case $\varrho_\varepsilon(i,m,a)<0$ we use instead the following observation. Given $\ro\in\bN_+$ and $|a|<\ro$ there exists a map $\fX_\ro^a$ with good continuity properties that takes $f^{i,m,b}_{\sUV,\sIR}$ with $|b|<\ro$ and $F^{i,m,b}_{\sUV,\sIR}$ with $|b|=\ro$ and returns $F^{i,m,a}_{\sUV,\sIR}$. The existence of such a map follows from the Taylor theorem. We have
\begin{equation}
F^{i,m,a}_{\sUV,\sIR} = \fX_\ro^a(f^{i,m,b}_{\sUV,\sIR},F^{i,m,b}_{\sUV,\sIR}).
\end{equation}
The parameter $\ro\in\bN_+$ is chosen in such a way that the arguments of the map $\fX_\ro^a$ above involve only the effective force coefficients for which the bounds have already been proved. Finally we note that $F^{i,m}_{\sUV,\sIR}=F^{i,m,0}_{\sUV,\sIR}$.

Using the technique described above one proves uniform bounds for the coefficients $F^{i,m}_{\sUV,\sIR}$, the convergence as $\sUV\searrow0$ as well as the convergence of the series~\eqref{eq:intro_ansatz} for sufficiently small random $\lambda$. This allows to construct the solution of Eq.~\eqref{eq:intro_stationary_mild} and prove the existence of the limit $\sUV\searrow0$ under the assumption that noise $\varXi_\sUV$ is spacetime periodic and $\lambda$ is sufficiently small. The solution of the initial value problem for any $\lambda\in\bR$, which is the aim of this work, requires additional analysis presented in the next section.

The crucial input necessary for the application of the above technique are uniform bounds for the relevant coefficients $f^{i,m,a}_{\sUV,\sIR}\in\sD'(\bM)$. We prove such bounds using probability. Recall that $f^{i,m,a}_{\sUV,\sIR}$ is defined in terms of $F^{i,m,a}_{\sUV,\sIR}$. In order to bound $f^{i,m,a}_{\sUV,\sIR}$ we first bound the cumulants of $\partial_\sIR^s F^{i,m,a}_{\sUV,\sIR}$, $s\in\{0,1\}$, and subsequently use a Kolmogorov type argument. We also establish the existence of the limit $\lim_{\sUV\searrow0} f^{i,m,a}_{\sUV,\sIR}$. To this end, we use a diagonal argument adapted from~\cite{hairer2017clt}. We introduce a noise $\varXi_{\uv;\sUV}$ obtained by convolving $\varXi_\sUV$ with some regularizing kernel at scale $\uv\in(0,1]$. Next, we construct the force $F_{\uv;\sUV}[\varphi]$ and the effective force $F_{\uv;\sUV,\sIR}[\varphi]$ as well as the coefficients $F^{i,m,a}_{\uv;\sUV,\sIR}$ and $f^{i,m,a}_{\uv;\sUV,\sIR}$ corresponding to the regularized noise $\varXi_{\uv;\sUV}$. The existence of the limit $\lim_{\sUV\searrow0} f^{i,m,a}_{\uv;\sUV,\sIR}$ follows from the existence of the limit $\lim_{\sUV\searrow0}\varXi_{\uv;\sUV}$ for any $\uv\in(0,1]$ and uniform bounds for cumulants of $\partial^s_\sIR\partial^r_\uv F^{i,m,a}_{\uv;\sUV,\sIR}$, $s,r\in\{0,1\}$.

The proof of the bounds for the cumulants of the effective force coefficients is based on a certain flow equation for the cumulants that follows from the flow equation~\eqref{eq:intro_flow_eq_i_m}. As in the case of the effective force coefficients the cumulants can be divided into the relevant and irrelevant ones. The bounds for the irrelevant cumulants are proved inductively by integrating the flow equation with the boundary condition at $\sIR=0$. The only relevant cumulants are the expected values of the relevant coefficients $F^{i,m,a}_{\uv;\sUV,\sIR}$. In order to bound such expectations we first bound the expectations of the relevant $f^{i,m,a}_{\uv;\sUV,\sIR}$. Subsequently, we use the Taylor theorem and the bounds for the irrelevant cumulants. More specifically, we prove the following uniform bounds
\begin{equation}\label{eq:intro_bound_mean_f}
 |\llangle \partial_\sIR^s\partial_\uv^r f^{i,m,a}_{\uv;\sUV,\sIR}\rrangle|
 \lesssim  [\uv]^{(\varepsilon-\sigma)r}[\sIR]^{(\varepsilon-\sigma)s}[\uv\vee\sUV\vee\sIR]^{\varrho_\varepsilon(i,m,a)-\varepsilon s}
\end{equation}
where $s,r\in\{0,1\}$, $\varepsilon\in(0,1]$ is sufficiently small and $\uv,\sUV,\sIR\in(0,1]$. Note that the expected value of $\partial_\sIR^s\partial_\uv^r f^{i,m,a}_{\uv;\sUV,\sIR}(x)$ does not depend on $x\in\bM$ by the assumed stationarity of the noise $\varXi_\sUV$. The above bound with $s=1$ follows from the induction hypothesis and the flow equation. In order to prove the bound for $s=0$ and $\varrho_\varepsilon(i,m,a)<0$ we use the identity
\begin{equation}\label{eq:intro_relevant_f}
 \llangle \partial_\uv^r f^{i,m,a}_{\uv;\sUV,\sIR}\rrangle =
 \llangle \partial_\uv^r f^{i,m,a}_{\uv;\sUV,1}\rrangle
 -\int_\sIR^1
 \llangle \partial_\uIR\partial_\uv^r f^{i,m,a}_{\uv;\sUV,\uIR}\rrangle\,\rd\uIR.
\end{equation}
with the following boundary condition
\begin{equation}
 \llangle f^{i,m,a}_{\uv;\sUV,1}\rrangle = \mathfrak{f}^{i,m,a}_\sUV,
 \qquad
 \llangle \partial_\uv f^{i,m,a}_{\uv;\sUV,1}\rrangle=0.
\end{equation}
Note that for $\uv=0$ the above boundary condition coincides with the renormalization condition expressed in Assumption~\ref{ass:renormalization_conditions}. The bound~\eqref{eq:intro_bound_mean_f} with $s=0$ follows now from the following estimate
\begin{equation}
 \int_\sIR^1 [\sIR]^{\varepsilon-\sigma}\, [\sUV\vee\sIR]^{\rho-\varepsilon}\,\rd\uIR\lesssim 1 \vee [\sUV\vee\sIR]^\rho
\end{equation}
valid for any $\rho\neq0$ and $\varepsilon\in(0,1]$ and the fact that $1 \vee [\sUV\vee\sIR]^\rho\leq [\sUV\vee\sIR]^\rho$ for $\rho<0$ and $\sUV,\sIR\in(0,1]$. Finally, we note that using Eq.~\eqref{eq:intro_relevant_f} with $r=0$, $\uv=0$ and $\sIR=0$ one can compute the relevant force coefficients $f^{i,m,a}_{\sUV\phantom{0}}=f^{i,m,a}_{\sUV,0}=f^{i,m,a}_{0;\sUV,0}$. The above choice of the values of these coefficients guarantees the existence of the limit $\lim_{\sUV\searrow0} f^{i,m,a}_{\sUV,\sIR}$ for any $\sIR\in(0,1]$. We also obtain $|f^{i,m,a}_\sUV|\lesssim [\sUV]^{\varrho_\varepsilon(i,m,a)}$. Since for relevant $f^{i,m,a}_{\sUV,\sIR}$ we have $\varrho_\varepsilon(i,m,a)<0$ the limit $\lim_{\sUV\searrow0} f^{i,m,a}_{\sUV}$ need not exist. We stress that the values of the relevant coefficients $f^{i,m,a}_{\sUV}$ for any $\sUV\in(0,1]$ are uniquely determined by the flow equation and the renormalization conditions expressed in terms of the effective force coefficients $f^{i,m,a}_{\sUV,\sIR}$ with $\sIR=1$.

\begin{rem}
In the approach based on the theory of regularity structures the renormalization problem is tackled at the level of the stochastic estimates of a model and a very general solution of this problem (applicable to sub-critical singular SPDEs with local differential operators) was given in~\cite{chandra2016bphz}. Recall that a model consists of certain multi-linear functionals of the noise that can be represented by tree diagrams. The stochastic estimates of a model involve expectations of products of trees and are represented by diagrams with loops. The approach to the renormalization problem taken in~\cite{chandra2016bphz} exploits this diagrammatical representation. One of the main obstacles in this approach is the problem of overlapping divergences which appears when a given diagram have two or more divergent sub-diagrams which are neither nested nor disjoint. The solution of above problem given in~\cite{chandra2016bphz} uses the ideas from the BPHZ renormalization technique in perturbative QFT and involves a decomposition of the integration domain into the Hepp sectors and the application of a variant of the Zimmerman forest formula.

In contrast, in the flow equation approach the solution of the renormalization problem does not rely on any diagrammatical representation. As explained above, we construct the effective force coefficients recursively using the flow equation. To prove bounds for the cumulants of the effective force coefficients we impose appropriate boundary conditions when solving the flow equation. In particular, the boundary conditions for the expectations of the relevant effective force coefficients coincide with the BPHZ renormalization conditions. The correct values of the counterterms are determined automatically by the flow equation and the boundary conditions. The problem of overlapping divergences does not exist in our approach. Note that the decomposition into different spatial scales (like the Hepp sectors) needed to solve the problem of overlapping divergences in the BPHZ approach is built into the foundation of the flow equation approach. Let us also mention that our technique does not require any positive renormalization corresponding to a recentering procedure. In particular, we do not need to define a structure group.
\end{rem}

\subsection{Initial value problem}\label{sec:intro_initial_value}

In this section we show how to use the ideas described in the previous two sections to solve the initial value problem for the parabolic equation~\eqref{eq:intro_spde_macro}. Even though this equation is not of the form~\eqref{eq:intro_stationary_mild}, as we will argue below, the renormalization problem for the mild form of this equation is the same as the renormalization problem for Eq.~\eqref{eq:intro_stationary_mild}, which was discussed in the previous section.

In what follows, the Dirac delta at $t\in\bR$ is denoted by $\delta_t\in\sD'(\bR)$ and the characteristic function of an interval $I\subset\bR$ is denoted by $1_I$. We also set $(\mathtt{w}f)(\mathring x,\bar x)=\mathtt{w}(\mathring x)f(\mathring x,\bar x)$ for any $\mathtt{w}\in L^\infty(\bR)$ and $f\in L^\infty(\bM)$. Let $\sUV\in(0,1]$. We would like to construct the solution of the following equation
\begin{equation}\label{eq:init_idea_main_eq}
 \varPhi_\sUV = G\ast (1_{(0,\infty)} F_\sUV[\varPhi_\sUV] +\delta_0\otimes\phi_\sUV)
\end{equation}
and study its limit $\sUV\searrow0$. Because of the singularity of $\lim_{\sUV\searrow0}\varXi_\sUV$ and the necessity to renormalize the equation the solution theory developed in this paper requires that in the limit $\sUV\searrow0$ the initial condition $\phi_\sUV$ is a small perturbation (in the sense specified below) of the initial condition that would lead to a stationary solution. We note that a similar assumption was made in~\cite{bruned2021renormalising}. More specifically, we suppose that the initial data $\phi_\sUV$ is of the following form $\phi_\sUV=\phi^\triangleright_\sUV+\phi^\vartriangle_\sUV$.

\label{par:initial_data}
The process $\phi_\sUV^\triangleright$ is given by $\phi_\sUV^\triangleright:=\varPhi_\sUV^\triangleright(0,\Cdot)$, where
\begin{equation}
 \varPhi_\sUV^\triangleright:=
 (G-G_1)\ast f^\triangleright_\sUV=G\ast (f^\triangleright_{\sUV}- \check f^\triangleright_\sUV ),
 \qquad
 f^\triangleright_{\sUV}:=\sum_{i=0}^{i_\triangleright} F^{i,0}_{\sUV,1},
 \qquad
 \check f^\triangleright_\sUV :=
 \fQ G_1\ast f^\triangleright_{\sUV}
\end{equation}
and the effective force coefficients $F^{i,m}_{\sUV,\sIR}$ were constructed in the previous section. By definition $i_\triangleright\in\bN_0$ is the largest integer such that $\dim(\varPhi)\geq i_\triangleright \dim(\lambda)$. We show that $\lim_{\sUV\searrow0} f^\triangleright_\sUV$ exists in the space of distributions. Since the kernel $\fQ G_1$ is smooth the limit $\lim_{\sUV\searrow0}\check f^\triangleright_\sUV$ is a smooth function. Moreover, the limit $\lim_{\sUV\searrow0} \varPhi_\sUV^\triangleright =\varPhi^\triangleright_0$ exists in the parabolic Besov space $\sC^\alpha(\bH)$ for every $\alpha<-\dim(\varPhi)$.

\begin{dfn}\label{dfn:dim_phi}
The dimension of the regular part of the initial data, denoted by $\dim(\phi)$, is fixed  such that \mbox{$0<\dim(\phi)<\dim(\lambda)$} and if $\dim(\varPhi)<\dim(\lambda)$, then $\dim(\phi)=\dim(\varPhi)$.
\end{dfn}

\begin{rem}
This dimension need not have any physical interpretation.
\end{rem}

The process $\phi^\vartriangle_\sUV$ is such that the limit $\lim_{\sUV\searrow0}\phi^\vartriangle_\sUV=\phi_0^\vartriangle$ exists in the Besov space $\sC^\beta(\bT)$ for every $\beta<-\dim(\phi)$. We note that if $\dim(\lambda)>\dim(\varPhi)$, that is Eq.~\eqref{eq:init_idea_main_eq} is far from criticality, then $\dim(\phi)=\dim(\varPhi)$ and the initial data $\phi_0^\vartriangle$ is allowed to have the same regularity as the expected regularity of the solution. If $\dim(\lambda)\leq\dim(\varPhi)$, then $0<\dim(\phi)<\dim(\lambda)$. Consequently, if $\dim(\lambda)$ is small, that is  Eq.~\eqref{eq:init_idea_main_eq} is close to criticality, then the regularity of $\phi_0^\vartriangle$ is allowed to be only slightly negative.

In order to take advantage of our assumption about the initial data we introduce the following ansatz
\begin{equation}\label{eq:intro_idea_init_ansatz_breve}
 \varPhi_\sUV = \varPhi_\sUV^\triangleright + \breve\varPhi_\sUV
\end{equation}
and show that Eq.~\eqref{eq:init_idea_main_eq} is equivalent to a certain equation for $\breve\varPhi_\sUV$ with initial data coinciding with $\phi_\sUV^\vartriangle$. To this end, let us first note that the equality $\phi_\sUV^\triangleright=(G\ast (f^\triangleright_{\sUV} - \check f^\triangleright_\sUV))(0,\Cdot)$ and the semi-group property of the fractional heat kernel imply that
\begin{equation}
 G\ast(\delta_0\otimes \phi_\sUV^\triangleright) = G\ast1_{(-\infty,0]} (f^\triangleright_{\sUV} - \check f^\triangleright_\sUV)
\end{equation}
in the region $[0,\infty)\times\bR^\rdim\subset\bM$. Consequently, in that region region Eq.~\eqref{eq:init_idea_main_eq} is equivalent to the equation
\begin{equation}
 \varPhi_\sUV = G\ast (1_{(0,\infty)} F_\sUV[\varPhi_\sUV] + 1_{(-\infty,0]}(f^\triangleright_{\sUV} - \check f^\triangleright_\sUV) +\delta_0\otimes\phi_\sUV^\vartriangle).
\end{equation}
It is easy to see that \mbox{$\breve\varPhi_\sUV = \varPhi_\sUV - \varPhi_\sUV^\triangleright$} satisfies the following equation
\begin{equation}\label{eq:idea_eq_breve}
 \breve\varPhi_\sUV = G\ast (1_{(0,\infty)} (F_\sUV[\breve\varPhi_\sUV+\varPhi_\sUV^\triangleright] -f^\triangleright_{\sUV} + \check f^\triangleright_\sUV)+ \delta_0\otimes\phi_\sUV^\vartriangle).
\end{equation}
Our goal is to prove that the classical maximal solution $\breve\varPhi_\sUV$ of Eq.~\eqref{eq:idea_eq_breve} converges to some distribution $\breve\varPhi_0$ as $\sUV\searrow0$.

\begin{rem}
The decomposition~\eqref{eq:intro_idea_init_ansatz_breve} is a higher order variant of the so-called Da Prato-Debussche trick~\cite{daprato2003}. Here its purpose is to deal with the effect of the initial condition. A similar idea was used in~\cite{bruned2021renormalising}.
\end{rem}

The solution of Eq.~\eqref{eq:idea_eq_breve} up to a possible explosion time can be constructed by patching together solutions of the following equations
\begin{equation}\label{eq:idea_eq_breve_loc}
 \breve\varPhi_\sUV = (G-G_T)\ast (1_{(t,t+6i_\dagger)}  ( F_\sUV[\breve\varPhi+\varPhi_\sUV^\triangleright] -f^\triangleright_{\sUV} + \check f^\triangleright_\sUV) +\delta_t\otimes\phi_\sUV^\vartriangle)
\end{equation}
posed in the domains of the form $[t,t+T]\times\bR^\rdim\subset \bM$ for sufficiently small $T\in(0,1)$ and appropriate initial data $\phi^\vartriangle_\sUV$ at time $t\in[0,\infty)$. Recall that $\supp\,G\subset[0,\infty)\times\bR^\rdim$, $\supp\,G_T\subset(T,\infty)\times\bR^\rdim$ and the force $F_\sUV[\varphi]$ is a local functional that involves only spatial derivatives of $\varphi$. We stress that by abusing the notation we now set $\phi_\sUV^\vartriangle = \breve\varPhi_\sUV(t,\Cdot)$. We also note that one could replace the time interval $(t,t+6 i_\dagger)$ with any time interval containing $(t,t+1)$. It is convenient to choose $i_\dagger\in\bN_+$ as in Def.~\ref{dfn:im}. Since $\delta_t\otimes\phi_\sUV^\vartriangle\in\sD'(\bM)$ is not a function of time we introduce another ansatz
\begin{equation}\label{eq:intro_idea_init_ansatz_tilde}
 \breve\varPhi_\sUV
 =
 \tilde\varPhi_\sUV+\varPhi_\sUV^\vartriangle,
\end{equation}
where
\begin{equation}
 \varPhi_\sUV^\vartriangle:=(G-G_1)\ast
 f_\sUV^\vartriangle,
 \qquad
 f_\sUV^\vartriangle:=\delta_t\otimes\phi^\vartriangle_\sUV.
\end{equation}
One shows that in the region $[t,t+T]\times\bR^\rdim\subset \bM$ the function $\varPhi_\sUV^\vartriangle$ coincides with $(G-G_T)\ast f_\sUV^\vartriangle$ and Eq.~\eqref{eq:idea_eq_breve_loc} is equivalent to the following equation
\begin{equation}\label{eq:intro_tilde}
 \tilde\varPhi_\sUV = (G-G_T)\ast 1_{(t,t+6i_\dagger)}
 (F_\sUV[\tilde\varPhi_\sUV+\varPhi_\sUV^\triangleright+\varPhi_\sUV^\vartriangle]
 -f^\triangleright_{\sUV} + \check f^\triangleright_\sUV).
\end{equation}
We rewrite the above equation in the form
\begin{equation}\label{eq:intro_tilde2}
 \tilde\varPhi_\sUV = (G-G_T)\ast \tilde F_\sUV[\tilde \varPhi_\sUV]
\end{equation}
upon introducing the notation
\begin{equation}\label{eq:intro_idea_forces}
\begin{gathered}
 \tilde F_\sUV[\varphi] := 1_{(t,t+6i_\dagger)} \check F_\sUV[\varphi]
 =1_{(t,t+6i_\dagger)} (\hat F_{\sUV}[\varphi + \varPhi_\sUV^\vartriangle] + \check f_\sUV^\triangleright),
 \\
 \check F_{\sUV}[\varphi] := \hat F_{\sUV}[\varphi + \varPhi_\sUV^\vartriangle] - f_\sUV^\vartriangle + \check f_\sUV^\triangleright,
 \qquad
 \hat F_{\sUV}[\varphi]:=F_{\sUV}[\varphi+\varPhi^\triangleright_\sUV] - f_\sUV^\triangleright.
\end{gathered}
\end{equation}
In order to avoid confusion observe that the measure $1_{(t,t+6i_\dagger)}\delta_t$ vanishes identically. Hence, $1_{(t,t+6i_\dagger)}f_\sUV^\vartriangle=0$.

In order to solve Eq.~\eqref{eq:intro_tilde} in the domain $[t,t+T]\times\bR^\rdim$ with sufficiently small $T\in(0,1)$ we construct the effective force
\begin{equation}\label{eq:intro_ansatz_F_tilde}
 \langle \tilde F_{\sUV,\sIR}[\varphi],\psi\rangle
 :=\sum_{i=0}^\infty \sum_{m=0}^\infty \lambda^i\,\langle \tilde F^{i,m}_{\sUV,\sIR},\psi\otimes\varphi^{\otimes m}\rangle
\end{equation}
for $\sUV\in(0,1]$, $\sIR\in(0,T]$. By definition $\tilde F_{\sUV,\sIR}[\varphi]$ satisfies the flow equation~\eqref{eq:intro_flow_eq} together with the initial condition $\tilde F_{\sUV,0}[\varphi]=\tilde F_\sUV[\varphi]$. Using the technique described in Sec.~\ref{sec:intro_flow} one shows that
\begin{equation}
 \tilde\varPhi_\sUV:=(G-G_T)\ast \tilde F_{\sUV,T}[0] = \sum_{i=0}^\infty \lambda^i\,(G-G_T)\ast \tilde F^{i,0}_{\sUV,T}
\end{equation}
solves Eq.~\eqref{eq:intro_tilde}. The effective force coefficients $\tilde F^{i,m}_{\sUV,\sIR}$ are defined recursively as outlined in Sec.~\ref{sec:intro_flow} and satisfy the flow equation~\eqref{eq:intro_flow_eq_i_m}. In order to prove uniform bounds for $\tilde F^{i,m}_{\sUV,\sIR}$ and show the existence of the limit $\lim_{\sUV\searrow0}\tilde F^{i,m}_{\sUV,\sIR}$ we use the technique presented below. The solution of Eq.~\eqref{eq:idea_eq_breve} in the domain $[t,t+T]\times\bR^\rdim$ coincides with
\begin{equation}
 \breve\varPhi_\sUV
 =
 \varPhi^\vartriangle_\sUV+\tilde\varPhi_\sUV,
 \qquad
 \varPhi^\vartriangle_\sUV=(G-G_1)\ast(\delta_t\otimes\phi^\vartriangle_\sUV).
\end{equation}
The maximal solution of Eq.~\eqref{eq:idea_eq_breve} is constructed by patching together solutions in time intervals $[t,t+T]$. 

Finally, let us describe the technique we use to prove uniform bounds for $\tilde F^{i,m}_{\sUV,\sIR}$ and show the existence of the limit $\lim_{\sUV\searrow0}\tilde F^{i,m}_{\sUV,\sIR}$. We first study the effective force $F_{\sUV,\sIR}[\varphi]$ for the equation $\varPsi_\sUV = (G-G_1)\ast F_\sUV[\varPsi_\sUV]$. The above equation coincides with Eq.~\eqref{eq:intro_stationary_mild} with $T=1$ and the effective force coefficients $F^{i,m}_{\sUV,\sIR}$ are constructed and bounded uniformly in $\sUV\in(0,1]$, $\sIR\in[0,1]$ as described in Sec.~\ref{sec:intro_renormalization}. This step requires the renormalization of the relevant coefficients $f^{i,m,a}_\sUV$ of the force $F_\sUV[\varphi]$. Note that the effective force coefficients $F^{i,m}_{\sUV,\sIR}$ do not involve the initial condition. We call $F_{\sUV,\sIR}[\varphi]$ the effective force at stationarity and $F^{i,m}_{\sUV,\sIR}$ the effective force coefficients at stationarity.

\begin{rem}\label{rem:eff_force_shift}
We shall use below the following observation. Assume that a functional $F_{\sUV,\sIR}[\varphi]$ satisfies the flow equation~\eqref{eq:intro_flow_eq}, then for any $f_\sUV$ and $\varPsi_\sUV$ the functional $\tilde F_{\sUV,\sIR}[\varphi]=F_{\sUV,\sIR}[\varphi + (G_\sIR-G_1) \ast f_\sUV + \varPsi_\sUV] - f_\sUV$ also satisfies the flow equation~\eqref{eq:intro_flow_eq}.
\end{rem}

We introduce a new effective force $\hat F_{\sUV,\sIR}[\varphi]$ by the following relation between formal power series
\begin{equation}\label{eq:eff_force_hat}
 \hat F_{\sUV,\sIR}[\varphi]:=F_{\sUV,\sIR}[\varphi+(G_\sIR-G_1) \ast f^\triangleright_\sUV] - f_\sUV^\triangleright,
\end{equation}
where $f^\triangleright_{\sUV}=\sum_{i=0}^{i_\triangleright} F^{i,0}_{\sUV,1}$. The effective force $\hat F_{\sUV,\sIR}[\varphi]$ satisfies the flow equation~\eqref{eq:intro_flow_eq} with the boundary condition $\hat F_{\sUV,0}[\varphi]=\hat F_\sUV[\varphi]$, where $\hat F_\sUV[\varphi]$ was introduced in Eq.~\eqref{eq:intro_idea_forces}. Subsequently, we define another effective force formally given by
\begin{equation}\label{eq:eff_force_check}
 \check F_{\sUV,\sIR}[\varphi] = \hat F_{\sUV,\sIR}[\varphi + (G_\sIR-G_1)\ast
 f_\sUV^\vartriangle+(G-G_\sIR)\ast \check f_\sUV^\triangleright]-f_\sUV^\vartriangle + \check f_\sUV^\triangleright,
\end{equation}
where $\check f^\triangleright_{\sUV}=\fQ G_1\ast f^\triangleright_{\sUV}$ and $f^\vartriangle_{\sUV}=\delta_t\otimes\phi_\sUV^\vartriangle$. The effective force $\check F_{\sUV,\sIR}[\varphi]$ satisfies the flow equation~\eqref{eq:intro_flow_eq} with the boundary condition $\check F_{\sUV,0}[\varphi]=\check F_\sUV[\varphi]$, where $\check F_\sUV[\varphi]$ was introduced in Eq.~\eqref{eq:intro_idea_forces}.

Recall that we would like to construct and bound the effective force $\tilde F_{\sUV,\sIR}[\varphi]$. Using the flow equation~\eqref{eq:intro_flow_eq_i_m} one proves that
\begin{equation}\label{eq:intro_support_identity}
 \langle \tilde F^{i,m}_{\sUV,\sIR},\psi\otimes\varphi^{\otimes m}\rangle=
 \langle \check F^{i,m}_{\sUV,\sIR},\psi\otimes\varphi^{\otimes m}\rangle
\end{equation}
for any $\psi,\varphi\in C^\infty_\rc(\bM)$ such that $\supp\,\psi\subset (t+2i\sIR,t+6 i_\dagger)\times\bR^\rdim$. The above formula plays a crucial role in the construction since it provides a relation between the effective forces $\tilde F_{\sUV,\sIR}[\varphi]$ and $\check F_{\sUV,\sIR}[\varphi]$. Note that Eq.~\eqref{eq:intro_support_identity} holds true for $\sIR=0$ by Eq.~\eqref{eq:intro_idea_forces}. In order to prove Eq.~\eqref{eq:intro_support_identity} for $\sIR\in(0,1]$ one uses a simple inductive argument based on the following support property $\supp\,\dot G_\sIR\subset [0,2\sIR]\times\bR^\rdim$ of the scale decomposition of the fractional heat kernel.

The construction of the effective force $\tilde F^{i,m}_{\sUV,\sIR}$ in a small neighborhood of the complement of the domain $(t+2i\sIR,t+6 i_\dagger)\times\bR^\rdim\times\bM^m$ does not involve renormalization. Here we crucially use the fact that the coefficients $\tilde F^{i,m}_{\sUV,\sIR}$ are not too singular thanks to the Da Prato-Debussche trick, $\tilde F^{i,m}_{\sUV,\sIR}$ is supported in the domain $[t,t+6i_\dagger]\times\bR^\rdim\times\bM^m$ and the Lebesgue measure of $[t,t+6i_\dagger]\setminus (t+2i\sIR,t+6 i_\dagger)$ is proportional to $\sIR=[\sIR]^\sigma$.

\section{Statement of the result}\label{sec:result}

In this section we list the standing assumptions of the paper and give the precise statement of the main result. We present several concrete examples to which our result applies. We also discuss related results known in the literature.

\subsection{Assumptions}\label{sec:assumptions}

\begin{dfn}\label{dfn:spacetime_distance}
Let $\bar\bM=\bR^\rdim$, where $\rdim\in\bN_+$ is called the spatial dimension. We call \mbox{$\bM:=\bR\times\bar\bM$} the spacetime. We denote points in $\bM$ by $x=(\mathring x,\bar x)$, where $\bar x=(x^1,\ldots,x^\rdim)\in\bar\bM$. The coordinate $\mathring x\in\bR$ is called the time. For $\sUV\in(0,1]$ let $\bS(\sUV)=\bR/(2\pi/[\sUV]\,\bZ)$ be the circle of size $[\sUV]^{-1}=\sUV^{-1/\sigma}$, let $\bT(\sUV):=\bS(\sUV)^\rdim$ be the $\rdim$-dimensional torus of size $[\sUV]^{-1}$ and let $\bH(\sUV):=\bR\times\bT(\sUV)$. We also set $\bT(0)=\bar\bM$ and $\bH(0)=\bM$. The absolute value of $t\in\bR$ is denoted by $|t|$. For $\bar x=(x^1,\ldots,x^\rdim)\in\bar\bM$ we set  
\begin{equation}
 |\bar x|_{\bar\bM}\equiv |\bar x|:=\max_{k\in\{1,\ldots,\rdim\}}|x^k|,
 \qquad 
 |\bar x|_{\bT(\sUV)} := \min_{\bar y\in(2\pi/[\sUV]\,\bZ)^\rdim}|\bar x+\bar y|_{\bar\bM},
 \quad \sUV\in(0,1].
\end{equation}
For $x=(\mathring x,\bar x)\in\bR\times\bR^\rdim=\bM$ we set 
\begin{equation}
 |x|_\bM:=[\mathring x]\vee|\bar x|_{\bar\bM},
 \qquad
 |x|_{\bH(\sUV)}:=[\mathring x]\vee|\bar x|_{\bT(\sUV)},
 ~~ \sUV\in(0,1],
 \qquad
 [\mathring x]:=|\mathring x|^{1/\sigma}.
\end{equation}
We omit $\sUV$ in the notation if $\sUV=1$. Functions $\bH(\sUV)\to\bR$ are always identified with periodic in space functions $\bM\to\bR$. 
\end{dfn}

\begin{ass}\label{ass:noise}
The family of random fields $\microXi_\sUV$, $\sUV\in[0,1]$, defined on the spacetime \mbox{$\bM=\bR\times\bR^\rdim$} satisfies the following conditions.
\begin{enumerate}
\item[(A)] For any $\sUV\in[0,1]$ the field $\microXi_\sUV\in C^\infty(\bH(\sUV))$ is stationary and centered. In particular, if $\sUV\in(0,1]$, then $\microXi_\sUV$ is periodic in space with period $2\pi/[\sUV]$.
\item[(B)] It holds
\begin{equation}\label{eq:noise_assumptions_bound}
 \int_\bM \llangle\microXi_0(0)\microXi_0(x)\rrangle\,\rd x = 1,
 \qquad
 \sup_{\sUV \in [0,1]} \llangle |\partial^a\microXi_\sUV(0)|^n \rrangle < \infty
\end{equation}
for all $n\in\bN_+$ and multi-indices $a\in\frM$ such that $|a|\leq3$.
\item[(C)] For any $\sUV\in[0,1]$ and any compact sets $\cO_1$, $\cO_2 \subset \bH(\sUV)$ such that 
\begin{equation}
 \inf_{x_1 \in \cO_1,x_2\in\cO_2}|x_1 - x_2|_{\bH(\sUV)} \geq 1
\end{equation}
the \mbox{$\sigma$-algebras} generated by $\{\microXi_\sUV(x)\,|\, x\in \cO_1\}$ and $\{\microXi_\sUV(x)\,|\, x\in \cO_2\}$ are independent.
\item[(D)] For every $\sUV \in(0,1]$, there is a coupling of $\microXi_0$ and $\microXi_\sUV$ such that for every $T > 0$ and $R\in(0,\pi)$ it holds
\begin{equation}\label{eq:coupling}
\lim_{\sUV \searrow0} 
\sup_{\substack{|\mathring x| \leq T/\sUV\\
|\bar x| \leq R/[\sUV]}}
[\sUV]^{-2\dim(\varXi)}\,
\llangle |\microXi_0(\mathring x,\bar x)-\microXi_\sUV(\mathring x,\bar x)|^2\rrangle = 0\,.
\end{equation}
\end{enumerate}
\end{ass}

\begin{rem}
The assumption stated above is a trivial adaptation of the assumption formulated in~\cite{hairer2017clt}.
\end{rem}

\begin{rem}
Note that for any $\sUV\in(0,1]$ the field $\varXi_\sUV$ obtained by rescaling $\microXi_\sUV$ as in Eq.~\eqref{eq:noise_rescaled} is periodic in space with period $2\pi$ and $\varXi_\sUV\in C^\infty(\bH)$. Recall that $\varXi_0\in\sS'(\bH)$ denotes the white noise on~$\bH$. There is no relation between the fields $\varXi_0\in\sS'(\bH)$ and~$\microXi_0\in C^\infty(\bM)$.
\end{rem}

\begin{rem}
Let $\xi_0$ to be the white noise on $\bM$ and for $\sUV\in(0,1]$ let $\xi_\sUV$ be the periodic extension of $\xi_0$ restricted to $\bR\times [-\pi/[\sUV],\pi/[\sUV])^\rdim\subset\bM$. It is easy to see that for any function $M\in C^\infty_\rc(\bM)$ supported in the ball of radius $1/2$ centered at the origin such that $\int_\bM M(x)\,\rd x=1$ the family of random fields $\microXi_\sUV=M\ast\xi_\sUV$, $\sUV\in[0,1]$, satisfies the above assumption. Note that for any $\sUV\in(0,1]$ the law of the field $\varXi_\sUV$ obtained by rescaling $\microXi_\sUV$ coincides with the law of the field \mbox{$\varXi_\sUV = M_\sUV\ast\varXi_0$}, where $\varXi_0$ is the white noise on $\bH$ and $M_\sUV(\mathring x,\bar x):=[\sUV]^{-\rDim}M(\mathring x/\sUV,\bar x/[\sUV])$, $\rDim=\sigma+\rdim$.
\end{rem}

\begin{rem}
A non-trivial example of a family of random fields $\microXi_\sUV$, $\sUV\in[0,1]$, satisfying the above assumption is obtained as follows~\cite{hairer2017clt}. Let $\pi_0$ be the homogeneous Poisson process on $\bR\times\bM$ and for $\sUV\in(0,1]$ let $\pi_\sUV$ be the periodic extension of $\pi_0$ restricted to $\bR\times\bR\times [-\pi/[\sUV],\pi/[\sUV])^\rdim\subset\bR\times\bM$. For any functions $M\in C^\infty_\rc(\bR\times\bM)$ supported in the ball of radius $1/2$ centered at the origin such that $\int_{\bR\times\bM} M(s,x)\,\rd s\rd x=0$ and $\int_{\bR\times\bM\times\bM} M(s,x)M(s,y)\,\rd s\rd x\rd y=1$ the family of random fields $\microXi_\sUV(x)=\int_{\bR\times\bM} M(s,x-y)\,\pi_\sUV(\rd s\rd y)$, $\sUV\in[0,1]$, satisfies the above assumption.
\end{rem}

\begin{rem}\label{rem:noise_bound_law}
By the above assumption for any $n\in\bN_+$ and $a\in\frM$ such that \mbox{$|a|\leq 3$} the bound
\begin{equation}
 \sup_{x\in\bH}\llangle |\partial^a\varXi_\sUV(x)|^n\rrangle \lesssim [\sUV]^{-n\dim(\varXi)-n[a]}
\end{equation}
holds uniformly in $\sUV\in(0,1]$. Moreover, as proved in~\cite[Sec.~6]{hairer2017clt}, it holds $\lim_{\sUV\searrow0}\varXi_\sUV=\varXi_0$ in distribution in the parabolic Besov space $\sC^\alpha(\bH)$ for any $\alpha<-\dim(\varXi)$, where $\varXi_0$ is the white noise on $\bH=\bR\times\bT$.
\end{rem}

\begin{ass}\label{ass:initial}
The random fields $\phi^\vartriangle_\sUV$, $\sUV\in(0,1]$, defined on $\bT$ are smooth. For any $n\in\bN_+$ and any $a\in\bar\frM$ such that \mbox{$|a|\leq 2\floor{\sigma/2}+4$} the bound
\begin{equation}
 \sup_{\bar x\in\bT}\llangle |\partial^{a} \phi^\vartriangle_\sUV(\bar x)|^n \rrangle \lesssim [\sUV]^{-n\dim(\phi)-n|a|}
\end{equation}
holds uniformly in $\sUV\in(0,1]$. Moreover, there exists a random distribution $\phi^\vartriangle_0\in \sC^\beta(\bT)$ such that $\lim_{\sUV\searrow0}\phi^\vartriangle_\sUV=\phi^\vartriangle_0$ in law in the Besov space $\sC^\beta(\bT)$ for any $\beta<-\dim(\phi)$.
\end{ass}

\begin{rem}\label{rem:ass_joint_convergence}
If the random fields $\varXi_\sUV$ and $\phi_\sUV^\vartriangle$ are not independent, then we require in addition that
\begin{equation}\label{eq:ass_joint_limit}
 \lim_{\sUV\searrow0}\,(\varXi_\sUV,\phi_\sUV^\vartriangle) 
 =
 (\varXi_0,\phi_0^\vartriangle)
\end{equation}
in law in $\sC^\alpha(\bH)\times \sC^\beta(\bT)$ for any $\alpha<-\dim(\varXi)$ and any $\beta<-\dim(\phi)$.
\end{rem}

\begin{rem}
Since it is reasonable to expect that both the noise and the initial condition are smooth at the microscopic scale no effort was made to state optimal assumptions regarding differentiability.
\end{rem}

\begin{ass}\label{ass:renormalization_conditions}
The force coefficients $f^{i,m,a}_\sUV=f^{i,m,a}_{\sUV,0}$, $(i,m,a)\in\bar\frI$, satisfy the following conditions. We assume that for the irrelevant coefficients the following bounds
\begin{equation}\label{eq:ass_ren_cond_bound}
 |f^{i,m,a}_{\sUV,0}|\lesssim [\sUV]^{\varrho(i,m,a)-\varepsilon},
 \qquad
 (i,m,a)\in \bar\frI^+,
\end{equation}
hold uniformly in $\sUV\in(0,1]$. Moreover, we demand that there exist $i_\flat,m_\flat\in\bN_+$ such that $f^{i,m,a}_{\sUV,0}$ vanishes unless $i\in\{0,\ldots,i_\flat\}$, $m\in\{0,\ldots,m_\flat\}$ and $a\in\bar\frM^m$ is such that $|a|<\sigma$. Furthermore, we suppose that there exists a family of continuous functions $[0,1]\ni\sUV\mapsto \mathfrak{f}^{i,m,a}_\sUV\in\bR$, $(i,m,a)\in\bar\frI^-$, such that it holds
\begin{equation}\label{eq:renormalization_conditions_intro}
\llangle f^{i,m,a}_{\sUV,1}\rrangle = \mathfrak{f}^{i,m,a}_\sUV,
\qquad
(i,m,a)\in \bar\frI^-,
\end{equation}
where the relevant coefficients $f^{i,m,a}_{\sUV,1}$, $(i,m,a)\in \bar\frI^-$, are defined in terms of the force coefficients $f^{i,m,a}_\sUV$, $(i,m,a)\in\bar\frI$, as explained in the introduction.
\end{ass}

\begin{rem}
Note that the above assumption implies that $\lim_{\sUV\searrow0}\mathfrak{f}^{i,m,a}_\sUV = \mathfrak{f}^{i,m,a}_0$ and $|\mathfrak{f}^{i,m,a}_\sUV|\lesssim 1$ uniformly in $\sUV\in[0,1]$. Moreover, as we will see the bound~\eqref{eq:ass_ren_cond_bound} actually holds for all $(i,m,a)\in\bar\frI$.
\end{rem}
  
\begin{ass}\label{ass:infrared}
If $\sigma\notin 2\bN_+$, then we assume that $\ceil{\sigma}-1>\sigma-2\dim(\lambda)$. 
\end{ass}

\begin{rem}
It plays an important role in our construction that the function 
\begin{equation}\label{eq:infrared_function}
 \bR^\rdim\ni\bar x\mapsto x^a G(x)\in\bR,
 \qquad\quad
 x\equiv (\mathring x,\bar x),
\end{equation}
is absolutely integrable for every $\mathring x\in(0,\infty)$ and every \mbox{$a\in\frM_\sigma$}, where $G$ is the fractional heat kernel. In general, our construction requires that the above function is absolutely integrable for all $a\in\frM$ such that \mbox{$|a|\leq\floor{\sigma-2\dim(\lambda)}+1$}. However, if \mbox{$\sigma\notin 2\bN_+$}, because of the long-range character of the fractional heat kernel the function~\eqref{eq:infrared_function} is integrable only for $a=(\mathring a,\bar a)\in\frM$ such that $|\bar a|<\sigma$. Under Assumption~\ref{ass:infrared} our construction uses only the functions~\eqref{eq:infrared_function} with $a\in\frM_\sigma$ which are absolutely integrable.
\end{rem}

\subsection{Main theorem}

Consider the mild formulation of the macroscopic equation~\eqref{eq:intro_spde_macro}
\begin{equation}\label{eq:statement_mild}
 \varPhi_\sUV = G\ast (1_{(0,\infty)} F_\sUV[\varPhi_\sUV] +\delta_0\otimes\phi_\sUV),
\end{equation}
where $\ast$ denotes the space-time convolution, $1_{(0,\infty)} F_\sUV[\varphi](x):=1_{(0,\infty)}(\mathring x) F_\sUV[\varphi](x)$, $\delta_0\in\sD'(\bR)$ is the Dirac delta distribution at $\mathring x=0$, the initial data is of the form $\phi_\sUV=\phi_\sUV^\triangleright+\phi_\sUV^\vartriangle$ and $G(\mathring x,\Cdot)\in L^1(\bR^\rdim)$, $\mathring x\in(0,\infty)$, is the heat kernel of the fractional Laplacian $(-\Delta)^{\sigma/2}$ extended for $\mathring x\in(-\infty,0]$ in such a way that $G(0,\Cdot)=\delta_{\bR^\rdim}\in \cS'(\bR^\rdim)$ and $G(\mathring x,\Cdot)=0$ for $\mathring x\in(-\infty,0)$. Recall that the force $F_\sUV[\varphi]$ is defined by Eq.~\eqref{eq:force_macro} in terms of the noise $\varXi_\sUV$, the force coefficients $f^{i,m,a}_\sUV$, $(i,m,a)\in\bar\frI$,  and the parameter $\lambda\in\bR$. Note that $G\ast(\delta_0\otimes\phi_\sUV)$ is the unique solution of the equation $\fQ\varPhi_\sUV=0$ in the domain $[0,\infty)\times\bR^\rdim$ with the initial condition $\varPhi_\sUV(0,\Cdot)=\phi_\sUV$. We also remind the reader that the stochastic process $\varPhi_\sUV^\triangleright$ introduced in Eq.~\eqref{eq:intro_initial} coincides with the stationary solution $\varPsi_\sUV$ of the equation $\varPsi_\sUV = (G-G_1)\ast F_\sUV[\varPsi_\sUV]$ truncated at order $\lambda^{i_\triangleright}$, where $i_\triangleright\in\bN_0$ is chosen in such that $\varPsi_\sUV-\varPhi_\sUV^\triangleright$, $\sUV\in(0,1]$, is uniformly bounded (as a formal power series in~$\lambda$) in some space with positive regularity. Moreover, by definition $\phi_\sUV^\triangleright:=\varPhi_\sUV^\triangleright(0,\Cdot)$.

\begin{rem}
In Sec.~\ref{sec:intro_initial_value}, using the fact that the force $F_\sUV[\varphi]$ is a local functional that involves only spatial derivatives of $\varphi$, we showed that if $\breve\varPhi_\sUV$ satisfies the equation
\begin{equation}\label{eq:statement_rem_breve}
 \breve\varPhi_\sUV = G\ast (1_{(0,\infty)} (F_\sUV[\breve\varPhi_\sUV+\varPhi^\triangleright_\sUV]-\fQ\varPhi^\triangleright_\sUV) +\delta_0\otimes\phi^\vartriangle_\sUV),
\end{equation} 
then $\varPhi_\sUV=\breve\varPhi_\sUV+\varPhi_\sUV^\triangleright$ satisfies Eq.~\eqref{eq:statement_mild}.
\end{rem}

\begin{dfn}
For $\varepsilon>0$ let $\alpha\equiv\alpha_\varepsilon:=-\dim(\varPhi)-\varepsilon$, \mbox{$\beta\equiv\beta_\varepsilon:=-\dim(\phi)-\varepsilon$} and \mbox{$\gamma\equiv\gamma_\varepsilon:=\sigma-\varepsilon$}.
\end{dfn}

\begin{thm}\label{thm:main}
Suppose that the assumptions formulated in the previous section hold true. Let $\varepsilon>0$ be sufficiently small. There exist:
\begin{itemize}
\item a coupling between $(\varXi_\sUV,\phi_\sUV^\vartriangle)$, $\sUV\in[0,1]$,
\item a family of random stopping times $\breve T_\sUV\in[0,\infty]$, $\sUV\in[0,1]$,
\item a family of stochastic processes $\breve\varPhi_\sUV\in C([0,\breve T_\sUV),\cC^\gamma(\bT))$, $\sUV\in(0,1]$,
\item a stationary stochastic process $\varPhi_0^\triangleright\in \sC^\alpha(\bR\times\bT)$,
\item a stochastic process $\breve\varPhi_0\in C([0,\breve T_0),\sC^\beta(\bT))$
\end{itemize}
such that:
\begin{itemize}
 \item[(A)] for every $\sUV\in(0,1]$ the stochastic process $\varPhi_\sUV^\triangleright\in C(\bR,\cC^\gamma(\bT))$ is well-defined and stationary, it holds $\lim_{t\nearrow \breve T_\sUV}\|\breve\varPhi_\sUV(t,\Cdot)\|_{\cC^\gamma(\bT)}=\infty$ on the event \mbox{$\{\breve T_\sUV<\infty\}$} and the stochastic process $\varPhi_\sUV=\varPhi_\sUV^\triangleright+\breve\varPhi_\sUV\in C([0,\breve T_\sUV),\cC^\gamma(\bT))$ is the unique pathwise solution of Eq.~\eqref{eq:statement_mild} with $\phi_\sUV=\phi_\sUV^\triangleright + \phi_\sUV^\vartriangle$,
 \item[(B)] the law of the stochastic process $\varPhi^\triangleright_0$ depends only on the choice of the renormalization constants $\mathfrak{f}^{i,m,a}_0\in\bR$, $(i,m,a)\in\bar\frI^-$,  
 \item[(C)] $\varPhi_0^\triangleright = \lim_{\sUV\searrow0} \varPhi_\sUV^\triangleright$ in the space $\sC^\alpha(\bR\times\bT)$ in probability, 
 \item[(D)] $\lim_{t\nearrow \breve T_0}\|\breve\varPhi_0(t,\Cdot)\|_{\sC^\beta(\bT)}=\infty$ on the event $\{\breve T_0<\infty\}$ and the law of the stochastic process $\breve\varPhi_0$ depends only on the initial data $\phi_0^\vartriangle\in \sC^\beta(\bT)$ and the choice of the renormalization constants $\mathfrak{f}^{i,m,a}_0\in\bR$, $(i,m,a)\in\bar\frI^-$,  
 \item[(E)] for any $\Tmax\in(0,\infty)$ on the event $\{\Tmax<\breve T_0\}$ it holds $\Tmax<\breve T_\sUV$ for all sufficiently small $\sUV\in(0,1]$ and $\breve\varPhi_0 = \lim_{\sUV\searrow0} \breve\varPhi_\sUV$ in the space $C([0,\Tmax],\sC^\beta(\bT))$ in probability.
\end{itemize}
\end{thm}
\begin{rem}
For the definitions of the functions spaces appearing in the above statement see Sec.~\ref{sec:topology}. The space $\sC^\beta(\bT)$  coincides with (some separable subspace of) the standard Besov space $\cB_{\infty,\infty}^\beta(\bT)$. The space $\sC^\alpha(\bR\times\bT)\equiv \sC^\alpha(\bH)$ locally coincides with (a separable subspace of) the parabolic Besov space $\cB^\alpha_{\infty,\infty}(\bR\times\bT)$. The space $\cC^\gamma(\bT)$ is contained in $C^{\floor{\gamma}}(\bT)$. Recall that $\alpha\leq\beta<0$ and $\gamma>0$.
\end{rem}
\begin{proof}
Part (A) follows from Theorem~\ref{thm:regular_eq}. Its proof uses only classical PDE tools and is completely standard. The remaining statements rely on probabilistic estimates established in Sec.~\ref{sec:probabilistic} and deterministic construction presented in Sec.~\ref{sec:deterministic}. Parts (B) and (C) follow from Lemma \ref{lem:singular_terms_main}. Parts (D) and (E) are consequences of Theorem~\ref{thm:solution_max}.
\end{proof}

\begin{rem}
The method of the proof of the above theorem was discussed in Sec.~\ref{sec:technique}. Let us recall that it is based on the ansatz~\eqref{eq:idea_ansatz} for the effective force $F_{\sUV,\sIR}[\varphi]$ and a differential equation of the form~\eqref{eq:flow_eq_form}, called the flow equation, that governs the evolution of the effective force in the parameter~$\sIR$. The above-mentioned differential equation is closely related to the Polchinski flow equation in Euclidean QFT~\cite{polchinski1984} (see also~\cite{wegner1973}) that governs the evolution of the effective potential in the cutoff scale. The Polchinski flow equation is very useful tool to investigate properties of models of perturbative QFT. There is a very simple and general proof of perturbative renormalizability of QFT models based on the Polchinski flow equation (see the review paper \cite{muller2003}; for the position space approach more similar to ours in this paper see~\cite{kopper2007}). In particular, the proof avoids completely the problem of overlapping divergences of Feynman integrals which is one of the major difficulties in the approaches based on the Feynman diagrammatical expansion (such as the BPHZ approach). For the state of the art applications of the Polchinski flow equation in perturbative QFT see~\cite{frob2016all,efremov2017}.

The Polchinski flow equation requires a certain formal power series ansatz for the effective potential. The above-mentioned series typically does not converge because of the so-called large field problem. As a result, the Polchinski flow equation was not used extensively in constructive quantum field theory. An exception is the Sine-Gordon model for which the perturbative series converges (for application of the flow equation for this model see~\cite{brydges1987Mayer,bauerschmidt2021}). Note that in the context of the initial value problem for the parabolic PDEs we avoid the large field problem by first constructing the solution in sufficiently small random time intervals.
\end{rem}

\begin{rem}\label{rem:kupiainen_approach}
The technique we use is related to the technique developed earlier by Kupiainen in~\cite{kupiainen2016rg} for the dynamical $\Phi^4_3$ model and later applied to the generalized KPZ equation in~\cite{kupiainen2017kpz} (a similar method was used earlier in~\cite{bricmont1994} to study the long-time asymptotics of nonlinear parabolic PDEs). Both techniques are based on the Wilsonian renormalization group theory~\cite{wilson1971}. Let us briefly explain the Wilson approach to the renormalization problem adapted to the context of SPDEs. Consider an equation of the form
\begin{equation}
 \varPhi_\sUV = (G-G_1)\ast F_\sUV[\varPhi_\sUV],
\end{equation}
where $F_\sUV[\varphi]$ is some force and the parameter $\sUV$ plays the role of the UV cutoff. We note that $\varPhi_\sUV = (G-G_1)\ast f_{\sUV,1}$ is a solution of the above equation provided $f_{\sUV,1}$ is such that
\begin{equation}\label{eq:wilson2}
 f_{\sUV,1} = F_\sUV[(G-G_1)\ast f_{\sUV,1}].
\end{equation}
Let $F_{\sUV,\sIR}[\varphi]$ be the effective force satisfying Eq.~\eqref{eq:effective_force_property} for every $\sIR\in[0,1]$ together with the boundary condition $F_{\sUV,0}[\varphi]=F_{\sUV}[\varphi]$. Then $f_{\sUV,1}=F_{\sUV,1}[0]$ fulfills Eq.~\eqref{eq:wilson2} and it holds
\begin{equation}\label{eq:wilson_idea}
 f_{\sUV,1}
 = F_\sUV[(G-G_\sIR)\ast f_{\sUV,1}+(G_\sIR-G_1)\ast f_{\sUV,1}]
 =
 F_{\sUV,\sIR}[(G_\sIR-G_1)\ast f_{\sUV,1}]
\end{equation}
for $\sIR\in[0,1]$. Note that the characteristic temporal and spatial scales of the kernel $G_\sIR$ are of order $\sIR$ and $[\sIR]=\sIR^{1/\sigma}$, respectively. The functions $f_{\sUV,1}$ as well as $(G-G_1)\ast f_{\sUV,1}$ and $(G-G_\sIR)\ast f_{\sUV,1}$ are essentially constant at temporal and spatial scales less than $\sUV$ and $[\sUV]$. These functions converge as $\sUV\searrow0$ only in the space of distributions. On the other hand, the function $(G_\sIR-G_1)\ast f_{\sUV,1}$ is essentially constant at temporal and spatial scales less than $\sUV\vee\sIR$ and $[\sUV\vee\sIR]$ and for every $\sIR\in(0,1]$ its limit $\sUV\searrow0$ is a regular function. It is clear by Eq.~\eqref{eq:wilson_idea} that in order to define the effective force $F_{\sUV,\sIR}[\varphi]$ one has to control all potentially problematic products involving the rough part $(G-G_\sIR)\ast f_{\sUV,1}$ of the solution $\varPhi_\sUV= (G-G_1)\ast f_{\sUV,1}$.

The basic idea of the Wilsonian renormalization group theory is to construct the effective force $F_{\sUV,\sIR}[\varphi]$ for $\sIR\in\{\delta^n\,|\,n\in\bN_0\}$, $\delta\in(0,1)$, recursively using the fact that $F_{\sUV,0}[\varphi]=F_{\sUV}[\varphi]$ and for $\uIR\leq\sIR$ the effective forces $F_{\sUV,\uIR}[\varphi]$ and $F_{\sUV,\sIR}[\varphi]$ are related by Eq.~\eqref{eq:effective_force_property}. This is the approach taken in~\cite{kupiainen2016rg,kupiainen2017kpz} in the context of singular SPDEs. The same strategy is used in most applications of the renormalization group in constructive QFT~\cite{gallavotti1985,gawedzki1989}. The discrete renormalization group method described above is very powerful. Unfortunately, constructions based on this method are typically highly technical. For this reason, we do not believe it would be feasible to give a general solution theory for a large class of sub-critical SPDEs using the technique of~\cite{kupiainen2016rg,kupiainen2017kpz}.
\end{rem}

\begin{rem}
Our approach to the renormalization problem is similar in spirit to the approach proposed recently in \cite{linares2021} in the context of quasi-linear SPDEs. The model introduced in that work consists of certain multi-linear functionals of the noise indexed by multi-indices. The above functionals can be represented by linear combinations of trees but this representation does not play any role in the proof. The stochastic estimates for the model are established inductively with the use of the spectral gap inequality. The counterterms are fixed by the symmetries and the BPHZ renormalization conditions.
\end{rem}

\begin{rem}\label{rem:other_approaches}
Let us briefly describe the standard approaches that are used to give meaning to singular SPDEs. The case of equations such as the dynamical $\Phi^4_2$ model for which Wick ordering is sufficient to define the non-linear term was resolved by Da Prato and Debussche~\cite{daprato2003} using an expansion around the solution of the linear equation. This type of expansion breaks down in the case of more singular equations like the dynamical $\Phi^4_3$ model. The renormalization problem for such equations stayed open until the breakthrough by Hairer~\cite{hairer2014structures}, who introduced the theory of regularity structures allowing a local description of a distribution by a generalized Taylor expansion around each spacetime point. Another approach that can be used to tackle a large class of singular SPDEs was developed later by Gubinelli, Imkeller and Perkowski in~\cite{gubinelli2015} using the paracontrolled calculus. Let us also mention an approach based on the rough path theory proposed by Otto and Weber in~\cite{otto2019rough} in the context of singular quasi-linear SPDEs.
\end{rem}

Now we would like to discuss some applications of our result. In the situation when the physical system under consideration has some symmetries it is natural to assume that the noise and the force coefficients respect these symmetries. This requirement typically restricts the choice of the values of the renormalization parameters.

\begin{dfn}
Let $\mathcal{H}$ be a group and $\fM$ and $\fN$ be actions of $\mathcal{H}$ on the space of spacetime functions $C(\bM)$ such that $\fQ\,\fM_g\varphi =\fN_g\fQ\varphi$ for all $\varphi\in C^\infty_\rc(\bM)$. We say that the system has symmetry group $\mathcal{H}$ if the noise $\microXi_\sUV$ and the microscopic force coefficients $\microf^{i,m,a}_\sUV$, $(i,m,a)\in\bar\frI$, are such that the corresponding macroscopic force satisfies the condition
\begin{equation}
 F_\sUV[\fM_g\varphi] \stackrel{\mathrm{(law)}}{=} \fN_g F_\sUV[\varphi]
\end{equation}
for every $\sUV\in(0,1]$, $\varphi\in C^\infty_\rc(\bM)$ and $g\in\mathcal{H}$. We will only consider situation when the above condition implies automatically an analogous condition for the effective force $F_{\sUV,\sIR}[\varphi]$ (we restrict attention to symmetries which are preserved by the cutoff propagator $G_\sIR$).
\end{dfn}

\begin{dfn}
Let $\mathcal{G}$ be the group of isometries of $\bR^\rdim$ preserving the cubic lattice with spacing $2\pi$. The action of the group $\mathcal{G}$ on the space of spacetime functions $C(\bM)$ is defined by
\begin{equation}
 (\fM_g\varphi)(\mathring x,\bar x):=\varphi(\mathring x,g^{-1}\bar x),
 \qquad
 g\in\mathcal{G},\quad
 (\mathring x,\bar x)\in\bR\times\bR^\rdim=\bM.
\end{equation}
The action of the group $\bZ_2$ on $C(\bM)$ is defined by
\begin{equation}
 (\fM_g\varphi)(\mathring x,\bar x) = g\,\varphi(\mathring x,\bar x),\qquad g\in \{-1,1\}=\bZ_2.
\end{equation}
We set $\fN_g=\fM_g$ for all $g\in\mathcal{G}\times\bZ_2$.
\end{dfn}

\begin{example}\label{example:phi_4_3}
Let $\rdim=3$, $\sigma=2$ and $\dim(\lambda)=1$. Consider any microscopic equation of the form~\eqref{eq:intro_spde} describing a system with the symmetry group \mbox{$\mathcal{G}\times\bZ_2$}. The above-mention equation describes a phase coexistence model studied in~\cite{hairer2018coexistence,shen2018} (see also \cite{furlan2019weak,erhard2020,zhu2018}). The only relevant force coefficients not violating the symmetries of the equation are
\begin{equation}
 \microf^{1,3,0}_\sUV,\qquad \microf^{1,1,0}_{\sUV},\qquad \microf^{2,1,0}_{\sUV}.
\end{equation}
The relevant force coefficients are fixed by the following renormalization conditions
\begin{equation}
 \lim_{\sUV\searrow0}\,\llangle f^{1,3,0}_{\sUV,1} \rrangle = \mathfrak{f}_0^{1,3,0},
 \qquad
 \lim_{\sUV\searrow0}\,\llangle f^{1,1,0}_{\sUV,1} \rrangle = \mathfrak{f}_0^{1,1,0},
 \qquad
 \lim_{\sUV\searrow0}\,\llangle f^{2,1,0}_{\sUV,1} \rrangle = \mathfrak{f}_0^{2,1,0}
\end{equation}
and for all other $(i,m,a)\in\bar\frI^-$
\begin{equation}
 \lim_{\sUV\searrow0}\,\llangle f^{i,m,a}_{\sUV,1} \rrangle = 0.
\end{equation}
Our analysis implies that the macroscopic behavior of this model is described by the following sub-critical singular SPDE
\begin{equation}\label{eq:example_phi_4_3}
 (\partial_{\mathring x}-\Delta_{\bar x}) \varPhi_0(x)
 = \varXi_0(x) + \lambda\,f^{1,3,0}_0\varPhi_0(x)^3 + \lambda f^{1,1,0}_0\varPhi_0(x) + \lambda^2 f^{2,1,0}_0\varPhi_0(x),
\end{equation}
where $\varXi_0$ is the periodization of the white noise on $\bM$, $f^{1,3,0}_0=\mathfrak{f}^{1,3,0}_0$ and the mass counterterms $f^{1,1,0}_0$, $f^{2,1,0}_0$ are not well-defined. This together with the a priori bounds proved in~\cite{mourrat2017down} allows us to recover the universality results established earlier in~\cite{shen2018}.
\end{example}

\begin{rem}
The singular SPDE~\eqref{eq:example_phi_4_3} is called the dynamical $\varPhi^4_3$ model. Let us mention that the renormalization problem for this equation was solved for the first time in~\cite{hairer2014structures} and later in~\cite{catellier2018,kupiainen2016rg}. The universality class of the dynamical $\varPhi^4_3$ model is expected to contain the three-dimensional Ising model with Glauber dynamics and Kac interactions near the critical temperature (this was conjectured in~\cite{giacomin99}; the two dimensional version of this result was established in~\cite{mourrat2017}). The invariant measure of the above SPDE describes the $\Phi^4_3$ Euclidean QFT (see for example \cite{gubinelli2018pde}).
\end{rem}

\begin{rem}\label{eq:parametrization_phi_4_3}
As mentioned above $\lambda\in\bR$ is a bookkeeping parameter that can be set equal to one. In our approach the entire solution manifold of the dynamical $\Phi^4_3$ model is parameterized by the initial data and the renormalization parameters \mbox{$\mathfrak{f}_0^{1,3,0},\mathfrak{f}_0^{1,1,0},\mathfrak{f}_0^{2,1,0}\in\bR$}. However, this parametrization is not injective. It turns out that the whole solution manifold can be parameterized by the initial data and $\mathfrak{f}_0^{1,3,0},\mathfrak{f}_0^{2,1,0}\in\bR$ with fixed $\mathfrak{f}_0^{1,1,0}=0$ (it is natural to assume that $\mathfrak{f}_0^{1,3,0}<0$ to avoid finite time blow-up of the solution; in infinite volume one could in addition set $\mathfrak{f}_0^{1,3,0}=-1$ by a rescaling operation to get a one-parameter family of equations).
\end{rem}

\begin{example}
Let $\rdim=4$, $\sigma\in(2,3)$ and $\dim(\lambda)=2\sigma-4$. Consider any microscopic equation of the form~\eqref{eq:intro_spde} describing a system with the symmetry group $\mathcal{G}\times\bZ_2$. The only relevant force coefficients not violating the symmetries of the equation are
\begin{equation}
 \microf^{1,3,0}_\sUV,\qquad \microf^{i,1,0}_{\sUV},~~ i\in\{1,\ldots,i_\sharp\},
\end{equation}
where $i_\sharp:=\floor{\sigma/\dim(\lambda)}$. Note that $i_\sharp \to\infty$ as $\sigma\searrow2$. The relevant force coefficients are fixed by the following renormalization conditions
\begin{equation}
 \lim_{\sUV\searrow0}\,\llangle f^{1,3,0}_{\sUV,1} \rrangle = \mathfrak{f}_0^{1,3,0},
 \qquad
 \lim_{\sUV\searrow0}\,\llangle f^{i,1,0}_{\sUV,1} \rrangle = \mathfrak{f}_0^{i,1,0},~~ i\in\{1,\ldots,i_\sharp\},
\end{equation}
and for all other $(i,m,a)\in\bar\frI^-$
\begin{equation}
 \lim_{\sUV\searrow0}\,\llangle f^{i,m,a}_{\sUV,1} \rrangle = 0.
\end{equation}
The scaling limit $\varPhi_0:=\lim_{\sUV\searrow0}\varPhi_\sUV$ formally solves the following sub-critical singular SPDE
\begin{equation}\label{eq:spde_example}
 \fQ \varPhi_0(x)
 = \varXi_0(x) +\lambda\,f^{1,3,0}_0\varPhi_0(x)^3 + \sum_{i=1}^{i_\sharp}\lambda^i f^{i,1,0}_0\varPhi_0(x).
\end{equation}
Careful analysis shows that it holds $f^{1,3,0}_0=\mathfrak{f}^{1,3,0}_0$ and the coefficients $f^{i,1,0}_0$, $i\in\{1,\ldots,i_\sharp\}$, are generically ill-defined because the limit $\lim_{\sUV\searrow0}f^{i,1,0}_\sUV$ does not exist (unless the equation is linear, that is $\mathfrak{f}^{1,3,0}_0=0$).
We obtain essentially the same result for all $\dim(\lambda)\in(\sigma-2,2\sigma-4]$.
\end{example}

\begin{example}
Let $\rdim=4$, $\sigma\in(2,8/3)$ and $\dim(\lambda)=\sigma-2$. Consider again an arbitrary microscopic equation of the form~\eqref{eq:intro_spde} describing a system with the symmetry group $\mathcal{G}\times\bZ_2$. The only relevant force coefficients that are not fixed by the symmetries are
\begin{equation}
 \microf^{1,1,\Delta}_\sUV:=\microf^{1,1,a_1}_{\sUV}=\ldots=\microf^{1,1,a_\rdim}_{\sUV},
 \qquad
 \microf^{i,3,0}_\sUV,
 ~~
 i\in\{1,2\}
 \qquad
 \microf^{i,1,0}_{\sUV},~~ i\in\{1,\ldots,i_\sharp\},
\end{equation}
where the spatial multi-indices $a_j=(a_j^1,\ldots,a_j^\rdim)\in\bN_0^\rdim$, $j\in\{1,\ldots,\rdim\}$, are defined by the conditions $a^j_j=2$ and $a^k_j=0$ for all \mbox{$k\in\{1,\ldots,\rdim\}\setminus\{j\}$} and $i_\sharp$ is as in the previous example. We impose the following renormalization conditions
\begin{equation}
 \lim_{\sUV\searrow0}\,\llangle f^{1,1,\Delta}_{\sUV,1} \rrangle = \mathfrak{f}_0^{1,1,\Delta},
 \qquad
 \lim_{\sUV\searrow0}\,\llangle f^{i,3,0}_{\sUV,1} \rrangle = \mathfrak{f}_0^{i,3,0},~~i\in\{1,2\},
\end{equation}
\begin{equation}
 \lim_{\sUV\searrow0}\,\llangle f^{i,1,0}_{\sUV,1} \rrangle = \mathfrak{f}_0^{i,1,0},~~  i\in\{1,\ldots,i_\sharp\},
\end{equation}
and for all other $(i,m,a)\in\bar\frI^-$
\begin{equation}
 \lim_{\sUV\searrow0}\,\llangle f^{i,m,a}_{\sUV,1} \rrangle = 0.
\end{equation}
The scaling limit $\varPhi_0:=\lim_{\sUV\searrow0}\varPhi_\sUV$ formally solves the following sub-critical singular SPDE
\begin{equation}\label{eq:spde_example_laplacian}
 \fQ \varPhi_0(x)
 = \varXi_0(x) +\lambda f^{1,1,\Delta}_0\Delta_{\bar x}\varPhi_0(x)- \sum_{i=1}^2\lambda^i f^{i,3,0}_0\varPhi_0(x)^3 + \sum_{i=1}^{i_\sharp}\lambda^{i} f^{i,1,0}_0\varPhi_0(x).
\end{equation}
Careful analysis shows that it holds $f^{1,1,\Delta}_0=\mathfrak{f}^{1,1,\Delta}_0$, $f^{1,3,0}_0=\mathfrak{f}^{1,3,0}_0$ and the coefficients $f^{2,3,0}_0$ and $f^{i,1,0}_0$, $i\in\{1,\ldots,i_\sharp\}$, are generically ill-defined because the limits $\lim_{\sUV\searrow0}f^{2,3,0}_\sUV$ and $\lim_{\sUV\searrow0}f^{i,1,0}_\sUV$ need not exist. If $\mathfrak{f}^{1,3,0}_0=0$, then $f^{2,3,0}_0=\mathfrak{f}^{2,3,0}_0$ and $f^{1,1,0}_0=\mathfrak{f}^{1,1,0}_0$ are well-defined. If in addition $\mathfrak{f}^{2,3,0}_0=0$, then the equation is linear and all coefficients are well-defined.
We obtain essentially the same result for all $\dim(\lambda)\in(0,\sigma-2]$.
\end{example}

\begin{example}
Let $\rdim=4$, $\sigma=2$ and $\dim(\lambda)>0$. Consider again an arbitrary microscopic equation of the form~\eqref{eq:intro_spde} describing a system with the symmetry group~$\mathcal{G}\times\bZ_2$. This time the only relevant force coefficients not violating the symmetries of the equation are
\begin{equation}
 \microf^{i,1,0}_{\sUV},~~ i\in\{1,\ldots,i_\sharp\}.
\end{equation}
We impose the following renormalization conditions
\begin{equation}
 \lim_{\sUV\searrow0}\,\llangle f^{i,1,0}_{\sUV,1} \rrangle = \mathfrak{f}_0^{i,1,0},~~ i\in\{1,\ldots,i_\sharp\}.
\end{equation}
It is easy to see that
\begin{equation}
 f^{i,1,0}_0=\lim_{\sUV\searrow0} f^{i,1,0}_\sUV=\mathfrak{f}_0^{i,1,0}
\end{equation}
and the scaling limit $\varPhi_0:=\lim_{\sUV\searrow0}\varPhi_\sUV$ solves the following linear SPDE
\begin{equation}
 \fQ \varPhi_0(x)
 = \varXi_0(x) + \sum_{i=1}^{i_\sharp}\lambda^i f^{i,1,0}_0\,\varPhi_0(x).
\end{equation}
Obviously the above triviality result is of little interest because of the restrictive assumption $\dim(\lambda)>0$.
\end{example}

\begin{rem}
If $\sigma\in[8/3,4)$, then depending on the choice of $\dim(\lambda)>0$ there are other relevant force coefficients not violating the symmetries of the equation, for example $f^{1,5,0}_\sUV$. If $\sigma\geq \rdim=4$, then the number of possible relevant force coefficients is not bounded and no universality result is expected.
\end{rem}

\begin{rem}
If the spatial dimension $\rdim$ equals $4$, as assumed above, then our technique allows to renormalize the singular SPDE~\eqref{eq:spde_example} in the whole range of subcriticality $\sigma\in(2,4]$ (for $\sigma>4$ the above equation is not singular). If $\rdim=3$, then the SPDE~\eqref{eq:spde_example} is singular and sub-critical for $\sigma\in(3/2,3]$. In the case $\rdim=3$ our method covers only the range $\sigma\in(5/3,3]$. In general, for $\sigma\notin2\bN_+$ because of the long-range character of the fractional heat kernel our method needs the following assumption $\ceil{\sigma}-1>\sigma-2\dim(\lambda)$, see Assumption~\ref{ass:infrared}.
\end{rem}

\begin{rem}
For $\rdim=3$ and $\sigma=2$ the singular SPDE~\eqref{eq:spde_example} coincides with the dynamical $\Phi^4_3$ model. For $\rdim=3$ and $\sigma\in(21/11,2)$ the renormalization problem for the SPDE~\eqref{eq:spde_example} is very similar to the renormalization problem for the dynamical $\Phi^4_3$ model and was solved in~\cite{gubinelli2018pde} using paracontrolled calculus. Different types of non-local singular SPDEs were previously studied in~\cite{gubinelli2015,ignat2019,chiarini2019,ignat2020}. However, the non-local SPDEs considered in these references are not close to criticality.
\end{rem}

\begin{rem}
There exists a general blackbox machinery built upon the theory of regularity structures that can be used to systematically renormalize virtually all sub-critical singular SPDEs with local differential operators~\cite{chandra2016bphz,bruned2019algebraic,bruned2021renormalising}. It is applicable also in the case $\dim(\varPhi)<0$ which is not covered by our technique. However, its generalization for non-local singular SPDEs treated in the present work is non-trivial. Let us also mention that there are a number of results about specific families of singular SPDEs covering whole sub-critical region~\cite{chandra2018sine,chandra2019full,otto2021full,linares2021}.
\end{rem}

\begin{rem}
The singular SPDE~\eqref{eq:spde_example} belongs to the class of the stochastic quantization equations~\cite{parisi1981,damgaard1987}. The invariant measure for the SPDE~\eqref{eq:spde_example} describes a model of QFT with a modified propagator (this was established in~\cite{gubinelli2018pde} for $\rdim=3$ and $\sigma\in(21/11,2]$). Models of this type (especially for $\rdim=3$) attracted attention in the area of constructive QFT because they can be used to implement rigorously the Wilson-Fischer expansion used to study critical phenomena (see for example \cite{abdesselam2007} and references cited therein). In the case $\rdim=4$ the invariant measure for the SPDE~\eqref{eq:spde_example} is not supposed to be reflection positive but should satisfy all other Osterwalder-Schrader axioms. Let us mention that the SPDEs of the form~\eqref{eq:spde_example} are expected to describe the large scale behavior of certain long-range statistical models. For example, the SPDEs of the above form should govern the random fluctuations of a suitably rescaled coarse-grained spin field of the Ising-Kac model with power-law-decaying interactions of the form $|\bar x|^{-\rdim-\sigma}$, where $|\bar x|$ is the distance between the spins. However, we are not aware of any rigorous results in this direction.
\end{rem}

\begin{dfn}
The actions $\fM$ and $\fN$ of the group $\bR$ on $C(\bM)$ are defined by
\begin{equation}
 (\fM_g\varphi)(\mathring x,\bar x) = \varphi(\mathring x,\bar x)+g,\qquad
 (\fN_g\varphi)(\mathring x,\bar x) = \varphi(\mathring x,\bar x),\qquad g\in \bR.
\end{equation}
\end{dfn}

\begin{example}\label{example:kpz}
Let $\rdim=1$, $\sigma=2$ and $\dim(\lambda)=1/2$. Note that, strictly speaking, this example is outside the scope of the technique developed in this paper since $\dim(\varPhi)=(\rdim-\sigma)/2<0$. Consider an arbitrary microscopic equation of the form
\begin{equation}\label{eq:KPZ_example}
 (\partial_{\mathring x}-\Delta_{\bar x})\microPhi_\sUV(x) =
 \microF_\sUV[\microPhi_\sUV](x),
 \qquad x=(\mathring x,\bar x)\in\bR_+\times\bR
\end{equation}
describing a system with the symmetry group~$\mathcal{G}\times\bR$ with broken $\bZ_2$ symmetry. The equations of the above type belong to the class of continuous interface growth models that were studied in~\cite{hairer2017clt,hairerQuastel2018} (see also~\cite{hairer2019interface} as well as~\cite{gubinelli2016universality} and references cited therein). Since the system preserves the $\bR$ symmetry the force $\microF_\sUV[\varphi]$ can depend on $\varphi$ only thorough its spatial derivative $\partial_{\bar x}\varphi$. Consequently, there exists another force $\dot{\microF}_\sUV[\varphi]$ of the form~\eqref{eq:force_micro} such that it holds
\begin{equation}
 \microF_\sUV[\varphi] := \dot{\microF}_\sUV[\partial_{\bar x}\varphi]
\end{equation}
for all $\varphi\in C^\infty_\rc(\bM)$. Consider the microscopic equation
\begin{equation}\label{eq:KPZ_derivative}
 (\partial_{\mathring x}-\Delta_{\bar x})\microPsi_\sUV(x) = \partial_{\bar x}
 \dot{\microF}_\sUV[\microPsi_\sUV](x),
 \qquad x=(\mathring x,\bar x)\in\bR_+\times\bR,
\end{equation}
which is related to Eq.~\eqref{eq:KPZ_example} via the identification $\microPsi_\sUV=\partial_{\bar x}\microPhi_\sUV$. As described briefly in Remark~\ref{rem:kpz} our technique can be adapted to yield the existence and universality of the macroscopic scaling limit for the above equation. The only relevant coefficients of the force $\dot{\microF}_{\sUV}[\varphi]$ not violating the symmetries of the equation are
\begin{equation}
\microf^{1,2,0}_\sUV,\qquad\microf^{1,0,0}_{\sUV},\qquad\microf^{2,0,0}_{\sUV},\qquad\microf^{3,0,0}_{\sUV}.
\end{equation}
The relevant force coefficients are fixed by the following renormalization conditions
\begin{equation}
 \lim_{\sUV\searrow0}\,\llangle f^{1,2,0}_{\sUV,1} \rrangle = \mathfrak{f}_0^{1,2,0},
 \qquad
 \lim_{\sUV\searrow0}\,\llangle f^{i,0,0}_{\sUV,1} \rrangle = \mathfrak{f}_0^{i,0,0},
 \quad i\in\{1,2,3\}
\end{equation}
and for all other $(i,m,a)\in\bar\frI^-$
\begin{equation}
 \lim_{\sUV\searrow0}\,\llangle f^{i,m,a}_{\sUV,1} \rrangle = 0.
\end{equation}
We note that the solution of Eq.~\eqref{eq:KPZ_derivative} does not depend on the choice of the coefficients $\microf^{i,0,0}_\sUV\in\bR$, $i\in\{1,2,3\}$. The macroscopic scaling limit is described by the following singular SPDE
\begin{equation}\label{eq:burgers}
 (\partial_{\mathring x}-\Delta_{\bar x})\varPsi_0(x) = \partial_{\bar x}\big(\microXi_0(x) + \lambda f^{1,2,0}_0  \varPsi_0(x)^2 + \lambda f^{1,0,0}_0
 + \lambda^2 f^{2,0,0}_0
 + \lambda^3 f^{3,0,0}_0\big),
\end{equation}
where $f^{1,2,0}_0=\mathfrak{f}_0^{1,2,0}$ and the counterterms $f^{i,0,0}_0$, $i\in\{1,2,3\}$, need not be well-defined. This allows us to essentially reproduce the universality results of~\cite{hairer2017clt,hairerQuastel2018} in the weakly asymmetric regime.
\end{example}

\begin{rem}
The singular SPDE~\eqref{eq:burgers} is called the stochastic Burgers equation and is obtained by taking the spatial derivative of both sides of the KPZ equation
\begin{equation}
 (\partial_{\mathring x}-\Delta_{\bar x})\varPhi_0(x) = \microXi_0(x) + \lambda f^{1,2,0}_0  (\partial_{\bar x}\varPhi_0(x))^2 + \lambda f^{1,0,0}_0
 + \lambda^2 f^{2,0,0}_0
 + \lambda^3 f^{3,0,0}_0.
\end{equation}
One can define the notion of solution of the above equation using the Cole-Hopf transform and the classical It\^{o} calculus (for the Cole-Hopf solution it holds $f^{2,0,0}_0=f^{3,0,0}_0=0$ and $f^{1,0,0}_0$ is a certain infinite It\^{o} correction). The robust approximation theory for KPZ equation was first given in~\cite{hairer2013kpz}. Let us note that there are many microscopic models for which the convergence to the KPZ equation has been established rigorously (see the references cited in~\cite{hairerQuastel2018}).
\end{rem}

\begin{rem}\label{rem:kpz}
In order to prove that the large-scale behavior of solutions of the SPDE~\eqref{eq:KPZ_derivative} is described by to the stochastic Burgers equation we study the following equation
\begin{equation}
 \varPsi_\sUV = \partial_{\bar x}G\ast (1_{(0,\infty)} F_\sUV[\varPsi_\sUV] + \delta_0\otimes\phi_\sUV),
\end{equation}
where $G$ is the heat kernel in $\bR\times\bR$. We apply the technique develop in the paper with the following (non-standard) choice of the dimensions $\dim(\varPsi)=-1/2$, $\dim(\varXi)=-3/2$, $\dim(\lambda)=1/2$ (recall that $\rdim=1$, $\sigma=2$). Note that $\partial_{\bar x}G$ is a regularizing kernel of order $\dim(\varXi)-\dim(\varPsi)=1$ in the scale of the parabolic Besov spaces.
\end{rem}

\section{Regularizing kernels}\label{sec:kernels}

In this section we define kernels that will be used to control the regularity of the effective force coefficients and their cumulants. We also introduce the scale decomposition of the fractional heat kernel. We prove several properties of the regularizing kernels and the scale decomposition of the fractional heat kernel, which are frequently used in the paper.

Let $N\in\bN_+$. The space of compactly supported test functions is denoted by $C^\infty_\rc(\bR^N)$. The space of bounded smooth functions in denoted by $C^\infty_\rb(\bR^N)$. The space of Schwartz functions is denoted by $\sS(\bR^N)$. The space of distributions and the subspace of tempered distributions are denoted by $\sD'(\bR^N)$ and $\sS'(\bR^N)$, respectively. The Fourier transform in $\sS'(\bR^N)$ is denoted by $\fF$. In particular, for a Schwartz function \mbox{$\varphi\in\sS(\bR^N)$} we define
\begin{equation}
 \fF \varphi(p) = \int_{\bR^N} \exp(-\ri\, p \cdot x) \,\varphi(x) \,\rd x,
 \quad
 \varphi(x) = \frac{1}{(2\pi)^N} \int_{\bR^N} \exp(\ri\, p \cdot x)\, \fF \varphi(p) \,\rd p.
\end{equation}
Observe that
\begin{equation}
 \fF(\psi\ast\varphi)(p) = \fF(\psi)(p)\, \fF(\varphi)(p)
\end{equation}
for $\psi,\varphi\in\sS(\bR^N)$, where $\ast$ denotes the convolution.  We use the following notation $\delta_{\bR^N}\in\sS'(\bR^N)$ for the Dirac delta. We also set $\bR^0=\{0\}$.

\begin{dfn}\label{dfn:cK}
For $n\in\bN_+$ let $\cK^n\subset\sS'(\bM^n)$ be the space of signed measures $K$ on $\bM^n$ with finite total variation $|K|$. We set $\|K\|_{\cK^n} = \int_{\bM^n} |K(\rd x_1\ldots\rd x_n)|$. If $n=1$, then we write $\cK^1=\cK\subset\sS'(\bM)$. Given $K\in\cK$ we set $K^{\otimes0}:=\delta_\bM$ and $K^{\otimes n}:=K\otimes\ldots\otimes K\in\cK^n$ for $n\in\bN_+$.
\end{dfn}

\begin{dfn}\label{dfn:scaling_S}
For $n\in\bN_+$ and $\sIR>0$ and $K\in C(\bM^n)$ we define
\begin{equation}
 \fS_\sIR K(\mathring x_1,\bar x_1,\ldots,\mathring x_n,\bar x_n) := [\sIR]^{-\rDim n}\, K(\mathring x_1/\sIR,\bar x_1/[\sIR],\ldots,\mathring x_n/\sIR,\bar x_n/[\sIR]), 
\end{equation}
where $[\sIR]:=\sIR^{1/\sigma}$. The map $\fS_\sIR$ extends to $K\in\sD'(\bM^n)$ in natural way.
\end{dfn}
\begin{rem}
For any $\sIR>0$ the map $\fS_\sIR\,:\cK^n\to\cK^n$ is an isometry.
\end{rem}
\begin{dfn}
For $K\in\cK$ we set $K^{\ast0}:=\delta_\bM$ and $K^{\ast(\oo+1)}:=K\ast K^{\ast\oo}$ for $\oo\in\bN_0$. For an operator $\fP$ we set $\fP^0:=\rid$ and $\fP^{\oo+1}:=\fP\, \fP^{\oo}$ for $\oo\in\bN_0$.  
\end{dfn}

\begin{dfn}\label{dfn:varepsilon_sigma}
We define $\varepsilon_\sigma:=\sigma+1-\ceil{\sigma}\in(0,1\wedge\sigma]$. 
\end{dfn}
\begin{rem}\label{rem:epsilon_sigma}
In what follows, we assume that $\varepsilon\in(0,\varepsilon_\sigma)$. Later we shall introduce other possibly smaller strictly positive upper bounds for $\varepsilon$.
\end{rem}

\begin{dfn}\label{dfn:kernels}
For $\sIR\in(0,\infty)$ the kernels $K_\sIR,J_\sIR\in L^1(\bM)\subset\cK$ are defined by
\begin{equation}
 \fF K_\sIR(p)=\frac{1}{1+\ri\sIR\mathring p}\frac{1}{1+[\sIR]^2|\bar p|^2},
 \qquad
 \fF J_\sIR(p)=\frac{1}{1+[\sIR]^{\sigma-\varepsilon}|\bar p|^{\sigma-\varepsilon}}.
\end{equation}
For $\sIR\in(0,\infty)$ the kernels $\mathring K_\sIR\in L^1(\bR)$, $\bar K_\sIR\in L^1(\bR^\rdim)$ are defined by 
\begin{equation}
 \fF\mathring K_\sIR(\mathring p):=\frac{1}{1+\ri\sIR\mathring p}, 
 \qquad
 \fF\bar K_\sIR(\bar p):=\frac{1}{1+[\sIR]^2|\bar p|^2}.
\end{equation} 
We omit the index $\sIR$ if $\sIR=1$. We also set $K_0=\delta_\bM$, $J_0=\delta_\bM$, $\mathring K_0=\delta_\bR$, $\bar K_0=\delta_{\bR^\rdim}$.
\end{dfn}

\begin{dfn}\label{dfn:diff_op}
For $\sIR\in[0,\infty)$ we define the following (pseudo-)differential operators
\begin{equation}
 \fP_\sIR\equiv\fP_\sIR(\partial_x)=(1+\sIR\partial_{\mathring x})(1-[\sIR]^2\Delta_{\bar x}),
 \quad
 \fR_\sIR\equiv\fR_\sIR(\partial_x)=1+[\sIR]^{\sigma-\varepsilon}(-\Delta_{\bar x})^{(\sigma-\varepsilon)/2},
\end{equation} 
\begin{equation}
 \mathring\fP_\sIR\equiv\mathring\fP_\sIR(\partial_x)=1+\sIR\partial_{\mathring x}, 
 \qquad
 \bar\fP_\sIR\equiv\bar\fP_\sIR(\partial_x)=1-[\sIR]^2\Delta_{\bar x}.
\end{equation} 
We omit the index $\sIR$ if $\sIR=1$.
\end{dfn}

\begin{rem}
For $\sIR\in(0,\infty)$ we have $K_\sIR=\fS_\sIR K$, $J_\sIR=\fS_\sIR J$. For $\oo\in\bN_+$ and $\sIR\in[0,\infty)$ the kernels $K_\sIR^{\ast\oo}$, $\mathring K_\sIR^{\ast\oo}$, $\bar K_\sIR^{\ast\oo}$ are positive measures with total mass equal to $1$. Moreover, we have $K_\sIR^{\ast\oo}=\mathring K_\sIR^{\ast\oo}\otimes\bar K_\sIR^{\ast\oo}$, $\fP^\oo_\sIR = \mathring \fP^\oo_\sIR\otimes \bar \fP^\oo_\sIR$ and
\begin{equation}
 \fP^\oo_\sIR K^{\ast\oo}_\sIR=\delta_\bM,
 \quad
 \mathring\fP^\oo_\sIR \mathring K^{\ast\oo}_\sIR=\delta_\bR,
 \quad
 \bar\fP^\oo_\sIR \bar K^{\ast\oo}_\sIR=\delta_{\bR^\rdim},
 \quad
 \fR^\oo_\sIR J^{\ast\oo}_\sIR = \delta_\bM.
\end{equation}
The fact that $K$ and $J$ are the inverses of (pseudo-)differential operators plays an important role in our construction.
\end{rem}

\begin{lem}\label{lem:kernel_u_v}
For any $\sIR,\uIR>0$,
\begin{equation}\label{eq:kernel_u_v}
\begin{gathered}
 K_\sIR = \fP_\uIR K_\sIR \ast K_\uIR 
 \qquad
 \|\fP_\uIR K_\sIR \|_\cK= 1\vee(2\uIR/\sIR-1)(2[\uIR/\sIR]^2-1).
\end{gathered} 
\end{equation}
In particular, if $\sIR\geq\uIR$, then $\|\fP_\uIR K_\sIR\|_\cK=1$.
\end{lem}
\begin{proof}
By the exact scaling relations it is enough to consider the case $\sIR=1$. The statement of the lemma follows from the identity
\begin{equation}
\fP_\uIR K = (\uIR\delta_\bM + (1-\uIR)\mathring K\otimes\delta_{\bR^\rdim})\ast([\uIR]^2\delta_\bM + (1-[\uIR]^2) \delta_\bR\otimes\bar K)
\end{equation}
and $\|\delta_\bM\|_\cK=1$, $\|\mathring K\otimes\delta_{\bR^\rdim}\|_\cK=1$, $\|\delta_\bR\otimes\bar K\|_\cK=1$.
\end{proof}

\begin{dfn}\label{dfn:periodization}
For $H\in L^1(\bM)$ we define $\fT H\in L^1(\bH)$ by the formula
\begin{equation}
 \fT H(\mathring x,\bar x):= \sum_{\bar y\in(2\pi\bZ)^\rdim} V(\mathring x,\bar x+\bar y).
\end{equation}
For $H\in L^1(\bR^\rdim)$ we define $\fT H\in L^1(\bT)$ by an analogous formula.
\end{dfn}
\begin{rem}
For future reference, we note that $H\ast f = \fT H \star f$ for $f\in C(\bH)$ and $H\in L^1(\bM)$, where $\ast$ and $\star$ are the convolutions in $\bM$ and $\bH$, respectively.
\end{rem}

\begin{lem}\label{lem:kernel_simple_fact}
Let $\ooo\in\bN_+$, $a\in\frM$ and $p\in[1,\infty]$.
\begin{enumerate}
\item[(A)]
If $|a|\leq\ooo$, then $\|\partial^a K^{\ast\ooo}_\sIR\|_\cK
 \lesssim [\sIR]^{-[a]}$ uniformly in $\sIR>0$. 
\item[(B)]
If $a\in\bar\frM_\sigma$, then $\|\partial^a J_\sUV\|_\cK
 \lesssim [\sUV]^{-[a]}$ uniformly in $\sUV>0$. 
\item[(C)] 
If $2\ooo>\sigma$, then $\|\fR^{\phantom\ooo}_\sUV \bar K_\sIR^{\ast\ooo}\|_{L^1(\bR^\rdim)} \lesssim [\sUV\vee\sIR]^{\sigma-\varepsilon}\,[\sIR]^{\varepsilon-\sigma}$ uniformly in $\sUV,\sIR>0$. 
\item[(D)] 
If $\rdim-\ooo<\rdim/p$, then $\|\fT \bar K^{\ast\ooo}_\sIR\|_{L^p(\bT)}\lesssim [\sIR]^{-\rdim (p-1)/p}$ uniformly in $\sIR\in(0,1]$.
\end{enumerate}
\end{lem}
\begin{rem}\label{rem:kernel_simple_fact}
Recall that $K_\sIR=\mathring K_\sIR\otimes\bar K_\sIR$ and $\mathring K_\sIR(\mathring x)=\sIR^{-1}\,1_{[0,\infty)}(\mathring x) \exp(-\mathring x/\sIR)$. As a result, Part (C) of the above lemma implies that if $2\ooo>\sigma$, then $$\|\fR^{\phantom\ooo}_\sUV K_\sIR^{\ast\ooo}\|_\cK \lesssim [\sUV\vee\sIR]^{\sigma-\varepsilon}\,[\sIR]^{\varepsilon-\sigma}$$ uniformly in $\sUV,\sIR>0$ and Part (D) implies that if $\rdim-\ooo<\rdim/p$, then $$\|\fT K^{\ast\ooo}_\sIR\|_{L^p(\bH)}\lesssim [\sIR]^{-\rDim (p-1)/p}$$ uniformly in $\sIR\in(0,1]$. Recall also that the kernels $\bar K^{\ast\ooo}_\sIR$ and $ K^{\ast\ooo}_\sIR$ are positive. Hence, $\|\fT |\bar K^{\ast\ooo}_\sIR|\|_{L^p(\bT)}=\|\fT \bar K^{\ast\ooo}_\sIR\|_{L^p(\bT)}$ and $\|\fT |K^{\ast\ooo}_\sIR|\|_{L^p(\bH)}=\|\fT K^{\ast\ooo}_\sIR\|_{L^p(\bH)}$.
\end{rem}
\begin{rem}
Let $f\in C(\bM)$, $\oo\in\bN_+$ and $a\in\frM$ be such that $|a|\leq\oo$. We will often use the fact that $K^{\ast\oo}_{\sIR}\ast f\in C^\oo(\bM)$ and $\|\partial^a K^{\ast\oo}_{\sIR}\ast f\|_{L^\infty(\bM)}\lesssim [\sIR]^{-[a]}\|f\|_{L^\infty(\bM)}$ uniformly over $\sIR\in(0,\infty)$ and $f\in C(\bM)$. 
\end{rem}
\begin{proof}
We use the relations $K_\sIR=\fS_\sIR K$, $J_\sUV=\fS_\sUV J$. Part (A) follows from the fact that $\partial^a K\in \cK$ for $|a|\leq 1$. To prove Part (B) note that $a=(\mathring a,\bar a)\in\bar\frM_\sigma$ iff $\mathring a=0$ and $|\bar a|<\sigma$. Hence, by Remark~\ref{rem:epsilon_sigma} $|\bar a|<\sigma-\varepsilon$ and by Lemma~\ref{lem:kernel_fourier_transform}~(C) $\partial^aJ\in\cK$. Part (C) is a consequence of Lemma~\ref{lem:kernel_fourier_transform}~(D).

Let us turn to the proof of Part (D). First, note that the special case $\ooo=1$ imply the general case by the identity $\fT \bar K^{\ast(\ooo+1)}_\sIR = \fT \bar K^{\ast\ooo}_\sIR\star \fT \bar K_\sIR$ and the Young inequality for convolution, where $\star$ is the convolution in $\bT$. Identifying $\bT$ with $[-\pi,\pi)^\rdim\subset\bR^\rdim$ we get 
\begin{equation}
 \|\fT \bar K_\sIR - \bar K_\sIR\|_{L^p(\bT)}\lesssim [\sIR]^q
\end{equation}
for any $p\in[1,\infty]$ and $q>0$ since $\bar K(\bar x)\lesssim\exp(-|\bar x|)$ uniformly in \mbox{$\bar x\in\bR^\rdim$} for $|\bar x|>1$. Because $|\bar x|^{\rdim-1} \bar K(\bar x)$ is bounded we obtain $\bar K\in L^p(\bR^\rdim)$ for any $p\in[1,\rdim/(\rdim-1))$ and
\begin{equation}
 \|\bar K_\sIR\|_{L^p(\bT)}\leq \|\bar K_\sIR\|_{L^p(\bR^\rdim)} = [\sIR]^{-\rdim (p-1)/p} \|\bar K\|_{L^p(\bR^\rdim)}\lesssim [\sIR]^{-\rdim (p-1)/p}.
\end{equation}
The statement follows from the trivial identity $\fT K_\sIR=K_\sIR + (\fT K_\sIR-K_\sIR)$ and the bounds proved above.
\end{proof}

\begin{lem}\label{lem:kernel_derivative}
For $\varepsilon\in[0,2\wedge\sigma]$ we have $\|\fP_\sIR\partial_\sIR K_\sIR\ast K_\uIR\|_\cK \lesssim [\sIR]^{\varepsilon-\sigma}[\sIR\vee\uIR]^{-\varepsilon}$ uniformly in $\sIR\in(0,1]$ and $\uIR\in[0,1]$.
\end{lem}
\begin{proof}
Let $\dot \fP_\sIR:=\partial_\sIR \fP_\sIR$ and $\dot \fP=\dot \fP_1$. We note that $\fP_\sIR\partial_\sIR K_\sIR \ast K_\uIR= - \dot \fP_\sIR K_\sIR\ast K_\uIR$. The Fourier transform of $\dot\fP_\sIR K$ coincides with
\begin{equation}
 \frac{\rri\mathring p}{1+\rri\mathring p}
 \frac{1+[\sIR]^2|\bar p|^2}{1+|\bar p|^2} 
 +
 \frac{2}{\sigma}\frac{1+\sIR\rri\mathring p}{1+\rri\mathring p}
 \frac{[\sIR]^{2-\sigma}|\bar p|^2}{1+|\bar p|^2}.
\end{equation}
As a result, $\|\dot\fP_\sIR K\|_\cK \lesssim [\sIR]^{2}\vee [\sIR]^{2-\sigma}\vee 1$ uniformly for $\sIR\in(0,\infty)$. Since 
\begin{equation}
 \|\fP_\sIR\partial_\sIR K_\sIR \ast K_\uIR\|_\cK=
 \|\dot \fP_\sIR K_\sIR\ast K_\uIR\|_\cK=
 \|K_\sIR\ast \dot \fP_\sIR K_\uIR\|_\cK\leq 
 \|\dot\fP_\sIR K_{\sIR\vee\uIR}\|_\cK
\end{equation}
the statement follows from the equality $\|\dot\fP_\sIR K_\uv\|_\cK=\uv^{-1}\|\dot\fP_{\sIR/\uv} K\|_\cK$. 
\end{proof}

\begin{rem}
By the above lemma $\|\fP_\sIR\partial_\sIR K_\sIR\|_\cK \lesssim [\sIR]^{-\sigma}$, $\|\partial_\sIR K_\sIR\|_\cK \lesssim [\sIR]^{-\sigma}$ and $\|\partial_\sIR K_\sIR\ast f\|_{L^\infty(\bM)}\lesssim [\sIR]^{-\sigma}\,\|K_\sIR\ast f\|_{L^\infty(\bM)}$ uniformly in $\sIR\in(0,1]$ and $f\in C(\bM)$.
\end{rem}

\begin{lem}\label{lem:bounds_M}
Let $N\in C^\infty_\rc(\bM)$. For $\sIR\in(0,1]$ set $N_\sIR:=\fS_\sIR N$, where the rescaling operator $\fS_\sIR$ was introduced in Def.~\ref{dfn:scaling_S}. For any $a\in\frM$ and $r\in\{0,1\}$:
\begin{enumerate}
\item[(A)] $\|\partial^a\partial_\sIR^r N_\sIR \|_\cK \lesssim [\sIR]^{-[a]-r\sigma}$ uniformly in $\sIR\in(0,1]$,

\item[(B)] $\|\partial_\sIR N_\sIR \ast K_\uIR\|_\cK \lesssim [\sIR]^{\varepsilon-\sigma}\,[\uIR]^{-\varepsilon}$ uniformly in $\sIR,\uIR\in(0,1]$ for all $\varepsilon\in[0,1\wedge\sigma]$.
\end{enumerate}
\end{lem}
\begin{proof}
Part (A) follows from Def.~\ref{dfn:scaling_S} of the map $\fS_\sIR$.  Part (B) for $\sIR>\uIR$ follows from Part (A) and $\|K_\uIR\|_\cK=1$. To prove part (B) for $\sIR\leq\uIR$ it is enough to use the uniform bound $\|\partial^a K_\uIR\|_\cK\lesssim[\uIR]^{-[a]}$ for $a\in\frM$ such that $|a|=1$ and the identity
\begin{equation}
 \partial_\sIR N_\sIR = - \partial_{\mathring x} \fS_\sIR\mathring{N} -\sigma^{-1} [\sIR]^{1-\sigma} \textstyle\sum_{p=1}^\rdim\partial_{x^p} \fS_\sIR\bar{N}^p,
\end{equation}
where $\mathring{N}(\mathring x,\bar x):=\mathring x N(\mathring x,\bar x)$ and $\bar{N}^p(\mathring x,\bar x):=\bar x^p N(\mathring x,\bar x)$ for $p\in\{1,\ldots,\rdim\}$.
\end{proof}

\begin{dfn}\label{dfn:propagator_G}
Fix a function $\chi\in C^\infty(\bR)$ such that such that $\supp\,\chi\subset(1,\infty)$ and $\chi=1$ on some neighborhood of $[2,\infty)$. For $\sIR\in(0,\infty)$ and $a\in\frM$ we define
\begin{equation}
 G_\sIR(\mathring x,\bar x) := \chi(\mathring x/\sIR)\, G(\mathring x,\bar x),
 \qquad
 \dot G_\sIR:=\partial_\sIR G_\sIR,
 \qquad
 \dot G_\sIR^a:=\cX^a \dot G_\sIR,
\end{equation}
where $\cX^a(x):=x^a$ is a polynomial in $x\in\bM$ and the fractional heat kernel $G\in \cK$ is defined by
\begin{equation}
 \fF G(\mathring x,\Cdot)(\bar p) = 1_{[0,\infty)}(\mathring x)\,\exp(-\mathring x\,|\bar p|^\sigma),
 \qquad \mathring x\in\bR.
\end{equation}
We also set $G_0:=G$.
\end{dfn}

\begin{rem}
The kernel $G$ is supported in $[0,\infty)\times\bR^\rdim$ and is smooth outside $\{0\}\times\bR^\rdim$ (outside the origin if $\sigma\in2\bN_+$). We have
\begin{equation}
 \dot G^a_\sIR(\mathring x,\bar x) = -\mathring x/\sIR^2\,\dot \chi(\mathring x/\sIR)\, (\cX^a G)(\mathring x,\bar x),
 \qquad \dot\chi(t):=\partial_t\chi(t).
\end{equation}
Observe that for $\sIR>0$ it holds $\supp\, G_\sIR\subset (\sIR,\infty)\times\bR^\rdim$, $\supp\,\dot G_\sIR\subset (\sIR,2\sIR)\times\bR^\rdim$ and $G_\sIR = \sIR \,\fS_\sIR(G_1)$, $\dot G_\sIR = \fS_\sIR(\dot G_1)$.
\end{rem}

\begin{rem}\label{rem:dot_G_1}
For any $b=(\mathring b,\bar b)\in\frM$ it holds
\begin{equation}
 \fF (\partial^b G)(1,\Cdot)(\bar p) = (-1)^{\mathring b}\,\ri^{|\bar b|}\, |\bar p|^{\sigma\mathring b}\, \bar p^{\bar b} \exp(-|\bar p|^\sigma),
\end{equation}
By Lemma~\ref{lem:kernel_fourier_transform}~(E) $(\partial^b G)(1,\Cdot)\in C^\infty(\bR^\rdim)$ and $(1+|\bar x|^{\sigma+\rdim+|\bar b|})\,|(\partial^b G)(1,\bar x)|\lesssim 1$ uniformly in $\bar x\in\bR^\rdim$. If $\sigma\in2\bN_+$, then $\fF (\partial^b G)(1,\Cdot)\in\sS(\bR^\rdim)$ and consequently $(\partial^b G)(1,\Cdot)\in\sS(\bR^\rdim)$. By the scaling property of the fractional heat kernel we obtain
\begin{equation}
 \partial^b G(x)
 =
 [\mathring x]^{-\rdim-[b]} (\partial^b G)(1,\bar x/[\mathring x]),
 \qquad x=(\mathring x,\bar x)\in(0,\infty)\times \bR^\rdim.
\end{equation}
This shows that the kernel $G$ is smooth outside $\{0\}\times\bR^\rdim$ (outside the origin if $\sigma\in2\bN_+$) and the kernels $\partial^a (G-G_\sIR)$ and $\partial^b\dot G^c_\sIR$ belong to $\cK$ for any $\sIR\in(0,\infty)$, $a\in\bar\frM_\sigma$, $b\in\frM$, $c\in\frM_\sigma$. Moreover, for any $a=(\mathring a,\bar a),b=(\mathring b,\bar b)\in\frM$ the following bound
\begin{equation}\label{eq:G_bound}
 |\partial^b (\cX^a G)(x)|
 \lesssim 
 |\mathring x|^{1+\mathring a-\mathring b} \,|x|_\bM^{|\bar a|-|\bar b|-\sigma-\rdim}
\end{equation}
holds uniformly in $x\in (0,\infty)\times \bR^\rdim$, where $|x|_\bM=[\mathring x]\vee|\bar x|$ and $[\mathring x]=|\mathring x|^{1/\sigma}$.
\end{rem}

\begin{rem}
Recall that $\bar\frM_\sigma$ consists of spatial multi-indices $a$ such that $|a|<\sigma$ and $\frM_\sigma$ consists of space-time multi-indices $a$ such that $|a|\leq\sigma_\diamond$, where $\sigma_\diamond=\sigma$ if $\sigma\in2\bN_+$ and $\sigma_\diamond=\ceil{\sigma}-1$ otherwise. 
\end{rem}

\begin{lem}\label{lem:kernel_dot_G}
For any $\oo\in\bN_0$ and $a\in\frM_\sigma$ it holds
\begin{equation}
 \|\fR^{\phantom\oo}_\sUV\fP_\sIR^\oo\dot G^a_\sIR\|_\cK \lesssim
 [\sUV\vee\sIR]^{\sigma-\varepsilon}\,[\sIR]^{\varepsilon-\sigma+[a]},
\end{equation}
uniformly in $\sUV\in[0,\infty)$ and $\sIR\in(0,\infty)$.
\end{lem}
\begin{proof}  
We first show that
\begin{equation}
 \|\partial^b\dot G^a_\sIR\|_\cK\lesssim [\sIR]^{[a]-[b]}
\end{equation}
for all $a\in\frM_\sigma$ and all $b\in\frM$. Because $\partial^b \dot G_\sIR^a = [\sIR]^{[a]-[b]}\,\fS_\sIR(\partial^b\dot G_1^a)$, where the rescaling operator $\fS_\sUV$ was introduced in Def.~\ref{dfn:scaling_S}, it is enough to prove that $\partial^b(\cX^a\dot G_1)\in L^1(\bM)$ for all $a\in\frM_\sigma$ and all $b\in\frM$, which follows from Remark~\ref{rem:dot_G_1}. Using the above bound as well as the equalities $[b]=\sigma \mathring b+|\bar b|$ for any $b=(\mathring b,\bar b)\in\frM$ and $\fP_\sIR=(1+[\sIR]^\sigma\partial_{\mathring x})(1-[\sIR]^2\Delta_{\bar x})$  we obtain the statement of the lemma with $\sUV=0$. The general case follows from Lemma~\ref{lem:kernel_simple_fact}~(C) and the estimate
\begin{equation}
 \|\fR^{\phantom\oo}_\sUV\fP_\sIR^\oo\dot G^a_\sIR\|_\cK \leq
 \|\fR^{\phantom\oo}_\sUV K_\sIR^{\ast\ooo}\|_\cK\,\|\fP_\sIR^{\oo+\ooo}\dot G^a_\sIR\|_\cK
\end{equation}
applied with $\ooo=\floor{\sigma/2}+1$.
\end{proof}

We conclude this section with a technical lemma that was used several times above.

\begin{lem}\label{lem:kernel_fourier_transform}
Let $\vartheta\in\bR$, $a\in\bN_0^\rdim$, $n,m\in\bN_0$ and $H\in\sS'(\bR^\rdim)$.
\begin{enumerate}
\item[(A)] If $\rdim+|a|+\vartheta>0$ and $\fF H(\bar p)=p^a |\bar p|^\vartheta$, then $|\bar x|^{\rdim+|a|+\vartheta}H(\bar x)$ is smooth and bounded in the domain $\bR^\rdim\setminus\{0\}$.
\item[(B)] If $(-\Delta_{\bar p})^n \fF H(\bar p)$ is integrable, then $|\bar x|^{2n}H(\bar x)$ is bounded and continuous.
\item[(C)] If $|a|<\vartheta$ and $\fF H(\bar p)=p^a/(1+|\bar p|^\vartheta)$, then $H\in L^1(\bR^\rdim)$.
\item[(D)] If $\vartheta\in[0,2)$ and $\fF H(\bar p)=|\bar p|^{2m+\vartheta}/(1+|\bar p|^2)^{m+1}$, then $H\in L^1(\bR^\rdim)$.
\item[(E)] If $\vartheta\in(0,\infty)$ and $\fF H(\bar p)=\bar p^a\, |\bar p|^{\vartheta m}\,\exp(-|\bar p|^\vartheta)$, then $H\in C^\infty(\bR^\rdim)$ and $|\bar x|^{\rdim+|a|+\vartheta}H(\bar x)$ is bounded.
\end{enumerate}
\end{lem}
\begin{proof}
Parts (A) and (B) are well-known facts.

To prove Part (C) first note that if $2\vartheta>\rdim+2|a|$, then $\fF H\in L^2(\bR^\rdim)$ and $H$ is locally integrable. Otherwise let $k$ be the largest integer such that $\vartheta k<\rdim+|a|$. Since $2\vartheta\leq \rdim+2|a|$ and $\vartheta>|a|$ it holds $\vartheta<\rdim +|a|$ and $k\in\bN_+$. Observe that
\begin{equation}
 \fF H(\bar p)
 =
 \sum_{l=1}^k \frac{(-1)^{l-1}\bar p^a }{|\bar p|^{\vartheta l}}
 + \frac{(-1)^{k} \bar p^a}{|\bar p|^{\vartheta k}(1+|\bar p|^\vartheta)}.
\end{equation}
Using Part (A) and noting that the last term in the above decomposition is in $L^2(\bR^\rdim)$ we infer that $H$ is locally integrable. If $\vartheta$ is rational, then let $k$ be the smallest positive integer such that $\vartheta k\in2\bN_+$ and let $n\in\bN_+$ be such that $\vartheta (k-1)+|a|<2n-\rdim$. Otherwise let $n$ be the smallest positive integer such that $|a|<2n-\rdim$ and let $k\in\bN_+$ be such that $\vartheta (k-1)+|a|<2n-\rdim<\vartheta k+|a|$. Using the identity
\begin{equation}
\fF H(\bar p)
= \sum_{l=0}^{k-1}~(-1)^l \bar p^a
\,|\bar p|^{\vartheta l} +  \frac{(-1)^k\,\bar p^a\,|\bar p|^{\vartheta k}}{1+|\bar p|^{\vartheta}}
\end{equation}
as well as Parts (A) and (B) with $n$ as above we prove that $H$ is integrable outside some neighbourhood of the origin. This completes the proof of Part (C).

It is enough to prove Part (D) with $m=0$ since the generalization to arbitrary $m$ is trivial. If $4>\rdim+2\vartheta$, then $\fF H\in L^2(\bR^\rdim)$ and $H$ is locally integrable. Otherwise let $k$ be the biggest positive integer such that $2 k< \rdim+\vartheta$. Observe that
\begin{equation}
 \fF H(\bar p) =
 \sum_{l=1}^k \frac{(-1)^{l-1}|\bar p|^\vartheta }{|\bar p|^{2l}}
 + \frac{(-1)^k |\bar p|^\vartheta}{|\bar p|^{2k}(1+|\bar p|^2)}.
\end{equation}
Using Part (A) and noting that the last term in the above decomposition is in $L^2(\bR^\rdim)$ we infer that $H$ is locally integrable. Next, we observe
\begin{equation}
 \frac{|\bar p|^\vartheta}{1+|\bar p|^2}
 = |\bar p|^\vartheta
 -
 \frac{|\bar p|^{2+\vartheta}}{(1+|\bar p|^2)(1+|\bar p|^\vartheta)}
 +
 \frac{|\bar p|^{2+2\vartheta}}{(1+|\bar p|^2)(1+|\bar p|^\vartheta)}.
\end{equation}
To conclude we apply Part (A) to the first term of the above decomposition and Part (B) to the last two terms with $n$ being the smallest integer such that $\rdim<2n$ or $\rdim+\vartheta<2n$, respectively.

In order to prove Part (E) we first note that $H$ is clearly smooth. Next, we observe that
\begin{equation}
 \fF H(\bar p)=\sum_{l=0}^{k-1}~(-1)^l p^a |\bar p|^{\vartheta(l+m)}/l! + p^a\,|\bar p|^{\vartheta(k+m)}\, f(|\bar p|^\vartheta),
\end{equation}
where $f(t)=t^{-k}(\exp(-t)-\sum_{l=0}^{k-1}~(-t)^l/l!)$ is a smooth function of fast decay, $n$ is the smallest positive integer such that $\vartheta+|a|\leq 2n-\rdim$ and $k$ is the smallest positive integer such that $\vartheta (k+m)+|a|> 2n-\rdim$. Part (E) follows now from Part (A) and Part (B) with $n$ as above.
\end{proof}

\section{Function spaces and useful maps}\label{sec:topology}

In this section we introduce normed vector spaces for the effective force coefficients and their cumulants. We also define the maps $\fA$ and $\fB$ used to write various flow equations that will appear in the paper. We discuss the most important properties of the above-mentioned spaces and maps. We start with the definitions of Besov-type spaces.

\begin{dfn}\label{dfn:besov_T}
For $\alpha<0$ and $\phi\in C^\infty(\bT)$ we define
\begin{equation}
 \|\phi\|_{\sC^\alpha(\bT)}:= \sup_{\sIR\in(0,1]} [\sIR]^{-\alpha}\,
 \|\bar K_\sIR^{\ast\oo}\ast \phi\|_{L^\infty(\bR^\rdim)},
\end{equation}
where $\oo=\ceil{-\alpha}$. By definition the space $\sC^\alpha(\bT)$ is the completion of $C^\infty(\bT)$ equipped with the norm $\|\Cdot\|_{\sC^\alpha(\bT)}$. 
\end{dfn}
\begin{rem}
The space $\sC^\alpha(\bT)$ is a separable subspace of the standard Besov space. 
\end{rem}

\begin{dfn}
Let $\mathtt{w}\in C^\infty_\rc(\bR)$. For $v\in C(\bM)$ we set $(\mathtt{w}v)(x):=\mathtt{w}(\mathring x)\,v(x)$, where $x=(\mathring x,\bar x)\in\bM=\bR\times\bR^\rdim$. The above definition is extended to $v\in\sD'(\bM)$ in a natural way.
\end{dfn}

\begin{dfn}
For $\alpha<0$ and $f\in C^\infty(\bH)$ we define
\begin{equation}
 \|f\|_{\sC^\alpha(\bH)}:= \sup_{\sIR\in(0,1]} [\sIR]^{-\alpha}\,
 \|K_\sIR^{\ast\oo}\ast f\|_{L^\infty(\bH)},
\end{equation}
where $\oo=\ceil{-\alpha}$. Fix some sequence $\mathtt{w}_n\in C^\infty_\rc(\bR)$, $n\in\bN_+$, such that $\mathtt{w}_n(\mathring x)=1$ for $\mathring x\in[-n,n]$. By definition the space $\sC^\alpha(\bH)$ is a Fr{\'e}chet space obtained by the completion of $C^\infty(\bH)$ equipped with the family of seminorms $\|\mathtt{w}_n\Cdot\|_{\sC^\alpha(\bH)}$, $n\in\bN_+$.
\end{dfn}
\begin{rem}
Note that for every $f\in \sC^\alpha(\bH)$ and every $\mathtt{w}\in C^\infty_\rc(\bR)$ the distribution $\mathtt{w}f$ belongs to the standard parabolic Besov space. 
\end{rem}

\begin{dfn}\label{dfn:C_gamma}
For $\gamma\in(0,\infty)$ and $f\in C^\infty(\bT)$ we define
\begin{equation}
 \|\phi\|_{\cC^\gamma(\bT)}:= \|\fR \phi\|_{L^\infty(\bT)},
\end{equation}
where the operator $\fR$ is given by $\fF(\fR\phi)(\bar p) := (1+|\bar p|^\gamma)\fF\phi(\bar p)$. By definition the space $\cC^\gamma(\bT)$ is the completion of $C^\infty(\bT)$ equipped with the norm $\|\Cdot\|_{\cC^\gamma(\bT)}$. 
\end{dfn}

\begin{dfn}\label{dfn:C_gamma_time}
For $t_1,t_2\in\bR$, $t_1<t_2$, $\gamma\in(0,\infty)$ and $f\in C([t_1,t_2],\cC^\gamma(\bT))$ we define
\begin{equation}
 \|\varphi\|_{C([t_1,t_2],\cC^\gamma(\bT))}:= \sup_{\mathring x\in[t_1,t_2]}\|\varphi(\mathring x,\Cdot)\|_{\cC^\gamma(\bT)}.
\end{equation} 
\end{dfn}

\begin{rem}
Recall that $\varepsilon_\sigma:=\sigma+1-\ceil{\sigma}>0$ and $\varepsilon\in(0,\varepsilon_\sigma)$. In what follows, we always assume that $\gamma=\gamma_\varepsilon=\sigma-\varepsilon$. Note that the definition of the pseudo-differential operator $\fR$ given above coincides with the one in Def.~\ref{dfn:diff_op} and by Lemma~\ref{lem:kernel_fourier_transform}~(C) it holds $\cC^\gamma(\bT)\subset C^{\floor{\gamma}}(\bT)$, where $\floor{\gamma}=\ceil{\sigma}-1$.
\end{rem}

We now proceed to the definition of the spaces $\cV^{\mathsf{m}}$, \mbox{$\mathsf{m}=(m_1,\ldots,m_n)\in\bN_0^n$}, $n\in\bN_+$. If $n=1$, we write $\mathsf{m}=m\in\bN_0$ and $\cV^{\mathsf{m}}=\cV^m$. Let us mention that we shall use the norm $\|\Cdot\|_{\cV^m}$ to bound the effective force coefficients $F^{i,m}_{\sUV,\sIR}$ and the norm $\|\Cdot\|_{\cV^{\mathsf{m}}}$ to bound the cumulants of the effective force coefficients. We explain in more detail how these norms will be used in Remark~\ref{rem:spaces_motivation} below.

\begin{dfn}
Let $\bY$ be a topological space. We define $\sM(\bY)$ to be the space of signed Radon measures $V$ on $\bY$ with finite total variation $|V|$ equipped with the following norm
\begin{equation}
 \|V\|_{\sM(\bY)}:=|V|(\bY)=\int_\bY|V(\rd y)|.
\end{equation}
We make the following identification $\sM(\{0\})=\bR$.
\end{dfn}

\begin{dfn}\label{dfn:map_U}
Let $n\in\bN_+$, $\mathsf{m}=(m_1,\ldots,m_n)\in\bN_0^n$, $m=m_1+\ldots+m_n$ and $V\,:\,\bM^n\to\sM(\bM^{m})$. The map $\fU^{\mathsf{m}} V\,:\,\bM^n\to\sM(\bM^{m})$ is defined by the equality
\begin{equation}
 \fU^{\mathsf{m}} V(x_1,\ldots,x_n;A_1\times\ldots\times A_n)
 :=
 V(x_1,\ldots,x_n;(A_1+x_1)\times \ldots \times(A_n+x_n))
\end{equation}
for all Borel sets $A_q\subset\bM^{m_q}$, $q\in\{1,\ldots,n\}$, where
\begin{equation}
 A+x:=\{(y_1+x,\ldots,y_m+x)\,|\,(y_1,\ldots,y_m)\in A\}
\end{equation}
for any $A\subset\bM^m$ and $x\in\bM$. If $n=1$ we write $\fU^{\mathsf{m}}=\fU^m$.
\end{dfn}

\begin{rem}\label{rem:map_U}
The equation defining $\fU^{\mathsf{m}}V$ can be rewritten as
\begin{equation}
 \fU^{\mathsf{m}} V(x_1,\ldots,x_n;\rd \ry_1\ldots\rd \ry_n)
 =
 V(x_1,\ldots,x_n;\rd (\ry_1+x_1)\ldots\rd(\ry_n+x_n)),
\end{equation}
where we use the notation $\ry+x:=(y_1+x,\ldots,y_n+x)\in\bM^m$ for arbitrary $x\in\bM$ and $\ry=(y_1,\ldots,y_m)\in\bM^m$, $m\in\bN_0$.
\end{rem}

\begin{dfn}\label{dfn:cV}
Let $n\in\bN_+$, \mbox{$\mathsf{m}=(m_1,\ldots,m_n)\in\bN_0^n$} and $m=m_1+\ldots+m_n$. The vector space $\cV^{\mathsf{m}}$ consists of maps $V\,:\,\bM^n\to\sM(\bM^{m})$ such that the map $\fU^{\mathsf{m}} V$ belongs to $C(\bH^n,\sM(\bM^{m}))$ and the norm
\begin{equation}\label{eq:cVt_norm}
 \|V\|_{\cV^{\mathsf{m}}}
 :=
 \sup_{x_1\in\bM} \int_{\bH^{n-1}}\int_{\bM^m}
 |V(x_1,\ldots,x_n;\rd y_1\ldots\rd y_m)|\,\rd x_2\ldots\rd x_n
\end{equation}
is finite. The vector space $\cV^{\mathsf{m}}_\rt$ consists of $V\in\cV^{\mathsf{m}}$ such that $\fU^{\mathsf{m}}V(x_1,\ldots,x_n)=\fU^{\mathsf{m}} V(x_1+z,\ldots,x_n+z)$ for all $x_1,\ldots,x_n,z\in\bM$. If $n=1$ we write $\mathsf{m}=m\in\bN_0$ and $\cV^{\mathsf{m}}=\cV^m$, $\cV^{\mathsf{m}}_\rt=\cV^m_\rt$.
\end{dfn}

\begin{rem}\label{rem:cV}
$V\in C(\bH^n,\sM(\bM^{m}))$ iff $V\,:\,\bM^n\to\sM(\bM^{m})$ is continuous and for all $x_p,z_p\in\bM$ such that $\mathring x_p=\mathring z_p$ and $\bar x_p-\bar z_p\in(2\pi\bZ)^\rdim$, $p\in\{1,\ldots,n\}$ and all Borel sets $A_q\subset\bM^{m_q}$, $q\in\{1,\ldots,n\}$ we have
\begin{equation}
 V(x_1,\ldots,x_n;A_1\times\ldots\times A_n)
 =
 V(z_1,\ldots,z_n;A_1\times\ldots\times A_n).
\end{equation} 
\end{rem}

\begin{rem}
Let $V\in\cV^{\mathsf{m}}$. The equality
\begin{equation}
 \langle \tilde V,\psi\rangle = \int_{\bM^{n}}\int_{\bM^{m}} V(x_1,\ldots,x_n;\rd y_1\ldots\rd y_m)\,\psi(x_1,\ldots,x_n,y_1,\ldots,y_m)\,\rd x_1\ldots\rd x_n
\end{equation}
for all $\psi\in C_\rc^\infty(\bM^n\times\bM^m)$ defines a distribution $\tilde V\in \sD'(\bM^n\times\bM^m)$. In what follows, we implicitly identify $V\in\cV^{\mathsf{m}}$ with $V\in\sD'(\bM^n\times\bM^m)$ and view the space $\cV^{\mathsf{m}}$ as a subset of $\sD'(\bM^n\times\bM^m)$. We use this identification to define $\supp\, V\subset \bM^n\times\bM^m$. We warn the reader that in what follows $\supp\,V$ does not coincide with the support of the map $V\,:\,\bM^n\to\sM(\bM^{m})$.
\end{rem}

\begin{dfn}
Let $k\in\bN_+$ and $V\in\sD'(\bM^k)$. We say that $V$ is translationally invariant iff
\begin{equation}
 \langle V,\psi_1\otimes\ldots\otimes\psi_k\rangle
 =
 \langle V,\psi_1(\Cdot+x)\otimes\ldots\otimes\psi_k(\Cdot+x)\rangle
\end{equation}
for any $x\in\bM$ and $\psi_1,\ldots,\psi_k\in C^\infty_\rc(\bM)$.
\end{dfn}
\begin{rem}
Note that $V\in \cV^{\mathsf{m}}$ is an element of $\cV^{\mathsf{m}}_\rt$ iff $V$ is translationally invariant. For $V\in\cV^{\mathsf{m}}_\rt$ it holds
\begin{equation}\label{eq:cV_norm}
 \|V\|_{\cV^{\mathsf{m}}}
 =
 \frac{1}{n!}\sum_{\pi\in\cP_n}\sup_{x_1\in\bH} \int_{\bH^{n-1}}\int_{\bM^m}
 |V(x_{\pi(1)},\ldots,x_{\pi(n)};\rd y_1\ldots\rd y_m)|\,\rd x_2\ldots\rd x_n,
\end{equation}
where $\cP_n$ is the group of permutations of $\{1,\ldots,n\}$.
\end{rem}

\begin{rem}
Let us note that in our analysis we will only use the norm $\|V\|_{\cV^{\mathsf{m}}}$ in a situation when $V\in\cV^{\mathsf{m}}_\rt$ is translationally invariant, or $n=1$ and $\mathsf{m}=m\in\bN_0$ and $V\in\cV^m$. For this reason the fact that the norm $\|V\|_{\cV^{\mathsf{m}}}$ does not treat $x_1$ and $x_2,\ldots,x_n$ on equal footing is innocuous.
\end{rem}

%
%

\begin{dfn}\label{dfn:kernels_K_n_m}
For $n\in\bN_+$, $m\in\bN_0$, $\oo\in\bN_0$ and $\sUV,\sIR\in[0,\infty)$ the kernel $K^{n,m;\oo}_{\sUV,\sIR}\in\cK^{n+m}$ is defined by
\begin{equation}
 K^{n,m;\oo}_{\sUV,\sIR}:= (K^{\ast\oo}_\sIR)^{\otimes n}\otimes (K^{\ast\oo}_\sIR\ast J_\sUV^{\phantom{\oo}})^{\otimes m}.
\end{equation}
We also set $K^{m;\oo}_{\sUV,\sIR}:=K^{1,m;\oo}_{\sUV,\sIR}\in\cK^{1+m}$. We omit the indices $\sUV,\sIR$ if $\sUV=\sIR=1$.
\end{dfn}

\begin{dfn}
Let $n\in\bN_+$, $\mathsf{m}=(m_1,\ldots,m_n)\in\bN_0^n$, $m=m_1+\ldots+m_n$. For $\mathtt{w}\in C^\infty(\bR^n)$ and $V\in \sD'(\bM^n\times\bM^m)$ we define $\mathtt{w}V\in \sD'(\bM^n\times\bM^m)$ by the equality $\langle \mathtt{w}V,\psi\rangle:=\langle V,\mathtt{w}\psi\rangle$ for all $\psi\in C_\rc^\infty(\bM^n\times\bM^m)$, where
\begin{equation}
 (\mathtt{w}\psi)(x_1,\ldots,x_n;y_1,\ldots,y_m):=\mathtt{w}(\mathring x_1,\ldots,\mathring x_n)\,\psi(x_1,\ldots,x_n;y_1,\ldots,y_m).
\end{equation}
For $V\in \sD'(\bM^n\times\bM^m)$ we define $\fU^{\mathsf{m}} V\in \sD'(\bM^n\times\bM^m)$ by the equality $\langle \fU^{\mathsf{m}}V,\psi\rangle:=\langle V,\fU^{\mathsf{m}\dagger}\psi\rangle$ for all $\psi\in C_\rc^\infty(\bM^n\times\bM^m)$, where 
\begin{equation}
 (\fU^{\mathsf{m}\dagger}\psi)(x_1,\ldots,x_n;\ry_1,\ldots,\ry_n):=
 \psi(x_1,\ldots,x_n;\ry_1-x_1,\ldots,\ry_n-x_n)
\end{equation}
and we used the notation introduced in Remark~\ref{rem:map_U}.
\end{dfn}

\begin{dfn}\label{dfn:cD}
Let $n\in\bN_+$, $\mathsf{m}=(m_1,\ldots,m_n)\in\bN_0^n$, $m=m_1+\ldots+m_n$. For $\oo\in\bN_0$ the space $\cD^{\mathsf{m};\oo}$ consists of distributions $V\in \sD'(\bM^n\times\bM^m)$ such that $K^{n,m;\oo}\ast \mathtt{w}V\in\cV^{\mathsf{m}}$ for all $\mathtt{w}\in C^\infty_\rc(\bR)$ and $\supp\,\fU^{\mathsf{m}} V \subset\bM^n \times ([-T,T]\times\bR^\rdim)^m$ for some $T\in(0,\infty)$. If $n=1$ we set $\cD^{\mathsf{m};\oo}=\cD^{m;\oo}$. If $n=1$ and $m=0$ we set $\cD^{0;\oo}=\cD^{;\oo}$. The space $\cD^{\mathsf{m}}\subset \sD'(\bM^n\times\bM^m)$ is the union of the spaces $\cD^{\mathsf{m};\oo}$, $\oo\in\bN_0$. The spaces $\cD^m$ and $\cD$ are defined analogously. The spaces $\cD^{\mathsf{m};\oo}_\rt$, $\cD^{m;\oo}_\rt$, $\cD^{\oo}_\rt$, $\cD^{\mathsf{m}}_\rt$, $\cD^m_\rt$, $\cD_\rt$ consists of translationally invariant elements of the respective spaces defined above.
\end{dfn}

\begin{rem}\label{rem:Vm_K}
For $V\in\cV^\mathsf{m}$ and $K\in\cK^{n+m}$ it holds $\|K\ast V\|_{\cV^\mathsf{m}}\leq \|K\|_{\cK^{n+m}} \|V\|_{\cV^\mathsf{m}}$.
\end{rem}

\begin{rem}
Observe that $\cD^{;0}=C(\bH)$ and $\cV=C_\rb(\bH)$. Moreover, we have $\|v\|_{\cV}=\|v\|_{L^\infty(\bM)}$ for any $v\in\cV$. Using the fact that the kernel $K$ is the inverse of a differential operator it is easy to see that $\cD=\sD'(\bH)$. Furthermore, $\cV_\rt=\cD_\rt\simeq \bR$. We also note that for $V\in \sD'(\bM^n\times\bM^m)$ such that $K^{n,m;\oo}\ast V\in\cV^{\mathsf{m}}$ we have $K^{n,m;\oo}\ast \mathtt{w}V\in\cV^{\mathsf{m}}$ for all $\mathtt{w}\in C^\infty_\rc(\bR)$. The last claim is proved by induction on $\oo\in\bN_0$ (see related Lemma~\ref{lem:time_estimates_prep}).
\end{rem}

\begin{rem}\label{rem:spaces_motivation}
Recall that the kernels $K$ and $J$ as well as the kernels $K_\sIR$ and $J_\sIR$ obtained by rescaling $K$ and $J$ are not smooth. However, $\|\partial^a K_\sIR^{\ast\oo}\|\lesssim [\sIR]^{-[a]}$ for every $a\in\frM$ such that $|a|\leq \oo$ and $\|\partial^a J_\sUV\|_\cK\lesssim [\sUV]^{-[a]}$ for every $a\in\bar\frM_\sigma$. As a result, the function obtained by convolving $V\in C(\bM^n\times\bM^m)$ with the kernel $K^{n,m;\oo}_{\sUV,\sIR}$ is more regular but not smooth. Let us mention that we shall frequently bound expressions of the form $\|K^{n,m;\oo}_{\sUV,\sIR}\ast V\|_{\cV^{\mathsf{m}}}$ with the parameter $\oo\in\bN_0$ varying from bound to bound. Note that $\|K^{n,m;\oo+\ooo}_{\sUV,\sIR}\ast V\|_{\cV^{\mathsf{m}}}\leq \|K^{n,m;\oo}_{\sUV,\sIR}\ast V\|_{\cV^{\mathsf{m}}}$ for any $\oo,\ooo\in\bN_0$ by Remark~\ref{rem:Vm_K}. For this reason, the precise value of $\oo\in\bN_0$ is usually of little importance. Note also that it holds
\begin{equation}
 K^{n,m;\oo}_{\sUV,\sIR}:= K^{\ast\oo,\otimes (n+m)}_\sIR \ast  (\delta_\bM^{\otimes n} \otimes J_\sUV^{\otimes m}).
\end{equation}
In particular, if $\oo=0$ or $\sIR=0$ we have
\begin{equation}
 K^{n,m;\oo}_{\sUV,\sIR}:= \delta_\bM^{\otimes n} \otimes J_\sUV^{\otimes m}.
\end{equation}
As we will see, for any $\sUV\in(0,1]$ the effective force coefficients $F^{i,m}_{\sUV,\sIR}$ belong to the space $\cD^{m;0}$. We will also show that they can be uniformly bounded in the space $\cD^{m;\oo}$ for sufficiently big $\oo\in\bN_0$. More specifically, we will prove that for any $\mathtt{w}\in C_\rc^\infty(\bR)$ the effective force coefficients $F^{i,m}_{\sUV,\sIR}$ almost surely satisfy bounds of the type
\begin{equation}
 \|K^{m;\oo}_{\sUV,\sIR} \ast \mathtt{w} F^{i,m}_{\sUV,\sIR}\|_{\cV^m} \lesssim [\sUV\vee\sIR]^{\varrho_\varepsilon(i,m)}
\end{equation}
uniformly in $\sUV\in(0,1]$ and $\sIR\in[0,1]$. This bound implies, in particular, that the force coefficients $F^{i,m}_{\sUV}=F^{i,m}_{\sUV,0}$ satisfy the bound
\begin{equation}
 \|(\delta_\bM^{\otimes n} \otimes J_\sUV^{\otimes m})\ast \mathtt{w} F^{i,m}_{\sUV}\|_{\cV^m} \lesssim [\sUV]^{\varrho_\varepsilon(i,m)}
\end{equation}
uniformly in $\sUV\in(0,1]$. If the force $F_\sUV[\varphi]$ depends on derivatives of $\varphi$, the validity of the above bound depends crucially on the presence of the kernel $J_\sUV$ and the estimate $\|\partial^a J_\sUV\|_\cK\lesssim [\sUV]^{-[a]}$ for every $a\in\bar\frM_\sigma$ (the restriction to $a\in\bar\frM$ such that $|a|<\sigma$ is related to the assumption that the SPDE under consideration is semi-linear). Finally, we remark that we will prove that the cumulants $\llangle F^{i_1,m_1}_{\sUV,\sIR},\ldots,F^{i_n,m_n}_{\sUV,\sIR} \rrangle$ of the effective force coefficients at stationarity belong to $\cD^{\mathsf{m};0}_\rt$, where $\mathsf{m}=(m_1,\ldots,m_n)$, and can be uniformly bounded in the space $\cD^{\mathsf{m};\oo}_\rt$ for sufficiently big $\oo\in\bN_0$. More precisely, we shall establish bounds of the type
\begin{equation}
 \|K^{n,m;\oo}_{\sUV,\sIR}\ast \llangle F^{i_1,m_1}_{\sUV,\sIR},\ldots,F^{i_n,m_n}_{\sUV,\sIR} \rrangle\|_{\cV^{\mathsf{m}}} \lesssim [\sUV\vee\sIR]^{\varrho_\varepsilon(i_1,m_1)+\ldots+\varrho_\varepsilon(i_n,m_n)+(n-1)\rDim}
\end{equation}
uniform in $\sUV\in(0,1]$ and $\sIR\in[0,1]$. 
\end{rem}

\begin{dfn}\label{dfn:map_Y}
Let $n\in\bN_+$, $\mathsf{m}=(m_1,\ldots,m_n)\in\bN_0^n$, $m=m_1+\ldots+m_n$. For a permutation $\pi\in\cP_{m_1}$ we define the map $\fY_\pi\,:\,\cD^{\mathsf{m}}\to \cD^{\mathsf{m}}$ by
\begin{equation}
 \big\langle \fY_\pi(V),\botimes_{q=1}^n \psi_q\otimes
 \botimes_{q=1}^m \varphi_q\big\rangle
 :=
 \big\langle V,\botimes_{q=1}^n \psi_q\otimes
 \botimes_{q=1}^{m_1} \varphi_{\pi(q)}\otimes
 \botimes_{q=m_1+1}^{m} \varphi_q \big\rangle,
\end{equation}
where $\psi_1,\ldots,\psi_n,\varphi_1,\ldots,\varphi_m\in\ C^\infty_\rc(\bM)$. For a permutation $\omega\in\cP_{n}$ we define the map $\fY^\omega\,:\,\cD^{\mathsf{m}}_\rt\to \cD^{\omega(\mathsf{m})}_\rt$, where $\omega(\mathsf{m}):=(m_{\omega(1)},\ldots,m_{\omega(n)})$, by the equality
\begin{equation}
 \big\langle \fY^\omega(V),\botimes_{q=1}^n \psi_q\otimes
 \botimes_{q=1}^n \varphi_q\big\rangle
 :=
 \big\langle V,\botimes_{q=1}^n \psi_{\omega(q)}\otimes
 \botimes_{q=1}^{n} \varphi_{\omega(q)}\big\rangle,
\end{equation}
where $\psi_q\in\ C^\infty_\rc(\bM)$, $\varphi_q\in C^\infty_\rc(\bM^{m_q})$ for $q\in\{1,\ldots,n\}$.
\end{dfn}
\begin{rem}\label{rem:translational_permutation}
The maps $\fY_\pi:\cV^{\mathsf{m}}\to\cV^{\mathsf{m}}$ and $\fY^\omega:\cV_\rt^{\mathsf{m}}\to\cV_\rt^{\omega(\mathsf{m})}$, defined analogously to the above definition, are bounded with norm one.
\end{rem}

\begin{dfn}
The vector space $\cG$ consists of kernels $G\in C^\infty(\bM)$ such that $\supp\,G\subset [-T,T]\times\bR^\rdim\subset\bM$ for some $T\in(0,\infty)$ and  $\partial^a G\in L^1(\bM)$, $\fT |\partial^a G|\in L^\infty(\bH)$ for any $a\in\frM$, where $\fT$ denotes the periodization introduced in Def.~\ref{dfn:periodization}.
\end{dfn}

\begin{dfn}\label{dfn:maps_A_B}
Fix $n\in\bN_0$, $\hat n\in\{1,\ldots,n\}$, $m_1,\ldots,m_{n+1}\in\bN_0$ and
\begin{equation}
\begin{gathered}
 \mathsf{m}=(m_1+m_{n+1},m_2,\ldots,m_n)\in\bN_0^n,
 \qquad
 \tilde{\mathsf{m}}=(m_1+1,m_2,\ldots,m_{n+1})\in\bN_0^{n+1},
 \\
 \hat{\mathsf{m}}=(m_1+1,m_2,\ldots,m_{\hat n})\in\bN_0^{\hat n},
 \qquad
 \check{\mathsf{m}}=(m_{\hat n+1},\ldots,m_{n+1})\in\bN_0^{n-\hat n+1}.
\end{gathered} 
\end{equation} 
The bilinear map $\fA\,:\,\cG\times\cV^{\tilde{\mathsf{m}}}\to\cV^{\mathsf{m}}$ is defined by
\begin{multline}\label{eq:fA_dfn}
 \fA(G,V)(x_1,\ldots,x_n;\rd\ry_1\rd\ry_{n+1}\rd\ry_2\ldots\rd\ry_n)
 \\
 :=
 \int_{\bM^2} V(x_1,\ldots,x_{n+1};\rd y\rd\ry_1\ldots\rd\ry_{n+1})\,G(y-x_{n+1})\,\rd x_{n+1}
\end{multline}
for every $G\in\cG$ and $V\in\cV^{\tilde{\mathsf{m}}}$. The trilinear map $\fB\,:\,\cG\times \cV^{\hat{\mathsf{m}}}\times\cV^{\check{\mathsf{m}}}\to\cV^{\mathsf{m}}$ is defined by
\begin{multline}\label{eq:fB_dfn}
 \fB(G,W,U)(x_1,\ldots,x_{n};\rd\ry_1\rd\ry_{n+1}\rd\ry_2\ldots\rd \ry_{n})
 \\
 :=
 \int_{\bM^2} W(x_1,\ldots,x_{\hat n};\rd y\rd\ry_1\ldots\rd\ry_{\hat n})
 \,G(y-x_{n+1})
 \\
 \times 
 U(x_{n+1},x_{\hat n+1},\ldots,x_{n};\rd\ry_{n+1}\rd\ry_{\hat n+1}\ldots\rd\ry_{n})\,\rd x_{n+1}
\end{multline}
for every $G\in\cG$ and $U\in\cV^{\hat{\mathsf{m}}}$, $W\in\cV^{\check{\mathsf{m}}}$. In the above equations $\ry_j\in\bM^{m_j}$, $j\in\{1,\ldots,n+1\}$ and the integral is over $x_{n+1},y\in\bM$.
\end{dfn}

\begin{rem}
The RHS of the flow equation for the effective force coefficients $F^{i,m}_{\sUV,\sIR}$ involves the map $\fB$. Both of the maps $\fA$ and $\fB$ appear in the flow equation for the cumulants of the  effective force coefficients. 
\end{rem}

\begin{lem}\label{lem:fA_fB_bounds}
Fix some $n\in\bN_0$, $\hat n\in\{1,\ldots,n\}$ and $\mathsf{m}\in\bN_0^n$, $\tilde{\mathsf{m}}\in\bN_0^{n+1}$, $\hat{\mathsf{m}}\in\bN_0^{\hat n}$,
$\check{\mathsf{m}}\in\bN_0^{n-\hat n+1}$ as in Def.~\ref{dfn:maps_A_B}. Let $G\in\cG$, $V\in\cV^{\tilde{\mathsf{m}}}$, $W\in\cV^{\hat{\mathsf{m}}}$ and $U\in\cV^{\check{\mathsf{m}}}$. Then
\begin{equation}\label{eq:fA_ieq}
 \|\fA(G,V)\|_{\cV^{\mathsf{m}}}
 \leq
 \|\fT |G|\|_\cV\,
 \|V\|_{\cV^{\tilde{\mathsf{m}}}},
\end{equation}
\begin{equation}\label{eq:fB_ieq}
 \|\fB(G,W,U)\|_{\cV^{\mathsf{m}}}
 \leq
 \|G\|_\cK\,
 \|W\|_{\cV^{\check{\mathsf{m}}}}\, 
 \|U\|_{\cV^{\hat{\mathsf{m}}}},
\end{equation}
where $\|\fT |G|\|_\cV=\|\fT |G|\|_{L^\infty(\bH)}$ and $\|G\|_\cK=\|G\|_{L^1(\bM)}$. If $V,W,U$ are translationally invariant, then $\fA(G,V)$ and $\fB(G,W,U)$ are also translationally invariant.
\end{lem}
\begin{proof}
It is easy to see that the functions $\fU^{\mathsf{m}}\fA(G,V)$, $\fU^{\mathsf{m}}\fB(G,W,U)$ are continuous and $2\pi$ periodic in space. It is also straightforward to check that $\fA(G,V)$, $\fB(G,W,U)$ are translationally invariant if $V$, $W$, $U$ are translationally invariant. To prove the first bound note that
\begin{multline}
 \|\fA(G,V)\|_{\cV^{\mathsf{m}}}
 \leq 
 \sup_{x_1\in\bH}\int_{\bH^{n-1}} \int_\bM \int_{\bM^{m+1}} |G(y+x_1-z)|
 \\
 \times |\fU^{\mathsf{m}}V(x_1,\ldots,x_{n},z; \rd y\rd \ry_1\ldots\rd\ry_{n+1})|\,\rd z\,~\rd x_2\ldots\rd x_n.
\end{multline}
Using periodicity of $\fU^{\mathsf{m}}V$ we arrive at
\begin{multline}
 \|\fA(G,V)\|_{\cV^{\mathsf{m}}}\leq\sup_{x_1\in\bH}\int_{\bH^n} \int_{\bM^{m+1}} (\fT|G|)(y+x_1-z)
 \\
 \times |\fU^{\mathsf{m}}V(x_1,\ldots,x_{n},z; \rd y\rd \ry_1\ldots\rd\ry_{n+1})|\,\rd z\rd x_2\ldots\rd x_{n}.
\end{multline}
This implies $\|\fA(G,V)\|_{\cV^{\mathsf{m}}}\leq \|\fT G\|_{\cV}\, \|\fU^{\mathsf{m}}V\|_{\cV^{\tilde{\mathsf{m}}}}= \|\fT G\|_{\cV}\, \|V\|_{\cV^{\tilde{\mathsf{m}}}}$, which proves the first bound. The second bound follows from the estimate
\begin{multline}
 \|\fB(G,W,U)\|_{\cV^{\mathsf{m}}}
 \leq
 \sup_{x_1\in\bH}\int_{\bH^{n-1}} \int_{\bM^{m+2}} |W(x_1,\ldots,x_{\hat n};\rd y\rd\ry_1\ldots\rd\ry_{\hat n})|
 \\
 \times |G(\rd z)|\,
 |U(z-y,x_{\hat n+1},\ldots,x_{n};\rd\ry_{n+1}\rd\ry_{\hat n+1}\ldots\rd\ry_{n})|\,\rd x_2\ldots\rd x_{n}.
\end{multline}
This finishes the proof.
\end{proof}


\begin{lem}\label{lem:fA_fB_Ks}
Fix some $n\in\bN_0$, $\hat n\in\{1,\ldots,n\}$ and $\mathsf{m}\in\bN_0^n$, $\tilde{\mathsf{m}}\in\bN_0^{n+1}$, $\hat{\mathsf{m}}\in\bN_0^{\hat n}$,
$\check{\mathsf{m}}\in\bN_0^{n-\hat n+1}$ as in Def.~\ref{dfn:maps_A_B}. Let $V\in\cV^{\tilde{\mathsf{m}}}$, $W\in\cV^{\hat{\mathsf{m}}}$, $U\in\cV^{\check{\mathsf{m}}}$ and $G\in\cG$. For any $\sUV,\sIR\in[0,1]$ and $\oo\in\bN_0$ it holds
\begin{equation}
\begin{gathered}
 K_{\sUV,\sIR}^{n,m;\oo}\ast\fA(G,V)=
 \fA\big(\fR^{\phantom{2}}_\sUV\fP^{2\oo}_\sIR G,\,
 K_{\sUV,\sIR}^{n+1,m+1;\oo}\ast V\big),
 \\
 K_{\sUV,\sIR}^{n,m;\oo}\ast\fB(G,W,U)=
 \fB\big(\fR^{\phantom{2}}_\sUV\fP^{2\oo}_\sIR G,\,
 K_{\sUV,\sIR}^{\hat n,\hat m+1;\oo}\ast W,\,
 K_{\sUV,\sIR}^{n-\hat n+1,\check m;\oo}\ast U\big).
\end{gathered} 
\end{equation}
\end{lem}
\begin{proof}
The lemma follows immediately from Def.~\ref{dfn:maps_A_B} of the maps $\fA$ and $\fB$. 
\end{proof}

\begin{rem}
By the above lemma and the support properties of elements of the spaces $\cG$ and $\cD^{\mathsf{m};\oo}$ the maps $\fA$ and $\fB$ admit the following extensions 
\begin{equation}
 \fA:\,\cG\times\cD^{\tilde{\mathsf{m}};\oo}\to\cD^{\mathsf{m};\oo},
 \qquad 
 \fB:\,\cG\times\cD^{\hat{\mathsf{m}};\oo}\times\cD^{\check{\mathsf{m}};\oo}\to\cD^{\mathsf{m};\oo},
\end{equation}
where $\mathsf{m},\tilde{\mathsf{m}},\hat{\mathsf{m}},\check{\mathsf{m}}$ are as in Def.~\ref{dfn:maps_A_B} and $\oo\in\bN_0$ is arbitrary. Note that for any $V\in\cD^{\tilde{\mathsf{m}};\oo}$, $W\in\cD^{\hat{\mathsf{m}};\oo}$, $U\in\cD^{\check{\mathsf{m}};\oo}$, $G\in\cG$ and $\mathtt{r}\in C^\infty_\rc(\bR^{n})$ there exist
$\mathtt{v}\in C^\infty_\rc(\bR^{n+1})$,
$\mathtt{w}\in  C^\infty_\rc(\bR^{\hat n})$,
$\mathtt{u}\in  C^\infty_\rc(\bR^{n-\hat n+1})$ such that 
\begin{equation}
 \mathtt{r}\fA(G,V)=
 \mathtt{r}\fA(G,\,\mathtt{v} V),
 \qquad
 \mathtt{r}\fB(G,W,U)=
 \mathtt{r}\fB(G,\,
 \mathtt{w}W,\,
 \mathtt{u}U).
\end{equation}
\end{rem}

\begin{rem}\label{rem:fA_fB_Ks}
Lemma~\ref{lem:fA_fB_Ks} will be used together with Lemma~\ref{lem:fA_fB_bounds} to bound RHS of various flow equations. In our applications the kernel $G$ will usually coincide with $\dot G_\sIR^a$ introduced in Def.~\ref{dfn:propagator_G}, where $a\in\frM_\sigma$ and $\sIR\in(0,1]$. Note that for any $\oo\in\bN_0$ it holds
\begin{equation}
 \|\fR^{\phantom{2}}_\sUV\fP^{2\oo}_\sIR \dot G_\sIR^a\|_\cK\lesssim [\sUV\vee\sIR]^{\sigma-\varepsilon}\,[\sIR]^{\varepsilon-\sigma+[a]},
 \quad
 \|\fT|\fR^{\phantom\oo}_\sUV\fP_\sIR^{2\oo}\dot G^a_\sIR|\|_\cV \lesssim
 [\sUV\vee\sIR]^{\sigma-\varepsilon}\,[\sIR]^{\varepsilon-\sigma-\rDim+[a]}
\end{equation}
uniformly in $\sUV\in[0,1]$, $\sIR\in(0,1]$ by Lemma~\ref{lem:kernel_dot_G} and Lemma~\ref{lem:kernel_simple_fact} (D) applied with $\ooo=\rdim+1$ and $p=\infty$ (recall that $\rDim=\sigma+\rdim$).
\end{rem}

\begin{rem}\label{rem:fB1_bound}
Let $m\in\bN_0$ $k\in\{0,\ldots,m\}$, $G\in\cG$, $W\in\cV^{k+1}$, $U\in\cV^{m-k}$ and $\oo\in\bN_0$. Lemma~\ref{lem:fA_fB_bounds} and Lemma~\ref{lem:fA_fB_Ks} imply that
\begin{equation}
\begin{gathered}
 \|\fB(G,W,U)\|_{\cV^{m}}
 \leq
 \|G\|_\cK\,\|W\|_{\cV^{k+1}} \|U\|_{\cV^{m-k}},
 \\
 K_{\sUV,\sIR}^{m;\oo}\ast\fB(G,W,U)=
 \fB\big(\fR^{\phantom{2}}_\sUV\fP^{2\oo}_\sIR G,\,
 K_{\sUV,\sIR}^{k+1;\oo}\ast W,\,
 K_{\sUV,\sIR}^{m-k;\oo}\ast U\big).
\end{gathered} 
\end{equation}
As a result, the map $\fB$ admits the extension $\fB:\,\cG\times\cD^{k+1;\oo}\times\cD^{m-k;\oo}\to\cD^{m;\oo}$ and
\begin{equation}
 \|K_{\sUV,\sIR}^{m;\oo}\ast\fB(G,W,U)\|_{\cV^{m}}\leq
 \|\fR^{\phantom{2}}_\sUV\fP^{2\oo}_\sIR G\|_\cK\,
 \|K_{\sUV,\sIR}^{k+1;\oo}\ast W\|_{\cV^{k+1}}\,
 \|K_{\sUV,\sIR}^{m-k;\oo}\ast U\|_{\cV^{m-k}}.
\end{equation}
\end{rem}

\section{Taylor polynomial and remainder}\label{sec:taylor}

In this section we use the Taylor theorem to decompose a distribution on $\bM^{1+m}$ into a local distribution, that is a distribution supported on the diagonal $$\{(x,\ldots,x)\in\bM^{1+m}\,|\,x\in\bM\},$$ and a certain non-local remainder that is easy to control. More specifically, we introduce and study maps $\fX^a_\ro$, where $m\in\bN_+$, $a\in\frM^m$ and $\ro\in\bN_+$ are such that $|a|<\ro$. The map $\fX^a_\ro$ allows to express
\begin{equation}
 \cX^{m,a} V\in\cD^{m;0},\quad |a|<\ro,
\end{equation}
in terms of 
\begin{equation}
 \cX^{m,b} V\in\cV^{m;0},\quad b\in\frM^m,~|b|=\ro
 \qquad
 \textrm{and}
 \qquad
 v^b\in\cD^{;0}, \quad b\in\frM^m,~|b|<\ro,
\end{equation}
where
\begin{equation}
 v^b=\fI(\cX^{m,b} V):=\int_{\bM^m}((\delta_\bM\otimes J^{\otimes m})\ast\cX^{m,b} V)(\,\Cdot\,;\rd y_1\ldots\rd y_m)\in\cD^{;0}.
\end{equation}
Recall that the polynomials $\cX^{m,a}\in C^\infty(\bM\times\bM^m)$ were introduced in Def.~\ref{dfn:polynomials}. We call $\fI V$ the local part of the distribution $V$. Note that distributions $V$ that we encounter in this work are typically essentially localized close to the diagonal. Consequently, the distribution $\cX^{m,b} V$ obtained after the multiplication of $V$ by a polynomial in the relative coordinates of the degree $|b|$ is typically better behaved than the distribution $V$ and its behavior improves as the degree $|b|$ increases.

The results of this section will be used in the proof of uniform bounds for the relevant cumulants of the effective force coefficients as well as in the deterministic analysis to prove uniform bounds for the relevant effective force coefficients at stationarity. We remark that the results established in this section are useful if the parameters $\sUV$ and $\sIR$ satisfy the condition $\sIR\geq\sUV$. For $\sUV>0$ and $\sIR\in[0,\sUV)$ one can bound the relevant coefficients using a simpler strategy that is also used to bound the irrelevant coefficients.

In the proof of bounds for the cumulants we apply the map $\fX^a_\ro$ to express the expectations of the relevant effective force coefficients 
\begin{equation}
 \llangle F^{i,m,a}_{\sUV,\sIR}\rrangle=\cX^{m,a}\llangle F^{i,m}_{\sUV,\sIR}\rrangle \in \cD^{m;0}
\end{equation}
in terms of the expectations of the irrelevant coefficients $\llangle F^{i,m,b}_{\sUV,\sIR}\rrangle\in \cD^{m;0}$, $|b|=\ro$, and the following expectations $\llangle f^{i,m,b}_{\sUV,\sIR}\rrangle=\fI(\llangle F^{i,m,b}_{\sUV,\sIR}\rrangle) \in \cD^{;0}=C(\bH)$, $|b|<\ro$, where $\ro\in\bN_+$ is chosen sufficiently big so that $\varrho(i,m)+\ro>0$. Note that by the stationarity $\llangle f^{i,m,b}_{\sUV,\sIR}(x)\rrangle=\llangle f^{i,m,b}_{\sUV,\sIR}\rrangle$ does not depend on $x\in\bM$. The relevant coefficients $\llangle f^{i,m,b}_{\sUV,\sIR}\rrangle\in\bR$ are fixed uniquely by the renormalization conditions at $\sIR=1$ and the flow equation. The bounds for relevant $\llangle F^{i,m,a}_{\sUV,\sIR}\rrangle$ follow from the bounds for the map $\fX^a_\ro$ proved in this section as well as the bounds for $\llangle f^{i,m,b}_{\sUV,\sIR}\rrangle$ and the bounds for irrelevant $\llangle F^{i,m,b}_{\sUV,\sIR}\rrangle$.

In the deterministic analysis we apply the map $\fX^a_\ro$ to express the relevant effective force coefficients $F^{i,m,a}_{\sUV,\sIR}=\cX^{m,a} F^{i,m}_{\sUV,\sIR}\in \cD^{m;0}$ in terms of the irrelevant coefficients $F^{i,m,b}_{\sUV,\sIR}\in \cD^{m;0}$, $|b|=\ro$, and the coefficients $f^{i,m,b}_{\sUV,\sIR}=\fI F^{i,m,a}_{\sUV,\sIR}\in\cD^{;0}$, $|b|<\ro$, where $\ro\in\bN_+$ is again chosen sufficiently big so that \mbox{$\varrho(i,m)+\ro>0$}. We prove the bounds for relevant $F^{i,m,a}_{\sUV,\sIR}$ using the bounds for the map $\fX^a_\ro$ proved in this section as well as the bounds for relevant $f^{i,m,b}_{\sUV,\sIR}$ established in the probabilistic analysis and the bounds for irrelevant $F^{i,m,b}_{\sUV,\sIR}$. The map $\fX^a_\ro$ is also used to construct the relevant effective force coefficients $F^{i,m,a}_{0,\sIR}\in\cD^{m;\oo}$ and bound $F^{i,m,a}_{\sUV,\sIR}-F^{i,m,a}_{0,\sIR}\in\cD^{m;\oo}$.

\begin{dfn}
For $m\in\bN_+$, $a\in\frM^m$ and $V\in C^\infty(\bM\times\bM^m)$ we define $\partial^a V\in C^\infty(\bM\times\bM^m)$ by
\begin{equation}
 \partial^a V(x;x_1,\ldots,x_m):=
 \partial^{a_1}_{x_1} \ldots\partial^{a_m}_{x_m}
 V(x;x_1,\ldots,x_m).
\end{equation}
This map extends in an obvious way to $\sD'(\bM\times\bM^m)\supset\cD^m$.
\end{dfn}

\begin{dfn}\label{dfn:map_L_m}
For $m\in\bN_0$ and $v\in\cD$ we define $\fL^m v\in\cD^m$ by
\begin{equation}
 \langle\,\fL^m v\,,\,\psi\otimes\varphi_1\otimes\ldots\otimes\varphi_m\,\rangle= \langle v\,,\,\psi\varphi_1\ldots\varphi_m\rangle,
\end{equation}
where $\psi,\varphi_1,\ldots,\varphi_m\in C^\infty_\rc(\bM)$ are arbitrary.
\end{dfn}

\begin{rem}
We call $\fL^m(v)\in\sD'(\bM^{1+m})$ the local extension of a distribution $v\in\sD'(\bM)$. Note that $\fL^m(v)$ is a distribution supported on the diagonal.
\end{rem}

\begin{lem}\label{lem:bound_map_L_m}
Fix $m\in\bN_+$. There exists a constant $c>0$ such that the following statement is true. Assume that $v^{a}\in\bR\simeq \cD_\rt$, $a\in\frM^m$, are such that \mbox{$|v^a|\leq C\,[\sUV]^{[a]}$} for some $C>0$ and $v^a=0$ unless $a\in\bar\frM_\sigma^m$. Let $V^a\in\cD^m$, $a\in\frM^m$, be defined by
\begin{equation}
 V^a
 =
 \sum_{b\in\frM^m}\frac{1}{b!}
 \partial^b \fL^m(v^{a+b}).
\end{equation} 
Then for all $a\in\frM^m$ it holds $V^a=\cX^{m,a} V^0$ and
\begin{equation}
\|(\delta_\bM\otimes J_\sUV^{\otimes m})\ast V^a\|_{\cV^m}
 \leq c\, C\,[\sUV]^{[a]}.
\end{equation} 
\end{lem}
\begin{proof}
The identity $V^a=\cX^{m,a} V^0$ follows from the fact that $\cX^{m,a} \partial^b\fL^m(v)=0$ unless $b=a+c$ for some $c\in \frM^m$ and $$\cX^{m,a} \frac{1}{(a+c)!}\partial^{a+c} \fL^m(v)=\frac{1}{c!}\partial^c \fL^m(v).$$ Next, recall that $a\in\bar\frM_\sigma^m$ iff $a=(a_1,\ldots,a_m)$ and $a_q=(\mathring a_q,\bar a_q)\in\frM=\bN_0^{1+\rdim}$, $\mathring a_q=0$, $|a_q|<\sigma$ for all $q\in\{1,\ldots,m\}$. We note that $a+b\in\bar\frM_\sigma^m$ implies that $b=(b_1,\ldots,b_m)\in\bar\frM_\sigma^m$. Moreover, $\|\partial^{b_q}J_\sUV\|_\cK\lesssim[\sUV]^{-|b_q|}$ for all $q\in\{1,\ldots,m\}$ by Lemma~\ref{lem:kernel_simple_fact}~(B). To conclude the proof it is now enough to use Def.~\ref{dfn:cV} of the norm $\|\Cdot\|_{\cV^m}$.
\end{proof}

\begin{dfn}\label{dfn:map_Z}
For $m\in\bN_0$, $\tau\in(0,\infty)$ and $V\in C(\bM\times\bM^m)$ we define $\fZ_\tau V\in C(\bM\times\bM^m)$ by
\begin{equation}
 \fZ_\tau V(x;y_1,\ldots,y_m)
 :=\tau^{-(1+\rdim)m}\,V(x;x+(y_1-x)/\tau,\ldots,x+(y_m-x)/\tau).
\end{equation}
This map extends in an obvious way to $\sD'(\bM\times\bM^m)\supset\cD^m$.
\end{dfn}

\begin{dfn}\label{dfn:map_I}
Let $m\in\bN_0$. The linear map $\fI\,:\,\cV^m\to\cV$ is defined by
\begin{equation}
 \fI V(x):= \int_{\bM^m} V(x;\rd y_1\ldots\rd y_m).
\end{equation}
The map $\fI$ is extended to $V\in\cD^{m;\oo}$ by the formula
\begin{equation}
 \langle \fI V,\psi\rangle := \langle (\delta_\bM\otimes(K^{\ast\oo}\ast J)^{\otimes m})\ast V, \psi\otimes 1^{\otimes m}_\bM\rangle,
\end{equation}
where $\psi\in C^\infty_\rc(\bM)$.
\end{dfn}
\begin{rem}
The map $\fI$ extracts the local part of a distribution. In particular, note that if $V\in\sD'(\bM^{1+m})$ is proportional to the Dirac delta on the diagonal, then $V=\fL^m(\fI V)$, where $\fL^m$ is the local extension of a distribution introduced in Def.~\ref{dfn:map_L_m}.
\end{rem}

\begin{lem}\label{lem:map_I}
The map $\fI$ is well defined and has the following properties.
\begin{enumerate}
\item[(A)] $\fI(K^{m;\oo}_{\sUV,\sIR}\ast V)=K^{\ast\oo}_\sIR\ast\fI V$
for any $V\in\cV^m$, $\oo\in\bN_0$ and $\sUV,\sIR\in[0,\infty)$.

\item[(B)] $\|\fI V\|_\cV \leq \|V\|_{\cV^m}$ for $V\in\cV^m$.

\item[(C)] Let $a\in\frM^m$, $v\in\cV$, $V=\partial^a \fL^m v$. It holds $\fI(\cX^{m,a} V) = a!\,v$ as well as $\fI(\cX^{m,b} V) = 0$ for all $b\in\frM$ such that $b\neq a$.
\end{enumerate}
\end{lem}
\begin{proof}
In order to prove that the map is well defined it is enough to use the identity $\int_\bM K(x)\,\rd x=\int_\bM J(x)\,\rd x=1$. Part (A) is a consequence of the fact that $\int_\bM K_\sIR(x)\,\rd x=\int_\bM J_\sUV(x)\,\rd x=1$. Part (B) follows immediately from Def.~\ref{dfn:cV} of the norm~$\cV^m$. To prove Part (C) we note that $\partial^b\fL^m(v)\in\cD^{m;\oo}$ for any $\oo\geq|b|$ and $\int_\bM (\partial^b K^{\ast\oo})(\rd x)=0$ for any $\oo\geq|b|$ such that $b\neq0$. Hence, $\fI(\fL^m v)=v$ and $\fI(\partial^b\fL^m v)=0$ unless $b=0$. Part (C) follows now from the above identities and the fact that $\cX^{m,b}\fL^m v=0$ unless $b=0$.
\end{proof}

\begin{dfn}\label{dfn:map_X}
For $m,\ro\in\bN_+$, $a\in\frM^m$ such that $|a|<\ro$ and two collections of distributions: $v^b\in\cD$, $b\in\frM^m$, $|b|<\ro$, and $V^b\in\cD^m$, $b\in\frM^m$, $|b|=\ro$, the distribution $\fX_\ro^a(v^b,V^b)\in \cD^m$ is defined by the equality 
\begin{multline}\label{eq:fX}
 \fX_\ro^a(v^b,V^b)
 :=
 \sum_{|a+b|<\ro} \,\frac{1}{b!}
 \partial^b \fL^m(v^{a+b})
 \\
 +
 \sum_{|a+b|=\ro} \frac{|b|}{b!}
 \int_0^1 (1-\tau)^{|b|-1}\,
 \partial^b \fZ_\tau(V^{a+b})\,\rd\tau,
\end{multline} 
where the sums above are over $b\in\frM^m$.
\end{dfn}
\begin{rem}
This map is well defined by the bounds proved in Lemma~\ref{lem:taylor_bounds_aux}.
\end{rem}
\begin{rem}
The usefulness of the map $\fX^a_\ro$ comes from the identity proved in Theorem~\ref{thm:taylor}. Observe that the first term on the RHS of Eq.~\eqref{eq:fX} is a distribution on $\fM^{1+m}$ that is supported on the diagonal. The second term of the RHS of this equality is a certain non-local remainder, which can be controlled provided $\ro\in\bN_+$ is chosen big enough. We show desirable bounds for $\fX_\ro^a(v^b,V^b)$ in Theorem~\ref{thm:taylor_bounds} stated below.
\end{rem}

\begin{thm}\label{thm:taylor}
Let $m,\ro\in\bN_+$, $\oo\in\bN_0$ and $V\in\cD^{m;0}$ such that $\cX^{m,b} V \in \cD^{m;0}$ for any $b\in\frM^m$, $|b|\leq\ro$. Then $\fI(\cX^{m,a} V)\in\cD^{;0}$ and
\begin{equation}\label{eq:fX_identity}
 \cX^{m,a} V = \fX_\ro^a(\fI(\cX^{m,b} V),\cX^{m,b} V)
\end{equation}
for all $a\in\frM^m$ such that $|a|<\ro$.
\end{thm}
\begin{proof}
First recall that for any $\varphi\in C^\infty(\bR^N)$, $N\in\bN_+$, it holds
\begin{multline}
 \varphi(y) = \sum_{|b|<\ro} \frac{1}{b!}(y-x)^b\,(\partial^b \varphi)(x) 
 \\
 + \sum_{|b|=\ro} \frac{|b|}{b!}\,(y-x)^b \int_0^1 (1-\tau)^{\ro-1}\,(\partial^b \varphi)(x+\tau(y-x))\,\rd\tau
\end{multline} 
by the Taylor theorem, where the sums are over $b\in\bN_0^N$. Consequently, for any $\varphi\in C^\infty_\rc(\bM\times\bM^m)$ we have
\begin{multline}
 (\cX^{m,a}\varphi)(x;y_1,\ldots,y_m) 
 = \sum_{|a+b|<\ro} \frac{(-1)^{|b|}}{b!}\cX^{m,a+b}(x;y_1,\ldots,y_m) (\partial^b \varphi)(x;x,\ldots,x)
 \\
 + \sum_{|a+b|=\ro} \frac{(-1)^{|b|}|b|}{b!}\cX^{m,a+b}(x;y_1,\ldots,y_m) \int_0^1 (1-\tau)^{\ro-1}(\fZ_{\tau}^\dagger\partial^b\varphi)(x;y_1,\ldots,y_m)\,\rd\tau,
\end{multline}
where the map $\fZ_{\tau}^\dagger:=\tau^{(1+\rdim)m}\fZ_{1/\tau}$ is a formal dual to the map $\fZ_\tau$ and the sums are over $b\in\frM$. In order to prove the theorem it is now enough to test $\cX^{m,a} V$ with $\varphi\in C^\infty_\rc(\bM\times\bM^m)$, apply the above formula and use the definitions of the maps $\fX_\ro^a$ and $\fI$.
\end{proof}

\begin{rem}
Note that the map $\fX^a_\ro$ treats space and time on equal footing and can be used to build the solution theory for both elliptic and parabolic equations. By the above theorem we have
\begin{equation}\label{eq:map_X_property}
 F^{i,m,a}_{\sUV,\sIR} = \cX^{m,a}F^{i,m}_{\sUV,\sIR} = \fX^a_\ro(f^{i,m,b}_{\sUV,\sIR},F^{i,m,b}_{\sUV,\sIR})\in\cD^{m;0},
\end{equation}
where $f^{i,m,b}_{\sUV,\sIR} = \fI(F^{i,m,b}_{\sUV,\sIR})\in\cD^{;0}$ and $\ro\in\bN_0$ is such that $\varrho(i,m)+|\ro|>0$. This identity will be used in the deterministic analysis to express the relevant effective force coefficients $F^{i,m,a}_{\sUV,\sIR}$, $\varrho(i,m,a)=\varrho(i,m)+[a]<0$, in terms of the irrelevant coefficients $F^{i,m,b}_{\sUV,\sIR}$, $|b|=\ro$, and the relevant or irrelevant coefficients $f^{i,m,b}_{\sUV,\sIR}$, $|b|<\ro$. As an aside, we remark that using the generalized Taylor formula from~\cite[Appendix A]{hairer2014structures} it is possible to define a map $\tilde\fX^a_\ro$ such that the property expressed in Eq.~\eqref{eq:map_X_property} holds for $\tilde\fX^a_\ro$ and $\tilde\fX^a_\ro$ takes as an input only irrelevant coefficients $F^{i,m,b}_{\sUV,\sIR}$ and relevant coefficients $f^{i,m,b}_{\sUV,\sIR}$. In contrast to the map $\fX^a_\ro$, the map $\tilde\fX^a_\ro$ does not treat space and time on equal footing. Although, the use of the map $\tilde\fX^a_\ro$ would be more natural in the context of parabolic equations we prefer to use the map $\fX^a_\ro$ since the analysis of properties of the map $\tilde\fX^a_\ro$ is more involved.
\end{rem}

\begin{thm}\label{thm:taylor_bounds}
Let $m,\ro\in\bN_+$ and $\oo\in\bN_0$. There exists a constant $c>0$ and $\ooo\in\bN_+$ such that the following statement is true. Let $V^b\in\cV^m$, $v^b\in\cV$ be as in Def.~\ref{dfn:map_X} and $\sUV\geq 0$ and $\sIR>\sUV$. Assume that there exists a constant $C>0$ such that 
\begin{equation}\label{eq:taylor_ass_V}
 \|K_{\sUV,\sIR}^{m;\oo}\ast V^b\|_{\cV^m} \leq C\, [\sIR]^{[b]},\quad |b|=\ro,
\quad\quad
 \|K_\sIR^{\ast\oo}\ast v^b\|_\cV \leq C\, [\sIR]^{[b]},
 \quad |b|<\ro,
\end{equation}
for $b\in\frM$. Then
\begin{equation}\label{eq:taylor_thm_bound}
 \|K_{\sUV,\sIR}^{m;\ooo}\ast \fX_\ro^a(v^b,V^b)\|_{\cV^m} \leq c\, C\, [\sIR]^{[a]},
\end{equation}
for $a\in\frM$ such that $|a|<\ro$.
\end{thm}
\begin{proof}
Let $\oooo=\floor{\sigma/2}+1$. Since $\fR_\sUV J_\sUV=\delta_\bM$ we obtain
\begin{equation}
 K^{\ast(\oo+\oooo),\otimes(1+m)}_\sIR = K^{m;\oo}_{\sUV,\sIR} \ast (K_\sIR^{\ast \oooo}\otimes (\fR_\sUV K_\sIR^{\ast \oooo})^{\otimes m}).
\end{equation}
By Lemma~\ref{lem:kernel_simple_fact}~(C) for $\sIR\geq\sUV$ it holds
\begin{equation}\label{eq:bound_taylor_ass}
 \|K^{\ast(\oo+\oooo),\otimes(1+m)}_\sIR \ast V^b\|_{\cV^m}
 \lesssim 
 \|K^{m;\oo}_{\sUV,\sIR}\ast V\|_{\cV^m}.
\end{equation}
On the other hand, since $\|J_\sUV\|_\cK\lesssim 1$ in order to conclude the proof it is enough to show that
\begin{equation}
 \|K^{\ast\ooo,\otimes(1+m)}_\sIR\ast \fX_\ro^a(v^b,V^b)\|_{\cV^m} \lesssim C\, [\sIR]^{[a]},
\end{equation}
where $\ooo=6\oo+6\oooo+\ro$. This bound follows from Def.~\ref{dfn:map_X} of the map $\fX^a_\ro$, the bound~\eqref{eq:bound_taylor_ass}, Lemma~\ref{lem:taylor_bounds_aux} and the bounds $\|\partial^c K^{\ast\ro}_\sIR\|_\cK\lesssim [\sIR]^{-[c]}$, $c\in\frM$, $|c|\leq\ro$, proved in Lemma~\ref{lem:kernel_simple_fact}~(A).
\end{proof}

\begin{lem}\label{lem:taylor_bounds_aux}
Let $\oooo\in\bN_0$. The following bounds:
\begin{enumerate}
\item[(A)] $\|K^{\ast3\oooo,\otimes(1+m)}_\sIR\ast \fL^m v\|_{\cV^m} \lesssim  \|K^{\ast\oooo}_\sIR\ast v\|_\cV$,
\item[(B)]
$\|K^{\ast6\oooo,\otimes(1+m)}_\sIR\ast\fZ_\tau V\|_{\cV^m}
 \lesssim
 \|K^{\ast\oooo,\otimes(1+m)}_\sIR\ast V\|_{\cV^m}$
\end{enumerate}
hold uniformly in $\tau,\sIR\in(0,1]$ and $v\in\cV$, $V\in\cV^m$.
\end{lem}
\begin{proof}
It follows easily from the definitions given in this section, Def.~\ref{dfn:scaling_S} of the rescaling map $\fS_\sIR$ and Eq.~\eqref{eq:cVt_norm} defining the norm $\|\Cdot\|_{\cV^m}$ that 
\begin{equation}
 \fS_\sIR \fL^m v=\fL^m\fS_\sIR v, 
 \quad \fS_\sIR \fZ_\tau V=\fZ_\tau\fS_\sIR V,
 \quad
 K_\sIR=\fS_\sIR K,
 \quad 
 \fS_\sIR(H\ast V)=\fS_\sIR H\ast\fS_\sIR V,
\end{equation}
\begin{equation}
 \|V\|_{\cV^m} = [\sIR]^\rDim \|\fS_\sIR V\|_{\cV^m},
 \qquad
 \|V\|_{\cV^m} = \|\fZ_\tau V\|_{\cV^m}
\end{equation}
for any $\tau,\sIR>0$ and $v\in\cV$, $V\in\cV^m$, $H\in\cK^{1+m}$. By the above identities it is enough to prove the lemma for $\sIR=1$ and all $v\in C_\rb(\bM)$ and $V\,:\,\bM\to\sM(\bM^m)$ such that $\|V\|_{\cV^m}<\infty$. Using $\fP^{\oooo} K^{\ast\oooo}=\delta_\bM$ we obtain
\begin{multline}
 (K^{\ast3\oooo,\otimes(1+m)}\ast\fL^m v)(x;y_1,\ldots,y_m)
 \\
 =\int_\bM
 (K^{\ast\oooo}\ast v)(z)~
 \fP^{\oooo}(\partial_z) K^{\ast3\oooo}(x-z)K^{\ast3\oooo}(y_1-z)\ldots K^{\ast3\oooo}(y_m-z)
 \,\rd z
 \\
 =\int_\bM
 (K^{\ast\oooo}\ast v)(z)~
 H^{\ast\oooo}_0(x-z,y_1-z,\ldots,y_m-z)
 \,\rd z,
\end{multline}
where $H_0\in\cK^{1+m}$ is defined in the lemma below and $\fP(\partial_z)=1-\Delta_z$. This proves the bound (A) since
\begin{equation}
 K^{\ast3\oooo,\otimes(1+m)}\ast \fL^m v = H_0^{\ast\oooo}\ast \fL^m(K^{\ast\oooo}\ast v).
\end{equation}
The bound (B) follows from the identities
\begin{equation}\label{eq:lem_taylor_useful_identities}
\begin{gathered}
 K^{\ast3,\otimes(1+m)}
 =(K^{\ast 3}\otimes\check K^{\otimes m}_\tau) \ast \fZ_\tau(\delta_\bM\otimes K^{\otimes m}),
 \\
 K^{\ast 3,\otimes (1+m)}
 = H_\tau\ast\fZ_\tau(K\otimes\delta_\bM^{\otimes m}),
\end{gathered}
\end{equation}
where the kernels $\check K_\tau$ and $H_\tau$ are defined in the lemma below. Indeed, applying the above equalities recursively and using $\fZ_\tau H\ast \fZ_\tau V=\fZ_\tau(H\ast V)$ we arrive at
\begin{equation}\label{eq:taylor_S_hat_property}
 K^{\ast6\oooo,\otimes(1+m)}\ast\fZ_\tau V
 =
 \check H^{\ast\oooo}_\tau\ast \fZ_\tau(K^{\ast\oooo,\otimes(1+m)}\ast V),
 \quad  
 \check H_\tau:=(K^{\ast3}\otimes \check K^{\otimes m}_\tau) \ast H_\tau,
\end{equation}
which implies Part (B) since $\|\check H_\tau\|_{\cK^{1+m}}\lesssim 1$ uniformly in $\tau\in(0,1]$ by the lemma below. It remains to establish the identities~\eqref{eq:lem_taylor_useful_identities}. The first one follows from the fact that $\fZ_\tau(\delta_\bM\otimes K^{\otimes m}) = \delta_\bM\otimes\breve K_\tau^{\otimes m}$, where $\breve K_\tau(x):=\tau^{-1-\rdim}K(x/\tau)$ satisfies the equation $\fP(\tau \partial_x)\breve K_\tau = \delta_\bM$. To show the second one we observe that by the definition of $\check H_\tau$ given in the lemma below we have
\begin{multline}
 \fZ_{1/\tau}(H_\tau)(x;y_1,\ldots,y_m)
 \\=
 \tau^{(1+\rdim)m} \fP(\partial_x)\, K^{\ast3}(x) 
 K^{\ast3}(x+\tau(y_1-x)) \ldots K^{\ast3}(x+\tau(y_1-x)).
\end{multline}
After convolving both sides with the kernel $K\otimes\delta_\bM^{\otimes m}$ we obtain
\begin{equation}
 \fZ_{1/\tau}(H_\tau)\ast\big(K\otimes\delta_\bM^{\otimes m}\big)
 =
 \fZ_{1/\tau}(K^{\ast3,\otimes (1+m)})
\end{equation}
The second of the identities \eqref{eq:lem_taylor_useful_identities} follows after applying the map $\fZ_\tau$ to both sides of the above equality.
\end{proof}

\begin{lem}\label{lem:taylor_aux}
The distributions $\check K_\tau\in\sS'(\bM)$ and $H_\tau\in\sS'(\bM^{1+m})$ defined by
\begin{equation}
\begin{gathered}
\check K_\tau(x):= \fP(\tau \partial_x) K^{\ast3}(x),
 \\ \
 H_\tau(x;y_1,\ldots,y_m)
 :=
 \fP(\partial_x+(1-\tau)\partial_y) \, K^{\ast3}(x) K^{\ast3}(y_1) \ldots K^{\ast3}(y_m), 
\end{gathered}
\end{equation}
where $\partial_y:=\partial_{y_1}+\ldots+\partial_{y_m}$, are polynomials in $\tau>0$ of degree $3$ whose coefficients belong to $\cK$ and $\cK^{1+m}$, respectively.
\end{lem}
\begin{rem}
The lemma follows immediately from the fact that $\partial^a K\in\cK$ for any multi-index $a$ such that $|a|=1$. Note that the statement remains true when the kernel $K^{\ast3}$ is replaced by $K$. However, the proof of this stronger version is slightly non-trivial.
\end{rem}

\section{Effective force coefficients at stationarity}\label{sec:effective_force}

In this section we define effective force coefficients
\begin{equation}
 F^{i,m,a}_{\uv;\sUV,\sIR}\in\cD^m\subset \sD'(\bM\times\bM^m), 
 \quad
 (i,m,a)\in\frI_0,~(\uv,\sUV)\in[0,1]^2\setminus\{(0,0)\},~\sIR\in[0,1].
\end{equation}
The coefficients $F^{i,m,a}_{\uv;\sUV,\sIR}$ are multi-linear translationally invariant functionals of the noise $\varXi_\sUV$. Since the noise $\varXi_\sUV$ is a stationary random field the same is true for the coefficients $F^{i,m,a}_{\uv;\sUV,\sIR}$. At later stage of the construction we shall introduce other effective force coefficients which depend on the initial data and in general are not stationary.

We have $F^{i,m,a}_{\uv;\sUV,\sIR}=\cX^{m,a}F^{i,m}_{\uv;\sUV,\sIR}$, where $\cX^{m,a}$ is some polynomial and the coefficients $F^{i,m}_{\uv;\sUV,\sIR}\in\cD^m$ are such that the effective force functional $F_{\uv;\sUV,\sIR}[\varphi]$ given by the following formal power series
\begin{equation}\label{eq:ansatz}
 \langle F_{\uv;\sUV,\sIR}[\varphi],\psi\rangle
 :=\sum_{i=0}^\infty \sum_{m=0}^\infty \lambda^i\,\langle F^{i,m}_{\uv;\sUV,\sIR},\psi\otimes\varphi^{\otimes m}\rangle
\end{equation}
satisfies the flow equation
\begin{equation}\label{eq:flow_eq}
 \langle\partial_\sIR F_{\uv;\sUV,\sIR}[\varphi],\psi\rangle
 =
 -\langle \rD F_{\uv;\sUV,\sIR}[\varphi,\dot G_\sIR\ast F_{\uv;\sUV,\sIR}[\varphi]],\psi\rangle
\end{equation}
with the boundary condition $F_{\uv;\sUV,0}[\varphi]=F_{\uv;\sUV}[\varphi]$, where $\varphi,\psi\in C^\infty_\rc(\bM)$ and we used the notation introduced in Sec.~\ref{sec:intro_flow}. The force $F_{\uv;\sUV}[\varphi]$  is a generalization of the macroscopic force $F_{\sUV}[\varphi]$ given by Eq.~\eqref{eq:force_macro}. The force $F_{\uv;\sUV}[\varphi]$ involves the regularized noise $\varXi_{\uv;\sUV}$ obtained by convolving the noise $\varXi_\sUV$ with some smooth kernel. Note that if $\uv=0$, then the force $F_{\uv;\sUV}[\varphi]$ and the effective force $F_{\uv;\sUV,\sIR}[\varphi]$ coincide with the force $F_{\sUV}[\varphi]$ and the effective force $F_{\sUV,\sIR}[\varphi]$ at stationarity introduced in Sec.~\ref{sec:intro_flow}. The extra parameter $\uv\in[0,1]$ of the generalized effective force $F_{\uv;\sUV,\sIR}[\varphi]$ will play a role in Sec.~\ref{sec:probabilistic} in which we establish the convergence in probability as $\sUV\searrow0$ of the relevant coefficients $f^{i,m,a}_{\sUV,\sIR}=f^{i,m,a}_{0;\sUV,\sIR}$ of the effective force $F_{\sUV,\sIR}[\varphi]$ using a certain diagonal argument applied previously in a similar context for example in~\cite{hairer2017clt,mourrat2017}.

The results of this section have a preliminary character and rely only on the fact that the realizations of the noise $\varXi_\sUV$ belong to $C(\bH)=\cD^{;0}$ for any $\sUV\in(0,1]$ and the realizations of (a modification of) the white noise $\varXi_0$ belong to $\sC^\alpha(\bH)\subset \sD'(\bH)=\cD$ for any $\alpha<-\dim(\varXi)$. We construct the effective force coefficients $F^{i,m,a}_{\uv;\sUV,\sIR}$ pathwise for $(\uv,\sUV)\in[0,1]^2\setminus\{(0,0)\}$ and $\sIR\in[0,1]$ and prove that they satisfy some simple bounds.

\subsection{Regularized noise and force}

\begin{dfn}\label{dfn:noise_two_parameters}
Fix $M\in C^\infty_\rc(\bM)$ such that $\supp\,M\subset\{x\in\bM\,|\,|x|_\bM<1\}$. For $(\uv,\sUV)\in[0,1]^2\setminus\{(0,0)\}$ the regularized noise is defined by the equality
\begin{equation}
 \varXi_{\uv;\sUV}:= M_\uv\ast \varXi_{\sUV}.
\end{equation}
Here $M_0=\delta_\bM$ and $M_\uv:=\fS_\uv M\in C^\infty_\rc(\bM)$ for $\uv\in(0,1]$, where $\fS_\uv$ is the rescaling map introduced in Def.~\ref{dfn:scaling_S}. For $\varphi\in C^\infty_\rc(\bM)$ the force $F_{\uv;\sUV}[\varphi]\in C(\bM)$ is defined by the equality
\begin{equation}
 F_{\uv;\sUV}[\varphi](x):= \varXi_{\uv;\sUV}(x) + \sum_{(i,m,a)\in\frI} \lambda^i\, f^{i,m,a}_{\uv;\sUV}~\partial^{a_1}\varphi(x)\ldots \partial^{a_m}\varphi(x),
\end{equation}
where the force coefficients $f^{i,m,a}_{\uv;\sUV}$, \mbox{$(i,m,a)\in\frI$}, by assumption depend continuously on $(\uv,\sUV)\in[0,1]^2\setminus\{(0,0)\}$, are differentiable in the parameter $\uv$ if $\uv\in(0,1]$ and vanish identically unless $i\in\{0,\ldots,i_\flat\}$, $m\in\{0,\ldots,m_\flat\}$ and $a\in\bar\frM_\sigma^m$. The force coefficients $F^{i,m}_{\uv;\sUV}\in\sD'(\bM\times\bM^m)$ are defined by the equality
\begin{equation}
 \langle F_{\uv;\sUV}[\varphi],\psi\rangle =\sum_{i=0}^\infty\sum_{m=0}^\infty \lambda^i\,\langle F^{i,m}_{\uv;\sUV},\psi\otimes\varphi^{\otimes m}\rangle,
\end{equation}
where $\psi,\varphi\in C^\infty_\rc(\bM)$. We also set $F^{i,m,a}_{\uv;\sUV} = \cX^{m,a} F^{i,m}_{\uv;\sUV}$, where the polynomial $\cX^{m,a}$ was introduced in Def.~\ref{dfn:polynomials}. We omit index $\uv$ if $\uv=0$.
\end{dfn}

\begin{rem}
The force coefficients satisfy the conditions
\begin{equation}
 F_{\uv;\sUV}^{0,0}=\varXi_{\uv;\sUV},
 \qquad
 F_{\uv;\sUV}^{i,m,a}=\sum_{b\in\bar\frM_\sigma^m} \frac{1}{b!}\,\partial^b \fL^m(f^{i,m,a+b}_{\uv;\sUV}),\quad i\in\bN_+,m\in\bN_0,
\end{equation}
where the map $\fL^m$ was introduced in Def.~\ref{dfn:map_L_m}. The second of the above equalities follows from Lemma~\ref{lem:bound_map_L_m}. We also have $\fI(F_{\uv;\sUV}^{i,m,a})=f_{\uv;\sUV}^{i,m,a}$, where the map $\fI$ was introduced in Def.~\ref{dfn:map_I}. Let us also remark that all sums in the above definition are finite.
\end{rem}

\subsection{Effective force and flow equation}

In this section we define the effective force coefficients $F^{i,m,a}_{\uv;\sUV,\sIR}=\cX^{m,a}F^{i,m}_{\uv;\sUV,\sIR}$. For a list of multi-indices $a=(a_1,\ldots,a_m)\in\frM^m$ and a permutation $\pi\in\cP_n$ we set $\pi(a):=(a_{\pi(1)},\ldots,a_{\pi(m)})$. Recall that the polynomial $\cX^{m,a}$ was introduced in Def.~\ref{dfn:polynomials}. We have
\begin{equation}\label{eq:poly_perm}
 \cX^{m,a}(x;y_1,\ldots,y_m)
 \\
 =
 \cX^{m,\pi(a)}(x;y_{\pi(1)},\ldots,y_{\pi(m)})
\end{equation}
For all $y_0,z\in\bM$ it holds
\begin{multline}\label{eq:poly_binom}
 \cX^{m,a}(x;y_1,\ldots,y_m)
 =
 \sum_{b,c,d}\frac{a!}{b!c!d!}
 \cX^{1+k,(b_{1+k}+\ldots+b_m,b_1,\ldots,b_k)}(x;y_0,y_1,\ldots,y_k)
 \\
 \cX^{1,c_{1+k}+\ldots+c_m}(y_0;z)
 \,\cX^{m-k,(d_{1+k},\ldots,d_m)}(z;y_{1+k},\ldots,y_m),
\end{multline}
where the sum is over all $b_p,c_p,d_p$ with $p\in\{1,\ldots,m\}$ such that $b_p=a_p$, $c_p=0$, $d_p=0$ for $p\in\{1,\ldots,k\}$ and $b_p+c_p+d_p=a_p$ for $p\in\{1+k,\ldots,m\}$. Throughout the paper we use the following schematic notation for the sums of the above type
\begin{multline}\label{eq:poly_binom_notation}
 \cX^{m,a}(x;y_1,\ldots,y_m)
 \\
 =\sum_{b+c+d=a}  \frac{a!}{b!c!d!}\cX^{1+k,b}(x;y_0,y_1,\ldots,y_k)
 \cX^{1,c}(y_0,z) \cX^{m-k,d}(z;y_{1+k},\ldots,y_m).
\end{multline}
We note that the formula~\eqref{eq:poly_binom} was obtained by rewriting $(x-y_p)^{a_p}$, $p\in\{1+k,\ldots,m\}$, as
\begin{equation}
 (x-y_p)^{a_p} 
 =\sum_{\substack{b_p,c_p,d_p\\b_p+c_p+d_p=a_p}}\frac{a_p!}{b_p!c_p!d_p!}~(x-y_0)^{b_p}(y_0-z)^{c_p}(z-y_p)^{d_p}.
\end{equation}

The effective force coefficients at stationarity $F^{i,m}_{\uv;\sUV,\sIR}$ satisfy the following flow equation 
\begin{multline}\label{eq:flow_i_m}
 \big\langle \partial_\sIR^{\phantom{i}} F^{i,m}_{\uv;\sUV,\sIR}\,,
 \,\psi\otimes\botimes_{q=1}^m\varphi_q\big\rangle
 =
 -\frac{1}{m!}
 \sum_{\pi\in\cP_m}\sum_{j=1}^i\sum_{k=0}^m
 \,(k+1)
 \\
 \times\big\langle F^{j,k+1}_{\uv;\sUV,\sIR}\otimes F^{i-j,m-k}_{\uv;\sUV,\sIR},
 \psi\otimes \botimes_{q=1}^k\varphi_{\pi(q)}
 \otimes \fV\dot G_\sIR
 \otimes \botimes_{q=k+1}^m\varphi_{\pi(q)}
 \big\rangle,
\end{multline}
where $\fV\dot G_\sIR\in C^\infty(\bM\times\bM)$ is defined by $\fV\dot G_\sIR(x,y):=\dot G_\sIR(x-y)$. By the above flow equation and the identity~\eqref{eq:poly_binom} the effective force coefficients $F^{i,m,a}_{\uv;\sUV,\sIR}$ satisfy the following flow equation
\begin{multline}\label{eq:flow_i_m_a}
 \big\langle \partial_\sIR^{\phantom{i}} F^{i,m,a}_{\uv;\sUV,\sIR}\,,
 \,\psi\otimes\botimes_{q=1}^m\varphi_q\big\rangle
 =
 -\frac{1}{m!}
 \sum_{\pi\in\cP_m}\sum_{j=1}^i\sum_{k=0}^m
 \sum_{b+c+d=\pi(a)}\frac{(k+1)a!}{b!c!d!}
 \\
 \times\big\langle F^{j,k+1,b}_{\uv;\sUV,\sIR}\otimes F^{i-j,m-k,d}_{\uv;\sUV,\sIR},
 \psi\otimes \botimes_{q=1}^k\varphi_{\pi(q)}
 \otimes \fV\dot G^c_\sIR
 \otimes \botimes_{q=k+1}^m\varphi_{\pi(q)}
 \big\rangle,
\end{multline}
where $\fV\dot G^c_\sIR=\cX^{1,c}\fV \dot G_\sIR$ for $c\in\frM$. In the above equation we employed the schematic notation specified below Eq.~\eqref{eq:poly_binom}. These flow equations can be written compactly with the use of the map $\fB$ introduced in Def.~\ref{dfn:maps_A_B} and the map $\fY_\pi$ introduced in Def.~\ref{dfn:map_Y}. We have
\begin{equation}\label{eq:flow_deterministic_i_m}
 \partial_\sIR  F^{i,m}_{\uv;\sUV,\sIR}
 =
 -\frac{1}{m!}
 \sum_{\pi\in\cP_m}\sum_{j=1}^i\sum_{k=0}^m
 (k+1)\,\fY_\pi\fB\big(\dot G_\sIR, F^{j,k+1}_{\uv;\sUV,\sIR}, F^{i-j,m-k}_{\uv;\sUV,\sIR}\big)
\end{equation}
and
\begin{multline}\label{eq:flow_deterministic_i_m_a}
 \partial_\sIR^{\phantom{i}} F^{i,m,a}_{\uv;\sUV,\sIR}
 \\
 =
 -\frac{1}{m!}
 \sum_{\pi\in\cP_m}\sum_{j=1}^i\sum_{k=0}^m
 \sum_{b+c+d=\pi(a)}\frac{(k+1)a!}{b!c!d!}
 \fY_\pi\fB\big(\dot G^c_\sIR,F^{j,k+1,b}_{\uv;\sUV,\sIR},F^{i-j,m-k,d}_{\uv;\sUV,\sIR}\big).
\end{multline}

\begin{dfn}\label{dfn:eff_force_uv}
For $(\uv,\sUV)\in[0,1]^2\setminus\{(0,0)\}$ and $\sIR\in[0,1]$ the effective force coefficients $F^{i,m,a}_{\uv;\sUV,\sIR}$, $(i,m,a)\in\frI_0$, are defined by the flow equation~\eqref{eq:flow_deterministic_i_m_a} together with the conditions 
\begin{equation}
 F^{i,m,a}_{\uv;\sUV,\sIR}=0~~\textrm{if}~~m>i m_\flat,
 \qquad
 F^{i,m,a}_{\uv;\sUV,0}=F_{\uv;\sUV}^{i,m,a},\quad (i,m,a)\in\frI_0,
\end{equation}
where the force coefficients $F_{\uv;\sUV}^{i,m,a}$ were introduced in Def.~\ref{dfn:noise_two_parameters}. We also set $F^{i,m}_{\uv;\sUV,\sIR}:=F^{i,m,0}_{\uv;\sUV,\sIR}$ and $f^{i,m,a}_{\uv;\sUV,\sIR}:=\fI F^{i,m,a}_{\uv;\sUV,\sIR}$, where the map $\fI$ was introduced in Def.~\ref{dfn:map_I}.
\end{dfn}

\begin{rem}
By the properties of the maps $\fL^m$ and $\fI$ we have $f^{i,m,a}_{\uv;\sUV,0}=f^{i,m,a}_{\uv;\sUV\phantom{0}}$.
\end{rem}

\begin{lem}\label{lem:eff_force_basic_properties}
For every $(\uv,\sUV)\in[0,1]^2\setminus\{(0,0)\}$, $\sIR\in[0,1]$, $(i,m,a)\in\frI_0$ the effective force coefficients $F^{i,m,a}_{\uv;\sUV,\sIR}\in\cD^{m;0}$ are uniquely defined random distributions satisfying the condition $F^{0,0,0}_{\uv;\sUV,\sIR}=\varXi_{\uv;\sUV}$ and $F^{i,m,a}_{\uv;\sUV,\sIR}=\cX^{m,a}F^{i,m,0}_{\uv;\sUV,\sIR}$. Moreover, for any $\mathtt{w}\in C^\infty_\rc(\bR)$ the functions
\begin{equation}
\begin{aligned}
 [0,1]^2\ni(\uv,\sIR)&\mapsto (\delta_\bM\otimes J^{\otimes m})\ast \mathtt{w}F^{i,m,a}_{\uv;\sUV,\sIR}\in\cV^m,\quad \sUV\in(0,1],
 \\
 (0,1]\times[0,1]\ni(\uv,\sIR)&\mapsto (\delta_\bM\otimes J^{\otimes m})\ast \mathtt{w}F^{i,m,a}_{\uv;\sUV,\sIR}\in\cV^m,\quad \sUV=0,
\end{aligned} 
\end{equation}
are almost surely continuous in the full domain and continuously differentiable in the interior of the domain.
\end{lem}
\begin{proof}
The effective force coefficients $F^{i,m,a}_{\uv;\sUV,\sIR}$ are constructed pathwise. The construction is recursive. We first note that $F^{0,m,a}_{\uv;\sUV,\sIR}=0$ for any $m\in\bN_+$, $a\in\frM^m_\sigma$. Hence, $\partial_\sIR F^{0,0,0}_{\uv;\sUV,\sIR} =0$ by the flow equation~\eqref{eq:flow_deterministic_i_m_a}. Consequently,
$F^{0,0,0}_{\uv;\sUV,\sIR}=\varXi_{\uv;\sUV}$. Note that almost surely $\varXi_\sUV\in C(\bH)$ for $\sUV\in(0,1]$ and $\varXi_0\in\cD=\sD'(\bH)$. Hence for every $\mathtt{w}\in C^\infty_\rc(\bR)$ the functions
\begin{equation}
\begin{aligned}
 [0,1]\ni\uv&\mapsto \mathtt{w}\varXi_{\uv;\sUV}\in\cV=C_\rb(\bH),\qquad\sUV\in(0,1],
 \\
 (0,1]\ni\uv&\mapsto \mathtt{w}\varXi_{\uv;\sUV}\in\cV=C_\rb(\bH),\qquad\sUV=0,
\end{aligned}  
\end{equation}
are almost surely continuous in the full domain and continuously differentiable in the interior of the domain. Now pick any $i_\circ\in\bN_+$, $m_\circ\in\bN_0$ and assume that all $F^{i,m,a}_{\uv;\sUV,\sIR}$ with $i<i_\circ$, or $i=i_\circ$ and $m>m_\circ$ have been constructed. We define $\partial_\sIR F^{i,m,a}_{\uv;\sUV,\sIR}$ with $i=i_\circ$ and $m=m_\circ$ by the RHS of the flow equation~\eqref{eq:flow_deterministic_i_m_a}. Subsequently, $F^{i,m,a}_{\uv;\sUV,\sIR}$ is defined by the equation
\begin{equation}
 F^{i,m,a}_{\uv;\sUV,\sIR} = F^{i,m,a}_{\uv;\sUV} + \int_0^\sIR \partial_\uIR F^{i,m,a}_{\uv;\sUV,\uIR}\,\rd\uIR.
\end{equation} 
The continuity and differentiability of the first term on the RHS of the above equality follows from the properties of the force coefficients assumed in Def.~\ref{dfn:noise_two_parameters}. The well-definiteness of the second term and its continuity and differentiability properties follow from the support property of the effective force coefficients stated in Lemma~\ref{lem:support} and the bounds established in Lemma~\ref{lem:eff_force_preliminary_bound2}. The equality $F^{i,m,a}_{\uv;\sUV,\sIR}=\cX^{m,a} F^{i,m,0}_{\uv;\sUV,\sIR}$ follows from the fact that the coefficients $F^{i,m,a}_{\uv;\sUV,\sIR}$ and $\cX^{m,a}F^{i,m,0}_{\uv;\sUV,\sIR}$ satisfy the same flow equation~\eqref{eq:flow_i_m_a} and both coincide with $F^{i,m,a}_{\uv;\sUV}=\cX^{m,a}F^{i,m,0}_{\uv;\sUV}$ if $\sIR=0$.
\end{proof}

\subsection{Support properties}

\begin{lem}\label{lem:support}
For every $(\uv,\sUV)\in[0,1]^2\setminus\{(0,0)\}$, $\sIR\in[0,1]$ and $(i,m,a)\in\frI$ it holds almost surely
\begin{equation}
 \supp\,F^{i,m,a}_{\uv;\sUV,\sIR}\subset\{ (x,x_1,\ldots,x_m)\in\bM^{1+m}\,|\, \mathring x_1,\ldots,\mathring x_m \in [\mathring x-2(i-1)\sIR,\mathring x]\}.
\end{equation}
\end{lem}
\begin{proof}
We first note that the lemma is true if $i=0$ or $\sIR=0$ or $m>i m_\flat$. Next, we fix some $i_\circ\in\bN_+$ and $m_\circ\in\bN_0$ and assume that the lemma is true for all $(i,m,a)\in\frI_0$ such that either $i<i_\circ$, or $i=i_\circ$ and $m>m_\circ$. We shall prove the lemma for all $(i,m,a)\in\frI_0$ such that $i=i_\circ$ and $m=m_\circ$. To this end, we observe that by the flow equation~\eqref{eq:flow_i_m_a}, the induction hypothesis and $\supp\,\dot G_\sIR^c\subset [0,2\sIR]\times\bR^\rdim$ it holds
\begin{equation}
 \supp\,\partial_\sIR F^{i,m,a}_{\uv;\sUV,\sIR}\subset\{ (x,x_1,\ldots,x_m)\in\bM^{1+m}\,|\, \mathring x_1,\ldots,\mathring x_m \in [\mathring x-2(i-1)\sIR,\mathring x]\}.
\end{equation}
Since
\begin{equation}
 F^{i,m,a}_{\uv;\sUV,\sIR} = F^{i,m,a}_{\uv;\sUV} + \int_0^\sIR \partial_\uIR F^{i,m,a}_{\uv;\sUV,\uIR}\,\rd\uIR
\end{equation}
this implies the statement of the lemma.
\end{proof}

\begin{rem}\label{rem:support}
Since $\supp\,\dot G_\sIR^c\subset [0,2]\times\bR^\rdim$ for $\sIR\in[0,1]$ the above lemma implies that for every $\mathtt{w},\mathtt{v}\in C^\infty_\rc(\bR)$ such that $\mathtt{v}=1$ on some neighborhood of $\supp\,\mathtt{w}-[0,2i]$ it holds
\begin{multline}\label{eq:flow_deterministic_i_m_a_w}
 \mathtt{w}\partial_\sIR^{\phantom{i}} F^{i,m,a}_{\uv;\sUV,\sIR}
 \\
 =
 -\frac{1}{m!}
 \sum_{\pi\in\cP_m}\sum_{j=1}^i\sum_{k=0}^m
 \sum_{b+c+d=\pi(a)}\frac{(k+1)a!}{b!c!d!}\,
 \fY_\pi\fB\big(\dot G^c_\sIR,\mathtt{w}F^{j,k+1,b}_{\uv;\sUV,\sIR},\mathtt{v} F^{i-j,m-k,d}_{\uv;\sUV,\sIR}\big).
\end{multline} 
We will often use the following observation. Fix $\mathtt{u}\in C^\infty_\rc(\bR)$ and $i\in\bN_0$. For any $\mathtt{w}\in C^\infty_\rc(\bR)$ such that $\mathtt{u}=1$ on some neighborhood of $\supp\,\mathtt{w}-[0,i(i+1)]$ there exists $\mathtt{v}\in C^\infty_\rc(\bR)$ such that $\mathtt{v}=1$ on some neighborhood of \mbox{$\supp\,\mathtt{w} - [0,2i]$} and $\mathtt{u}=1$ on some neighborhood of $\supp\,\mathtt{v}-[0,(i-1)i]$.
\end{rem}

\subsection{Preliminary deterministic analysis}

\begin{dfn}
For $R>0$ and $i,m\in\bN_0$ we set $R^{i,m}=R^{1+i m_\flat-m}$. 
\end{dfn}

\begin{lem}\label{lem:eff_force_preliminary_bound2}
Fix $\sUV\in(0,1]$ and $\mathtt{u}\in C^\infty_\rc(\bR)$. Assume that for some $R>1$, $\varepsilon\in(0,1]$ and all $\uv\in[0,1]$, $r\in\{0,1\}$, $(i,m,a)\in\frI$ it holds
\begin{equation}
 \|\mathtt{u}\partial_\uv^r \varXi_{\uv;\sUV}^{\phantom{r}}\|_\cV\leq R\,[\uv]^{(\varepsilon-\sigma)r},
 \qquad
 |\partial_\uv^r f^{i,m,a}_{\uv;\sUV}|\leq R^{i,m}\,[\uv]^{(\varepsilon-\sigma)r}.
\end{equation}
Then it holds
\begin{equation}
 \|(\delta_\bM\otimes J^{\otimes m})\ast \mathtt{w}\partial_\sIR^s\partial_\uv^r  F^{i,m,a}_{\uv;\sUV,\sIR}\|_{\cV^m}\lesssim R^{i,m}\,[\uv]^{(\varepsilon-\sigma)r}\,[\sIR]^{(\varepsilon-\sigma)s},
\end{equation} 
for all $\uv,\sIR\in[0,1]$, $r,s\in\{0,1\}$, $(i,m,a)\in\frI_0$ and $\mathtt{w}\in C^\infty_\rc(\bR)$ satisfying the condition $\mathtt{u}=1$ on some neighborhood of $\supp\,\mathtt{w}-[0,i(i+1)]$.
The constants of proportionality in the above bounds depend only on $(i,m,a)\in\frI_0$ and \mbox{$\mathtt{w},\mathtt{u}\in C^\infty_\rc(\bR)$} and are otherwise universal. 
\end{lem}
\begin{rem}
The proof is very similar to that of the next lemma and is omitted.
\end{rem}

\begin{lem}\label{lem:eff_force_preliminary_bound}
Fix $\varepsilon,\uv,\sUV\in(0,1]$ and $\mathtt{u}\in C^\infty_\rc(\bR)$. Assume that for some $R>1$ and $\delta\in(0,1]$ it holds
\begin{equation}
\begin{gathered}
 \|\mathtt{u}\varXi_{\uv;0}\|_\cV\leq R,
 \qquad
 \|\mathtt{u}(\varXi_{\uv;\sUV}-\varXi_{\uv;0})\|_\cV\leq \delta R,
\\
 |f^{i,m,a}_{\uv;0}|\leq R^{i,m},
 \qquad
 |f^{i,m,a}_{\uv;\sUV}-f^{i,m,a}_{\uv;0}|\leq \delta R^{i,m}
 \qquad\textrm{for all}\quad (i,m,a)\in\frI^-.
\end{gathered} 
\end{equation}
Then:
\begin{enumerate}
\item[(A)] $\|(\delta_\bM\otimes J^{\otimes m}) \ast \mathtt{w}\partial_\sIR^s F^{i,m,a}_{\uv;0,\sIR}\|_{\cV^m}\lesssim R^{i,m}\,[\sIR]^{(\varepsilon-\sigma)s}$,

\item[(B)] $\|(\delta_\bM\otimes J^{\otimes m}) \ast \mathtt{w}\partial_\sIR^s F^{i,m,a}_{\uv;\sUV,\sIR}\|_{\cV^m}\lesssim R^{i,m}\,[\sIR]^{(\varepsilon-\sigma)s}$,

\item[(C)] $\|(\delta_\bM\otimes J^{\otimes m})\ast \mathtt{w}\partial_\sIR^s(F^{i,m,a}_{\uv;\sUV,\sIR}-F^{i,m,a}_{\uv;0,\sIR})\|_{\cV^m}\lesssim \delta R^{i,m}\,[\sIR]^{(\varepsilon-\sigma)s}$
\end{enumerate}
for all $\sIR\in[0,1]$, $s\in\{0,1\}$, $(i,m,a)\in\frI_0$ and $\mathtt{w}\in C^\infty_\rc(\bR)$ satisfying the condition $\mathtt{u}=1$ on some neighborhood of $\supp\,\mathtt{w}-[0,i(i+1)]$. The constants of proportionality in the above bounds depend only on $(i,m,a)\in\frI_0$ and \mbox{$\mathtt{w},\mathtt{u}\in C^\infty_\rc(\bR)$} and are otherwise universal. 
\end{lem}

\begin{rem}\label{rem:eff_force_preliminary_bound}
A slightly stronger statement is in fact true. Fix $i_\circ,m_\circ\in\bN_0$. If the assumption of the above lemma holds for all $i,m\in\bN_0$ such that $i<i_\circ$ or $i=i_\circ$ and $m>m_\circ$, then the statement holds for all $i,m\in\bN_0$ such that either (1)~$i<i_\circ$ or (2)~$i=i_\circ$, $m>m_\circ$ or (3) $i=i_\circ$, $m=m_\circ$, $s=1$.
\end{rem}

\begin{proof}
We first note that for $m>i m_\flat$ we have $F^{i,m,a}_{\uv;\sUV,\sIR}=0$ and the bounds (A), (B), (C) hold trivially. Next, we observe that $F^{0,0,0}_{\uv;0,\sIR}=\varXi_{\uv;0}$ and $F^{0,0,0}_{\uv;\sUV,\sIR}=\varXi_{\uv;\sUV}$ for $\sIR\in[0,1]$. Recall that for $m=0$ it holds $\|v\|_{\cV^m}=\|v\|_\cV$. As a result, the bounds (A), (B), (C) hold true if $i=0$.

Fix $i_\circ\in\bN_+$ and $m_\circ\in\bN_0$. Assume that the bounds (A), (B), (C) are true for all $(i,m,a)\in\frI_0$ such that either $i<i_\circ$, or $i=i_\circ$ and $m>m_\circ$. We shall prove the bounds for all $(i,m,a)$ such that $i=i_\circ$ and $m=m_\circ$. Fix some $\mathtt{w}\in C^\infty_\rc(\bR)$ such that $\mathtt{u}=1$ on some neighborhood of $\supp\,\mathtt{w}-[0,i(i+1)]$. We first observe that by the flow equation~\eqref{eq:flow_deterministic_i_m_a} and the support properties of the effective force coefficients stated in Lemma~\ref{lem:support} it holds
\begin{multline}\label{eq:flow_preliminary_w_i_m_a}
 \mathtt{w}\partial_\sIR^{\phantom{i}} F^{i,m,a}_{\uv;\sUV,\sIR}
 =
 -\frac{1}{m!}
 \sum_{\pi\in\cP_m}\sum_{j=1}^i\sum_{k=0}^m
 \sum_{b+c+d=\pi(a)}\frac{(k+1)a!}{b!c!d!}\,
 \\\times
 \fY_\pi\fB\big(\dot G^c_\sIR,\mathtt{w} F^{j,k+1,b}_{\uv;\sUV,\sIR},\mathtt{v} F^{i-j,m-k,d}_{\uv;\sUV,\sIR}\big),
\end{multline}
where $\mathtt{v}\in C^\infty_\rc(\bR)$ is arbitrary such that $\mathtt{v}=1$ on some neighborhood of \mbox{$\supp\,\mathtt{w} - [0,2i]$} and $\mathtt{u}=1$ on some neighborhood of $\supp\,\mathtt{v}-[0,(i-1)i]$. Note that Eq.~\eqref{eq:flow_preliminary_w_i_m_a} is also valid for $\sUV=0$. Recall that $\delta_\bM\otimes J^{\otimes m}=K^{m;0}$. Using Remark~\ref{rem:fB1_bound} with $\oo=0$ and $\sUV=1$ we obtain
\begin{multline}
 \|K^{m;0} \ast \mathtt{w}\partial_\sIR F_{\uv;\sUV,\sIR}^{i,m,a}\|_{\cV^m} 
 \leq
 \sum_{\pi\in\cP_m}\sum_{j=1}^i\sum_{k=0}^m\sum_{b+c+d=\pi(a)}
 \frac{(k+1)a!}{b!c!d!}\,\|\fR\dot G^c_\sIR\|_\cK
 \\
 \times
 \|K^{k+1;0} \ast \mathtt{w} F_{\uv;\sUV,\sIR}^{j,k+1,b}\|_{\cV^{k+1}} 
 \,
 \|K^{m-k;0} \ast \mathtt{v} F_{\uv;\sUV,\sIR}^{i-j,m-k,d}\|_{\cV^{m-k}}
\end{multline}
and
\begin{multline}
 \|K^{m;0}\ast \mathtt{w}\partial_\sIR (F_{\uv;\sUV,\sIR}^{i,m,a}-F_{\uv;0,\sIR}^{i,m,a})\|_{\cV^m} 
 \leq 
 \sum_{\pi\in\cP_m}\sum_{j=1}^i\sum_{k=0}^m\sum_{b+c+d=\pi(a)}
 \frac{(k+1)a!}{b!c!d!}\,\|\fR\dot G^c_\sIR\|_\cK
 \\
 \times\Big(
 \|K^{k+1;0} \ast \mathtt{w}(F_{\uv;\sUV,\sIR}^{j,k+1,b}-F_{\uv;0,\sIR}^{j,k+1,b})\|_{\cV^{k+1}} 
 \,
 \|K^{m-k;0} \ast \mathtt{v} F_{\uv;\sUV,\sIR}^{i-j,m-k,d}\|_{\cV^{m-k}}
 \\[2mm]
 +
 \|K^{k+1;0} \ast\mathtt{w} F_{\uv;0,\sIR}^{j,k+1,b}\|_{\cV^{k+1}} 
 \,
 \|K^{m-k;0} \ast \mathtt{v}(F_{\uv;\sUV,\sIR}^{i-j,m-k,d}-F_{\uv;0,\sIR}^{i-j,m-k,d})\|_{\cV^{m-k}}
 \Big).
\end{multline}
To get the last bound we used the multi-linearity of the map $\fB$. Let us also recall that by Lemma~\ref{lem:kernel_dot_G} we have
\begin{equation}
 \|\fR
 \dot G_\sIR^{c}\|_\cK \lesssim
 [\sIR]^{\varepsilon-\sigma}
\end{equation}
for $c\in\frM^m_\sigma$. For $s=1$ the bounds (A), (B), (C) follow now from the induction hypothesis. For $s=0$ and $\sIR=0$ the bounds (A), (B), (C) follow from the assumed bounds for the force coefficients $f^{i,m,a}_{\uv;\sUV}$, the equation
\begin{equation}\label{eq:lem_preliminary_bound_initial}
 F^{i,m,a}_{\uv;\sUV}=\sum_{b\in\bar\frM^m_\sigma}\frac{1}{b!}\partial^b \fL^m(f^{i,m,a+b}_{\uv;\sUV}).
\end{equation}
and Lemma~\ref{lem:bound_map_L_m} applied with $\sUV=1$. To prove the case $s=0$ and $\sIR\in(0,1]$ we use the bounds
\begin{equation}\label{eq:lem_preliminary_bound_integration}
 \|K^{m;0}\ast \mathtt{w} F^{i,m,a}_{\uv;\sUV,\sIR}\|_{\cV^m} 
 \leq 
 \|K^{m;0}\ast \mathtt{w} F^{i,m,a}_{\uv;\sUV}\|_{\cV^m} 
 +
 \int_0^\sIR 
 \|K^{m;0}\ast \mathtt{w} \partial_\uIR F^{i,m,a}_{\uv;\sUV,\uIR}\|_{\cV^m}\,\rd\uIR,
\end{equation}
\begin{multline}\label{eq:lem_preliminary_bound_integration2}
 \|K^{m;0}\ast \mathtt{w} (F^{i,m,a}_{\uv;\sUV,\sIR}-F^{i,m,a}_{\uv;0,\sIR})\|_{\cV^m} 
 \leq 
 \|K^{m;0}\ast \mathtt{w} (F^{i,m,a}_{\uv;\sUV}-F^{i,m,a}_{\uv;0})\|_{\cV^m} \\
 +
 \int_0^\sIR 
 \|K^{m;0}\ast \mathtt{w} \partial_\uIR (F^{i,m,a}_{\uv;\sUV,\uIR}-F^{i,m,a}_{\uv;0,\uIR})\|_{\cV^m}\,\rd\uIR
\end{multline}
and the results established above.
\end{proof}

\section{Cumulants of effective force coefficients}\label{sec:cumulants_estimates}

In this section we prove uniform bounds for the joint cumulants of the effective force coefficients $F^{i,m,a}_{\uv;\sUV,\sIR}$ at stationarity introduced in Sec.~\ref{sec:effective_force}. The bounds are stated in Theorem~\ref{thm:cumulants}. Their validity relies crucially on certain renormalization conditions that fix the relevant force coefficients $f^{i,m,a}_{\uv;\sUV}$, $(i,m,a)\in\frI^-$. The results established in this section are used in Sec.~\ref{sec:probabilistic} to prove uniform bounds for the relevant force coefficient $f^{i,m,a}_{\sUV,\sIR}=f^{i,m,a}_{0;\sUV,\sIR}$ as well as to show their convergence in probability as~$\sUV\searrow0$. 

\begin{dfn}\label{dfn:cumulants}
Let $p\in\bN_+$, $I=\{1,\ldots,p\}$ and $\zeta_q$, $q\in I$, be random variables. The joint cumulant of the multi-set $(\zeta_q)_{q\in I}=(\zeta_1,\ldots,\zeta_p)$ is defined by
\begin{equation}
 \llangle \zeta_1,\ldots,\zeta_p\rrangle
 \equiv
 \llangle(\zeta_q)_{q\in I}\rrangle
 = 
 (-\ri)^p \partial_{t_1}\ldots\partial_{t_p}\log\llangle \exp(\ri t_1 \zeta_1+\ldots+\ri t_p \zeta_p)\rrangle \big|_{t_1=\ldots=t_p=0}.
\end{equation}
For example, $\llangle \zeta_1,\zeta_2\rrangle=\llangle \zeta_1\zeta_2\rrangle-\llangle \zeta_1\rrangle\llangle \zeta_2\rrangle$.
\end{dfn}

\begin{lem}\label{lem:cumulants}
Let $p\in\bN_+$, $I=\{1,\ldots,p\}$ and $\zeta_1,\ldots,\zeta_p,\Phi,\Psi$ be random variables. Then
\begin{equation}\label{eq:expectation_cumulants}
 \llangle\zeta_1\ldots\zeta_p\rrangle=
 \llangle{\textstyle\prod}_{q\in I} \zeta_q\rrangle
 =\sum_{r=1}^p
 \sum_{\substack{I_1,\ldots,I_r\subset I,\\I_1\cup\ldots\cup I_r=I\\I_1,\ldots,I_r\neq \emptyset}}
 \prod_{l=1}^r \llangle(\zeta_q)_{q\in I_l}\rrangle,
\end{equation} 
\begin{equation}\label{eq:cumulants_expectation}
 \llangle\zeta_1,\ldots,\zeta_p\rrangle=
 \llangle(\zeta_q)_{q\in I}\rrangle
 =\sum_{r=1}^p (-1)^{r-1}(r-1)! 
 \sum_{\substack{I_1,\ldots,I_r\subset I,\\I_1\cup\ldots\cup I_r=I\\I_1,\ldots,I_r\neq \emptyset}}
 \prod_{l=1}^r \llangle {\textstyle\prod}_{q\in I_l}\zeta_q\rrangle
\end{equation} 
and
\begin{equation}\label{eq:cumulants_product}
 \llangle(\zeta_q)_{q\in I},\Phi\Psi\rrangle
 =
 \llangle(\zeta_q)_{q\in I},\Phi,\Psi\rrangle
 +
 \sum_{\substack{I_1,I_2\subset I\\I_1\cup I_2= I}}
 \llangle(\zeta_q)_{q\in I_1},\Phi\rrangle
 ~
 \llangle(\zeta_q)_{q\in I_2},\Psi\rrangle.
\end{equation}
\end{lem}
\begin{rem}
For the proof of the above lemma see Proposition~3.2.1 in~\cite{peccati2011wiener}. 
\end{rem}

\begin{dfn}\label{dfn:notation_cumulants_distributions}
Let $n\in\bN_+$, $I=\{1,\ldots,n\}$, $m_1,\ldots,m_n\in\bN_0$ and \mbox{$\zeta_q\in\cD^{m_q}$}, \mbox{$q\in I$}, be random distributions. By definition the joint cumulant of the list $(\zeta_q)_{q\in I}=(\zeta_1,\ldots,\zeta_n)$ is the distribution
\begin{equation}
 \llangle (\zeta_q)_{q\in I}\rrangle
 \equiv
 \llangle \zeta_1,\ldots,\zeta_n\rrangle
 \in 
 \cD^{\mathsf{m}},
 \qquad 
 \mathsf{m}=(m_1,\ldots,m_n)\in\bN_0^n,
\end{equation}
defined by the equality
\begin{equation}
 \langle\llangle \zeta_1,\ldots,\zeta_n\rrangle,\psi_1\otimes\ldots\otimes\psi_n\otimes\varphi_1\otimes\ldots\otimes\varphi_n\rangle
 :=
 \llangle
 \langle\zeta_1,\psi_1\otimes\varphi_1\rangle,\ldots,\langle\zeta_n,\psi_n\otimes\varphi_n\rangle\rrangle,
\end{equation}
where $\psi_q\in C^\infty_\rc(\bM)$, $\varphi_q\in C^\infty_\rc(\bM^{m_q})$, $q\in I$, are arbitrary. 
\end{dfn}

\begin{dfn}\label{dfn:cumulants_eff_force}
A list $(i,m,a,s,r)$, where $i,m\in\bN_0$, $a\in\frM^m_\sigma$, $s,r\in\{0,1\}$ is called an index. For $n\in\bN_+$ we call
\begin{equation}\label{eq:list_indices}
 \vI\equiv ((i_1,m_1,a_1,s_1,r_1),\ldots,(i_n,m_n,a_n,s_n,r_n))
\end{equation}
a list of indices. We set
\begin{itemize}
 \item $\rn(\vI):=n$, 
 \item $\rri(\vI):=i_1+\ldots+i_n$, 
 \item $\mathsf{m}(\vI):=(m_1,\ldots,m_n)$,
 \item $\rrm(\vI):=m_1+\ldots+m_n$,
 \item $\ra(\vI):=[a_1]+\ldots+[a_m]$,
 \item $\rs(\vI):=s_1+\ldots+s_n$,
 \item $\rr(\vI):=r_1+\ldots+r_n$.
\end{itemize}
We use the following notation for the joint cumulants of the effective force coefficients at stationarity
\begin{equation}
 E^\vI_{\uv;\sUV,\sIR}:=
 \llangle \partial_\sIR^{s_1}\partial_\uv^{r_1} F^{i_1,m_1,a_1}_{\uv;\sUV,\sIR} ,\ldots,
 \partial_\sIR^{s_n}\partial_\uv^{r_n} F^{i_n,m_n,a_n}_{\uv;\sUV,\sIR}\rrangle\in\cD^{\mathsf{m}(\vI);0}_\rt.
\end{equation}
\end{dfn}

\begin{rem}\label{rem:cumulants_continuity}
Note that if $\rr(\vI)=0$ and $\rs(\vI)=0$, then $E^\vI_{\uv;\sUV,\sIR}$ is well defined for all $(\uv,\sUV)\in[0,1]^2\setminus\{(0,0)\}$, $\sIR\in[0,1]$. If $\rr(\vI)=1$ or $\rs(\vI)=1$, then $E^\vI_{\uv;\sUV,\sIR}$ is well defined if $\uv\in(0,1]$ or $\sIR\in(0,1]$, respectively.
\end{rem}

\begin{rem}
The main result of this section stated in Theorem~\ref{thm:cumulants} implies that $E^\vI_{\uv;\sUV,\sIR}$ can be bounded uniformly in the space $\cD^{\mathsf{m}(\vI);\oo}_\rt$ for some $\oo\in\bN_0$ depending on the list of indices $\vI$.
\end{rem}

\begin{dfn}\label{dfn:varrho}
Let $\varepsilon\in(0,1]$. For any list of indices $\vI$ we define $\varrho_\varepsilon(\vI)\in\bR$ by the following equality
\begin{multline}\label{eq:dfn_varrho_I}
 \varrho_\varepsilon(\vI) = 
 -\rn(\vI) (\dim(\varXi)+\varepsilon)
 +\rrm(\vI) (\dim(\varPhi)+\varepsilon)
 \\
 + \rri(\vI) (\dim(\lambda)-\varepsilon m_\flat) + \ra(\vI).
\end{multline}
For any $(i,m,a)\in\frI_0$ we define $\varrho_\varepsilon(i,m,a)\in\bR$ by the following equality
\begin{equation}\label{eq:dfn_varrho}
 \varrho_\varepsilon(i,m,a) = 
 -\dim(\varXi)-\varepsilon
 + m\, (\dim(\varPhi)+\varepsilon)
 + i\, (\dim(\lambda)-\varepsilon m_\flat) 
 + [a].
\end{equation}
We also set
\begin{equation}
 \varrho_\varepsilon(i,m) = \varrho_\varepsilon(i,m,0),
 \qquad
 \varrho_\varepsilon(i)=\varrho_\varepsilon(i,0)
\end{equation}
and omit $\varepsilon$ in the notation if $\varepsilon=0$. Moreover, we define
\begin{equation}
 \varepsilon_\diamond:=\dim(\varXi)/3\wedge\dim(\lambda)/(m_\flat+3) \wedge \min_{(i,m,a)\in\frI^+} \varrho(i,m,a)/(2+(i_\diamond+1) m_\flat).
\end{equation}
\end{dfn}

\begin{rem}
We have
$
 \varrho_\varepsilon(\vI) = \varrho_\varepsilon(i_1,m_1,a_1) +\ldots +\varrho_\varepsilon(i_n,m_n,a_n),
$
for a list of indices $\vI$ of the form Eq.~\eqref{eq:list_indices}. 
\end{rem}

\begin{rem}
Observe that the above definition of $\varrho_0(i,m,a)=\varrho(i,m,a)$ is consistent with Def.~\ref{dfn:relevant_irrelevant}. Moreover, for any $\varepsilon>0$ and $(i,m,a)\in\frI_0$ such that $m\leq i m_\flat$ it holds $\varrho_\varepsilon(i,m,a)<\varrho(i,m,a)$. Similarly, for any $\varepsilon>0$ and any list of indices $\vI$ such that $\rrm(\vI)\leq \rri(\vI) m_\flat$ it holds $\varrho_\varepsilon(\vI)<\varrho(\vI)$.
\end{rem}

\begin{rem}\label{rem:epsilon}
We claim that $\varrho_\varepsilon(i,m,a)>\varepsilon$ for all $\varepsilon\in(0,\varepsilon_\diamond)$ and $(i,m,a)\in\frI^+$. Indeed, for all $(i,m,a)\in\frI^+$ such that $i\in\{0,\ldots,i_\diamond+1\}$ we have
\begin{equation}
 \varrho_\varepsilon(i,m,a)-\varepsilon = \varrho(i,m,a)
 -\varepsilon(2+im_\flat-m) > \varrho(i,m,a)-\varepsilon_\diamond (2+(i_\diamond+1) m_\flat) >0.
\end{equation}
Moreover, for all $(i,m,a)\in\frI^+$ such that $i>i_\diamond+1$ it holds
\begin{equation}
 \varrho_\varepsilon(i,m,a) = 
 \varrho_\varepsilon(i_\diamond+1,m,a) + (i-i_\diamond-1)(\dim(\lambda)-\varepsilon m_\flat)\geq \varrho_\varepsilon(i_\diamond+1,m,a)>\varepsilon
\end{equation}
since $(i_\diamond+1,m,a)\in\frI^+$ for any $m\in\bN_0$ and $a\in\frM^m_\sigma$. 

We also claim that $\varrho_\varepsilon(\vI)+(\rn(\vI)-1)\rDim>0$ for all $\varepsilon\in(0,\varepsilon_\diamond)$ and all lists of indices $\vI$ such that $\varrho(\vI)+(\rn(\vI)-1)\rDim>0$ and $\rrm(\vI)\leq\rri(\vI)m_\flat$. If $\rn(\vI)=1$, then the above statement follows from the previous paragraph. If $\rn(\vI)=2$, then $\rri(\vI)>0$ and 
\begin{equation}
 \varrho_\varepsilon(\vI)+(\rn(\vI)-1)\rDim=\varrho_\varepsilon(\vI)+2\dim(\varXi)\geq-2\varepsilon+(\dim(\lambda)-\varepsilon m_\flat)>0
\end{equation}
If $\rn(\vI)>2$, then 
\begin{equation}
 \varrho_\varepsilon(\vI)+(\rn(\vI)-1)\rDim\geq\rn(\vI)\,(\dim(\varXi)-\varepsilon)-\rDim \geq \rn(\vI)\,(\dim(\varXi)/3-\varepsilon)>0
\end{equation}
\end{rem}
\begin{rem}
By the above remarks $\varrho_\varepsilon(i,m,a)\neq 0$ for all $(i,m,a)\in\frI_0$ such that $m\leq i m_\flat$ and $\varrho_\varepsilon(\vI)\neq 0$ for all list of indices such that $\rrm(\vI)\leq \rri(\vI) m_\flat$.
\end{rem}

\subsection{Flow equation for cumulants}

\begin{lem}\label{lem:flow_E_general}
Let $n\in\bN_+$, $i_1\in\bN_0$, $m_1,\ldots,m_n\in\bN_0$, $a_1\in\frM^{m_1}_\sigma$, $r_1\in\{0,1\}$ and $I\equiv \{2,\ldots,n\}$. For any random  distributions $\zeta_q\in\cD^{m_q}$, $q\in I$, we have the following flow equation, in the notation of Def.~\ref{dfn:notation_cumulants_distributions}:
\begin{multline}\label{eq:flow_E_general}
 \llangle[\big] 
 \partial_\sIR \partial_\uv^{r_1} F^{i_1,m_1,a_1}_{\uv;\sUV,\sIR},
 (\zeta_q)_{q\in I}\rrangle[\big]
 =
 -\frac{1}{m_1!}
 \sum_{\pi\in\cP_{m_1}}\sum_{j=1}^{i_1}\sum_{k=0}^{m_1}\sum_{r=0}^{r_1}
 \sum_{b+c+d=\pi(a_1)}\frac{(k+1)a_1!}{b!c!d!}\,
 \\\times\Bigg(
 \sum_{\substack{I_1,I_2\subset I\\I_1\cup I_2= I\\I_1\cap I_2=\emptyset}}
 \fY_\pi\fB\big(\dot G^c_\sIR,
 \llangle[\big]
 \partial_\uv^r F^{j,k+1,b}_{\uv;\sUV,\sIR},(\zeta_q)_{q\in I_1}\rrangle[\big]
 ,\llangle[\big]
 \partial_\uv^{r_1-r} F^{i_1-j,m_1-k,d}_{\uv;\sUV,\sIR},
 (\zeta_q)_{q\in I_2}
 \rrangle[\big]\big)
 \\+
 \fY_\pi\fA\big(\dot G^c_\sIR,
 \llangle[\big]
 \partial_\uv^r F^{j,k+1,b}_{\uv;\sUV,\sIR},
 (\zeta_q)_{q\in I},
 \partial_\uv^{r_1-r} F^{i_1-j,m_1-k,d}_{\uv;\sUV,\sIR}
 \rrangle[\big]\big)
 \Bigg).
\end{multline}
\end{lem}
\begin{proof}
This follows immediately from Eq.~\eqref{eq:flow_i_m_a} and the relation between cumulants expressed in Eq.~\eqref{eq:cumulants_product}.
\end{proof}

\begin{lem}\label{lem:flow_E_form_bound}
Let $\varepsilon\in(0,1]$, $n\in\bN_+$ and
\begin{equation}\label{eq:flow_thm_list_indices}
 \vJ\equiv (\vJ_1,\ldots,\vJ_n)=((i_1,m_1,a_1,s_1,r_1),\ldots,(i_n,m_n,a_n,s_n,r_n))
\end{equation}
be a list of indices such that $s_l=1$ for some $l\in\{1,\ldots,n\}$.
\begin{enumerate}
\item[(A)]
The distribution $E^\vJ_{\uv;\sUV,\sIR}\in\cD^{\mathsf{m}(\vJ)}_\rt$ can be expressed as a linear combination of distributions of the form
\begin{equation}
 \fY^\omega\fY_\pi\fA\big(\dot G^c_\sIR,E^{\vK}_{\uv;\sUV,\sIR}\big)
 \qquad\qquad
 \textrm{or}
 \qquad\qquad
 \fY^\omega\fY_\pi\fB\big(\dot G^c_\sIR,
 E^{\vL}_{\uv;\sUV,\sIR},
 E^{\vM}_{\uv;\sUV,\sIR}
 \big),
\end{equation}
where $c\in\frM$ is some multi-index, $\vK$, $\vL$, $\vM$ are some lists of indices, $\omega\in\cP_n$, $\pi\in\cP_{m_l}$ are some permutations and the maps $\fY^\omega$, $\fY_\pi$ were introduced in Def.~\ref{dfn:map_Y}. The lists of indices $\vK$, $\vL$, $\vM$ satisfy the following conditions
\begin{equation} 
\begin{split}
&\rn(\vK)=\rn(\vJ)+1,\\
&\rri(\vK)=\rri(\vJ),\\
&\rrm(\vK)=\rrm(\vJ)+1,\\
&\ra(\vK)+[c]=\ra(\vJ),\\
&\rs(\vK)=\rs(\vJ)-1,\\
&\rr(\vK)=\rr(\vJ),\\
&\varrho_{2\varepsilon}(\vJ) = \varrho_{2\varepsilon}(\vK)-\rDim +\sigma,
\end{split}
\quad\quad\textrm{or}\qquad\quad
\begin{split}
&\rn(\vL)+\rn(\vM)=\rn(\vJ)+1,\\
&\rri(\vL)+\rri(\vM)=\rri(\vJ),\\
&\rrm(\vL)+\rrm(\vM)=\rrm(\vJ)+1,\\ 
& \ra(\vL)+\ra(\vM)+[c]=\ra(\vJ),\\
&\rs(\vL)+\rs(\vM)=\rs(\vJ)-1,\\
&\rr(\vL)+\rr(\vM)=\rr(\vJ),\\
&\varrho_{2\varepsilon}(\vJ) = \varrho_{2\varepsilon}(\vL)+\varrho_{2\varepsilon}(\vM)+\sigma.
\end{split}
\end{equation}

\item[(B)] Fix some $\oo\in\bN_+$. Suppose that the bound
\begin{equation}\label{eq:lem_cumulants}
 \|K^{\rn(\vI),\rrm(\vI);\oo}_{\uv\vee\sUV,\sIR}\ast E^\vI_{\uv;\sUV,\sIR}\|_{\cV^{\mathsf{m}(\vI)}}
 \lesssim [\uv]^{(\varepsilon-\sigma)\rr(\vI)}\,
 [\sIR]^{\varrho_{2\varepsilon}(\vI)-\sigma \rs(\vI)+(\rn(\vI)-1)\rDim}
\end{equation}
holds uniformly in $(\uv,\sUV)\in[0,1]^2\setminus\{(0,0)\}$, $\sIR\in[\uv\vee\sUV,1]$ for all lists of indices $\vI\in\{\vK,\vL,\vM\}$, where $\vK,\vL,\vM$ are arbitrary lists of indices satisfying the conditions specified in Part (A) given the list of indices~$\vJ$. Then the above bound holds uniformly in $(\uv,\sUV)\in[0,1]^2\setminus\{(0,0)\}$, $\sIR\in[\uv\vee\sUV,1]$ for the list of indices $\vI=\vJ$.
\end{enumerate}
\end{lem}
\begin{proof}
Without loss of generality we can assume that $l=1$. Part (A) of the theorem follows from Lemma~\ref{lem:flow_E_general} applied with
\begin{equation}
 \zeta_q\equiv \partial_\sIR^{s_q}\partial_\uv^{r_q} F^{i_q,m_q,a_q}_{\uv;\sUV,\sIR}, 
 \qquad q\in\{2,\ldots,n\}.
\end{equation}
We observe that the multi-index $c\in\frM$ and the permutation $\pi\in\cP_{m_1}$ in the statement of the theorem coincide with the respective objects in the expression~\eqref{eq:flow_E_general}. The permutation $\omega\in\cP_n$ is trivial because of the assumption $l=1$. Moreover, it holds
\begin{equation}
\begin{gathered}
 \vK=((j,k+1,b,0,r),\,\vJ_2,\ldots,\vJ_n,\,(i_1-j,m_1-k,d,0,r_1-r)),
 \\
 \vL= (j,k+1,b,0,r)\sqcup (\vJ_q)_{q\in I_1},
 \qquad
 \vM =(i_1-j,m_1-k,d,0,r_1-r) \sqcup (\vJ_q)_{q\in I_2},
\end{gathered} 
\end{equation}
where $\sqcup$ denotes the concatenation of lists, $I_1\cup I_2=I=\{2,\ldots,n\}$, $I_1\cap I_2=\emptyset$ and $j\in\{1,\ldots,i_1\}$, $k\in\{0,\ldots,m_1\}$, $b\in\frM_\sigma^{1+k}$, $d\in\frM_\sigma^{m-k}$ and $r\in\{0,1\}$ coincide with the respective objects in the expression~\eqref{eq:flow_E_general}. This implies that the lists $\vK$, $\vL$, $\vM$ satisfy the conditions specified in Part~(A) of the theorem. Only the last of these conditions is nontrivial and it follows from Eq.~\eqref{eq:dfn_varrho_I} and the fact that $\dim(\varXi)+\dim(\varPhi)=\rdim=\rDim-\sigma$. Part~(B) follows from Part~(A), Lemma~\ref{lem:fA_fB_bounds}, Lemma~\ref{lem:fA_fB_Ks} and the bounds
\begin{equation}
 \|\fR^{\phantom{2}}_{\uv\vee\sUV}\fP^{2\oo}_\sIR \dot G_\sIR^a\|_\cK \lesssim [\sIR]^{[a]},
 \qquad
 \|\fT\fR^{\phantom\oo}_{\uv\vee\sUV}\fP_\sIR^\oo\dot G^a_\sIR\|_\cV \lesssim [\sIR]^{[a]-\rDim}.
\end{equation}
uniform in $\sUV\in[0,1]$ and $\sIR\in[\uv\vee\sUV,1]$. These bounds follow immediately from the bounds stated in Remark~\ref{rem:fA_fB_Ks}.
\end{proof}

\begin{rem}\label{rem:cumulants_problem}
Note that for $\sIR\in[\uv\vee\sUV,1]$ the bound~\eqref{eq:lem_cumulants} coincides with the bound~\eqref{eq:thm_cumulants} stated in Theorem~\ref{thm:cumulants}, which is the main result of Sec.~\ref{sec:cumulants_estimates}. Unfortunately, the result stated in Part~(B) of the above theorem with the bound~\eqref{eq:lem_cumulants} replaced by the bound~\eqref{eq:thm_cumulants} and $\sIR\in[\uv\vee\sUV,1]$ replaced by $\sIR\in(0,\uv\vee\sUV]$ is in general false even for $\oo=0$. Let us mention that such a result for $\oo=0$ would be true if the following bounds
\begin{equation}
 \|\fR^{\phantom{2}}_{\uv\vee\sUV} \dot G_\sIR^a\|_\cK\lesssim [\uv\vee\sUV]^{\sigma-\varepsilon+[a]}\,[\sIR]^{\varepsilon-\sigma},
 \quad
 \|\fT|\fR^{\phantom\oo}_{\uv\vee\sUV}\dot G^a_\sIR|\|_\cV \lesssim
 [\uv\vee\sUV]^{\sigma-\varepsilon-\rDim+[a]}\,[\sIR]^{\varepsilon-\sigma},
\end{equation} 
were satisfied uniformly in $\sUV\in[0,1]$ and $\sIR\in(0,\uv\vee\sUV]$. However, the second of the above bounds does not hold (compare the bounds stated in Remark~\ref{rem:fA_fB_Ks}). For the above-mentioned reason the proof of the main bound~\eqref{eq:thm_cumulants} for $\sIR\in(0,\uv\vee\sUV]$ requires a special treatment and is discussed in Sec.~\ref{sec:cumulants_preliminary}.
\end{rem}

\subsection{Bounds for cumulants of noise}

In this section we prove bounds for moments and cumulants of the noises $\varXi_\sUV$ and $\varXi_{\uv;\sUV}=M_\uv\ast\varXi_\sUV$, which follow from Assumption~\ref{ass:noise}.

\begin{rem}\label{rem:mixing}
By Assumption~\ref{ass:noise} (C) for any $\sUV\in(0,1]$ and any compact sets $\cO_1$, $\cO_2 \subset \bH$ such that 
\begin{equation}
 \inf_{x_1 \in \cO_1,x_2\in\cO_2}|x_1 - x_2|_{\bH} \geq [\sUV]=\sUV^{1/\sigma}
\end{equation}
the \mbox{$\sigma$-algebras} generated by $\{\varXi_\sUV(x)\,|\, x\in \cO_1\}$ and $\{\varXi_\sUV(x)\,|\, x\in \cO_2\}$ are independent. Recall that the distance function $|\Cdot|_{\bH}$ was introduced in Def.~\ref{dfn:spacetime_distance}.
\end{rem}

\begin{lem}\label{lem:moments_cumulants}
Let $C,c>0$ and $n\in\bN_+$ be fixed. Suppose that the kernel $H\in\cK$ is such that $\fT H\in L^\infty(\bH)$ and $\|H\|_\cK\leq 1$ and $\|\fT H\|_{L^\infty(\bH)} \leq c$ and the cumulants of a random field $\zeta\in C(\bH)$ satisfy the bound
\begin{equation}
 \sup_{x_1\in\bH}\int_{\bH^{m-1}} |\llangle \zeta(x_1),\zeta(x_2),\ldots,\zeta(x_m) \rrangle|\, \rd x_2\ldots \rd x_m
 \leq c\,C^m
\end{equation}
for all $m\in\{1,\ldots,n\}$. Then
\begin{equation}
 \sup_{x\in\bH}\,|\llangle (H \ast \zeta(x))^n \rrangle| \leq B_n\,c^n\,C^n,
\end{equation}
where $B_n$ is the number of all partitions of the set $\{1,\ldots,n\}$.
\end{lem}
\begin{rem}\label{rem:moments_cumulants}
Recall that $\fT H$ denotes the periodization in space of a kernel $H$, which was introduced in Def.~\ref{dfn:periodization}. The above lemma follows immediately from the identity $H\ast \zeta=\fT H\star \zeta$, where $\star$ is the convolution in $\bH$, and the relation~\eqref{eq:expectation_cumulants} between moments and cumulants. Note also that it holds
\begin{equation}
 \sup_{x\in\bH}\, |\llangle (H \ast \zeta(x))^n \rrangle|
 \leq \|H\|_\cK^n \sup_{x_1,\ldots,x_n\in\bH}\,\llangle |\zeta(x_1)|\ldots|\zeta(x_n)|\rrangle
 \leq 
 \sup_{x\in\bH}\, \llangle |\zeta(x)|^n\rrangle.
\end{equation}
\end{rem}

\begin{lem}\label{lem:ass_noise_general}
Assumption~\ref{ass:noise} implies the following:
\begin{enumerate}
 \item[(A)] For any $n\in\bN_+$ the bound
\begin{equation}
 \sup_{x\in\bH} |\llangle (\fP_\sUV\varXi_{\sUV}(x))^n \rrangle| \lesssim [\sUV]^{-n \dim(\varXi)}
\end{equation}
holds uniformly in $\sUV\in(0,1]$.
\item[(B)] For any $n\in\bN_+$ the bound
\begin{equation}
 \sup_{x_1\in\bH}\int_{\bH^{n-1}}|\llangle \fP_\sUV\varXi_{\sUV}(x_1),\fP_\sUV\varXi_{\sUV}(\rd x_2), \ldots,\fP_\sUV\varXi_{\sUV}(\rd x_n)\rrangle|
 \lesssim [\sUV]^{n \dim(\varXi)- \rDim}
\end{equation}
holds uniformly in $\sUV\in[0,1]$.
\item[(C)] For any $n\in\bN_+$ and $\varepsilon\in[0,1\wedge\sigma]$ the bounds
\begin{equation}
\begin{gathered}
 \sup_{x\in\bH}\,|\llangle (\fP_{\sUV}\fP_{\uv} \varXi_{\uv;\sUV}(x))^n\rrangle| \lesssim [\uv\vee\sUV]^{-n \dim(\varXi)},
 \\
 \sup_{x\in\bH}\,|\llangle (\fP_\uv \partial_\uv \varXi_{\uv;\sUV}(x))^n\rrangle| \lesssim [\uv]^{n(\varepsilon-\sigma)} [\uv\vee\sUV]^{-n \dim(\varXi)-n \varepsilon}
\end{gathered} 
\end{equation}
hold uniformly in $\uv,\sUV\in(0,1]$.
\item[(D)] For any $n\in\bN_+$, $r_1,\ldots,r_n\in\{0,1\}$ and $\varepsilon\in[0,1\wedge\sigma]$ the bound
\begin{multline}
 \sup_{x_1\in\bH}\int_{\bH^{n-1}}|\llangle K_\sIR\ast \partial_\uv^{r_1}\varXi_{\uv;\sUV}(x_1),K_\sIR\ast \partial_\uv^{r_2}\varXi_{\uv;\sUV}(\rd x_2),\ldots,K_\sIR\ast \partial_\uv^{r_n}\varXi_{\uv;\sUV}(\rd x_n)\rrangle|
 \\
 \lesssim 
 [\uv]^{(\varepsilon-\sigma)r}\,
 [\uv\vee\sUV\vee\sIR]^{n\dim(\varXi)-\rDim-\varepsilon r}
\end{multline}
holds uniformly in $(\uv,\sUV)\in[0,1]^2\setminus\{(0,0)\}$, $\sIR\in[0,1]$, where $r=r_1+\ldots+r_n$.
\end{enumerate}
\end{lem}
\begin{rem}\label{rem:cumulant_noise}
Let $\mathsf{m}=(0,\ldots,0)\in\bN_0^n$. By the above lemma 
\begin{equation}
 \|\llangle \partial_\uv^{r_1}\varXi_{\uv;\sUV},\ldots,\partial_\uv^{r_n}\varXi_{\uv;\sUV}\rrangle\|_{L^\infty(\bH^n)}\lesssim 
 [\uv]^{(\varepsilon-\sigma)r}\,
 [\uv\vee\sUV]^{-n\dim(\varXi)-\varepsilon r},
\end{equation}
\begin{equation}
 \|\llangle \partial_\uv^{r_1}\varXi_{\uv;\sUV},\ldots,\partial_\uv^{r_n}\varXi_{\uv;\sUV}\rrangle\|_{\cV^{\mathsf{m}}}\lesssim 
 [\uv]^{(\varepsilon-\sigma)r}\,
 [\uv\vee\sUV]^{-n\dim(\varXi)-\varepsilon r+(n-1)\rDim}
\end{equation}
uniformly in $(\uv,\sUV)\in[0,1]^2\setminus\{(0,0)\}$. The extra factor $[\uv\vee\sUV]^{(n-1)\rDim}$ on the RHS of the second bound is due to the fact that by
\begin{equation}
 \varXi_{\uv;\sUV}=M_\uv\ast\varXi_\sUV,\qquad 
 \supp\, M_\uv\subset\{x\in\bM\,|\, |x|_\bM<[\uv]\},
\end{equation}
and Remark~\ref{rem:mixing} the cumulant $\llangle \partial_\uv^{r_1}\varXi_{\uv;\sUV},\ldots,\partial_\uv^{r_n}\varXi_{\uv;\sUV}\rrangle \in C(\bH^n)$ is not zero only on a subset of $\bH^n$ whose volume is proportional to $[\uv\vee\sUV]^{(n-1)\rDim}$.
\end{rem}

\begin{proof}
First note that if $n=1$, then the lemma holds true trivially since by assumption $\varXi_\sUV$ is centered. The bound (A) follows from Assumption~\ref{ass:noise} (B) and the equality
\begin{equation}
\fP_\sUV\varXi_\sUV(\mathring x,\bar x) = [\sUV]^{-\dim(\varXi)}\,(\fP\microXi_\sUV)(\mathring x/\sUV,\bar x/[\sUV]).
\end{equation} 
If $\sUV=0$, then $\fP_\sUV\varXi_\sUV$ coincides with the white noise in $\bH$ and the only non-vanishing cumulant is the covariance. Consequently, the bound (B) for $\sUV=0$ follows from the equality $n\dim(\varXi)-\rDim= 0$ for $n=2$ and
\begin{equation}
 \sup_{x_1\in\bH}\int_{\bH}|\llangle \varXi_0(x_1),\varXi_0(\rd x_2)\rrangle| = 1.
\end{equation}
The bound (A) implies that 
\begin{equation}\label{eq:bound_cumulant_noise_sup}
 |\llangle \fP_\sUV\varXi_\sUV(x_1), \ldots,\fP_\sUV\varXi_\sUV(x_n)\rrangle|
 \lesssim [\sUV]^{-n \dim(\varXi)}
\end{equation}
holds uniformly in $x_1,\ldots,x_n\in\bH$ and $\sUV\in(0,1]$. By the independence condition stated in Remark~\ref{rem:mixing} the cumulant $\llangle \fP_\sUV\varXi_\sUV(x_1),\ldots,\fP_\sUV\varXi_\sUV(x_n)\rrangle$ vanishes unless
\begin{equation}
 |\mathring x_p-\mathring x_q|\leq \sUV\qquad\mathrm{and}\qquad
 |\bar x_p^r-\bar x_q^r+ 2\pi k|\leq [\sUV]\quad\mathrm{for~some}\quad k\in\bZ
\end{equation}
for all $r\in\{1,\ldots,\rdim\}$ and $p,q\in\{1,\ldots,n\}$, $p\neq q$. Using the above mention property together with the bound~\eqref{eq:bound_cumulant_noise_sup} we obtain the bound (B) for $\sUV\in(0,1]$.

Recall that $\varXi_{\uv;\sUV}=M_\uv\ast\varXi_\sUV$, where $M_\uv=\fS_\uv(M)$. Set $\ooo=\rdim+1$. Let $N=\fP^{\ooo+1} M\in C_\rc^\infty(\bM)$ and $N_\uv=\fS_\uv(N)$. Thus, $\fP_\sUV\fP_\uv\varXi_{\uv;\sUV}=K_\uv^{\ast\ooo}\ast N_\uv\ast\fP_\sUV\varXi_\sUV$. Note that
$\|\fT K_\uv^{\ast\ooo}\|_{L^\infty(\bH)}\lesssim [\uv]^{-\rDim}$ by Lemma~\ref{lem:kernel_simple_fact}~(D) and $\|K_\uv\|_\cK,\|N_\uv\|_\cK\lesssim 1$. As a result, the first of the bounds (C) follows for $\uv\leq\sUV$ from the bound (A) and for $\uv>\sUV$ from the bound (B), Lemma~\ref{lem:moments_cumulants} and the fact that $\varXi_\sUV$ is centered. The second of the bounds (C) follows by the same argument since
\begin{multline}
 \fP_\uv\partial_\uv \varXi_{\uv;\sUV}
 =\fP_\uv \partial_\uv M_\uv\ast\varXi_\sUV
 =\fP_\uv \partial_\uv (K^{\ast(\ooo+1)}_\uv\ast N_\uv)\ast\varXi_\sUV
 \\
 =K^{\ast\ooo}_\uv \ast(\partial_\uv N_\uv
 +(\ooo+1) N_\uv\ast\fP_\uv \partial_\uv K_\uv)\ast K_\sUV\ast \fP_\sUV \varXi_\sUV
\end{multline}
and for $\varepsilon\in[0,1\wedge\sigma]$
\begin{equation}
 \|(\partial_\uv N_\uv
 +(\ooo+1) N_\uv\ast\fP_\uv \partial_\uv K_\uv)\ast K_\sUV\|_\cK \lesssim [\uv]^{\varepsilon-\sigma}[\uv\vee\sUV]^{-\varepsilon}
\end{equation}
uniformly in $\uv,\sUV\in(0,1]$ by Lemma~\ref{lem:kernel_derivative} and Lemma~\ref{lem:bounds_M}.

In order to show the bound (D) observe that we have
\begin{equation}\label{eq:noise_proof_cumulants}
 K_\sIR\ast\partial_\uv^r\varXi_{\uv;\sUV} 
 = 
 \partial_\uv^r M_\uv\ast K_\sUV\ast K_\sIR \ast \fP_\sUV\varXi_\sUV. 
\end{equation}
and by Lemma~\ref{lem:bounds_M} it holds
\begin{equation}\label{eq:derivative_proof_cumulants}
 \|\partial_\uv^r M_\uv\ast K_\sUV\ast K_\sIR\|_\cK\lesssim [\uv]^{(\varepsilon-\sigma)r} [\uv\vee\sUV\vee\sIR]^{-\varepsilon r}
\end{equation}
for $r\in\{0,1\}$. The bound (D) follows now from the estimate
\begin{multline}
 \sup_{x_1\in\bH}\int_{\bH^{n-1}}|\llangle K_\sIR\ast \partial_\uv^{r_1}\varXi_{\uv;\sUV}(x_1),K_\sIR\ast \partial_\uv^{r_1}\varXi_{\uv;\sUV}(\rd x_2),\ldots,K_\sIR\ast \partial_\uv^{r_n}\varXi_{\uv;\sUV}(\rd x_n)\rrangle|
 \\
 \leq 
 \|\partial_\uv^{r_1} M_\uv\ast K_\sUV\ast K_\sIR\|_\cK\ldots 
 \|\partial_\uv^{r_n} M_\uv\ast K_\sUV\ast K_\sIR\|_\cK\,
 \\
 \times \sup_{x_{1}\in\bH}\int_{\bH^{n-1}}|\llangle \fP_\sUV\varXi_\sUV(x_1),\fP_\sUV\varXi_\sUV(\rd x_2),\ldots,\fP_\sUV\varXi_\sUV(\rd x_n)\rrangle|,
\end{multline}
the bound (B) and the fact that $\varXi_\sUV$ is centered.
\end{proof}

\subsection{Preliminary bounds}\label{sec:cumulants_preliminary}

The aim of this section is to establish an auxiliary result that will be used in Sec.~\ref{sec:cumulants_uniform_bounds} to prove the bounds~\eqref{eq:thm_cumulants} for the cumulants $E^\vI_{\uv;\sUV,\sIR}$ in the range of parameters $(\uv,\sUV)\in[0,1]^2\setminus\{(0,0)\}$, $\sIR\in[0,\uv\vee\sUV]$. This range of parameters requires a special treatment because of the issue mentioned in Remark~\ref{rem:cumulants_problem}. In contrast to the proof of the bounds~\eqref{eq:thm_cumulants} in the full range of parameters, which is the main result of Sec.~\ref{sec:cumulants_estimates}, the analysis of this subsection does not involve any fine tuning (or renormalization) of the force coefficients. All the bounds proved in this subsection are straightforward in nature and morally follow from dimensional analysis. Unfortunately, to establish them rigorously we need to introduce several new objects.

Recall that the fractional heat kernel $G_\sIR$ with UV cutoff $\sIR\in(0,\infty)$ was introduced in Def.~\ref{dfn:propagator_G}. Our analysis so far was based on the decomposition of scales expressed by the first equality below
\begin{equation}
 G-G_\sIR = \int_0^\sIR \dot G_\uIR\,\rd\uIR = \int_0^\sIR \partial_\uIR G_{\uv\vee\sUV;\uIR,0}\,\rd\uIR + \int_0^\sIR \partial_\uIR G_{\uv\vee\sUV;\sIR,\uIR}\,\rd\uIR.
\end{equation}
In this section we use a more sophisticated decomposition expressed by the second equality. The kernel $\partial_\uIR G_{\uv\vee\sUV;\uIR,0}$ have compact support and the force coefficients $F^{i,m,a}_{\uv;\sUV,\sIR,0}$, which are defined as in Sec.~\ref{sec:effective_force} but with the kernel $\dot G_\sIR$ replaced by $\partial_\uIR G_{\uv\vee\sUV;\uIR,0}$, satisfy a certain independence condition similar to the independence condition of the noise $\varXi_\sUV$ stated in Remark~\ref{rem:mixing}. This allows us to exploit an argument similar to the one mentioned in Remark~\ref{rem:cumulant_noise} in the proof of bounds for cumulants of $F^{i,m,a}_{\uv;\sUV,\sIR,0}$. We note that the original effective force coefficients $F^{i,m,a}_{\uv;\sUV,\sIR}$ can be obtained by integrating the RHS of a certain flow equation involving the kernel $\partial_\uIR G_{\uv\vee\sUV;\sIR,\uIR}$ with the initial condition coinciding with $F^{i,m,a}_{\uv;\sUV,\sIR,0}$. The kernels $\partial_\uIR G_{\uv\vee\sUV;\sIR,\uIR}$ are more regular at spatial scales of order $[\uv\vee\sUV]$ and in particular satisfy the bounds stated in Remark~\ref{rem:cumulants_problem}. Recall that the second of these bounds is violated by the kernel $\dot G_\sIR$, which is the reason why the analysis presented in this section is needed.

\begin{dfn}
For $\sUV\in(0,\infty)$, $\sIRa,\sIRb\in[0,\infty)$ we define
\begin{equation}
 G_{\sUV;\sIRa,\sIRb}(\mathring x,\bar x):= 
 G_{\sIRa}(\mathring x,\bar x)\,(1-\chi(|\bar x|^\sigma/\sUV))
 + 
 G_{\sIRb}(\mathring x,\bar x)\,\chi(|\bar x|^\sigma/\sUV).
\end{equation}
where $G_\sIR$ is the fractional heat kernel with the cutoff $\sIR\in[0,\infty)$ introduced in Def.~\ref{dfn:propagator_G} and the function $\chi\in C^\infty(\bR)$ is such that $\chi(t)=0$ for $t\leq 1$ and $\chi(t)=1$ for $t\geq 2$ was fixed in Def.~\ref{dfn:propagator_G}. For $\sIR\in(0,\infty)$ and $a\in\frM$ we set
\begin{equation}
\begin{gathered}
 \dot G_{\sUV;\sIR,0}:=\partial_{\sIR} G_{\sUV;\sIR,0},
 \qquad
 \dot G_{\sUV;0,\sIR}:=\partial_{\sIR} G_{\sUV;\sIR,0},
 \\
 \dot G^a_{\sUV;\sIR,0}:=\cX^a \dot G_{\sUV;\sIR,0},
 \qquad
 \dot G^a_{\sUV;0,\sIR}:=\cX^a \dot G_{\sUV;0,\sIR},
\end{gathered} 
\end{equation}
where $\cX^a(x)=x^a$.
\end{dfn}

\begin{rem}
Note that $G_{\sUV;\sIR,\sIR}=G_\sIR$ and $\dot G^a_{\sUV;\sIR,0}+\dot G^a_{\sUV;0,\sIR}=\dot G^a_\sIR$. Moreover, for all $\sUV\in[0,1]$, $\sIR\in(0,\sUV]$ we have
\begin{equation}
 \supp \, \dot G_{\sUV;\sIR,0}^a
 \subset \{x\in\bM\,|\,\mathring x\in [\sIR,2\sIR], |\bar x|\leq 2[\sUV]\}.
\end{equation}
\end{rem}

\begin{lem}\label{lem:kernel_dot_G_2}
For any $a\in\frM_\sigma$ and $r\in\{0,1\}$ it holds
\begin{enumerate}
 \item[(A)] $\|\fR_\sUV\partial_\sUV^r\dot G^a_{\sUV;\sIR,0}\|_\cK \lesssim [\sUV]^{\sigma-\sigma r-\varepsilon+[a]}\,[\sIR]^{\varepsilon-\sigma}$,
 \item[(B)] $\|\fR_\sUV\partial_\sUV^r\dot G^a_{\sUV;0,\sIR}\|_\cK \lesssim [\sUV]^{\sigma-\sigma r-\varepsilon+[a]}\,[\sIR]^{\varepsilon-\sigma}$,
 \item[(C)] $\|\fT|\fR_\sUV\partial_\sUV^r\dot G^a_{\sUV;0,\sIR}|\|_\cV \lesssim [\sUV]^{\sigma-\sigma r-\varepsilon-\rDim}\,[\sIR]^{\varepsilon-\sigma}$
\end{enumerate}
uniformly in $\sUV\in[0,1]$, $\sIR\in(0,\sUV)$, where $\|\Cdot\|_\cV=\|\Cdot\|_{L^\infty(\bH)}$, $\|\Cdot\|_\cK=\|\Cdot\|_{L^1(\bM)}$, the periodization $\fT$ was introduced in Def.~\ref{dfn:periodization} and $\rDim=\rdim+\sigma$.
\end{lem}
\begin{rem}
The bound (C) does not hold for the kernel $\dot G^a_{\sUV;\sIR,0}$. Only this bound does: $$\|\fT|\fR_\sUV\partial_\sUV^r\dot G^a_{\sUV;\sIR,0}|\|_\cV \lesssim [\sUV]^{\sigma-\sigma r-\varepsilon}\,[\sIR]^{\varepsilon-\sigma-\rDim}.$$
\end{rem}

\begin{proof}
First note that the bound (A) follows from the bound (B) by the equality $\dot G_{\sUV;\sIR,0}=\dot G^a_\sIR-\dot G^a_{\sUV;\sIR,0}$ and Lemma~\ref{lem:kernel_dot_G}. For any spatial multi-index $b\in\bar\frM$ it holds 
\begin{equation}
\begin{gathered}
 \partial^b\partial_\sUV^r\dot G^a_{\sUV;0,\sIR}
 =
 [\sUV]^{[a]-|b|-\sigma r}\, \fS_\sUV\big(\partial^b \partial_\sUV^r \dot G^a_{\sUV;0,\tau}|_{\sUV=1,\tau=\sIR/\sUV}\big),
\end{gathered} 
\end{equation}
where the rescaling operator $\fS_\sUV$ was introduced in Def.~\ref{dfn:scaling_S} and $\partial^b\equiv\partial^b_x$. Using the properties of the function $\chi$ and Remark~\ref{rem:dot_G_1} we obtain that for $\sUV=1$, $a\in\frM$, $b\in\bar\frM$, $r\in\{0,1\}$ the bound
\begin{multline}
 |\partial^b \partial_\sUV^r \dot G^a_{\sUV;0,\tau}(\mathring x,\bar x)|
 =|\mathring x/\tau^2\,\dot \chi(\mathring x/\tau)\, \partial^b x^a G(x)\, \partial_\sUV^r (1-\chi(|\bar x|^\sigma/\sUV))|
 \\
 \lesssim |\mathring x|/\tau^2\,1_{[\tau,2\tau]}(\mathring x)\,
 |\mathring x|^{1+\mathring a}\,|x|_\bM^{|\bar a|-\sigma-\rdim}\,1_{[1,\infty)}(|\bar x|)
 \lesssim
 1_{[\tau,2\tau]}(\mathring x)\,1_{[1,\infty)}(|\bar x|)\,|\bar x|^{|\bar a|-\sigma-\rdim}, 
\end{multline}
where $\dot\chi(t)=\partial_t\chi(t)$, holds uniformly in $\tau\in(0,1]$ and $x=(\mathring x,\bar x)\in\bM$. This bound implies that for $\sUV=1$, $a\in\frM_\sigma$, $b\in\bar\frM$, $r\in\{0,1\}$ the bounds
\begin{equation}
 \|\partial^b \partial_\sUV^r \dot G^a_{\sUV;0,\tau}\|_{\cK} \lesssim 1,
 \qquad
 \sup_{\mathring x\in\bR}\|\partial^b \partial_\sUV^r \dot G^a_{\sUV;0,\tau}(\mathring x,\Cdot)\|_{L^1(\bR^\rdim)} \lesssim 1
\end{equation}
hold uniformly in $\tau\in(0,1]$. As a result, using the equality $\bar\fP_\sUV=1-[\sUV]^2\Delta_{\bar x}$ and the properties of the rescaling operator $\fS_\sUV$ we obtain that for any $\oo\in\bN_0$ the bounds
\begin{equation}
\begin{gathered}
 \|\bar\fP_\sIR^{\oo} \partial_{\sUV}^r \dot G_{\sUV;0,\sIR}\|_\cK \lesssim [\sUV]^{[a]-\sigma r}
 \qquad
 \sup_{\mathring x\in\bR}\|\bar\fP_\sIR^{\oo}\partial_{\sUV}^r\dot G^a_{\sUV;0,\sIR}(\mathring x,\Cdot)\|_{L^1(\bR^\rdim)}
 \lesssim [\sUV]^{[a]-\sigma r-\sigma}
\end{gathered} 
\end{equation}
hold uniformly in $\sUV\in(0,1]$ and $\sIR\in(0,\sUV)$. The bounds (B) and (C) follow now from Lemma~\ref{lem:kernel_simple_fact} (C), (D) and the bounds
\begin{equation}
\begin{gathered}
 \|\fR_\sUV\partial_{\sUV}^r \dot G^a_{\sUV;0,\sIR}\|_\cK 
 \leq
 \|\fR_\sUV  \bar K_\sUV^{\ast\oooo}\|_{L^1(\bR^\rdim)}\,
 \|\bar\fP^{\oooo}_{\sUV} \partial_{\sUV}^r \dot G^a_{\sUV;0,\sIR}\|_\cK,
 \\
 \|\fT |\fR_\sUV \partial_{\sUV}^r\dot G^a_{\sUV;0,\sIR}|\|_\cV 
 \leq 
 \|\fT \bar K_\sUV^{\ast\ooo}\|_{L^\infty(\bT)}\,
 \|\fR_\sUV \bar K_\sUV^{\ast\oooo}\|_{L^1(\bR^\rdim)}\,
 \sup_{\mathring x\in\bR}\|\bar\fP^{\ooo+\oooo}_{\sUV} \partial_{\sUV}^r \dot G_{\sUV;0,\sIR}(\mathring x,\Cdot)\|_{L^1(\bR^\rdim)}
\end{gathered} 
\end{equation}
applied with $\ooo=\rdim+1$ and $\oooo=\floor{\sigma/2}+1$. This finishes the proof.
\end{proof}

\begin{dfn}\label{dfn:eff_force_uv_2sIR}
For $(\uv,\sUV)\in[0,1]^2\setminus\{(0,0)\}$ and $\sIRa,\sIRb\in[0,1]$ the effective force coefficients $F^{i,m,a}_{\uv;\sUV,\sIRa,\sIRb}$, $(i,m,a)\in\frI_0$, are defined by the flow equations
\begin{multline}\label{eq:flow_deterministic_i_m_a_hat}
 \partial_{\sIRa}^{\phantom{i}} F^{i,m,a}_{\uv;\sUV,\sIRa,0}
 =
 -\frac{1}{m!}
 \sum_{\pi\in\cP_m}\sum_{j=1}^i\sum_{k=0}^m
 \sum_{b+c+d=\pi(a)}\frac{(k+1)a!}{b!c!d!}\,
 \\\times
 \fY_\pi\fB\big(\dot G^c_{\uv\vee\sUV;\sIRa,0},F^{j,k+1,b}_{\uv;\sUV,\sIRa,0},F^{i-j,m-k,d}_{\uv;\sUV,\sIRa,0}\big),
\end{multline}
\begin{multline}\label{eq:flow_deterministic_i_m_a_check}
 \partial_{\sIRb}^{\phantom{i}} F^{i,m,a}_{\uv;\sUV,\sIRa,\sIRb}
 =
 -\frac{1}{m!}
 \sum_{\pi\in\cP_m}\sum_{j=1}^i\sum_{k=0}^m
 \sum_{b+c+d=\pi(a)}\frac{(k+1)a!}{b!c!d!}\,
 \\\times
 \fY_\pi\fB\big(\dot G^c_{\uv\vee\sUV;0,\sIRb,},F^{j,k+1,b}_{\uv;\sUV,\sIRa,\sIRb},F^{i-j,m-k,d}_{\uv;\sUV,\sIRa,\sIRb}\big). 
\end{multline}
together with the conditions 
\begin{equation}
 F^{i,m,a}_{\uv;\sUV,\sIRa,\sIRb}=0~~\textrm{if}~~m>i m_\flat,
 \qquad
 F^{i,m,a}_{\uv;\sUV,0,0}=F_{\uv;\sUV}^{i,m,a},\quad (i,m,a)\in\frI_0,
\end{equation}
where the force coefficients $F_{\uv;\sUV}^{i,m,a}$ were introduced in Def.~\ref{dfn:noise_two_parameters}.
\end{dfn}

\begin{rem}
For every $(\uv,\sUV)\in[0,1]^2\setminus\{(0,0)\}$ and $\sIRa,\sIRb\in[0,1]$ the effective force coefficients $F^{i,m,a}_{\uv;\sUV,\sIRa,\sIRb}\in\cD^{m;0}$, $(i,m,a)\in\frI_0$, are uniquely defined random distributions. Their construction is pathwise and similar to the recursive construction described in the proof of Lemma~\ref{lem:eff_force_basic_properties}. At each step of the recursion we first define $F^{i,m,a}_{\uv;\sUV,\sIRa,0}$ with the use of the flow equation~\eqref{eq:flow_deterministic_i_m_a_hat} with the initial condition $F^{i,m,a}_{\uv;\sUV,0,0}=F^{i,m,a}_{\uv;\sUV}$. Next, we define $F^{i,m,a}_{\uv;\sUV,\sIRa,\sIRb}$ using the flow equation~\eqref{eq:flow_deterministic_i_m_a_check} with the initial condition $F^{i,m,a}_{\uv;\sUV,\sIRa,0}$ that has already been constructed above.
\end{rem}

\begin{rem}
Let us make a few crucial observation based on Remark~\ref{rem:discrete_RG}. First, note that it holds
\begin{equation}
 F_{\uv;\sUV,\sIRa,\sIRb}[\varphi] = F_{\uv;\sUV}[\varphi+(G-G_{\uv\vee\sUV;\sIRa,\sIRb})\ast F_{\uv;\sUV,\sIRa,\sIRb}[\varphi]]
\end{equation}
In particular, because $G_{\uv\vee\sUV;\sIR,\sIR}=G_{\sIR}$ we get
\begin{equation}
 F_{\uv;\sUV,\sIR,\sIR}[\varphi] = F_{\uv;\sUV}[\varphi+(G-G_\sIR)\ast F_{\uv;\sUV,\sIR,\sIR}[\varphi]].
\end{equation}
Since the above equation defines the effective force coefficients $F^{i,m,a}_{\uv;\sUV,\sIR,\sIR}$ uniquely and since the effective force coefficients $F^{i,m,a}_{\uv;\sUV,\sIR}$ satisfy
\begin{equation}
 F_{\uv;\sUV,\sIR}[\varphi] = F_{\uv;\sUV}[\varphi+(G-G_\sIR)\ast F_{\uv;\sUV,\sIR}[\varphi]]
\end{equation}
we obtain that
\begin{equation}\label{eq:effective_force_eq_hat_check}
 F^{i,m,a}_{\uv;\sUV,\sIR} = F^{i,m,a}_{\uv;\sUV,\sIR,\sIR},
 \qquad
 \partial_\sIR F^{i,m,a}_{\uv;\sUV,\sIR} = \partial_{\sIRa} F^{i,m,a}_{\uv;\sUV,\sIRa,\sIR}\big|_{\sIRa=\sIR}
 + \partial_{\sIR} F^{i,m,a}_{\uv;\sUV,\sIRa,\sIR}\big|_{\sIRa=\sIR}
\end{equation} 
for all $(i,m,a)\in\frI_0$ and $(\uv,\sUV)\in[0,1]^2\setminus\{(0,0)\}$, $\sIR\in[0,1]$.
\end{rem}

We start our analysis of the properties of the effective force coefficients $F^{i,m,a}_{\uv;\sUV,\sIRa,\sIRb}$ with the following deterministic lemma.

\begin{rem}
Recall that for $R>0$ and $i,m\in\bN_0$ we set $R^{i,m}=R^{1+i m_\flat-m}$. 
\end{rem}
\begin{lem}\label{lem:bounds_effective_force_cutoff_space}
Fix $\varepsilon\in(0,1]$, $(\sUV,\uv)\in[0,1]^2\setminus\{(0,0)\}$ and $\mathtt{u}\in C^\infty_\rc(\bR)$. Assume that for some $R>1$ and all $r\in\{0,1\}$, $(i,m,a)\in\frI$ it holds
\begin{equation}
\begin{gathered}
 \|\mathtt{u}\partial_\uv^r \varXi_{\uv;\sUV}^{\phantom{r}}\|_\cV\leq R\,[\uv]^{(\varepsilon-\sigma)r}\,[\uv\vee\sUV]^{-\dim(\varXi)-2\varepsilon},
 \\
 |\partial_\uv^r f^{i,m,a}_{\uv;\sUV}|\leq R^{i,m}\,[\uv]^{(\varepsilon-\sigma)r}\,[\uv\vee\sUV]^{\varrho_{2\varepsilon}(i,m,a)}.
\end{gathered} 
\end{equation}
Then it holds
\begin{multline}
 \|(\delta_\bM\otimes J^{\otimes m}_{\uv\vee\sUV})\ast \mathtt{w}\partial_\uv^r\partial_{\sIRa}^{\rsa}\partial_{\sIRb}^{\rsb}  F^{i,m,a}_{\uv;\sUV,\sIRa,\sIRb}\|_{\cV^m}
 \\
 \lesssim R^{i,m}\,[\uv]^{(\varepsilon-\sigma)r}
 \,[\sIRa]^{(\varepsilon-\sigma)\rsa}
 \,[\sIRb]^{(\varepsilon-\sigma)\rsb}
 \,[\uv\vee\sUV]^{\varrho_{2\varepsilon}(i,m,a)-\varepsilon \rsa-\varepsilon \rsb},
\end{multline}
for all $r,\rsa,\rsb\in\{0,1\}$, $(i,m,a)\in\frI_0$, $\sIRa,\sIRb\in[0,\uv\vee\sUV]$ and all $\mathtt{w}\in C^\infty_\rc(\bR)$ satisfying the condition $\mathtt{u}=1$ on some neighborhood of $\supp\,\mathtt{w}-[0,i(i+1)]$. The constants of proportionality in the above bounds depend only on $(i,m,a)\in\frI_0$ and \mbox{$\mathtt{w},\mathtt{u}\in C^\infty_\rc(\bR)$} and are otherwise universal.
\end{lem}
\begin{proof}
The proof uses the strategy of the proof of Lemma~\ref{lem:eff_force_preliminary_bound} and is also by induction on $i,m\in\bN_0$. In the inductive step we first prove the statement for: (1) $\rsa=1$, $\rsb=0$, $\sIRa\in(0,\uv\vee\sUV]$, $\sIRb=0$ and (2) $\rsa\in\{0,1\}$, $\rsb=1$, $\sIRa,\sIRb\in(0,\uv\vee\sUV]$ using the induction hypothesis, the flow equations
 \begin{multline}
 \partial_{\sIRa}^{\phantom{i}} \partial_\uv^r F^{i,m,a}_{\uv;\sUV,\sIRa,0}
 =
 -\frac{1}{m!}
 \sum_{\pi\in\cP_m}\sum_{j=1}^i\sum_{k=0}^m\sum_{u+v+w=r}
 \sum_{b+c+d=\pi(a)}\frac{(k+1)a!}{b!c!d!}\,
 \\\times
 \fY_\pi\fB\big(\partial_{\uv}^u \dot G^c_{\uv\vee\sUV;\sIRa,0},\partial_\uv^{v} F^{j,k+1,b}_{\uv;\sUV,\sIRa,0},\partial_\uv^{w}F^{i-j,m-k,d}_{\uv;\sUV,\sIRa,0}\big),
\end{multline}
\begin{multline}
 \partial_{\sIRb}^{\phantom{i}} \partial_{\sIRa}^{\rsa} \partial_\uv^r F^{i,m,a}_{\uv;\sUV,\sIRa,\sIRb}
 =
 -\frac{1}{m!}
 \sum_{\pi\in\cP_m}\sum_{j=1}^i\sum_{k=0}^m\sum_{\tilde v+\tilde w=\rsa}\sum_{u+v+w=r}
 \sum_{b+c+d=\pi(a)}\frac{(k+1)a!}{b!c!d!}\,
 \\\times
 \fY_\pi\fB\big(\partial_{\uv}^u \dot G^c_{\uv\vee\sUV;0,\sIRb},\partial_{\sIRa}^{\tilde v} \partial_\uv^{v} F^{j,k+1,b}_{\uv;\sUV,\sIRa,0},
 \partial_{\sIRa}^{\tilde w}\partial_\uv^{w}F^{i-j,m-k,d}_{\uv;\sUV,\sIRa,0}\big),
\end{multline}
the bounds for $\partial_{\uv}^r \dot G^c_{\uv\vee\sUV;\sIRa,0}$ and $\partial_{\uv}^r \dot G^c_{\uv\vee\sUV;0,\sIRb}$ established in Lemma~\ref{lem:kernel_dot_G_2}~(A),~(B) as well as Remark~\ref{rem:fB1_bound} applied with $\oo=0$. The statement for $\rsa=\rsb=0$, $\sIRa=\sIRb=0$ follows from the assumed bounds for the force coefficients $f^{i,m,a}_{\uv;\sUV}$, an analog of Eq.~\eqref{eq:lem_preliminary_bound_initial} as well as the Lemma~\ref{lem:bound_map_L_m}. To prove the statement for $\rsa=\rsb=0$, $\sIRa,\sIRb\in[0,\uv\vee\sUV]$ we use the results established so far, the equations
\begin{equation}
 F^{i,m,a}_{\uv;\sUV,\sIRa,0} = F^{i,m,a}_{\uv;\sUV} + \int_0^{\sIRa} F^{i,m,a}_{\uv;\sUV,\uIR,0}\,\rd\uIR,
 \qquad
 F^{i,m,a}_{\uv;\sUV,\sIRa,\sIRb} = F^{i,m,a}_{\uv;\sUV,\sIRa,0} + \int_0^{\sIRb} \partial_{\uIR} F^{i,m,a}_{\uv;\sUV,\sIRa,\uIR}\,\rd\uIR,
\end{equation}
and the estimate
\begin{equation}
 \int_0^{\sIR} [\uIR]^{\varepsilon-\sigma}\,[\uv\vee\sUV]^{\rho-\varepsilon}\,\rd \uIR \leq \sigma/\varepsilon~ [\uv\vee\sUV]^{\rho}
\end{equation}
valid for $\sIR\in[0,\uv\vee\sUV]$ and any $\rho\in\bR$. 
\end{proof}

\begin{lem}\label{lem:support_hat}
For every $(\uv,\sUV)\in[0,1]^2\setminus\{(0,0)\}$, $\sIRa\in[0,\uv\vee\sUV]$ and $(i,m,a)\in\frI$ it holds almost surely
\begin{equation}
 \supp\,F^{i,m,a}_{\uv;\sUV,\sIRa,0}\subset
 \big\{ (x,x_1,\ldots,x_m)\in\bM^{1+m}\,\big|\,\forall_{p\in\{1,\ldots,m\}} |x_p-x|_\bM\leq 2(i-1)\,[\sUV\vee\uv]\big\}.
\end{equation} 
\end{lem}
\begin{rem}
The proof of the above lemma is similar to the proof of Lemma~\ref{lem:support} and is based on the fact that for $\sIRa\in[0,\uv\vee\sUV]$ the kernel $\dot G^c_{\uv\vee\sUV;\sIRa,0}$ has the following support property
\begin{equation}
 \supp\,\dot G^c_{\uv\vee\sUV;\sIRa,0}\subset 
 \{x\in\bM\,|\, |x|_\bM \leq 2\,[\uv\vee\sUV]\}.
\end{equation} 
\end{rem}

\begin{dfn}\label{dfn:cumulants_eff_force2}
A list $(i,m,a,\rsa,\rsb,r)$, where $i,m\in\bN_0$, $a\in\frM^m_\sigma$, $\rsa,\rsb,r\in\{0,1\}$ is called an extended index. For $n\in\bN_+$ we call 
\begin{equation}
 \vI\equiv ((i_1,m_1,a_1,\rsa_1,\rsb_1,r_1),\ldots,(i_n,m_n,a_n,\rsa_n,\rsb_n,r_n))
\end{equation}
a list of extended indices. The quantities $\rn(\vI)$, $\rri(\vI)$, $\mathsf{m}(\vI)$, $\rrm(\vI)$, $\ra(\vI)$, $\rs(\vI)$, $\rr(\vI)$ are as defined in Def.~\ref{dfn:cumulants_eff_force} and $\rsa(\vI)=\rsa_1+\ldots+\rsa_n$. We introduce the following notation for the joint cumulants of the effective force coefficients
\begin{equation}
 E^\vI_{\uv;\sUV,\sIRa,\sIRb}:=
 \llangle \partial_{\sIRa}^{\rsa_1}\partial_{\sIRb}^{\rsb_1}\partial_\uv^{r_1} F^{i_1,m_1,a_1}_{\uv;\sUV,\sIRa,\sIRb} ,\ldots,
 \partial_{\sIRa}^{\rsa_n}\partial_{\sIRb}^{\rsb_n}\partial_\uv^{r_n} F^{i_n,m_n,a_n}_{\uv;\sUV,\sIRa,\sIRb}\rrangle\in\cD^{\mathsf{m}(\vI);0}_\rt,
\end{equation}
where we used the notation introduced in Def.~\ref{dfn:notation_cumulants_distributions}.
\end{dfn}

The following lemma is an analog of Lemma~\ref{lem:flow_E_form_bound} for cumulants introduced in the above definition. Note that in the lemma the parameter $\sIRa$ plays a similar role to the parameter $\uv$. The proof of the lemma relies on the flow equation in the parameter $\sIRb$.

\begin{lem}\label{rem:extended_cumulants}
Let $\varepsilon\in(0,1]$, $n\in\bN_+$ and
\begin{equation}
 \vJ\equiv (\vJ_1,\ldots,\vJ_n)=((i_1,m_1,a_1,\rsa_1,,\rsb_1,r_1),\ldots,(i_n,m_n,a_n,\rsa_n,\rsb_n,r_n))
\end{equation}
be a list of extended indices such that $\rsb_l=1$ for some $l\in\{1,\ldots,n\}$. 
\begin{enumerate}
\item[(A)] The distribution $E^\vJ_{\uv;\sUV,\sIRa,\sIRb}$ can be expressed as a linear combination of distributions of the form
\begin{equation}
 \fY^\omega\fY_\pi\fA\big(\dot G^c_{\uv\vee\sUV;0,\sIR},E^{\vK}_{\uv;\sUV,\sIRa,\sIRb}\big)
 \quad
 \textrm{or}
 \quad
 \fY^\omega\fY_\pi\fB\big(\dot G^c_{\uv\vee\sUV;0,\sIR},
 E^{\vL}_{\uv;\sUV,\sIRa,\sIRb},
 E^{\vM}_{\uv;\sUV,\sIRa,\sIRb}
 \big),
\end{equation}
where the objects appearing above satisfy all the conditions specified in Part~(A) of Lemma~\ref{lem:flow_E_form_bound} and $\rsa(\vK)=\rsa(\vJ)$, $\rsa(\vL)+\rsa(\vM)=\rsa(\vJ)$. 

\item[(B)]
Suppose that the bound~\eqref{eq:lem_cumulants_peliminary_bound}, stated in the lemma below, holds uniformly in $(\uv,\sUV)\in[0,1]^2\setminus\{(0,0)\}$, $\sIRa,\sIRb\in(0,\uv\vee\sUV]$ for all lists of indices \mbox{$\vI\in\{\vK,\vL,\vM\}$}, where $\vK,\vL,\vM$ are arbitrary lists of indices satisfying the conditions mentioned in Part (A) given the list of indices~$\vJ$. Then the above-mentioned bound holds uniformly in $(\uv,\sUV)\in[0,1]^2\setminus\{(0,0)\}$, $\sIRa,\sIRb\in(0,\uv\vee\sUV]$ for the list of indices $\vI=\vJ$. 
\end{enumerate}
\end{lem}
\begin{proof}
The claim (A) is proved using the fact that the cumulants $E^\vI_{\uv;\sUV,\sIRa,\sIRb}$ satisfy a flow equation analogous to the flow equation~\eqref{eq:flow_E_general}. The claim (B) is a consequence of Lemma~\ref{lem:fA_fB_bounds}, Lemma~\ref{lem:fA_fB_Ks} applied with $\oo=0$ and the bounds for $\dot G^c_{\uv\vee\sUV;0,\sIR}$ proved in Lemma~\ref{lem:kernel_dot_G_2}~(B),~(C). 
\end{proof}

\begin{lem}\label{lem:cumulants_preliminary_bound}
Let $\varepsilon\in(0,\varepsilon_\sigma)$. Suppose that the family of noises $\varXi_\sUV$, $\sUV\in[0,1]$, satisfies Assumption~\ref{ass:noise} and the force coefficients satisfy the bounds
\begin{equation}
 |\partial_\uv^r f^{i,m,a}_{\uv;\sUV}|\lesssim [\uv]^{(\varepsilon-\sigma)r}\,[\uv\vee\sUV]^{\varrho_{2\varepsilon}(i,m,a)},\qquad
 (i,m,a)\in\frI,
\end{equation}
uniformly in $(\uv,\sUV)\in[0,1]^2\setminus\{(0,0)\}$.
\begin{enumerate}
\item[(A)] Let $\cO\subset\bH$ be an open bounded region. The random variable $F^{i,m,a}_{\uv;\sUV,\sIRa,0}\in\cD^{m;0}$ restricted to the region $\cO\times\bM^m\subset\bM\times\bM^m$, is measurable with respect to the $\sigma$-algebra generated by $\{\varXi(z)\,|\,z\in\cO^i_{\uv,\sUV}\}$, where
\begin{equation}
 \cO^i_{\uv,\sUV}:=\{z\in\bH\,|\, \exists_{x\in\overline{\cO}} ~|z-x|_\bH < [\uv]+2i\,[\uv\vee\sUV]\,\},\qquad i\in\bN_0,\uv,\sUV\in[0,1].
\end{equation}

\item[(B)] For all $r,\rsa,\rsb\in\{0,1\}$, $(i,m,a)\in\frI_0$ and $\mathtt{w}\in C^\infty_\rc(\bR)$ the following bound
\begin{multline}
 \llangle\|(\delta_\bM\otimes J^{\otimes m}_{\uv\vee\sUV})\ast \mathtt{w}\partial_\uv^r\partial_{\sIRa}^{\rsa}\partial_{\sIRb}^{\rsb}  F^{i,m,a}_{\uv;\sUV,\sIRa,\sIRb}\|_{\cV^m}^n \rrangle
 \\
 \lesssim \,[\uv]^{n(\varepsilon-\sigma)r}
 \,[\sIRa]^{n(\varepsilon-\sigma)\rsa}
 \,[\sIRb]^{n(\varepsilon-\sigma)\rsb}
 \,[\uv\vee\sUV]^{n\varrho_{2\varepsilon}(i,m,a)-n\varepsilon \rsa-n\varepsilon \rsb},
\end{multline}
holds uniformly in $(\sUV,\uv)\in[0,1]^2\setminus\{(0,0)\}$, $\sIRa,\sIRb\in[0,\uv\vee\sUV]$.

\item[(C)] For every list of extended indices $\vI$ the following bound
\begin{multline}\label{eq:lem_cumulants_peliminary_bound}
 \|(\delta_\bM^{\otimes \rn(\vI)}\otimes J_{\uv\vee\sUV}^{\otimes \rrm(\vI)})\ast E^\vI_{\uv;\sUV,\sIRa,\sIRb}\|_{\cV^{\mathsf{m}}}
 \\
 \lesssim 
 [\uv]^{(\varepsilon-\sigma)\rr(\vI)}
 \,[\sIRa]^{(\varepsilon-\sigma)\rsa(\vI)} [\sIRb]^{(\varepsilon-\sigma)\rsb(\vI)}
 \,[\uv\vee\sUV]^{\varrho_\varepsilon(\vI)-\varepsilon\rsa(\vI)-\varepsilon\rsb(\vI)+(\rn(\vI)-1)\rDim}
\end{multline}
holds uniformly in $(\uv,\sUV)\in[0,1]^2\setminus\{(0,0)\}$ and $\sIRa,\sIRb\in[0,\uv\vee\sUV]$.
\end{enumerate}

\end{lem}
\begin{rem}\label{rem:cumulants_moments_preliminary_bound}
The following slightly stronger statement is in fact true. Fix $i_\circ,m_\circ\in\bN_0$. If the assumption of the above lemma holds for all $i,m\in\bN_0$ such that $i<i_\circ$ or $i=i_\circ$ and $m>m_\circ$, then the statement holds for all $i,m\in\bN_0$ and all lists of indices $\vI$ such that $i=\ri(\vI)$, $m=\rrm(\vI)$ and either (1)~$i<i_\circ$ or (2)~$i=i_\circ$ and $m>m_\circ$ or (3)~$i=i_\circ$, $m=m_\circ$, $s=1$.
\end{rem}
\begin{rem}\label{rem:moments_preliminary_bound}
Part (B) of Lemma~\ref{lem:cumulants_preliminary_bound} applied with $\sIRa=\sIRb$ and Eq.~\eqref{eq:effective_force_eq_hat_check} together imply that for all $r,s\in\{0,1\}$, $(i,m,a)\in\frI_0$ and $\mathtt{w}\in C^\infty_\rc(\bR)$ the following bound
\begin{equation}
 \llangle\|(\delta_\bM\otimes J^{\otimes m}_{\uv\vee\sUV})\ast \mathtt{w}\partial_\uv^r \partial_{\sIR}^{s}  F^{i,m,a}_{\uv;\sUV,\sIR}\|_{\cV^m}^n \rrangle
 \\
 \lesssim \,[\uv]^{n(\varepsilon-\sigma)r}
 \,[\sIR]^{n(\varepsilon-\sigma) s}
 \,[\uv\vee\sUV]^{n\varrho_{2\varepsilon}(i,m,a)-n\varepsilon s}
\end{equation}
holds uniformly in $(\sUV,\uv)\in[0,1]^2\setminus\{(0,0)\}$, $\sIR\in[0,\uv\vee\sUV]$.
\end{rem}

\begin{rem}\label{rem:cumulants_preliminary_bound}
Part (C) of Lemma~\ref{lem:cumulants_preliminary_bound} and Eq.~\eqref{eq:effective_force_eq_hat_check} together imply that for all lists of indices $\vI$ the following bound
\begin{multline}
 \|(\delta_\bM^{\otimes \rn(\vI)}\otimes J_{\uv\vee\sUV}^{\otimes \rrm(\vI)})\ast E^\vI_{\uv;\sUV,\sIR}\|_{\cV^{\mathsf{m}}}
 \\
 \lesssim 
 [\uv]^{(\varepsilon-\sigma)\rr(\vI)}
 \,[\sIR]^{(\varepsilon-\sigma)\rs(\vI)}
 \,[\uv\vee\sUV]^{\varrho_\varepsilon(\vI)-\varepsilon\rs(\vI)+(\rn(\vI)-1)\rDim}
\end{multline}
holds uniformly in $(\uv,\sUV)\in[0,1]^2\setminus\{(0,0)\}$ and $\sIRa,\sIRb\in[0,\uv\vee\sUV]$, where $E^\vI_{\uv;\sUV,\sIR}$ was introduced in Def.~\ref{dfn:cumulants_eff_force}.
\end{rem}

\begin{proof}
Part~(A) is proved by induction on $i,m\in\bN_0$ using a strategy similar to the strategy of the proof of Lemma~\ref{lem:support}. Note that $\varXi_{\uv;\sUV}=M_\uv\ast \varXi_\sUV$, where $\supp\,M_\uv\subset\{x\in\bM\,|\,|x|_\bM<[\uv]\}$. This implies Part~(A) for $i=0$. To prove Part~(A) for $i\in\bN_+$ we use the flow equation~\eqref{eq:flow_deterministic_i_m_a_hat} and the support property of the effective force coefficients established in Lemma~\ref{lem:support_hat}.

To prove Part (B) note that $F^{0,0,0}_{\uv;\sUV,\sIRa,\sIRb}=\varXi_{\uv;\sUV}$ is the only non-zero effective force coefficient $F^{i,m,a}_{\uv;\sUV,\sIRa,\sIRb}$ with $i=0$. As a result, for $i=0$ the bound stated in Part (B) follows from the result of Lemma~\ref{lem:noise_moment_bound}, which relies only on Lemma~\ref{lem:ass_noise_general} (C). This together with the assumed bound for the force coefficients $f^{i,m,a}_{\uv;\sUV}$ implies that for any $\mathtt{u}\in C^\infty_\rc(\bR)$ 
the assumption of Lemma~\ref{lem:bounds_effective_force_cutoff_space} is satisfied with some random constant $R=R_{\uv;\sUV}(\mathtt{u})$ such that $\llangle R_{\uv;\sUV}^n\rrangle \lesssim 1$ uniformly in $(\uv,\sUV)\in[0,1]^2\setminus\{(0,0)\}$ for any $n\in\bN_+$. Consequently, Part (B) follows from Lemma~\ref{lem:bounds_effective_force_cutoff_space}.

Let us turn to the proof of Part~(C). First note that Part (B) implies that for any $\mathtt{w}_p\in C^\infty_\rc(\bR)$, $(i_p,m_p,a_p)\in\frI_0$, $\rsa_p,r_p\in\{0,1\}$ it holds
\begin{multline}
 \llangle[\big] \textstyle\prod_{p=1}^n \|(\delta_\bM \otimes J_\sUV^{\otimes m_p}) \ast \mathtt{w}_p \partial_\uv^{r_p}\partial_{\sIRa}^{\rsa_p} F^{i_p,m_p,a_p}_{\uv;\sUV,\sIRa,0}\|_{\cV^{m_p}}
 \rrangle[\big]
 \\
 \lesssim 
 \prod_{p=1}^n \,
 [\uv]^{(\varepsilon-\sigma)r_p}
 [\sIRa]^{(\varepsilon-\sigma)\rsa_p} [\uv\vee\sUV]^{\varrho_{2\varepsilon}(i_p,m_p,a_p)-\varepsilon \rsa_p}
\end{multline}
uniformly in $(\uv,\sUV)\in[0,1]^2\setminus\{(0,0)\}$ and $\sIRa\in(0,\uv\vee\sUV]$. By the Fubini theorem, the relation~\eqref{eq:cumulants_expectation} between cumulants and moments and Def.~\ref{dfn:cV} of the norm $\|\Cdot\|_{\cV^m}$ it holds
\begin{equation}
 \int_{\bM^m} |\llangle V_1(x_1;\rd \ry_1),\ldots,V_n(x_n;\rd \ry_n) \rrangle|
 \lesssim
 \llangle \|V_1\|_{\cV^{m_1}} \ldots \|V_2\|_{\cV^{m_n}}\rrangle
\end{equation}
for any $V_1\in\cV^{m_1},\ldots,V_n\in\cV^{m_n}$ and $x_1,\ldots,x_n\in\bM$, where $m=m_1+\ldots+m_n$ and the constant of proportionality depends only on $n$ (note that the expression on the LHS of the above inequality involves the cumulant and the expression on the RHS is the expectation of a product). Using the above bounds we obtain for any list $\vI$ of extended indices such that $\rsb(\vI)=0$ and any $T>0$ the bound
\begin{multline}\label{eq:lem_cumulant_preliminary_bound}
 \int_{\bM^{m(\vI)}} |(\delta_\bM^{\otimes \rn(\vI)}\otimes J_{\uv\vee\sUV}^{\otimes \rrm(\vI)})\ast E^\vI_{\uv;\sUV,\sIRa,0}(x_1,\ldots,x_n;\rd\ry_1,\ldots,\rd\ry_n)|
 \\\lesssim 
 [\uv]^{(\varepsilon-\sigma)\rr(\vI)}
 [\sIRa]^{(\varepsilon-\sigma)\rsa(\vI)} [\uv\vee\sUV]^{\varrho_\varepsilon(\vI)-\varepsilon\rsa(\vI)}
\end{multline}
uniform over $(\sUV,\uv)\in[0,1]^2\setminus\{(0,0)\}$, $\sIRa\in[0,\uv\vee\sUV]$ and $x_1,\ldots,x_n\in [-T,T]\times\bT$.

Next, observe that Part~(A) implies that given two regions $\cO_1,\cO_2\subset\bH$ the random variables $F^{i_1,m_1,a_1}_{\uv;\sUV,\sIRa,0}\in\cD^{m_1;0}$ and $F^{i_2,m_2,a_2}_{\uv;\sUV,\sIRa,0}\in\cD^{m_2;0}$ restricted to $\cO_1\times\bM^{m_1}$ and $\cO_2\times\bM^{m_2}$, respectively, are independent if $|x_1-x_2|_\bH> 2\,[\uv]+2(i_1+i_2)\,[\uv\vee\sUV]$ for all $x_1\in\cO_1$ and $x_2\in\cO_2$. In consequence, the cumulant $E^\vI_{\uv;\sUV,\sIRa,0}$ vanishes outside the region
\begin{equation}
 \{(x_1,\ldots,x_n,y_1,\ldots,y_m)\in\bM^{n+m}\,|\,\forall_{p,q\in\{1,\ldots,n\},p\neq q}\,|x_p-x_q|_\bH \leq 2(\rri(\vI) +1)\,[\uv\vee\sUV]\}.
\end{equation}
This together with the bound~\eqref{eq:lem_cumulant_preliminary_bound} and translational invariance of  $E^\vI_{\uv;\sUV,\sIRa,0}$ implies Part~(C) for $\sIRb=0$ and any list of extended indices $\vI$ such that $\rsb(\vI)=0$. 

We prove Part~(C) for $\sIRb\in[0,\uv\vee\sUV]$ by induction. The statement is clearly true for all list of indices $\vI$ such that $\rrm(\vI)>\rri(\vI)m_\flat$ since then $E^\vI_{\uv;\sUV,\sIRa,\sIRb}=0$. It is also true if $\rri(\vI)=0$ as then $\rrm(\vI)=0$ for non-zero cumulants and $E^\vI_{\uv;\sUV,\sIRa,\sIRb}$ coincides with a cumulant of $\partial_\uv^r\varXi_{\uv;\sUV}$. In particular, it is independent of $\sIRb$ and vanishes if $\rsb(\vI)\neq 0$. Hence, for all lists of extended indices $\vI$ such that $\rri(\vI)=0$ the bound stated in Part~(C) is implied by the special case of this bound for $\sIRb=0$, which was proved in the previous paragraph. Fix $i_\circ\in\bN_+$ and $m_\circ\in\bN_0$. Assume that the statement is true for all lists of indices $\vI$ such that either $\rri(\vI)<i_\circ$, or $\rri(\vI)=i_\circ$ and $\rrm(\vI)>m_\circ$. We shall prove the statement for all $\vI$ such that $\rri(\vI)=i_\circ$ and $\rrm(\vI)=m_\circ$. We first consider the case $\rsb(\vI)\neq 0$. Note that by Lemma~\ref{rem:extended_cumulants}~(A) the cumulants $E^\vI_{\uv;\sUV,\sIRa,\sIRb}$ can be expressed in terms of the cumulants for which Part~(C) has already been established. Hence, the statement with $\rs(\vI)\neq 0$ is a consequence of the inductive assumption and Lemma~\ref{rem:extended_cumulants}~(B). To prove the statement for lists of extended indices $\vI$ such that $\rs(\vI)=0$ we use the equality
\begin{equation}
 E^\vI_{\uv;\sUV,\sIRa,\sIRb} = E^\vI_{\uv;\sUV,\sIRa,0} + \sum_{q=1}^n\int_0^{\sIRb} E^{\vI_q}_{\uv;\sUV,\sIRa,\uIR}\,\rd\uIR,
\end{equation}
where 
\begin{equation}
 \vI_q = ((i_1,m_1,a_1,\rsa_1,0,r_1),\ldots, (i_q,m_q,a_q,\rsa_q,1,r_q),\ldots,(i_n,m_n,a_n,\rsa_n,0,r_n)).
\end{equation}
The claim follows from the bound
\begin{multline}
 \|(\delta_\bM^{\otimes \rn(\vI)}\otimes J_{\uv\vee\sUV}^{\otimes \rrm(\vI)})\ast E^\vI_{\uv;\sUV,\sIRa,\sIRb}\|_{\cV^{\mathsf{m}}} 
 \leq
 \|(\delta_\bM^{\otimes \rn(\vI)}\otimes J_{\uv\vee\sUV}^{\otimes \rrm(\vI)})\ast  E^\vI_{\uv;\sUV,\sIRa,0}\|_{\cV^{\mathsf{m}}}  
 \\
 +
 \sum_{q=1}^n \int_0^{\sIRb} \|(\delta_\bM^{\otimes \rn(\vI)}\otimes J_{\uv\vee\sUV}^{\otimes \rrm(\vI)})\ast E^{\vI_q}_{\uv;\sUV,\sIRa,\uIR}\|_{\cV^{\mathsf{m}}}  \,\rd\uIR
\end{multline}
as well as the statement with $\sIRb=0$ or $\rsb(\vI)\neq0$, which has already been proved above.
\end{proof}

\subsection{Renormalization conditions}

\begin{ass}\label{ass:renormalization}
Let us fix constants $i_\flat,m_\flat\in\bN_+$ such that $\varrho(i_\flat+1)>0$ and $\varrho(1,m_\flat+1)>0$ and a family of continuous functions \mbox{$[0,1]\ni\sUV\mapsto \mathfrak{f}^{i,m,a}_\sUV\in\bR$}, $(i,m,a)\in\frI$ such that $|\mathfrak{f}^{i,m,a}_\sUV|\lesssim [\sUV]^{\varrho(i,m,a)-\varepsilon}$ uniformly in $\sUV\in[0,1]$ for any \mbox{$\varepsilon\in(0,1]$} and $\mathfrak{f}^{i,m,a}_\sUV$ vanishes identically unless $i\in\{0,\ldots,i_\flat\}$, $m\in\{0,\ldots,m_\flat\}$ and \mbox{$a\in\bar\frM^m_\sigma$}. We assume the following.
\begin{enumerate}
\item[(A)] The irrelevant force coefficients $f^{i,m,a}_{\uv;\sUV}$, \mbox{$(i,m,a)\in \frI^+$}, satisfy the conditions
\begin{equation}\label{eq:renormalization_conditions2}
 f^{i,m,a}_{\uv;\sUV}=f^{i,m,a}_\sUV=\mathfrak{f}^{i,m,a}_\sUV,
 \qquad
 (i,m,a)\in \frI^+,
\end{equation}
for all $(\sUV,\uv)\in[0,1]^2\setminus\{(0,0)\}$.

\item[(B)] The relevant force coefficients $f^{i,m,a}_{\uv;\sUV}$, \mbox{$(i,m,a)\in \frI^-$}, are chosen in such a way that the conditions
\begin{equation}\label{eq:renormalization_conditions}
\llangle f^{i,m,a}_{\uv;\sUV,1}\rrangle = \llangle f^{i,m,a}_{\sUV,1}\rrangle = \mathfrak{f}^{i,m,a}_\sUV,
\qquad
(i,m,a)\in \frI^-,
\end{equation}
hold for all $(\sUV,\uv)\in[0,1]^2\setminus\{(0,0)\}$, where the effective force coefficients $f^{i,m,a}_{\uv;\sUV,1}$, \mbox{$(i,m,a)\in \frI^-$}, are defined in terms of the force coefficients $f^{i,m,a}_{\uv;\sUV}$, $(i,m,a)\in\frI$ as specified in Def.~\ref{dfn:eff_force_uv}.
\end{enumerate}
\end{ass}

\begin{rem}
Note that $(i,m,a)\in\frI^-$ implies in particular that $i\in\{0,\ldots,i_\flat\}$, $m\in\{0,\ldots,m_\flat\}$ and $a\in\bar\frM_\sigma$. Moreover, observe that if $(i,m,a)\in\frI^-$, then $\varrho(i,m,a)<0$ and the bound $|\mathfrak{f}^{i,m,a}_\sUV|\lesssim [\sUV]^{\varrho(i,m,a)-\varepsilon}$ uniform in $\sUV\in[0,1]$ holds trivially for any $\varepsilon\in(0,1]$ by the assumed continuity of the function $\mathfrak{f}^{i,m,a}_\Cdot$. Furthermore, if $(i,m,a)\in\frI^+$, then $\mathfrak{f}^{i,m,a}_0=0$.
\end{rem}

\begin{rem}
The conditions~\eqref{eq:renormalization_conditions} are called the renormalization conditions and the parameters $\mathfrak{f}^{i,m,a}_\sUV$, $\sUV\in[0,1]$, $(i,m,a)\in\frI^-$, are called the renormalization parameters.
\end{rem}

\begin{rem}
We shall prove that given $\mathfrak{f}^{i,m,a}_\sUV$, $(i,m,a)\in\frI$, as in the above assumption there exist a unique choice of the force coefficients $f^{i,m,a}_{\uv;\sUV}$, \mbox{$(i,m,a)\in\frI$}, such that the conditions (A) and (B) are satisfied. We show that the force coefficients $f^{i,m,a}_{\uv;\sUV}$, \mbox{$(i,m,a)\in\frI$}, depend continuously on $(\uv,\sUV)\in[0,1]^2\setminus\{(0,0)\}$, are differentiable in the parameter $\uv$ for $\uv\in(0,1]$, vanish identically unless $i\in\{0,\ldots,i_\flat\}$, $m\in\{0,\ldots,m_\flat\}$ and $a\in\bar\frM_\sigma$ and satisfy the bound assumed in Lemma~\ref{lem:cumulants_preliminary_bound}. Note that the irrelevant force coefficients $f^{i,m,a}_{\uv;\sUV}$, $(i,m,a)\in\frI^+$, do not depend on the parameter $\uv\in[0,1]$. In general this is not the case for the relevant coefficients $f^{i,m,a}_{\uv;\sUV}$, $(i,m,a)\in\frI^-$.
\end{rem}

\subsection{Uniform bounds}\label{sec:cumulants_uniform_bounds}

We now proceed to the prove of the uniform bounds for the cumulants of the effective force coefficients. The bounds are stated in Theorem~\ref{thm:cumulants}. In the proof of this theorem we use the following simple lemma.

\begin{lem}\label{lem:integration}
Let $n\in\bN_+$, $\mathsf{m}=(m_1,\ldots,m_n)\in\bN_0^n$, $m=m_1+\ldots+m$, $\epsilon>0$, $m\in\bN_0$, $\rho\in\bR$ and $(0,1]\ni\sIR\mapsto V_{\sUV,\sIR} \in \cD^{\mathsf{m}}$ be a family of integrable functions parameterized by $\sUV\in[0,1]$. Assume that for some $C>0$, $\oo\in\bN_0$, $\sUV\in[0,1]$ and all $\sIR\in(0,1]$ it holds
\begin{equation}
 \|K^{n,m;\oo}_{\sUV,\sIR}\ast V_{\sUV,\sIR}\|_{\cV^{\mathsf{m}}}
 \leq C\,[\sIR]^{\epsilon-\sigma}\,[\sUV\vee\sIR]^{\rho-\epsilon}.
\end{equation}
Set $U_{\sUV,\sIR}:=\int_0^\sIR V_{\sUV,\uIR}\,\rd \uIR$. 
\begin{enumerate}
\item[(A)] For all $\sIR\in[0,\sUV]$ it holds
\begin{equation}
 \|K^{n,m;\oo}_{\sUV,\sIR}\ast U_{\sUV,\sIR}\|_{\cV^{\mathsf{m}}}
 \leq C\,\sigma/\epsilon\, [\sIR]^\epsilon\,[\sUV]^{\rho-\epsilon}
 \leq C\,\sigma/\epsilon\, [\sUV]^{\rho}.
\end{equation} 
\item[(B)] If $\rho>0$, then for all $\sIR\in[0,1]$ it holds
\begin{equation}
 \|K^{n,m;\oo}_{\sUV,\sIR}\ast U_{\sUV,\sIR}\|_{\cV^{\mathsf{m}}}
 \leq C\,(\sigma/\epsilon+\sigma/\rho)\, [\sIR]^\epsilon\,[\sUV\vee\sIR]^{\rho-\epsilon}
 \leq C\,(\sigma/\epsilon+\sigma/\rho)\,[\sUV\vee\sIR]^{\rho}.
\end{equation}
\end{enumerate}
\end{lem}
\begin{proof}
By Lemma~\ref{lem:kernel_u_v} and Remark~\ref{rem:Vm_K} we have
\begin{equation}
 \|K^{n,m;\oo}_{\sUV,\sIR}\ast U_{\sUV,\sIR}\|_{\cV^{\mathsf{m}}} 
 \leq \int_0^\sIR 
 \|K^{n,m;\oo}_{\sUV,\uIR}\ast V_{\sUV,\uIR}\|_{\cV^{\mathsf{m}}}\,\rd\uIR.
\end{equation}
For $\sIR\leq\sUV$ we obtain
\begin{equation}
 \|K^{n,m;\oo}_{\sUV,\sIR}\ast U_{\sUV,\sIR}\|_{\cV^{\mathsf{m}}} \leq C\sigma/\epsilon\,[\sIR]^\epsilon\,[\sUV]^{\rho-\epsilon}.
\end{equation}
For $\sIR>\sUV$ we get
\begin{equation}
 \|K^{n,m;\oo}_{\sUV,\sIR}\ast U_{\sUV,\sIR}\|_{\cV^{\mathsf{m}}} \leq C\sigma/\epsilon\,[\sUV]^{\rho}
 + C\sigma/\rho\,[\sIR]^{\rho}\leq C\,(\sigma/\epsilon+\sigma/\rho)\,[\sIR]^{\rho}.
\end{equation}
This completes the proof.
\end{proof}

\begin{thm}\label{thm:cumulants}
Let $\varepsilon\in(0,\varepsilon_\sigma\wedge\varepsilon_\diamond/2)$. Suppose that the family of noises $\varXi_\sUV$, $\sUV\in[0,1]$, satisfies Assumption~\ref{ass:noise}. There exists a unique choice of the force coefficients $f^{i,m,a}_{\uv;\sUV}$ such that Assumption~\ref{ass:renormalization} holds true. With the above choice of the force coefficients for every list of indices $\vI$ there exists $\oo\in\bN_+$ depending on $\vI$ such that the following bound
\begin{multline}\label{eq:thm_cumulants}
 \|K^{\rn(\vI),\rrm(\vI);\oo}_{\uv\vee\sUV,\sIR}\ast E^\vI_{\uv;\sUV,\sIR}\|_{\cV^{\mathsf{m}(\vI)}}
 \\
 \lesssim [\uv]^{(\varepsilon-\sigma)\rr(\vI)}\,
 [\sIR]^{(\varepsilon-\sigma)\rs(\vI)}\,
 [\uv\vee\sUV\vee\sIR]^{\varrho_{2\varepsilon}(\vI)-\varepsilon\rs(\vI)+(\rn(\vI)-1)\rDim}
\end{multline}
holds uniformly in $(\uv,\sUV)\in[0,1]^2\setminus\{(0,0)\}$ and $\sIR\in[0,1]$.
\end{thm}

\begin{rem}\label{rem:force_coeff_cumulants}
Together, the equality
\begin{equation}
 \partial^r_\uv f^{i,m,a}_{\uv;\sUV} = \fI(\partial^r_\uv F^{i,m,a}_{\uv;\sUV}) = \fI E^\vI_{\uv;\sUV,0},
 \quad \vI=(i,m,a,0,r),~~(i,m,a)\in\frI,~~r\in\{0,1\},
\end{equation}
the property of the map $\fI$ proved in Lemma~\ref{lem:map_I}~(B) and the bound stated in the above theorem imply that the force coefficient $f^{i,m,a}_{\uv;\sUV}$ satisfies the bound
\begin{equation}
 |\partial^r_\uv f^{i,m,a}_{\uv;\sUV}|\leq\|\fI E^\vI_{\uv;\sUV,0}\|_{\cV^m}\lesssim  [\uv]^{(\varepsilon-\sigma)r}\,[\uv\vee\sUV]^{\varrho_{2\varepsilon}(\vI)} 
 =[\uv]^{(\varepsilon-\sigma)r}\,[\uv\vee\sUV]^{\varrho_{2\varepsilon}(i,m,a)}
\end{equation}
uniformly in $(\uv,\sUV)\in[0,1]^2\setminus\{(0,0)\}$. 
\end{rem}

\begin{proof}
The proof is by ascending induction on $\rri(\vI)\in\bN_0$ and descending induction on $\rrm(\vI)\in\bN_0$. Note that the theorem is trivially true for all list of indices $\vI$ such that $\rrm(\vI)>\rri(\vI)m_\flat$ since then $E^\vI_{\uv;\sUV,\sIR}=0$.

\textit{The base case:} 
We first prove the theorem for list of indices $\vI$ such that \mbox{$\rri(\vI)=0$}. In this case $\rrm(\vI)=0$ for non-zero cumulants and the cumulants $E^\vI_{\uv;\sUV,\sIR}$ coincide with the cumulants of the noise $\varXi_{\uv;\sUV}$. Evidently, $E^\vI_{\uv;\sUV,\sIR}$ is independent of~$\sIR$. Hence, it is enough to consider the case $\rs(\vI)=0$. As a result, the statement of the theorem for list of indices $\vI$ such that \mbox{$\rri(\vI)=0$} follows from Lemma~\ref{lem:ass_noise_general}~(D) since for list of indices $\vI$ of the above-mentioned form it holds
\begin{equation}
 \varrho_{2\varepsilon}(\vI) + (\rn(\vI)-1)\rDim =
 -\rn(\vI)(\dim(\varXi)+2\varepsilon)+ (\rn(\vI)-1)\rDim
 < \rn(\vI)\dim(\varXi)-\rDim-\varepsilon \rr(\vI)
\end{equation}
since $\rr(\vI)\leq \rn(\vI)$.

\textit{Induction step:}
Fix $i_\circ\in\bN_+$ and $m_\circ\in\bN_0$. Assume that the theorem is true for all lists of indices $\vI$ such that either $\rri(\vI)<i_\circ$, or $\rri(\vI)=i_\circ$ and $\rrm(\vI)>m_\circ$. We shall prove the theorem for all $\vI$ such that $\rri(\vI)=i_\circ$ and $\rrm(\vI)=m_\circ$. 

We first consider the case $\rs(\vI)\neq 0$. By Remark~\ref{rem:force_coeff_cumulants} the bound assumed in Lemma~\ref{lem:cumulants_preliminary_bound} is satisfied for all $(i,m,a)\in\frI$ such that $i<i_\circ$, or $i=i_\circ$ and $m>m_\circ$. As a result, for $\sIR\in(0,\uv\vee\sUV]$ the bound~\eqref{eq:thm_cumulants} with $\rs(\vI)\neq 0$ follows from the inductive assumption, Lemma~\ref{lem:cumulants_preliminary_bound}~(C), Remark~\ref{rem:cumulants_moments_preliminary_bound} and Remark~\ref{rem:cumulants_preliminary_bound}. Note that by Lemma~\ref{lem:flow_E_form_bound}~(A) the cumulants $E^\vI_{\uv;\sUV,\sIR}$ can be expressed in terms of the cumulants for which the statement of the theorem has already been established. In consequence, for \mbox{$\sIR\in[\uv\vee\sUV,1]$} the bound~\eqref{eq:thm_cumulants} with $\rs(\vI)\neq 0$ follows from the inductive assumption and Lemma~\ref{lem:flow_E_form_bound}~(B). 

Now consider the case $\rs(\vI)=0$, that is
\begin{equation}
 \vI = ((i_1,m_1,a_1,0,r_1),\ldots,(i_n,m_n,a_n,0,r_n)).
\end{equation}
It follows from Def.~\ref{dfn:cumulants_eff_force} of the cumulants $E^\vI_{\uv;\sUV,\sIR}$ that 
\begin{equation}\label{eq:thm_cumulants_ind_step}
 E^\vI_{\uv;\sUV,\sIR}
 = 
 E^\vI_{\uv;\sUV,0} + \sum_{q=1}^n \int_0^\sIR E^{\vI_q}_{\uv;\sUV,\uIR}\,\rd\uIR,
 \quad
 E^\vI_{\uv;\sUV,\sIR}
 = 
 E^\vI_{\uv;\sUV,1} - \sum_{q=1}^n \int_\sIR^1 E^{\vI_q}_{\uv;\sUV,\uIR}\,\rd\uIR,
\end{equation}
where 
\begin{equation}
 \vI_q = ((i_1,m_1,a_1,0,r_1),\ldots, (i_q,m_q,a_q,1,r_q),\ldots,(i_n,m_n,a_n,0,r_n)).
\end{equation}
Note that $\rs(\vI_q)=1$, hence the bound~\eqref{eq:thm_cumulants} has already been established for $E^{\vI_q}_{\uv;\sUV,\uIR}$. We will use the first of Eqs.~\eqref{eq:thm_cumulants_ind_step} to bound the irrelevant cumulants $E^\vI_{\uv;\sUV,\sIR}$, i.e. those with $\vI$ such that $\varrho(\vI)+(\rn(\vI)-1)\rDim>0$. The second equality will be used to bound certain relevant contributions to the cumulants $E^\vI_{\uv;\sUV,\sIR}$ with $\vI$ such that $\varrho(\vI)+(\rn(\vI)-1)\rDim\leq0$. 

We first analyse the relevant contributions. We note that the inequality $\varrho(\vI)+(\rn(\vI)-1)\rDim\leq0$ implies \mbox{$\rn(\vI)=1$}. Consequently, $\vI=(i,m,a,0,r)$ for some $r\in\{0,1\}$ and $a\in\frM_\sigma$ such that $\varrho(i,m,a)\leq0$, or equivalently, $(i,m,a)\in\frI^-$. Hence,
\begin{equation}\label{eq:thm_cumulant_relevant_form}
 E^\vI_{\uv;\sUV,\sIR}=\llangle \partial^r_\uv F^{i,m,a}_{\uv;\sUV,\sIR} \rrangle,
 \qquad 
 (i,m,a)\in\frI^-,~r\in\{0,1\}.
\end{equation}
The relevant contributions mentioned in the previous paragraph have the following form
\begin{equation}
 \llangle \partial_\uv^r f^{i,m,a}_{\uv;\sUV,\sIR}\rrangle =\fI(\llangle \partial_\uv^r F^{i,m,a}_{\uv;\sUV,\sIR}\rrangle)=\fI(K^{m;\oo}_{\uv\vee\sUV,1}\ast E^\vI_{\uv;\sUV,\sIR})\in\bR,
\end{equation}
where the map $\fI$ was introduced in Def.~\ref{dfn:map_I}. Note that by the translational invariance $\llangle \partial_\uv^r f^{i,m,a}_{\uv;\sUV,\sIR}\rrangle$ is a constant function. The application of the map $\fI$ to both sides of the second of Eqs.~\eqref{eq:thm_cumulants_ind_step} yields
\begin{equation}\label{eq:thm_cumulants_f_ren}
 \llangle \partial_\uv^r f^{i,m,a}_{\uv;\sUV,\sIR}\rrangle
 =
 \llangle \partial_\uv^r f^{i,m,a}_{\uv;\sUV,1}\rrangle
 -
 \int_\sIR^1 \fI(E^{\vI_1}_{\uv;\sUV,\uIR} )\,\rd\uIR,
\end{equation}
where $E^{\vI_1}_{\uv;\sUV,\uIR} = \llangle \partial_\uIR\partial_\uv^r F^{i,m,a}_{\uv;\sUV,\uIR}\rrangle$.  Recalling that $\cV^m\equiv\cV^{\mathsf{m}}$ and using Lemma~\ref{lem:map_I} we arrive at
\begin{equation}\label{eq:thm_cumulants_f_ren_bound}
 |\llangle \partial_\uv^r f^{i,m,a}_{\uv;\sUV,\sIR}\rrangle|
 \leq
 |\llangle \partial_\uv^r f^{i,m,a}_{\uv;\sUV,1}\rrangle|
 +
 \int_\sIR^1 \|K^{m;\oo}_{\uv\vee\sUV,\uIR}\ast E^{\vI_1}_{\uv;\sUV,\uIR}\|_{\cV^m}\,\rd\uIR.
\end{equation}
By the renormalization conditions~\eqref{eq:renormalization_conditions} we have 
\begin{equation}
 |\llangle f^{i,m,a}_{\uv;\sUV,1}\rrangle| = 
 |\mathfrak f^{i,m,a}_\sUV| 
 \lesssim 1,
 \qquad
 \llangle \partial_\uv^{\phantom{j}} f^{i,m,a}_{\uv;\sUV,1}\rrangle=0
\end{equation}
for all $(i,m,a)\in\frI^-$. Hence, the bound~\eqref{eq:thm_cumulants_f_ren_bound} and the bound~\eqref{eq:thm_cumulants} for $E^{\vI_1}_{\uv;\sUV,\sIR}$ imply that for $(i,m,a)\in\frI^-$ it holds
\begin{multline}\label{eq:thm_cumulants_relevant}
 |\llangle \partial_\uv^r f^{i,m,a}_{\uv;\sUV,\sIR}\rrangle|
 \lesssim 
 1+\int_\sIR^1 [\uv]^{(\varepsilon-\sigma)r}
 [\uIR]^{\varepsilon-\sigma}
 [\uv\vee\sUV\vee\uIR]^{\varrho_{2\varepsilon}(i,m,a)-\varepsilon}\,\rd\uIR
 \\\lesssim 
 [\uv]^{(\varepsilon-\sigma)r}
 [\uv\vee\sUV\vee\sIR]^{\varrho_{2\varepsilon}(i,m,a)}.
\end{multline}
We will use this bound twice in the remainder of the proof.

First, note that by the above bound 
\begin{equation}
 |\llangle \partial_\uv^r f^{i,m,a}_{\uv;\sUV}\rrangle|
 \lesssim 
 [\uv]^{(\varepsilon-\sigma)r}
 [\uv\vee\sUV]^{\varrho_{2\varepsilon}(i,m,a)}
\end{equation}
holds uniformly in $(\uv,\sUV)\in[0,1]^2\setminus\{(0,0)\}$. As a result, the bound assumed in Lemma~\ref{lem:cumulants_preliminary_bound} is satisfied for $i=i_\circ$, $m=m_\circ$ and all $a\in\frM^m_\sigma$. Consequently, for $\sIR\in(0,\uv\vee\sUV]$ the bound~\eqref{eq:thm_cumulants} follows from Lemma~\ref{lem:cumulants_preliminary_bound}~(C), Remark~\ref{rem:cumulants_moments_preliminary_bound} and Remark~\ref{rem:cumulants_preliminary_bound}.

It remains to prove the theorem for $\sIR\in[\uv\vee\sUV,1]$ and $\rs(\vI)=0$. Let us first consider the case \mbox{$\varrho(\vI)+(\rn(\vI)-1)\rDim>0$}. We use the first of Eqs.~\eqref{eq:thm_cumulants_ind_step}. As argued in the previous paragraph, the first term on the RHS of this equation is bounded by the RHS of the inequality~\eqref{eq:thm_cumulants}. The same bound for the second term follows from the bounds for $E^{\vI_q}_{\uv;\sUV,\sIR}$ proved above and Lemma~\ref{lem:integration}~(B) applied with $\rho=\varrho_{2\varepsilon}(\vI)+(\rn(\vI)-1)\rDim$ and $\epsilon=\varepsilon$. Note that $\rho>0$ by Remark~\ref{rem:epsilon}.

The case $\varrho(\vI)+(\rn(\vI)-1)\rDim\leq 0$ and $\sIR\in[\uv\vee\sUV,1]$ requires special treatment. As argued above, in this case the cumulant $E^\vI_{\uv;\sUV,\sIR}$ has the form~\eqref{eq:thm_cumulant_relevant_form}. If $m=0$, then $a=0$ and
\begin{equation}
 E^\vI_{\uv;\sUV,\sIR} = \llangle \partial^r_\uv F^{i,0,0}_{\uv;\sUV,\sIR}\rrangle=\partial_\uv^r f^{i,0,0}_{\uv;\sUV,\sIR}.
\end{equation}
Hence, in this case the statement of the theorem follows from the bound~\eqref{eq:thm_cumulants_relevant}. By the previous paragraph, for $\sIR\geq\uv\vee\sUV$ and all $(i,m,b)\in\frI^+$ it holds
\begin{equation}\label{eq:thm_cumulants_irrelevant1}
 \|K_{\uv\vee\sUV,\sIR}^{m;\oo}\ast \llangle \partial_\uv^r F^{i,m,b}_{\uv;\sUV,\sIR}\rrangle\|_{\cV^m} \lesssim [\uv]^{(\varepsilon-\sigma)r}[\sIR]^{\varrho_{2\varepsilon}(i,m,b)}.
\end{equation}
Moreover, for all $(i,m,b)\in\frI$ we have 
\begin{equation}\label{eq:thm_cumulants_irrelevant2}
 |\fI(\llangle \partial_\uv^r F^{i,m,b}_{\uv;\sUV,\sIR}\rrangle)|\lesssim [\uv]^{(\varepsilon-\sigma)r}[\sIR]^{\varrho_{2\varepsilon}(i,m,b)}.
\end{equation}
For $(i,m,b)\in\frI^+$ this follows from the bound~\eqref{eq:thm_cumulants_irrelevant1}, Lemma~\ref{lem:map_I} and translational invariance. For $(i,m,b)\in\frI^-$ the bound~\eqref{eq:thm_cumulants_irrelevant2} coincides with the bound~\eqref{eq:thm_cumulants_relevant}. We note that for any $(i,m,a)\in\frI^-$ such that $m\in\bN_+$ it holds
\begin{equation}
 \llangle \partial_\uv^r F^{i,m,a}_{\uv;\sUV,\sIR}\rrangle
 =
 \fX^a_\ro(\fI(\llangle \partial_\uv^r F^{i,m,b}_{\uv;\sUV,\sIR}\rrangle),
 \llangle \partial_\uv^r F^{i,m,b}_{\uv;\sUV,\sIR}\rrangle)
\end{equation}
by Theorem~\ref{thm:taylor}, where the map $\fX^a_\ro$ was introduced in Def.~\ref{dfn:map_X}. If $i=1$, then we set $\ro=\ceil{\sigma}$ and note that $\llangle \partial_\uv^r F^{i,m,b}_{\uv;\sUV,\sIR}\rrangle=0$ for all $|b|=\ro$ by locality. If $i>1$, then we set $\ro=\sigma_\diamond$ where $\sigma_\diamond=\sigma$ if $\sigma\in2\bN_+$ and $\sigma_\diamond=\ceil{\sigma}-1$ otherwise. By Assumption~\ref{ass:infrared} in both cases the RHS of the above equality involves only objects for which the bounds~\eqref{eq:thm_cumulants_irrelevant1} and~\eqref{eq:thm_cumulants_irrelevant2} hold true. Consequently, the statement of the theorem follows from Theorem~\ref{thm:taylor_bounds} applied with $C\lesssim [\sIR]^{\varrho_{2\varepsilon}(i,m)}$.
\end{proof}

\section{Probabilistic analysis}\label{sec:probabilistic}

The primary result of this section is Theorem~\ref{thm:probabilistic}, which exhibits and bounds the random fields $f^{i,m,a}_{0,\sIR}=\lim_{\sUV\searrow0} f^{i,m,a}_{\sUV,\sIR}$ for $(i,m,a)\in\frI^-_0$, $\sIR\in(0,1]$, where $f^{i,m,a}_{\sUV,\sIR}=f^{i,m,a}_{0;\sUV,\sIR}$, $(i,m,a)\in\frI^-_0$, $\sUV\in(0,1]$, $\sIR\in[0,1]$, are the relevant effective force coefficients. The proof relies on the Kolmogorov type argument presented in Sec.~\ref{sec:simple_probabilistic}, the bounds for cumulants of the effective force coefficients established in Sec.~\ref{sec:cumulants_estimates} and a certain diagonal argument taken from~\cite{hairer2017clt,mourrat2017} exploiting the extra parameter $\uv\in[0,1]$ of the coefficients $f^{i,m,a}_{\uv;\sUV,\sIR}$. Let us note that the probabilistic results of this section together with the deterministic results proved in Sec.~\ref{sec:deterministic} imply the main result of this paper stated in Theorem~\ref{thm:main}.

\begin{thm}\label{thm:probabilistic}
There exists $\oo\in\bN_0$, a coupling between the random fields $(\varXi_\sUV,\phi_\sUV^\vartriangle)$, $\sUV\in[0,1]$, a universal family of random continuous functions 
\begin{equation}\label{eq:probabilistic_thm_coefficients}
 (0,1]\ni\sIR\mapsto f^{i,m,a}_{0,\sIR}\in\sD'(\bH),
 \quad
 (i,m,a)\in\frI_0^-,
\end{equation}
and a random distribution $\phi_0^\vartriangle\in\sD'(\bT)$ such that for every $\varepsilon\in(0,1]$ and every \mbox{$\mathtt{u}\in C^\infty_\rc(\bR)$} the following conditions are satisfied:
\begin{enumerate}
\item [(A${}_1$)]
$
 \sup_{\sIR\in(0,1]}\,[\sIR]^{-\varrho_\varepsilon(i,m,a)}\,\|K^{\ast\oo}_\sIR\ast \mathtt{u} f^{i,m,a}_{0,\sIR}\|_{\cV} <\infty
$,
$(i,m,a) \in\frI_0^-$,

\item[(B${}_1$)]
$
 \lim_{\sUV\searrow0}\,[\sUV]^{-\varrho_\varepsilon(i,m,a)}\, \|\mathtt{u}f^{i,m,a}_\sUV\|_{\cV} = 0
$, $(i,m,a)\in\frI_0$,

\item[(C${}_1$)]
$
 \lim_{\sUV\searrow0}\sup_{\sIR\in(0,1]}[\sIR]^{-\varrho_\varepsilon(i,m,a)}\,\|K^{\ast\oo}_\sIR\ast \mathtt{u}(f^{i,m,a}_{\sUV,\sIR}-f^{i,m,a}_{0,\sIR})\|_{\cV}
=0$, $(i,m,a)\in\frI_0^-$.
\end{enumerate}
\begin{enumerate}
\item[(A${}_2$)]
$
 \sup_{\sIR\in(0,1]}\,[\sIR]^{-\beta_\varepsilon}\,\|\bar K_\sIR^{\ast\oo}\ast \phi^\vartriangle_0\|_{L^\infty(\bT)} <\infty
$,

\item[(B${}_2$)]
$
  \lim_{\sUV\searrow0}\,[\sUV]^{-\beta_\varepsilon}\,\|\fR_\sUV\phi_\sUV^\vartriangle\|_{L^\infty(\bT)}=0
$,

\item[(C${}_2$)]
$
 \lim_{\sUV\searrow0} \sup_{\sIR\in(0,1]}\,[\sIR]^{-\beta_\varepsilon}\,\|\bar K_\sIR^{\ast\oo}\ast (\phi^\vartriangle_\sUV-\phi^\vartriangle_0)\|_{L^\infty(\bT)} = 0
$.
\end{enumerate}
These bounds hold almost surely and the limits exist in probability.
\end{thm}

\begin{rem}
Recall that $f^{0,0,0}_{\sUV,\sIR}=\varXi_\sUV$, $\varrho_\varepsilon(0,0,0)=-\dim(\varXi)-\varepsilon=-(\rdim+\sigma)/2-\varepsilon$, $\|\Cdot\|_\cV=\|\Cdot\|_{L^\infty(\bH)}$, $\beta_\varepsilon=-\dim(\phi)-\varepsilon$, $\dim(\phi)$ was introduced in Def.~\ref{dfn:dim_phi} and $(\mathtt{u}f)(\mathring x,\bar x)=\mathtt{u}(\mathring x)f(\mathring x,\bar x)$.
\end{rem}

\begin{rem}
The random functions~\eqref{eq:probabilistic_thm_coefficients} are constructed in such a way that they are measurable with respect to the sigma algebra of the white noise $\varXi_0$ and depend only on the choice of the renormalization parameters $\mathfrak{f}^{i,m,a}_0$, $(i,m,a)\in\frI^-$. For a fixed choice of the renormalization parameters the family of functions~\eqref{eq:probabilistic_thm_coefficients} does not depend on a choice of the family of noises $\varXi_\sUV$, $\sUV\in(0,1]$, and a choice of the force coefficients $f^{i,m,a}_{\sUV}$, $\sUV\in(0,1]$, $(i,m,a)\in\frI$, satisfying our assumptions. This property is called universality.
\end{rem}

\begin{proof}
The construction of the coupling as well as the proofs of Parts (A${}_1$) and (C${}_1$) for $(i,m,a)=(0,0,0)$ and Parts (A${}_2$) and  
(C${}_2$) are contained in Sec.~\ref{sec:coupling}. Parts (A${}_1$) and (C${}_1$) for $(i,m,a)\in\frI^-=\frI_0^-\setminus\{(0,0,0)\}$ are established in Sec.~\ref{sec:probabilistic_convergence}. Part (B${}_1$) for $(i,m,a)=(0,0,0)$ and Part (B${}_2$) are proved in Sec.~\ref{sec:probabilistic_simple_bounds}. Part (B${}_1$) for $(i,m,a)\in\frI=\frI_0\setminus\{(0,0,0)\}$ follows from Remark~\ref{rem:force_coeff_cumulants}. Recall that the force coefficients $f^{i,m,a}_\sUV=f^{i,m,a}_{0;\sUV,0}\in\bR$, $(i,m,a)\in\frI$, are deterministic. 
\end{proof}

\subsection{Probabilistic estimate}\label{sec:simple_probabilistic}

In this section we prove a probabilistic estimate similar in spirit to the Kolmogorov continuity theorem. We will need the following simple lemma.

\begin{lem}\label{lem:expectation_sup}
Let $n\in2\bN_+$ be such that $n>\rdim$. 
\begin{enumerate}
\item[(A)] There exists a constant $C>0$ such that for all random fields $\zeta\in L^n(\bH)$ and $\sIR\in(0,1]$ it holds
\begin{equation}
 \llangle
 \|K_\sIR \ast \zeta\|^n_{L^\infty(\bH)}\rrangle
 \leq C\, [\sIR]^{-\rDim}\,
 \int_\bH \llangle|\zeta(x)|^n\rrangle\,\rd x.
\end{equation}
\item[(B)] There exists a constant $C>0$ such that for all random fields $\zeta\in L^n(\bT)$ and $\sIR\in(0,1]$ it holds
\begin{equation}
 \llangle
 \|\bar K_\sIR \ast \zeta\|^n_{L^\infty(\bT)}\rrangle
 \leq C\, [\sIR]^{-\rdim}\,
 \int_{\bT} \llangle|\zeta(\bar x)|^n\rrangle\,\rd \bar x.
\end{equation}
\item[(C)] For every $\mathtt{w}\in C^\infty_\rc(\bR)$ there exists a constant $C>0$ such that for all random fields $\hat\zeta\in L^\infty(\bH)$ and $\sIR\in(0,1]$ it holds
\begin{equation}
 \llangle
 \|\mathtt{w} (K_\sIR \ast \hat\zeta)\|^n_{L^\infty(\bH)}\rrangle
 \leq C\, [\sIR]^{-\rDim}\,
 \sup_{x\in\bH}\,\llangle|\hat\zeta(x)|^n\rrangle.
\end{equation}
\end{enumerate}
\end{lem}
\begin{proof}
Note that $K_\sIR \ast \zeta = \fT K_\sIR \star \zeta$, where $\star$ is the convolution in $\bH$ and $\fT K_\sIR$ is the periodization of $K_\sIR$ (see Def.~\ref{dfn:periodization}). Using the Young inequality for convolutions we obtain
\begin{equation}
 \llangle\|K_\sIR\ast \zeta\|^n_{L^\infty(\bH)}\rrangle
 \leq
 \|\fT K_\sIR\|_{L^{n/(n-1)}(\bH)}^n~
 \llangle\|\zeta\|^n_{L^n(\bH)}\rrangle.
\end{equation} 
By Lemma~\ref{lem:kernel_simple_fact}~(D) and the Fubini theorem we get
\begin{equation}\label{eq:lem_expectation_sup_aux}
 \llangle\|K_\sIR\ast \zeta\|^n_{L^\infty(\bH)}\rrangle
 \lesssim [\sIR]^{-\rDim}\,\int_\bH \llangle |\zeta(x)|^n\rrangle\,\rd x,
\end{equation}
which proves Part (A). The proof of Part (B) is the same. To prove Part (C) first recall that $\fP_\sIR=(1+\sIR\partial_{\mathring x})(1-[\sIR]^2\Delta_{\bar x})$ and set $\dot{\mathtt{w}}(\mathring x):=\partial_{\mathring x}\mathtt{w}(\mathring x)$. Note that it holds
\begin{multline}
 \|\mathtt{w}(K_\sIR \ast \hat\zeta)\|_{L^\infty(\bH)} 
 \leq\|K_\sIR\ast \fP_\sIR \mathtt{w}(K_\sIR \ast \hat\zeta)\|_{L^\infty(\bH)}
 \\
 \leq
 \|K_\sIR\ast \mathtt{w}\hat\zeta\|_{L^\infty(\bH)}
 +
 \sIR \|K_\sIR\ast \dot{\mathtt{w}}(K_\sIR \ast \hat\zeta)\|_{L^\infty(\bH)}.
\end{multline} 
Part (C) follows now from Part~(A) applied with $\zeta=\mathtt{w}\hat\zeta$ and $\zeta=\dot{\mathtt{w}}(K_\sIR \ast \hat\zeta)$ and the bound stated in Remark~\ref{rem:moments_cumulants} applied with $H=K_\sIR$.
\end{proof}

\begin{lem}\label{lem:probabilistic_estimate}
Fix $n\in2\bN_+$, such that $n>\rdim$, $\rho\in\bR$ and $\varepsilon>\rDim/n$. There exists a constant $c>0$ such that the following statement is true. Let $\zeta:(0,1]\to L^n(\bH)$ be a differentiable random function. Assume that
\begin{equation}
 \llangle\|\partial_\sIR^s \zeta_\sIR^{\phantom{l}}\|^n_{L^n(\bH)}\rrangle
 \\
 \leq C\,[\sIR]^{n(\rho+\varepsilon-\sigma s)}
\end{equation}
for some $C>0$ and all $s\in\{0,1\}$ and $\sIR\in(0,1]$. Then it holds
\begin{equation}
 \llangle[\Big] \sup_{\sIR\in(0,1]}
 \big\|[\sIR]^{-\rho}\,K_\sIR \ast \zeta_\sIR\big\|^n_{L^\infty(\bH)}\rrangle[\Big] \leq c\,C.
\end{equation}
\end{lem}
\begin{proof} 
Let $\hat\zeta:(0,1]\to L^\infty(\bH)$ be defined by
\begin{equation}
 \hat\zeta_\sIR:=[\sIR]^{-\rho}\, K_\sIR\ast \zeta_\sIR,
 \quad
 \sIR\in(0,1].
\end{equation}
By Lemma~\ref{lem:expectation_sup} and Lemma~\ref{lem:kernel_derivative} the function $\hat\zeta$ is differentiable and
\begin{equation}\label{eq:probabilistic_proof_ieq}
 \llangle\|\partial_\sIR^s \hat\zeta_\sIR\|^n_{L^\infty(\bH)}\rrangle
 \lesssim 
 C\, [\sIR]^{n(\varepsilon-\sigma s-\rDim/n)}
\end{equation}
for $s\in\{0,1\}$. Moreover, it holds
\begin{equation}
 \sup_{\sIR\in(0,1]}\|\hat\zeta_\sIR\|_{L^\infty(\bH)}
 \leq
 \|\hat\zeta_1\|_{L^\infty(\bH)}
 +
 \int_0^1 \|\partial_\sIR \hat\zeta_\sIR\|_{L^\infty(\bH)} \,\rd\sIR.
\end{equation}
By the Minkowski inequality we get
\begin{equation}\label{eq:probabilistic_proof_ieq2}
 \llangle[\Big] \sup_{\sIR\in(0,1]}\|\hat\zeta_\sIR\|^n_{L^\infty(\bH)} \rrangle[\Big]^{\!\frac{1}{n}}
 \leq
 \llangle\|\hat\zeta_1\|^n_{L^\infty(\bH)}\rrangle^{\frac{1}{n}} + \int_0^1 \llangle\|\partial_\sIR \hat\zeta_\sIR\|^n_{L^\infty(\bH)}\rrangle^{\frac{1}{n}} \,\rd\sIR.
\end{equation}
The statement of the lemma follows now from the bounds~\eqref{eq:probabilistic_proof_ieq} and~\eqref{eq:probabilistic_proof_ieq2}.
\end{proof}

\subsection{Simple bounds}\label{sec:probabilistic_simple_bounds}

In this section we show that
\begin{equation}
 \lim_{\sUV\searrow0}\,[\sUV]^{\dim(\varXi)+\varepsilon}\, \|\mathtt{u}\varXi_\sUV\|_{L^\infty(\bH)} = 0,
 \qquad
 \lim_{\sUV\searrow0}\,[\sUV]^{\dim(\phi)+\varepsilon}\,\|\fR_\sUV \phi_\sUV^\vartriangle\|_{L^\infty(\bT)}=0
\end{equation}
in distribution for every $\varepsilon\in(0,1]$ and every $\mathtt{u}\in C^\infty_\rc(\bR)$. This implies the convergence in probability for every coupling between $(\varXi_\sUV,\phi_\sUV^\vartriangle)$, $\sUV\in[0,1]$.

\begin{lem}\label{lem:noise_moment_bound}
For any $\varepsilon\in(0,1\wedge\sigma)$, $n\in\bN_+$ and $\mathtt{u}\in  C^\infty_\rc(\bR)$ it holds 
\begin{equation}
 \llangle \|\mathtt{u}\varXi_\sUV\|^n_{L^\infty(\bH)}\rrangle \lesssim [\sUV]^{-n(\dim(\varXi)+\varepsilon)}
\end{equation}
uniformly in $\sUV\in(0,1]$ and for $r\in\{0,1\}$
\begin{equation}
 \llangle \|\mathtt{u}\partial_\uv^r\varXi_{\uv;\sUV}\|^n_{L^\infty(\bH)}\rrangle \lesssim [\uv]^{(\varepsilon-\sigma)r}[\uv\vee\sUV]^{-n(\dim(\varXi)+2\varepsilon)}
\end{equation}
uniformly in $\uv\in(0,1]$, $\sUV\in[0,1]$.
\end{lem}
\begin{proof}
We prove the lemma for $n\in2\bN_+$ such that $n\varepsilon>\rDim$. The general case is implied by the Jensen inequality. The first bound follows from Lemma~\ref{lem:ass_noise_general} (A) and Lemma~\ref{lem:expectation_sup}~(C) applied with $\hat\zeta=\fP_\sUV \varXi_\sUV$. The second bound is the consequence of Lemma~\ref{lem:ass_noise_general} (C) and Lemma~\ref{lem:expectation_sup}~(C) applied with $\hat\zeta=\fP_{\uv\vee\sUV} \varXi_{\uv;\sUV}$ and $\hat\zeta=\fP_{\uv} \partial_\uv \varXi_{\uv;\sUV}$.
\end{proof}

\begin{lem}\label{lem:initial_data_moment_bound}
For any $\varepsilon\in(0,1]$ and $n\in\bN_+$ it holds
\begin{equation}
 \llangle \|\fR_\sUV \phi_\sUV^\vartriangle\|^n_{L^\infty(\bT)}\rrangle \lesssim  [\sUV]^{-n(\dim(\phi)+\varepsilon)}
\end{equation}
uniformly in $\sUV\in(0,1]$.
\end{lem}
\begin{proof}
We prove the lemma for $n\in2\bN_+$ such that $n\varepsilon>\rdim$. The general case is a consequence of the Jensen inequality. Since $\bar \fP_\sUV \bar K_\sUV=\delta_{\bR^\rdim}$ by Lemma~\ref{lem:kernel_simple_fact}~(C) we obtain
\begin{equation}
 \|\fR_\sUV \phi_\sUV^\vartriangle\|_{L^\infty(\bT)}
 =
 \|\fR_\sUV \bar K_\sUV^{\ast\ooo}\ast \bar\fP_\sUV^{\ooo} \phi_\sUV^\vartriangle\|_{L^\infty(\bT)}
 \\
 \lesssim
 \|\bar K_\sUV\ast \bar\fP_\sUV^{\ooo}\phi_\sUV^\vartriangle\|_{L^\infty(\bT)},
\end{equation}
where $\ooo=\floor{\sigma/2}+2$. Using the equality $\bar\fP_\sUV=(1-[\sUV]^2\Delta_{\bar x})$ and the bound stated in Assumption~\ref{ass:initial} we arrive at
\begin{equation}
 \llangle|\bar\fP_\sUV^\ooo \phi_\sUV^\vartriangle(x)|^n\rrangle
 \lesssim [\sUV]^{-n\dim(\phi)}
\end{equation}
uniformly in $\sUV\in(0,1]$ and $x\in\bT$. Finally, we apply Lemma~\ref{lem:expectation_sup}~(B) with $\zeta=\bar\fP_\sUV^\ooo \phi_\sUV^\vartriangle$.
\end{proof}

\subsection{Coupling}\label{sec:coupling}

\begin{lem}\label{lem:coupling}
Let $Y_\sUV$, $\sUV\in[0,1]$, be a family of random variables valued in some Polish space $(\cY,d)$. Assume that $\lim_{\sUV\searrow0}Y_\sUV=Y_0$ in law. Then there exists a coupling between $Y_\sUV$, $\sUV\in[0,1]$, such that $\lim_{\sUV\searrow0}Y_\sUV=Y_0$ in probability.
\end{lem}
\begin{proof}
We follow the method described in the proof of Theorem 6.5 in~\cite{hairer2017clt}. By~\cite[Theorem 6.9]{villani} the convergence in law is equivalent to the convergence in the Wasserstein sense for the bounded metric $d(\Cdot,\Cdot)\wedge1$. Hence,
\begin{equation}\label{eq:inf_wasserstein}
 \inf_{\bE} \bE(d(Y_\sUV,Y_0)\wedge1)
\end{equation}
converges to zero as $\sUV\searrow0$, where the infimum is taken over all joint distributions $\bE$ of the random variables $Y_\sUV$ and $Y_0$ with marginals coinciding with $\mathrm{Law}(Y_\sUV)$ and $\mathrm{Law}(Y_0)$. By~\cite[Theorem 4.1]{villani} for every $\sUV\in(0,1]$ the infimum~\eqref{eq:inf_wasserstein} is attained for some joint distribution $\bE=\bE_\sUV$. The desired coupling $\llangle\Cdot\rrangle$ is obtained with the use of \cite[Theorem A.1]{deAcosta1982} by gluing together the joint distributions $\bE_\sUV$, \mbox{$\sUV\in(0,1]$}, along the common marginal $\mathrm{Law}(Y_0)$. We have
\begin{equation}
 \lim_{\sUV\searrow0} \llangle d(Y_\sUV,Y_0)\wedge1\rrangle=\lim_{\sUV\searrow0} \bE_\sUV(d(Y_\sUV,Y_0)\wedge1) = 0,
\end{equation}
which is equivalent to convergence in probability. 
\end{proof}

\begin{lem}\label{lem:coupling_varXi_phi}
Let $\oo=\floor{\dim(\varXi)\vee\dim(\phi)}+1$. There exists a coupling between the random fields $(\varXi_\sUV,\phi_\sUV^\vartriangle)$, $\sUV\in[0,1]$, such that for every $\varepsilon\in(0,1]$ and every $\mathtt{u}\in C^\infty_\rc(\bR)$ it holds
\begin{enumerate}
 \item[(A)]
 $\lim_{\sUV\searrow0} \sup_{\sIR\in(0,1]}\,[\sIR]^{\dim(\varXi)+\varepsilon}\,\|K_\sIR^{\ast\oo}\ast \mathtt{u}(\varXi_\sUV-\varXi_0)\|_{L^\infty(\bH)} = 0$,
 
 \item[(B)]
 $\lim_{\sUV\searrow0} \sup_{\sIR\in(0,1]}\,[\sIR]^{\dim(\phi)+\varepsilon}\,\|\bar K_\sIR^{\ast\oo}\ast (\phi^\vartriangle_\sUV-\phi^\vartriangle_0)\|_{L^\infty(\bT)} = 0$
\end{enumerate} 
in probability.
\end{lem}
\begin{proof}
Let us fix some sequence of functions $\mathtt{u}_n\in C^\infty_\rc(\bR)$, $n\in\bN_0$, such that $\mathtt{u}_n(\mathring x)=1$ for $\mathring x\in[-n,n]$. For $(v,\phi)\in C(\bH)\times C(\bT)$ we define the family of semi-norms
\begin{equation}
 \sup_{\sIR\in(0,1]} [\sIR]^{\dim(\varXi)+\varepsilon}\|K^{\ast\oo}_\sIR \ast \mathtt{u} v\|_{L^\infty(\bH)}
 +
 \sup_{\sIR\in(0,1]} [\sIR]^{\dim(\phi)+\varepsilon}\|\bar K^{\ast\oo}_\sIR \ast \phi\|_{L^\infty(\bT)}
\end{equation}
parameterized by $\varepsilon\in\{2^{-n}\,|\,n\in\bN_0\}$ and $\mathtt{u}\in \{\mathtt{u}_n\,|\,n\in\bN_0\}$. Let $\cY$ be the Fr{\'e}chet space obtained by the completion of $C(\bH)\times C(\bT)$ equipped with the above family of semi-norms and let $d$ be a metric inducing the same topology. The statement follows now from the previous lemma, Remark~\ref{rem:weight} and the fact that $\lim_{\sUV\searrow0}\,(\varXi_\sUV,\phi_\sUV^\vartriangle) 
=
(\varXi_0,\phi_0^\vartriangle)$ in law in $(\cY,d)$. The existence of the above limit in law follows from Remark~\ref{rem:ass_joint_convergence}.
\end{proof}

\begin{lem}\label{lem:time_estimates_prep}
Let $\oo\in\bN_0$. The following bound
\begin{equation}
 \|K_\sIR^{\ast\oo}\ast \mathtt{v} v\|_\cV \lesssim
 \sum_{\ooo=0}^\oo \sIR^\ooo\,
 \|(\partial_{\mathring x}^\ooo \mathtt{v}) (K_\sIR^{\ast\oo}\ast v)\|_\cV 
\end{equation}
holds uniformly in $\sIR\in(0,1]$, $\mathtt{v}\in C^\infty_\rc(\bR)$ and $v\in\cV$.
\end{lem}
\begin{proof}
We have
\begin{equation}
 K_\sIR^{\ast\oo}\ast \mathtt{v} v = K_\sIR^{\ast\oo}\ast \mathtt{v} \fP_\sIR^\oo(K_\sIR^{\ast\oo}\ast v),
\end{equation}
where we used the equality $\fP_\sIR^\oo K_\sIR^{\ast\oo}=\delta_\bM$. Note that since $\fP_\sIR=1+\sIR\partial_{\mathring x}-[\sIR]^2\Delta$ and $\|K_\sIR\|_\cK=1$ we have
\begin{equation}
 \|K_\sIR^{\ast\ooo}\ast \mathtt{w} \fP_\sIR^\ooo f\|_\cV
 \leq
 \|K_\sIR^{\ast(\ooo-1)}\ast \mathtt{w} \fP_\sIR^{\ooo-1}f\|_\cV
 +\sIR \|K_\sIR^{\ast(\ooo-1)}\ast (\partial_{\mathring x}\mathtt{w}) \fP_\sIR^{\ooo-1}f\|_\cV
\end{equation}
for any $\ooo\in\bN_+$, $\mathtt{w}\in C^\infty_\rc(\bR)$ and $f\in\cV$. Let $\mathtt{v}^{(\oooo)}(\mathring x):=\partial^{\oooo}_{\mathring x}\mathtt{v}(\mathring x)$ for $\oooo\in\bN_0$. Applying the above bound recursively with $f=K_\sIR^{\ast\oo}\ast v$, $\mathtt{w}\in\{\mathtt{v}^{(0)},\ldots,\mathtt{v}^{(\oo-\ooo)}\}$ and $\ooo\in\{\oo,\oo-1,\ldots,1\}$ we arrive at the statement of the lemma.
\end{proof}

\begin{rem}\label{rem:weight}
Let $\oo\in\bN_0$, $\mathtt{u}\in C^\infty_\rb(\bR)$ and $\mathtt{v}\in C^\infty_\rc(\bR)$ be such that for all $\mathring x\in\bR$ it holds $\mathtt{v}(\mathring x)=0$ unless $\mathtt{u}(\mathring x)=1$. By Lemma~\ref{lem:time_estimates_prep} the following bound
\begin{equation}
 \|K^{\ast\oo}_{\sIR}\ast \mathtt{v} v\|_{\cV} 
 \lesssim \|K^{\ast\oo}_{\sIR}\ast\mathtt{u} v\|_{\cV}
\end{equation}
holds uniformly in $\sIR\in(0,1]$ and $v\in\cV$.
\end{rem}

\subsection{Moment bounds}

\begin{lem}\label{lem:moment_bounds}
Let $\varepsilon\in(0,\varepsilon_\sigma\wedge\varepsilon_\diamond/2]$, $n\in2\bN_+$, $s\in\{0,1\}$ be arbitrary. There exists $\oo\in\bN_0$ such that for every $(i,m,a)\in\frI^-_0$ the following bounds
\begin{equation}
\begin{gathered}
 \llangle|K_\sIR^{\ast\oo}
 \ast \partial_\sIR^s f^{i,m,a}_{\uv;\sUV,\sIR}(x)|^n\rrangle
 \lesssim [\sIR]^{n(\varrho_{2\varepsilon}(i,m,a)-\sigma s)},
 \\
 \llangle|K_\sIR^{\ast\oo}
 \ast \partial_\sIR^s (f^{i,m,a}_{\uv;\sUV,\sIR}-f^{i,m,a}_{0;\sUV,\sIR})(x)|^n\rrangle
 \lesssim [\uv]^{\varepsilon n}\, [\sIR]^{n(\varrho_{2\varepsilon}(i,m,a)-\sigma s)},
 \\
 \llangle|K_\sIR^{\ast\oo}
 \ast \partial_\sIR^s (f^{i,m,a}_{\uv_1;0,\sIR}-f^{i,m,a}_{\uv_2;0,\sIR})(x)|^n\rrangle
 \lesssim |[\uv_1]^{\varepsilon}-[\uv_2]^{\varepsilon}|^n\, [\sIR]^{n(\varrho_{2\varepsilon}(i,m,a)-\sigma s)}
\end{gathered} 
\end{equation}
hold uniformly in $x\in\bH$, $\uv,\sUV,\sIR\in(0,1]$.
\end{lem}
\begin{proof}
We first observe that by the Minkowski inequality and the estimate 
\begin{equation}
 \int_{\uv_1}^{\uv_2}[\uv]^{\varepsilon-\sigma}\rd\uv \lesssim |[\uv_1]^\varepsilon-[\uv_2]^\varepsilon|
\end{equation}
it is enough to prove that for $r\in\{0,1\}$ the following bound
\begin{equation}\label{eq:lem_moments_bound_r}
 \llangle|K_\sIR^{\ast\oo}
 \ast \partial_\sIR^s\partial^r_\uv f^{i,m,a}_{\uv;\sUV,\sIR}(x)|^n\rrangle
 \lesssim [\uv]^{n(\varepsilon-\sigma)r}\, [\sIR]^{n(\varrho_{2\varepsilon}(i,m,a)-\sigma s)}
\end{equation}
holds uniformly for $(\uv,\sUV)\in[0,1]\setminus\{(0,0)\}$, $\sIR\in(0,1]$. Since $\varrho_{2\varepsilon}(i,m,a)<0$ for any $(i,m,a)\in\frI_0^-$ and $\sIR\in(0,\uv\vee\sUV]$ the above bound follows from Lemma~\ref{lem:cumulants_preliminary_bound}~(B) and Remark~\ref{rem:moments_preliminary_bound} (note that the bound assumed in Lemma~\ref{lem:cumulants_preliminary_bound} is proved in Theorem~\ref{thm:cumulants} and Remark~\ref{rem:force_coeff_cumulants}). To prove the above bound for $\sIR\in(\uv\vee\sUV,1]$ consider the following list of indices
\begin{equation}
 \vI=((i,m,a,s,r),\ldots,(i,m,a,s,r)),\quad \rn(\vI)=n,~\rrm(\vI)=nm.
\end{equation}
Using the identity $\partial^s_\sIR\partial_\uv^r f^{i,m,a}_{\uv;\sUV,\sIR}=\fI(\partial^s_\sIR\partial_\uv^r F^{i,m,a}_{\uv;\sUV,\sIR})$ we obtain for some $\oooo\in\bN_0$
\begin{multline}
 \sup_{x_1\in\bH}\int_{\bH^{n-1}}|\llangle K^{\ast\oooo}_\sIR\ast \partial^s_\sIR\partial_\uv^r f^{i,m,a}_{\uv;\sUV,\sIR}(x_1),
 \ldots,
 K^{\ast\oooo}_\sIR\ast\partial^s_\sIR\partial_\uv^r f^{i,m,a}_{\uv;\sUV,\sIR}( x_n)\rrangle|\,\rd x_2\ldots\rd x_n
 \\
 \leq
 \big\|K^{\ast\oooo,\otimes(n+nm)}_\sIR\ast
 \llangle
 \partial^s_\sIR\partial_\uv^r F^{i,m,a}_{\uv;\sUV,\sIR},
 \ldots,
 \partial^s_\sIR\partial_\uv^r F^{i,m,a}_{\uv;\sUV,\sIR}\rrangle
 \big\|_{\cV^{\mathsf{m}}} 
 \\[1mm]
 =\big\|K^{\ast\oooo,\otimes(n+nm)}_\sIR\ast E^\vI_{\uv;\sUV,\sIR}\big\|_{\cV^{\mathsf{m}}} 
 \lesssim [\uv]^{n(\varepsilon-\sigma)r}\,[\sIR]^{n(\varrho_{2\varepsilon}(i,m,a)-\sigma s)+(n-1)\rDim}
\end{multline}
uniformly $(\uv,\sUV)\in[0,1]\setminus\{(0,0)\}$ and $\sIR\in[\uv\vee\sUV,1]$. The last inequality above follows from Theorem~\ref{thm:cumulants} and $n\varrho_{2\varepsilon}(i,m,a)= \varrho_{2\varepsilon}(\vI)$. Since \mbox{$\|K^{\ast\ooo}_\sIR\|_\cK=1$} and by Lemma~\ref{lem:kernel_simple_fact}~(D) $\|\fT K_\sIR^{\ast\ooo}\|_{L^\infty(\bH)}\lesssim [\sIR]^{-\rDim}$ uniformly in $\sIR\in(0,1]$ for \mbox{$\ooo=\rdim+1$} the bound~\eqref{eq:lem_moments_bound_r} with $\oo=\oooo+\ooo$ follows from Lemma~\ref{lem:moments_cumulants} applied with $H=K_\sIR^{\ast\ooo}$.
\end{proof}

\subsection{Uniform bounds and convergence}\label{sec:probabilistic_convergence}

In this section we assume that the random fields $\varXi_\sUV$, $\sUV\in[0,1]$ are coupled in such a way that the statement (A) of Lemma~\ref{lem:coupling_varXi_phi} holds true. We show that there exists $\oo\in\bN_0$ and a universal family of random continuous functions 
\begin{equation}\label{eq:probabilistic_convergence_limit}
 (0,1]\ni\sIR\mapsto f^{i,m,a}_{0,\sIR}\in\cD^{;\oo},
 \quad
 (i,m,a)\in\frI_0^-,
\end{equation}
such that for any $\varepsilon\in(0,1]$ it holds
\begin{equation}
 \sup_{\sIR\in(0,1]}[\sIR]^{-\varrho_{\varepsilon}(i,m,a)}\,\|K^{\ast\oo}_\sIR\ast \mathtt{u} f^{i,m,a}_{0,\sIR}\|_{L^\infty(\bH)} <\infty
\end{equation}
almost surely and
\begin{equation}
 \lim_{\sUV\searrow0}\sup_{\sIR\in(0,1]}[\sIR]^{-\varrho_{\varepsilon}(i,m,a)}\,\|K^{\ast\oo}_\sIR\ast \mathtt{u}(f^{i,m,a}_{\sUV,\sIR}-f^{i,m,a}_{0,\sIR})\|_{L^\infty(\bH)}
\end{equation}
in probability for all $\mathtt{u}\in C^\infty_\rc(\bR)$ and $(i,m,a)\in\frI_0^-$. Recall that $(i,m,a)\in\frI_0^-$ iff $(i,m,a)\in\frI_0$ and $\varrho(i,m,a)<0$. Let us also remark that $f^{0,0,0}_{\sUV,\sIR}=\varXi_\sUV$ for all $\sUV,\sIR\in[0,1]$.

We first establish the convergence of $f^{i,m,a}_{\uv;\sUV,\sIR}$ as $\sUV\searrow0$ for fixed $\uv\in(0,1]$. Note that this simple result, stated in the lemma below, relies only on rough estimates for moments of the regularized noise $\varXi_{\uv;\sUV}$ and its convergence as $\sUV\searrow0$.

\begin{lem}\label{lem:probabilistic_cutoff}
Fix $\varepsilon,\uv\in(0,1]$. We have
\begin{equation}
 \sup_{\sUV\in[0,1]}\llangle[\Big] \sup_{\sIR\in[0,1]} [\sIR]^{n(\sigma-\varepsilon)s}\,\|\mathtt{u} \partial_\sIR^s f^{i,m,a}_{\uv;\sUV,\sIR}\|_{L^\infty(\bH)}^n \rrangle[\Big] <\infty
\end{equation}
and
\begin{equation}
 \lim_{\sUV\searrow0}\,\llangle[\Big]  \sup_{\sIR\in[0,1]}
  [\sIR]^{n(\sigma-\varepsilon)s}\,
 \|\mathtt{u} \partial_\sIR^s(f^{i,m,a}_{\uv;\sUV,\sIR}-f^{i,m,a}_{\uv;0,\sIR})\|^n_{L^\infty(\bH)}\rrangle[\Big] =0
\end{equation} 
for all $s\in\{0,1\}$, $(i,m,a)\in\frI^-_0$, $n\in\bN_+$ and $\mathtt{u}\in C^\infty_\rc(\bR)$.
\end{lem}
\begin{proof}
The statement is true if $m>i m_\flat$ since then $f^{i,m,a}_{\uv;\sUV,\sIR}=0$. To prove the statement for $i=0$ it is enough to investigate $f^{0,0,0}_{\uv;\sUV,\sIR}=\varXi_{\uv;\sUV}=M_\uv\ast\varXi_\sUV$. Using the fact that $M_\uv\in C^\infty(\bM)$ is supported in a unit ball we obtain that for any $\mathtt{v}\in C^\infty_\rc(\bR)$ such that $\mathtt{v}=1$ on the set $\supp\,\mathtt{u}+[-1,1]$ it holds
\begin{multline}
\|\mathtt{u}(\varXi_{\uv;\sUV}-\varXi_{\uv;0})\|_{L^\infty(\bH)}
=
\|\mathtt{u} M_\uv\ast \mathtt{v}(\varXi_{\sUV}-\varXi_{0})\|_{L^\infty(\bH)}
\\
\leq
\|M_\uv\ast \mathtt{v}(\varXi_{\sUV}-\varXi_{0})\|_{L^\infty(\bH)}
\leq
\|\fP^\oo M_\uv\|_\cK \,\|K^{\ast\oo}\ast \mathtt{v}(\varXi_\sUV-\varXi_0)\|_{L^\infty(\bH)}
\end{multline}
for any $\oo\in\bN_0$. Note that $\|\fP^\oo M_\uv\|_\cK<\infty$. As a result, for $i=0$ the statement follows from Lemma~\ref{lem:noise_moment_bound} and Lemma~\ref{lem:coupling_varXi_phi} (A). 

Fix $i_\circ\in\bN_+$ and $m_\circ\in\bN_0$ and assume that the statement of the lemma is true for all $i,m\in\bN_0$ such that $i<i_\circ$ or $i=i_\circ$ and $m>m_\circ$. Since the force coefficients $f^{i,m,a}_{\uv;\sUV}=f^{i,m,a}_{\uv;\sUV,0}\in\bR$ are deterministic the induction hypothesis implies
\begin{equation}\label{eq:lem_probabilistic_cutoff_conditions}
 \sup_{\sUV\in[0,1]} |f^{i,m,a}_{\uv;\sUV}|<\infty,
 \qquad
 \lim_{\sUV\searrow0} |f^{i,m,a}_{\uv;\sUV}-f^{i,m,a}_{\uv;0}|=0
\end{equation}
for all $i,m\in\bN_0$ as above and all $a\in\frM^m_\sigma$. Hence, by the deterministic estimates established in Lemma~\ref{lem:eff_force_preliminary_bound} and Remark~\ref{rem:eff_force_preliminary_bound} the statement of the lemma is true for $s=1$ and $i=i_\circ$, $m=m_\circ$. Let us prove the statement for $s=0$ and $i=i_\circ$, $m=m_\circ$. Note that if $(i,m,a)\in\frI^+$, then the conditions~\eqref{eq:lem_probabilistic_cutoff_conditions} follow immediately from Assumption~\ref{ass:renormalization}. If $(i,m,a)\in\frI^-$, then we verify these conditions using the equality
\begin{equation}
 f^{i,m,a}_{\uv;\sUV} = \mathfrak{f}^{i,m,a}_{\sUV} 
 -
 \int_0^1 \llangle \partial_\sIR^s f^{i,m,a}_{\uv;\sUV,\sIR}\rrangle \,\rd \sIR,
\end{equation}
the fact that $\partial_\sIR^s f^{i,m,a}_{\uv;\sUV,\sIR}$ is stationary and the statement of the lemma with $s=1$. This proves that the conditions~\eqref{eq:lem_probabilistic_cutoff_conditions} hold true for $i=i_\circ$, $m=m_\circ$ and all $a\in\frM^m_\sigma$. Hence, by Lemma~\ref{lem:eff_force_preliminary_bound} and Remark~\ref{rem:eff_force_preliminary_bound} the statement is true for $s=0$ and $i=i_\circ$, $m=m_\circ$ and all $a\in\frM^m_\sigma$. This proves the induction step.
\end{proof}

\begin{lem}\label{lem:probabilistic}
Let $\mathtt{u}\in C^\infty_\rc(\bR)$ be arbitrary. There exists $\oo\in\bN_+$ such that it holds
\begin{equation}
\begin{gathered}
 \llangle[\Big] \sup_{\sIR\in(0,1]}\,[\sIR]^{-n\varrho_{3\varepsilon}(i,m,a)}\, \|K_\sIR^{\ast\oo} \ast \mathtt{u} f^{i,m,a}_{\uv;\sUV,\sIR}\|^n_{L^\infty(\bH)}
\rrangle[\Big] \lesssim 1,
 \\
\llangle[\Big] \sup_{\sIR\in(0,1]}\,[\sIR]^{-n\varrho_{3\varepsilon}(i,m,a)}\, \|K_\sIR^{\ast\oo} \ast \mathtt{u}(f^{i,m,a}_{\uv;\sUV,\sIR}-f^{i,m,a}_{0;\sUV,\sIR})\|^n_{L^\infty(\bH)}
\rrangle[\Big] \lesssim [\uv]^{n\varepsilon},
 \\
\llangle[\Big] \sup_{\sIR\in(0,1]}\,[\sIR]^{-n\varrho_{3\varepsilon}(i,m,a)}\, \|K_\sIR^{\ast\oo} \ast \mathtt{u}(f^{i,m,a}_{\uv_1;0,\sIR}-f^{i,m,a}_{\uv_2;0,\sIR})\|^n_{L^\infty(\bH)}
\rrangle[\Big] \lesssim |[\uv_1]^\varepsilon-[\uv_2]^\varepsilon|^n
\end{gathered}
\end{equation}
uniformly in $\uv,\uv_1,\uv_2,\sUV\in(0,1]$ for all $\varepsilon\in(0,\varepsilon_\sigma\wedge\varepsilon_\diamond/2]$, $n\in\bN_+$, $(i,m,a)\in\frI^-_0$.
\end{lem}
\begin{proof}
We first note that by the Jensen inequality if the above bounds hold for some $n=n_\circ$, then they hold for all $n<n_\circ$. We assume that $n\in\bN_+$ is such that $\varepsilon>\rDim/n$ and note that $\varrho_{3\varepsilon}(i,m,a)\leq\varrho_{2\varepsilon}(i,m,a)-\varepsilon$ if $m\leq i m_\flat$. The bounds in the statement of the lemma follow from Lemma~\ref{lem:probabilistic_estimate} applied with $\rho=\varrho_{2\varepsilon}(i,m,a)-\varepsilon$ and
\begin{equation}
 \zeta_\sIR = K_\sIR^{\ast\oo}\ast \mathtt{u} f^{i,m,a}_{\uv;\sUV,\sIR},\quad C\lesssim 1,
\end{equation}
\begin{equation}
 \zeta_\sIR = K_\sIR^{\ast\oo}\ast \mathtt{u}(f^{i,m,a}_{\uv;\sUV,\sIR}-f^{i,m,a}_{0;\sUV,\sIR}),\quad C\lesssim [\uv]^{n\varepsilon}, 
\end{equation}
\begin{equation}
 \zeta_\sIR = K_\sIR^{\ast\oo}\ast \mathtt{u}(f^{i,m,a}_{\uv_1;0,\sIR}-f^{i,m,a}_{\uv_2;0,\sIR}),\quad C\lesssim |[\uv_1]^\varepsilon-[\uv_2]^\varepsilon|^n.
\end{equation}
The above $\zeta_\sIR$ and $C$ satisfy the assumption of Lemma~\ref{lem:probabilistic_estimate} by Lemma~\ref{lem:time_estimates_prep}, Lemma~\ref{lem:moment_bounds}, the identity
\begin{equation}
 \partial_\sIR (K_\sIR^{\ast\oo}\ast \mathtt{u}f_\sIR) = \oo\,\fP_\sIR \partial_\sIR K_\sIR \ast K_\sIR^{\ast \oo} \ast \mathtt{u}f_\sIR +  K_\sIR^{\ast\oo}\ast\mathtt{u}\partial_\sIR f_\sIR, 
\end{equation}
and the uniform bound $\|\fP_\sIR \partial_\sIR K_\sIR\|_\cK\lesssim [\sIR]^{-\sigma}$ established in Lemma~\ref{lem:kernel_derivative}.
\end{proof}

\begin{lem}\label{lem:diagonal_argument}
Let $Y_{\uv,\sUV}$, $(\uv,\sUV)\in[0,1]^2\setminus\{(0,0)\}$ be a family of random variables valued in a Polish space $(\cY,d)$. Assume that:
\begin{enumerate}
\item[(A)] $\lim_{\uv\searrow0}\,\sup_{\sUV\in(0,1]} \llangle 1\wedge d(Y_{\uv,\sUV},Y_{0,\sUV})\rrangle = 0$,

\item[(B)] $\lim_{\uv_1,\uv_2\searrow0}\, \llangle 1\wedge d(Y_{\uv_1,0},Y_{\uv_2,0})\rrangle = 0$,

\item[(C)] $\forall_{\uv\in(0,1]}\,\lim_{\sUV\searrow0}\, \llangle 1\wedge d(Y_{\uv,\sUV},Y_{\uv,0})\rrangle = 0$.
\end{enumerate}
Then there exists a random variable $Y_{0,0}$ in $(\cY,d)$ such that 
\begin{equation}
 \lim_{\sUV\searrow0}\, \llangle 1\wedge d(Y_{0,\sUV},Y_{0,0})\rrangle = 0.
\end{equation}
\end{lem}
\begin{proof}
The assumption (B) says that $Y_{\uv,0}$, $\uv\in(0,1]$, is Cauchy in probability. This implies that there exists $Y_{0,0}$ such that $\lim_{\uv\searrow0}\, \llangle 1\wedge d(Y_{\uv,0},Y_{0,0})\rrangle = 0$. To conclude the proof it is enough to observe that by the triangle inequality for every $\uv\in(0,1]$ it holds
\begin{multline}
 \limsup_{\sUV\searrow0}\, \llangle 1\wedge d(Y_{0,\sUV},Y_{0,0})\rrangle
 \\
 \leq
 \sup_{\sUV\in(0,1]}\, \llangle 1\wedge d(Y_{0,\sUV},Y_{\uv,\sUV})\rrangle
 +
 \lim_{\sUV\searrow0}\,\llangle 1\wedge d(Y_{\uv,\sUV},Y_{\uv,0})\rrangle
 +
 \llangle 1\wedge d(Y_{\uv,0},Y_{0,0})\rrangle
\end{multline}
and subsequently take the limit $\uv\searrow0$ of the expression on the RHS of the above inequality.
\end{proof}

Let us prove the claim formulated at the beginning of this section. To this end, we fix some sequence of functions $\mathtt{u}_n\in C^\infty_\rc(\bR)$, $n\in\bN_0$, such that $\mathtt{u}_n(\mathring x)=1$ for $\mathring x\in[-n,n]$. For any $(i,m,a)\in\frI_0^-$ and any bounded continuous functions $(0,1]\ni\sIR \mapsto v_\sIR\in C(\bH)$ we define the family of semi-norms
\begin{equation}
 \sup_{\sIR\in(0,1]} [\sIR]^{-\rho}\|K^{\ast\oo}_\sIR \ast \mathtt{u} v_\sIR\|_{L^\infty(\bH)}
\end{equation}
parameterized by $\varepsilon\in\{2^{-n}\,|\,n\in\bN_0\}$ and $\mathtt{u}\in \{\mathtt{u}_n\,|\,n\in\bN_0\}$. Let $\cY^\rho$ be the Fr{\'e}chet space obtained by the completion of the space of the functions $v_\Cdot$ of the above form equipped with the above family of semi-norms and let $d^\rho$ be a metric inducing the same topology. The statement follows now from Lemma~\ref{lem:diagonal_argument} applied to $Y_{\uv,\sUV}=f^{i,m,a}_{\uv;\sUV,\Cdot}\in\cY^\rho$, $\rho=\varrho_{3\varepsilon}(i,m,a)$, Remark~\ref{rem:weight} and Lemma~\ref{lem:probabilistic_cutoff} and Lemma~\ref{lem:probabilistic}. In order to prove the universality of the family of the random functions~\eqref{eq:probabilistic_convergence_limit} we note that
\begin{equation}
 f^{i,m,a}_{0,\Cdot}:=\lim_{\uv\searrow0} f^{i,m,a}_{\uv;0,\Cdot},\quad (i,m,a)\in\frI_0^-,
\end{equation}
in probability in the space $\cY^\rho$, where the effective force coefficients $f^{i,m,a}_{\uv;0,\sIR}$ are constructed as specified in Def.~\ref{dfn:eff_force_uv} with the noise $\varXi_{\uv;0}$ obtained by mollification of the white noise and the (deterministic) force coefficients $f^{i,m,a}_{\uv;0}\in\bR$, $(i,m,a)\in\frI$ satisfying the following conditions specified in Assumption~\ref{ass:renormalization}. The irrelevant force coefficients $f^{i,m,a}_{\uv;0}$, \mbox{$(i,m,a)\in\frI^+$}, vanish identically and the relevant force coefficients $f^{i,m,a}_{\uv;0}$, $(i,m,a)\in\frI^+$, are fixed by the renormalization conditions $\llangle f^{i,m,a}_{\uv;0,1}\rrangle=\mathfrak{f}^{i,m,a}_0$.

\section{Deterministic analysis}\label{sec:deterministic}

The deterministic results established in this section allow to prove the existence of the macroscopic scaling limit and its universality using a pathwise construction. Given (generic realizations of) the relevant effective force coefficients $f^{i,m,a}_{0,\sIR}$, $(i,m,a)\in\frI^-_0$, $\sIR\in(0,1]$, and the initial data $\phi_0^\vartriangle$ we construct distributions $\varPhi_0^\triangleright$ and $\breve\varPhi_0$ having the following properties. If the effective force coefficients $f^{i,m,a}_{\sUV,\sIR}$, \mbox{$(i,m,a)\in\frI^-_0$}, \mbox{$\sIR\in(0,1]$}, constructed from a noise $\varXi_\sUV$ and force coefficients $f_\sUV^{i,m,a}$, $(i,m,a)\in\frI$, and the initial data $\phi_\sUV^\vartriangle$ are close to the effective force coefficients $f^{i,m,a}_{0,\sIR}$, \mbox{$(i,m,a)\in\frI^-_0$}, $\sIR\in(0,1]$, and the initial data $\phi_0^\vartriangle$, then $\varPhi_\sUV^\triangleright$ and $\breve\varPhi_\sUV$ are close to $\varPhi_0^\triangleright$ and $\breve\varPhi_0$, where $\varPhi_\sUV^\triangleright$ is the truncation of the formal power series solving Eq.~\eqref{eq:informal_statement_stationary} and $\breve\varPhi_\sUV$ is the classical solution of Eq.~\eqref{eq:statement_rem_breve}. Both $\varPhi_\sUV^\triangleright$ and $\breve\varPhi_\sUV$ are constructed as functions of the noise $\varXi_\sUV$, the force coefficients $f_\sUV^{i,m,a}$, $(i,m,a)\in\frI$, and the initial data $\phi_\sUV^\vartriangle$. The above-mentioned continuity result together with the results of the probabilistic analysis implies the existence of the limits $\lim_{\sUV\searrow0} \varPhi_\sUV^\triangleright=\varPhi_0^\triangleright$ and $\lim_{\sUV\searrow0} \breve\varPhi_\sUV=\breve\varPhi_0$ in probability.

\begin{dfn}\label{dfn:boundary_data_initial}
A collection of objects
\begin{equation}\label{eq:boundary_data}
 \phi_\sUV^\vartriangle\in\cC^\gamma(\bT),
 \qquad
 f^{0,0,0}_\sUV\equiv\varXi_\sUV\in \cD^{;0}=C(\bH),
 \qquad
 f^{i,m,a}_\sUV\in\bR,\quad~ (i,m,a)\in\frI,
\end{equation}
such that $f^{i,m,a}_\sUV=0$ unless $i\leq i_\flat$ and $m\leq m_\flat$ and $|a|<\sigma$ is called a boundary data for $(\varPhi_\sUV^\triangleright,\breve\varPhi_\sUV)$. A collection consisting of $\phi_0^\vartriangle\in\sD'(\bT)$ and a family of continuous functions 
\begin{equation}\label{eq:enhanced_data}
 (0,1]\ni\sIR\mapsto f^{i,m,a}_{0,\sIR}\in\cD=\sD'(\bH),
 \qquad (i,m,a)\in\frI^-_0,
\end{equation} 
is called a boundary data for $(\varPhi_0^\triangleright,\breve\varPhi_0)$.
\end{dfn}

\begin{rem}
Throughout this section $\sUV\in(0,1]$ is a fixed parameter. Recall that $\frI_0=\{(0,0,0)\}\cup\frI$ and $\frI_0^-\subset \frI_0$ is the set of relevant indices. We have $\gamma\equiv\gamma_\varepsilon:=\sigma-\varepsilon$, and the space $\cC^\gamma(\bT)\subset C^{\ceil{\sigma}-1}(\bT)$ was introduced in Def.~\ref{dfn:C_gamma}.
\end{rem}

Assuming that the boundary data for $(\varPhi_0^\triangleright,\breve\varPhi_0)$ is admissible in the sense specified below we give a construction of $(\varPhi_\sUV^\triangleright,\breve\varPhi_\sUV)$ and $(\varPhi_0^\triangleright,\breve\varPhi_0)$ as functions of the respective boundary data and show that if the boundary data for $(\varPhi_\sUV^\triangleright,\breve\varPhi_\sUV)$ and $(\varPhi_0^\triangleright,\breve\varPhi_0)$ are close to each other in the sense specified in Assumptions~\ref{ass:deterministic} and~\ref{ass:deterministic_initial}, then $(\varPhi_\sUV^\triangleright-\varPhi_0^\triangleright,\breve\varPhi_\sUV-\breve\varPhi_0)$ is small. To this end, we first construct the stationary effective force coefficients $F^{i,m,a}_{\sUV,\sIR}$, $F^{i,m,a}_{0,\sIR}$. We express $\varPhi_\sUV^\triangleright$, $\varPhi_0^\triangleright$ in terms of $F^{i,0,0}_{\sUV,\sIR}$, $F^{i,0,0}_{0,\sIR}$. Next, we construct another stationary effective force coefficients $\hat F^{i,m}_{\sUV,\sIR}$, $\hat F^{i,m}_{0,\sIR}$. After that we fix some $t\in[0,\infty)$ and $T\in(0,1)$ and construct the effective force coefficients $\check F^{i,m}_{\sUV,\sIR}$, $\check F^{i,m}_{0,\sIR}$ and $\tilde F^{i,m}_{\sUV,\sIR}$, $\tilde F^{i,m}_{0,\sIR}$ depending on the initial data at the time $\mathring x=t$. Subsequently, for sufficiently small $T\in(0,1)$ we express $\breve\varPhi_\sUV$, $\breve\varPhi_0$ restricted to the time interval $[t,t+T]$ in terms of $\tilde F^{i,0}_{\sUV,\sIR}$, $\tilde F^{i,0}_{0,\sIR}$ and the initial data at $\mathring x=t$. Finally, we construct the maximal solutions $\breve\varPhi_\sUV$, $\breve\varPhi_0$ as functions of the respective boundary data. The idea of the construction is explained in Sec.~\ref{sec:intro_initial_value}.

\begin{dfn}\label{dfn:force_ima}
Given a boundary data $f^{i,m,a}_\sUV$, $(i,m,a)\in\frI_0$, the effective force coefficients $F^{i,m,a}_{\sUV,\sIR}$, $\sIR\in[0,1]$, $(i,m,a)\in\frI_0$, are defined recursively by the flow equation~\eqref{eq:flow_deterministic_i_m_a} together with the conditions 
\begin{equation}
 F^{i,m,a}_{\sUV,\sIR}=0~~\textrm{if}~~m>i m_\flat,
 \qquad
 F^{0,0,0}_{\sUV,0}=\varXi_\sUV,
 \qquad
 F^{i,m,a}_{\sUV,0}=F_{\sUV}^{i,m,a},\quad (i,m,a)\in\frI,
\end{equation}
where for $(i,m,a)\in\frI$ we set
\begin{equation}\label{eq:determjnistic_force_coefficients}
 F^{i,m,a}_\sUV
 :=
 \sum_{b\in\frM^m}\frac{1}{b!}
 \partial^b \fL^m(f^{i,m,a+b}_\sUV)
\end{equation}
and the map $\fL^m$ was introduced in Def.~\ref{dfn:map_L_m}. We also set $F^{i,m}_{\sUV,\sIR}:=F^{i,m,0}_{\sUV,\sIR}$ and $f^{i,m,a}_{\sUV,\sIR}:=\fI F^{i,m,a}_{\sUV,\sIR}$, where the map $\fI$ was introduced in Def.~\ref{dfn:map_I}. Finally, we define a formal power series
\begin{equation}\label{eq:effective_force_determ}
 \langle F_{\sUV,\sIR}[\varphi],\psi\rangle
 :=\sum_{i=0}^\infty \sum_{m=0}^\infty \lambda^i\,\langle F^{i,m}_{\sUV,\sIR},\psi\otimes\varphi^{\otimes m}\rangle,
\end{equation}
where $\psi,\varphi\in C^\infty_\rc(\bM)$ are arbitrary. We omit $\sIR$ if $\sIR=0$.
\end{dfn}

\begin{rem}
By the properties of the maps $\fL^m$ and $\fI$ it holds $f^{i,m,a}_{\sUV,0}=f^{i,m,a}_{\sUV\phantom{0}}$. Furthermore, the terms of the sum on the RHS of Eq.~\eqref{eq:determjnistic_force_coefficients} vanish identically unless $a,b\in\bar\frM^m_\sigma$ by the assumptions about~$f_\sUV^{i,m,a}$.
\end{rem}

\begin{lem}\label{lem:eff_force_basic_properties_deterministic}
For every $\sIR\in[0,1]$ and $(i,m,a)\in\frI_0$ the effective force coefficients $F^{i,m,a}_{\sUV,\sIR}\in\cD^{m;0}$ are uniquely defined and satisfy the conditions $F^{0,0,0}_{\sUV,\sIR}=\varXi_{\sUV}$ and $F^{i,m,a}_{\sUV,\sIR}=\cX^{m,a}F^{i,m,0}_{\sUV,\sIR}$. Moreover, for any $\mathtt{w}\in C^\infty_\rc(\bR)$ the function 
\begin{equation}
 [0,1]\ni\sIR\mapsto (\delta_\bM\otimes J^{\otimes m})\ast \mathtt{w} F^{i,m,a}_{\sUV,\sIR}\in\cV^m
\end{equation} 
is continuous in the whole domain and differentiable in the domain $\sIR\in(0,1]$. 
\end{lem}

\begin{rem}
The above qualitative lemma is a special case of Lemma~\ref{lem:eff_force_basic_properties}. 
\end{rem}

\begin{rem}
Recall that for $R>0$ and $i,m\in\bN_0$ we set $R^{i,m}=R^{1+i m_\flat-m}$. 
\end{rem}

\begin{rem}\label{rem:varepsilon_bound_deterministic}
Throughout this section we fix some arbitrary $\varepsilon\in(0,\varepsilon_\sigma\wedge\varepsilon_\diamond)$, where $\varepsilon_\sigma$ and $\varepsilon_\diamond$ were introduced in Def.~\ref{dfn:varepsilon_sigma} and~Def.~\ref{dfn:varrho}.
\end{rem}

\begin{ass}\label{ass:deterministic}
Let $\mathtt{u}\in C^\infty_\rc(\bR)$, $R>1$, $\delta\in(0,1]$, $\oo\in\bN_0$. We assume that \mbox{$\sUV\in(0,1]$}, a boundary data $f^{i,m,a}_\sUV$, $(i,m,a)\in\frI_0$, and a boundary data $f^{i,m,a}_{0,\sIR}$, $\sIR\in(0,1]$, $(i,m,a)\in\frI_0^-$, are such that for all $\sIR\in(0,1]$ it holds:
\begin{enumerate}
\item[(A)] $\|K_\sIR^{\ast\oo}\ast \mathtt{u} f^{i,m,a}_{0,\sIR}\|_\cV \leq R^{i,m}\,[\sIR]^{\varrho_\varepsilon(i,m,a)}$,~~ $(i,m,a)\in\frI^-_0$,
\item[(B)] $\|\mathtt{u} f_\sUV^{i,m,a}\|_\cV\leq \delta R^{i,m}\,[\sUV]^{\varrho_\varepsilon(i,m,a)}$,~~ $(i,m,a)\in\frI_0$,
\item[(C)] $\|K^{\ast\oo}_\sIR\ast \mathtt{u}(f^{i,m,a}_{\sUV,\sIR}-f^{i,m,a}_{0,\sIR})\|_\cV\leq\delta R^{i,m}\,[\sIR]^{\varrho_\varepsilon(i,m,a)}$,~~ $(i,m,a)\in\frI^-_0$,
\end{enumerate}
where the effective force coefficients $f^{i,m,a}_{\sUV,\sIR}$, $(i,m,a)\in\frI_0$, $\sIR\in[0,1]$, are constructed from the boundary data $f^{i,m,a}_{\sUV}$, $(i,m,a)\in\frI_0$, as explained in Def.~\ref{dfn:force_ima}.
\end{ass}

\begin{rem}
At this stage the function $\mathtt{u}\in C^\infty_\rc(\bR)$ can be arbitrary. As we will see, in order to construct $(\varPhi_\sUV^\triangleright,\breve\varPhi_\sUV)$ and $(\varPhi_0^\triangleright,\breve\varPhi_0)$ up to some time $\Tmax$ one has to choose $\mathtt{u}$ such that $\mathtt{u}=1$ on some neighborhood of $[-i_\dagger(i_\dagger+1),6i_\flat+\Tmax]$, where $i_\dagger=i_\flat+i_\triangleright m_\flat$ and $i_\triangleright\in\bN_+$ is the largest integer such that $\varrho(i_\triangleright)+\sigma\leq0$. Recall also the useful property stated in Remark~\ref{rem:weight}.
\end{rem}

\begin{ass}\label{ass:deterministic_initial}
Let $R>1$, $\delta\in(0,1]$, $\oo\in\bN_0$. We assume that $\sUV\in(0,1]$, $\phi_\sUV^\vartriangle\in\cC^\gamma(\bT)$ and $\phi_0\in\sD'(\bT)$ are such that for all $\sIR\in(0,1]$ it holds:
\begin{enumerate}
\item[(A)]
$\|\bar K_\sIR^{\ast\oo}\ast \phi^\vartriangle_0\|_{L^\infty(\bT)}\leq R\,[\sIR]^{\beta_\varepsilon}$,
\item[(B)]
$\|\fR_\sUV \phi^\vartriangle_\sUV\|_{L^\infty(\bT)}\leq \delta R\, [\sUV]^{\beta_\varepsilon}$,
\item[(C)]
$\|\bar K_\sIR^{\ast\oo}\ast(\phi^\vartriangle_\sUV-\phi^\vartriangle_0)\|_{L^\infty(\bT)}\leq \delta R\,[\sIR]^{\beta_\varepsilon}$.
\end{enumerate}
\end{ass}

\begin{rem}
Recall that $\beta\equiv\beta_\varepsilon=-\dim(\phi)-\varepsilon$, $\dim(\phi)$ was introduced in Def.~\ref{dfn:dim_phi} and $\bar K_\sIR$, $\fR_\sUV$ were defined in Sec.~\ref{sec:kernels}.
\end{rem}

\begin{dfn}
Let $\mathtt{u}\in C^\infty_\rc(\bR)$ and $R>1$.
We say that a boundary data $\phi_0^\vartriangle$, $f^{i,m,a}_{0,\sIR}$, $\sIR\in(0,1]$, $(i,m,a)\in\frI_0^-$ for $(\varPhi_0^\triangleright,\breve\varPhi_0)$ is admissible iff for every $\delta\in(0,1]$ there exist $\sUV\in(0,1]$ and a boundary data $\phi_\sUV^\vartriangle$, $f^{i,m,a}_\sUV$, \mbox{$(i,m,a)\in\frI_0$} for $(\varPhi_\sUV^\triangleright,\breve\varPhi_\sUV)$ such that all of the conditions specified in Assumptions~\ref{ass:deterministic} and~\ref{ass:deterministic_initial} are satisfied.
\end{dfn}

\begin{rem}
We refrain from giving a characterization of the set of admissible boundary data for $(\varPhi_0^\triangleright,\breve\varPhi_0)$. Let us only note that $f^{i,m,a}_{0,\sIR}$, $\sIR\in(0,1]$, $(i,m,a)\in\frI_0$, satisfy certain non-linear constraints which are consequences of the fact that the coefficients $f^{i,m,a}_{\sUV,\sIR}$, $(i,m,a)\in\frI_0$, $\sIR\in[0,1]$, appearing in Assumption~\ref{ass:deterministic} are constructed from the boundary data $f^{i,m,a}_{\sUV}$, $(i,m,a)\in\frI_0$, with the use of the flow equation. Let us also remark that the set of admissible $\phi_0^\vartriangle\in\sD'(\bT)$ coincides with a ball in the Besov space $C^{\beta}(\bT)$.
\end{rem}

In Sec.~\ref{sec:stat_deterministic} we construct and bound the stationary effective force coefficients $F^{i,m,a}_{\sUV,\sIR}$, $F^{i,m,a}_{0,\sIR}$. Subsequently, in Sec.~\ref{sec:schauder} we establish a Schauder-type estimates. In Sec.~\ref{sec:d-p} we express the rough part of the solution $\varPhi_\sUV^\triangleright$, $\varPhi_0^\triangleright$ in terms of the effective force coefficients $F^{i,0,0}_{\sUV,\sIR}$, $F^{i,0,0}_{0,\sIR}$. We also construct another stationary effective force coefficients $\hat F^{i,m}_{\sUV,\sIR}$, $\hat F^{i,m}_{0,\sIR}$, which as argued in Sec.~\ref{sec:intro_initial_value}, are useful to deal with the effect of the initial condition. In Sec.~\ref{sec:effect_initial} we introduce the effective force coefficients $\check F^{i,m}_{\sUV,\sIR}$, $\check F^{i,m}_{0,\sIR}$ and $\tilde F^{i,m}_{\sUV,\sIR}$, $\tilde F^{i,m}_{0,\sIR}$ depending on the regular part of the initial data $\phi_\uv^\vartriangle$. In Sec.~\ref{sec:convergence_PT} we prove that the formal power series defining the effective force $\tilde F_{\sUV,\sIR}[\varphi]$ converges absolutely in sufficiently short time intervals. We also establish several properties of the effective force $\tilde F_{\sUV,\sIR}[\varphi]$. In Sec.~\ref{sec:local_solution} we construct the solution of the SPDE in a short random time interval by expressing it in terms of the effective force $\tilde F_{\sUV,\sIR}[\varphi]$. Finally, in Sec.~\ref{sec:maximal_solution} we show how to construct the maximal solution up to a possible finite-time blowup.

\subsection{Bounds at stationarity}\label{sec:stat_deterministic}

\begin{thm}\label{thm:deterministic_taylor}
Fix some $\mathtt{u}\in C^\infty_\rc(\bR)$, $R>1$, $\delta\in(0,1]$ and an admissible boundary data for $(\varPhi_0^\triangleright,\breve\varPhi_0)$. Let $\sUV\in(0,\delta^{1/\varepsilon}]$ and a boundary data for $(\varPhi_\sUV^\triangleright,\breve\varPhi_\sUV)$ be arbitrary such that all of the conditions specified in Assumption~\ref{ass:deterministic} are satisfied. There exists $\oo\in\bN_0$ and a universal family of differentiable functions 
\begin{equation}\label{eq:deterministic_taylor_family}
 (0,1]\ni \sIR \mapsto F^{i,m,a}_{0,\sIR}\in\cD^m, \qquad (i,m,a)\in\frI_0,
\end{equation}
satisfying the flow equation~\eqref{eq:flow_deterministic_i_m_a} such that it holds
\begin{enumerate}
\item[(A)] $\|K^{m;\oo}_{0,\sIR} \ast \mathtt{w} F^{i,m,a}_{0,\sIR}\|_{\cV^m}
 \lesssim R^{i,m}\, [\sIR]^{\varrho_\varepsilon(i,m,a)}$,~~ $\sIR\in(0,1]$,
\item[(B)] $\|K^{m;\oo}_{\sUV,\sIR} \ast \mathtt{w} F^{i,m,a}_{\sUV,\sIR}\|_{\cV^m}
 \lesssim R^{i,m}\, [\sUV\vee\sIR]^{\varrho_\varepsilon(i,m,a)}$,~~ $\sIR\in[0,1]$,
\item[(C)] $\|K^{m;\oo}_{\sUV,\sIR} \ast \mathtt{w} (F^{i,m,a}_{\sUV,\sIR}-F^{i,m,a}_{0,\sIR})\|_{\cV^m} 
 \lesssim \delta R^{i,m}\, [\sIR]^{\varrho_\varepsilon(i,m,a)-\varepsilon}$,~~ $\sIR\in[\sUV,1]$,
\end{enumerate}
for all $(i,m,a)\in\frI_0$ and all $\mathtt{w}\in C^\infty_\rc(\bR)$ such that $\mathtt{u}=1$ on some neighborhood of $\supp\,\mathtt{w}-[0,i(i+1)]$. The constants of proportionality in the above bounds depend only on $(i,m,a)\in\frI_0$ and \mbox{$\mathtt{w},\mathtt{u}\in C^\infty_\rc(\bR)$} and are otherwise universal.
\end{thm}
\begin{rem}
Note that we assume that $\sUV\in(0,\delta^{1/\varepsilon}]$. Under this assumption $[\sUV]^\varepsilon\leq\delta$ and the bounds (A), (B), (C) imply that
\begin{enumerate}
\item[($\tilde{\mathrm{C}}$)]
$\|K^{m;\oo}_{\sUV,\sIR} \ast \mathtt{w}(F^{i,m,a}_{\sUV,\sIR}-F^{i,m,a}_{0,\sIR})\|_{\cV^m} 
\lesssim \delta R^{i,m}\, ([\sIR]^{\varrho_\varepsilon(i,m,a)-\varepsilon}\vee
[\sUV\vee\sIR]^{\varrho_\varepsilon(i,m,a)-\varepsilon})$
\end{enumerate}
for all $\sIR\in(0,1]$.
\end{rem}
\begin{rem}\label{rem:varrho_R}
In the proof we use the equalities: $R^{i,m} = R^{j,k+1}\,R^{i-j,m-k}$ and
\begin{equation}\label{eq:varrho_equality_flow}
 \varrho_\varepsilon(i,m,a) = \varrho_\varepsilon(j,k+1,b) + [c] + \varrho_\varepsilon(i-j,m-k,d) + \sigma
\end{equation}
valid for $\varepsilon\in(0,1]$ and $i\in\bN_+$, $m\in\bN_0$, \mbox{$j\in\{1,\ldots,i\}$}, $k\in\{0,\ldots,m\}$, \mbox{$a\in\frM^m$}, \mbox{$b\in\frM^{k+1}$}, $c\in\frM$, $d\in\frM^{m-k}$ such that the multi-indices $a,b,c,d$ satisfy the condition specified below Eq.~\eqref{eq:poly_binom_notation}. The first of the above identities follows trivially from the definition of $R^{i,m}$. The second one is an immediate consequence of Def.~\ref{dfn:varrho} of $\varrho_\varepsilon(i,m,a)$. 
\end{rem}
\begin{proof}
We first note that for $m>i m_\flat$ we have $F^{i,m,a}_{\sUV,\sIR}=0$. For $\sUV=0$ the above equality is our definition of the LHS. As a result, for $m>i m_\flat$ the bounds (A), (B), (C) hold trivially. Next, we observe that $F^{0,0,0}_{\sUV,\sIR}=f^{0,0,0}_{\sUV,\sIR}$ for $\sIR\in[0,1]$. For $\sUV=0$ and $\sIR\in(0,1]$ this is our definition of the LHS. Recall also that for $m=0$ it holds $\|v\|_{\cV^m}=\|v\|_\cV$ and $K^{m;\oo}_{\sUV,\sIR}=K^{\ast\oo}_\sIR$. As a result, the bounds (A), (B), (C) with $i=0$ are implied by Remark~\ref{rem:weight} and the bounds stated in Assumption~\ref{ass:deterministic}. 

Fix $i_\circ\in\bN_+$, $m_\circ\in\bN_0$, $\ooo\in\bN_0$ and assume that the bounds (A), (B), (C) hold for $\oo=\ooo$ and all $(i,m,a)\in\frI_0$ such that either $i<i_\circ$, or $i=i_\circ$ and $m>m_\circ$. We shall show that there exists $\oooo\in\bN_0$ such that the bounds (A), (B), (C) hold for $\oo=\oooo$ and all $(i,m,a)\in\frI$ such that $i=i_\circ$ and $m=m_\circ$. Moreover, we will prove that if $i_\circ>i_\diamond$, then one can choose $\oooo=\ooo$. It is enough to verify the statement formulated above to prove the theorem since there are only finitely many $i_\circ\in\bN_+$ and $m_\circ\in\bN_0$ such that $i_\circ\leq i_\diamond$ and $m_\circ\leq i_\circ\,m_\flat$ and the bounds (A), (B), (C) hold for all $\oo\geq\ooo$ if they hold for $\oo=\ooo$. For convenience, we recall that $i_\diamond$ is the largest integer such that $\varrho(i_\diamond,0,0)\leq0$ and $(i,m,a)\in\frI$ iff $i\in\bN_+$, $m\in\bN_0$ and $a\in\frM^m$ is such that $|\bar a|<\sigma_\diamond$, where $\sigma_\diamond=\sigma$ if $\sigma\in2\bN_+$ and $\sigma_\diamond=\ceil{\sigma}-1$ otherwise.

Let us turn to the proof of the inductive step. Fix any $\mathtt{w}\in C^\infty_\rc(\bR)$ such that $\mathtt{u}=1$ on some neighborhood of $\supp\,\mathtt{w}-[0,i(i+1)]$. We first observe that by the flow equation~\eqref{eq:flow_deterministic_i_m_a} and the support properties of the effective force coefficients stated in Lemma~\ref{lem:support} it holds
\begin{multline}\label{eq:flow_deterministic_w_i_m_a}
 \mathtt{w}\partial_\sIR^{\phantom{i}} F^{i,m,a}_{\sUV,\sIR}
 \\
 =
 -\frac{1}{m!}
 \sum_{\pi\in\cP_m}\sum_{j=1}^i\sum_{k=0}^m
 \sum_{b+c+d=\pi(a)}\frac{(k+1)a!}{b!c!d!}\,
 \fY_\pi\fB\big(\dot G^c_\sIR,\mathtt{w} F^{j,k+1,b}_{\sUV,\sIR},\mathtt{v} F^{i-j,m-k,d}_{\sUV,\sIR}\big),
\end{multline}
where $\mathtt{v}\in C^\infty_\rc(\bR)$ is arbitrary such that $\mathtt{v}=1$ on some neighborhood of \mbox{$\supp\,\mathtt{w} - [0,2i]$} and $\mathtt{u}=1$ on some neighborhood of $\supp\,\mathtt{v}-[0,(i-1)i]$. We define $\partial_\sIR^{\phantom{i}} F^{i,m,a}_{0,\sIR}$ by the RHS of the flow equation~\eqref{eq:flow_deterministic_i_m_a} with $\sUV=0$. Hence, Eq.~\eqref{eq:flow_deterministic_w_i_m_a} remains valid for $\sUV=0$. Using Remark~\ref{rem:fB1_bound} we obtain
\begin{multline}
 \|K_{\sUV,\sIR}^{m;\oo}\ast \mathtt{w}\partial_\sIR F_{\sUV,\sIR}^{i,m,a}\|_{\cV^m} 
 \leq
 \sum_{\pi\in\cP_m}\sum_{j=1}^i\sum_{k=0}^m\sum_{b+c+d=\pi(a)}
 \frac{(k+1)a!}{b!c!d!}\,\|\fR_\sUV^{\phantom{\oo}}\fP_\sIR^{2\oo}\dot G^c_\sIR\|_\cK
 \\
 \times
 \|K_{\sUV,\sIR}^{k+1;\oo}\ast \mathtt{w} F_{\sUV,\sIR}^{j,k+1,b}\|_{\cV^{k+1}} 
 \,
 \|K_{\sUV,\sIR}^{m-k;\oo}\ast \mathtt{v} F_{\sUV,\sIR}^{i-j,m-k,d}\|_{\cV^{m-k}}
\end{multline}
and
\begin{multline}
 \|K_{\sUV,\sIR}^{m;\oo}\ast \mathtt{w}\partial_\sIR (F_{\sUV,\sIR}^{i,m,a}-F_{0,\sIR}^{i,m,a})\|_{\cV^m} 
 \leq 
 \sum_{\pi\in\cP_m}\sum_{j=1}^i\sum_{k=0}^m\sum_{b+c+d=\pi(a)}
 \frac{(k+1)a!}{b!c!d!}
 \\
 \times\|\fR_\sUV^{\phantom{\oo}}\fP_\sIR^{2\oo}\dot G^c_\sIR\|_\cK\Big(
 \|K_{\sUV,\sIR}^{k+1;\oo}\ast \mathtt{w}(F_{\sUV,\sIR}^{j,k+1,b}-F_{\sUV,\sIR}^{j,k+1,b})\|_{\cV^{k+1}} 
 \,
 \|K_{\sUV,\sIR}^{m-k;\oo}\ast \mathtt{v} F_{\sUV,\sIR}^{i-j,m-k,d}\|_{\cV^{m-k}}
 \\
 +
 \|K_{0,\sIR}^{k+1;\oo}\ast \mathtt{w} F_{0,\sIR}^{j,k+1,b}\|_{\cV^{k+1}} 
 \,
 \|K_{\sUV,\sIR}^{m-k;\oo}\ast \mathtt{v}(F_{\sUV,\sIR}^{i-j,m-k,d}-F_{0,\sIR}^{i-j,m-k,d})\|_{\cV^{m-k}}
 \Big).
\end{multline}
To get the last bound we used the fact that the map $\fB$ is multi-linear. Let us also recall that by Lemma~\ref{lem:kernel_dot_G} we have
\begin{equation}
 \|\fR_\sUV^{\phantom{\oo}}\fP_\sIR^{2\oo}\dot G_\sIR^{c}\|_\cK \lesssim
 [\sIR]^{\varepsilon-\sigma}
 [\sUV\vee\sIR]^{\sigma-\varepsilon+[c]},
\end{equation}
for $c\in\frM_\sigma$. We note that the above bounds remain valid for $\sUV=0$. Using the above bounds, the induction assumption and Remark~\ref{rem:varrho_R} we get
\begin{enumerate}
 \item[(A${}_1$)] $\|K^{m;\oo}_{0,\sIR} \ast \mathtt{w} \partial_\sIR F^{i,m,a}_{0,\sIR}\|_{\cV^m}
 \lesssim R^{i,m}\,[\sIR]^{\varrho_\varepsilon(i,m,a)-\sigma}$,~~$\sIR\in(0,1]$,
 \item[(B${}_1$)] $\|K^{m;\oo}_{\sUV,\sIR} \ast \mathtt{w} \partial_\sIR F^{i,m,a}_{\sUV,\sIR}\|_{\cV^m}
 \lesssim R^{i,m}\, [\sIR]^{\varepsilon-\sigma}[\sUV\vee\sIR]^{\varrho_\varepsilon(i,m,a)-\varepsilon}$,~~$\sIR\in(0,1]$,
 \item[(C${}_1$)] $\|K^{m;\oo}_{\sUV,\sIR} \ast \mathtt{w} \partial_\sIR (F^{i,m,a}_{\sUV,\sIR}-F^{i,m,a}_{0,\sIR})\|_{\cV^m} 
 \lesssim \delta R^{i,m}\, [\sIR]^{\varrho_\varepsilon(i,m,a)-\varepsilon-\sigma}$,~~$\sIR\in[\sUV,1]$.
\end{enumerate}
We observe that if $(i,m,a)\in\frI^+$, then $\varrho_\varepsilon(i,m,a)-\varepsilon\geq0$ by Remark~\ref{rem:epsilon}. As a result, for $(i,m,a)\in\frI^+$ the above bounds imply 
\begin{enumerate}
 \item[($\tilde{\mathrm{C}}_1$)] $\|K^{m;\oo}_{\sUV,\sIR} \ast \mathtt{w} \partial_\sIR (F^{i,m,a}_{\sUV,\sIR}-F^{i,m,a}_{0,\sIR})\|_{\cV^m} 
 \lesssim \delta R^{i,m}\, [\sIR]^{\varepsilon-\sigma}[\sUV\vee\sIR]^{\varrho_\varepsilon(i,m,a)-2\varepsilon}$,~~$\sIR\in(0,1]$.
\end{enumerate}

We note that by Assumption~\ref{ass:deterministic}, the equality
\begin{equation}
 F_{\sUV}^{i,m,a}=\sum_{b\in\bar\frM_\sigma^m} \frac{1}{b!}\,\partial^b \fL^m(f^{i,m,a+b}_{\sUV})
\end{equation}
and Lemma~\ref{lem:bound_map_L_m} the bound (B) is true for $\sIR=0$. Moreover, it holds
\begin{equation}
 F^{i,m,a}_{\sUV,\sIR}
 =
 F^{i,m,a}_{\sUV,0}
 +
 \int_0^\sIR \partial_\uIR F^{i,m,a}_{\sUV,\uIR}\,\rd\uIR.
\end{equation}
For $(i,m,a)\in\frI^+$ we define $F^{i,m,a}_{0,\sIR}$, $\sIR\in(0,1]$ by the RHS of the above equality with $\sUV=0$. Note that by Assumption~\ref{ass:renormalization} for $(i,m,a)\in\frI^+$ we have $F^{i,m,a}_{0,0}:=0$. As a result, if $\sIR\in(0,\sUV]$ or $(i,m,a)\in\frI^+$, then the bounds (A), (B), (C) follow from the bounds (A${}_1$), (B${}_1$), (C${}_1$) and Lemma~\ref{lem:integration} applied with $n=1$, $\epsilon=\varepsilon$ and $\rho=\varrho_\varepsilon(i,m,a)$ or $\rho=\varrho_\varepsilon(i,m,a)-\varepsilon$.

Recall that for $(i,m,a)\in\frI_0=\frI^-_0\cup \frI^+$ and $\sIR\in(0,1]$ it holds $f^{i,m,a}_{\sUV,\sIR}=\fI F^{i,m,a}_{\sUV,\sIR}$. For $(i,m,a)\in\frI^+$ and $\sIR\in[0,1]$ we set $f^{i,m,a}_{0,\sIR}:=\fI F^{i,m,a}_{0,\sIR}$. Note that $f^{i,m,a}_{0,\sIR}$, $(i,m,a)\in\frI^-_0$, is part of the boundary data for $(\varPhi_0^\triangleright,\breve\varPhi_0)$. By Lemma~\ref{lem:map_I} and the results established so far we obtain
\begin{enumerate}
\item[(A${}_2$)] $\|K_\sIR^{\ast\oo}\ast \mathtt{w} f^{i,m,a}_{0,\sIR}\|_\cV \leq R^{i,m}\,[\sIR]^{\varrho_\varepsilon(i,m,a)}$,~~$\sIR\in(0,1]$,
\item[(B${}_2$)] $\|K_\sIR^{\ast\oo}\ast\mathtt{w} f_\sUV^{i,m,a}\|_\cV\leq R^{i,m}\,[\sUV\vee\sIR]^{\varrho_\varepsilon(i,m,a)}$,~~$\sIR\in[0,1]$,
\item[(C${}_2$)] $\|K^{\ast\oo}_\sIR\ast \mathtt{w}(f^{i,m,a}_{\sUV,\sIR}-f^{i,m,a}_{0,\sIR})\|_\cV\leq\delta R^{i,m}\,[\sIR]^{\varrho_\varepsilon(i,m,a)-\varepsilon}$,~~$\sIR\in[\sUV,1]$,
\end{enumerate}
for all $(i,m,a)\in\frI_0$. In order to prove the bounds (A), (B), (C) for $\sIR\in(\sUV,1]$ and $(i,m,a)\in\frI^-$ we shall use the bounds (A${}_2$), (B${}_2$), (C${}_2$), the validity of the bounds (A), (B), (C) for $(i,m,a)\in\frI^+$ established earlier and the identities
\begin{equation}
 F^{i,0,0}_{\sUV,\sIR}=f^{i,0,0}_{\sUV,\sIR},
 \qquad
 F^{i,m,a}_{\sUV,\sIR} = \fX^a_\ro(f^{i,m,b}_{\sUV,\sIR},F^{i,m,b}_{\sUV,\sIR}),\quad m\in\bN_+,
\end{equation}
where the map $\fX^a_\ro$ was introduced in Def.~\ref{dfn:map_X}. For $(i,m,a)\in\frI^-$ we define $F^{i,m,a}_{0,\sIR}$ by the RHS of the above equalities with $\sUV=0$. If $i=1$, then we set $\ro=\ceil{\sigma}$ and note that $F^{i,m,b}_{\sUV,\sIR}=0$ for all $|b|=\ro$ by locality. If $i>1$, then we set $\ro=\sigma_\diamond$ and note that $(i,m,b)\in\frI^+$ for any $|b|=\ro$ by Assumption~\ref{ass:infrared}. We also observe that by Def.~\eqref{dfn:map_X} of the map $\fX^a_\ro$ it holds
\begin{equation}
 F^{i,m,a}_{\sUV,\sIR}-F^{i,m,a}_{0,\sIR} = 
 \fX^a_\ro(f^{i,m,b}_{\sUV,\sIR}-f^{i,m,b}_{0,\sIR},F^{i,m,b}_{\sUV,\sIR}-F^{i,m,b}_{0,\sIR}).
\end{equation}
The bounds (A), (B), (C) follow now from Theorem~\ref{thm:taylor_bounds} applied with the constant $C$ such that $C\lesssim [\sIR]^{\varrho_\varepsilon(i,m)}$ or $C\lesssim [\sIR]^{\varrho_\varepsilon(i,m)-\varepsilon}$.

The differentiability of the functions~\eqref{eq:deterministic_taylor_family} follows from the construction. In order to prove that the family of functions~\eqref{eq:deterministic_taylor_family} satisfies the flow equation~\eqref{eq:flow_deterministic_i_m_a} we use the bounds (A), (B), (C), the assumed admissibility of the family~\eqref{eq:enhanced_data} and the fact that the family $F^{i,m,a}_{\sUV,\sIR}$, $(i,m,a)\in\frI_0$, satisfies the flow equation.
\end{proof}

\subsection{Schauder estimates}\label{sec:schauder}

\begin{dfn}
For $\sIR\in[0,1]$ we set $\hat G_\sIR:=G_\sIR-G_1$ and $\hat G = G-G_1$, where $G_\sIR$ is the fractional heat kernel with the UV cutoff $\sIR$ introduced in Def.~\ref{dfn:propagator_G}.
\end{dfn}

\begin{lem}\label{lem:schauder}
Fix $\oo\in\bN_0$ and $\rho>0$. There exists a constant $C$ such that if for some $f\in\cV=C_\rb(\bH)$, $R>0$, $\sUV\in[0,1]$, $\uIR\in(0,1]$ and all $\sIR\in[\uIR,1]$ it holds
\begin{equation}
 \|K^{\ast\oo}_\sIR\ast f\|_\cV \leq R \, [\sUV\vee\sIR]^{-\sigma-\rho},
\end{equation}
then it holds:
\begin{enumerate}
\item[(A)] $\|\fR_\sUV^{\phantom{\oo}}\fP_\uIR^\oo\hat G_\uIR \ast f\|_\cV \leq C R \,[\sUV\vee\uIR]^{-\rho}$,

\item[(B)] $\|K_\uIR^{\ast\oo}\ast \fR_\sUV^{\phantom{\oo}}\hat G \ast f\|_\cV \leq C R \,[\sUV\vee\uIR]^{-\rho}$.
\end{enumerate}
\end{lem}
\begin{proof}
We first note that the equalities
\begin{equation}
 \hat G_\uIR = G_\uIR-G_1 =- \int_\uIR^1 \dot G_\sIR\,\rd\sIR,
 \qquad
 \hat G = \hat G_\uIR - \int_0^\uIR \dot G_\sIR\,\rd\sIR,
\end{equation}
imply the following bounds
\begin{multline}
 \|\fR_\sUV^{\phantom{\oo}}\fP_\uIR^\oo \hat G_\uIR \ast  f\|_\cV 
 \leq  \int_\uIR^1 \|\fR_\sUV^{\phantom{\oo}}\fP_\sIR^{2\oo}\dot G_\sIR\ast \fP_\uIR^{\oo} K_\sIR^{\ast2\oo}\ast f\|_\cV
 \,\rd\sIR
 \\
 \leq
 \int_\uIR^1 \|\fR_\sUV^{\phantom{\oo}}\fP_\sIR^{2\oo}\dot G_\sIR\|_\cK\,
 \|\fP_\uIR^\oo K_\sIR^{\ast \oo}\|_\cK\,
 \|K_\sIR^{\ast\oo}\ast f\|_\cV \,\rd\sIR
\end{multline}
and 
\begin{multline}
 \|K^{\ast\oo}_\uIR\ast \fR_\sUV^{\phantom{\oo}}\hat G \ast  f\|_\cV 
 \leq  
 \|K^{\ast\oo}_\uIR\ast \fR_\sUV^{\phantom{\oo}}\hat G_\uIR \ast  f\|_\cV 
 + \int_0^\uIR \|K_\uIR^{\ast\oo}\ast \fR_\sUV^{\phantom{\oo}}\dot G_\sIR \ast  f\|_\cV
 \,\rd\sIR
 \\
 \leq \|\fR_\sUV^{\phantom{\oo}}\hat G_\uIR \ast  f\|_\cV 
 + 
 \|K_\uIR^{\ast\oo} \ast  f\|_\cV \int_0^\uIR \|\fR_\sUV^{\phantom{\oo}} \dot G_\sIR^{\phantom{\oo}}\|_\cK\, 
 \,\rd\sIR.
\end{multline}
By Lemma~\ref{lem:kernel_u_v} $\|\fP_\uIR^\oo K_\sIR^{\ast \oo}\|_\cK\leq 1$ for $\sIR\geq\uIR$. Moreover, by Lemma~\ref{lem:kernel_dot_G} it holds
\begin{equation}
 \|\fR_\sUV^{\phantom{\oo}} \dot G_\sIR^{\phantom{\oo}}\|_\cK\leq c\,[\sUV\vee\sIR]^{\sigma-\varepsilon}[\sIR]^{\varepsilon-\sigma},
 \qquad
 \|\fR_\sUV^{\phantom{\oo}}\fP^{2\oo}_\sIR \dot G_\sIR^{\phantom{\oo}}\|_\cK\leq c\,[\sUV\vee\sIR]^{\sigma-\varepsilon}[\sIR]^{\varepsilon-\sigma}
\end{equation}
for some $c>0$ and all $\sUV\in[0,1]$, $\sIR\in(0,1]$. The lemma follows now from the above bounds and the estimates
\begin{equation}
 \int_\uIR^1 [\sUV\vee\sIR]^{-\rho-\varepsilon}\,[\sIR]^{\varepsilon-\sigma}\,\rd\sIR \leq \sigma (\rho^{-1}+\varepsilon^{-1})\,[\sUV\vee\uIR]^{-\rho}
\end{equation}
and
\begin{equation}
 [\sUV\vee\uIR]^{-\rho-\sigma} \int_0^\uIR [\sUV\vee\sIR]^{\sigma-\varepsilon}\,[\sIR]^{\varepsilon-\sigma}\,\rd\sIR \leq \sigma (\sigma^{-1}+\varepsilon^{-1})\,[\sUV\vee\uIR]^{-\rho}.
\end{equation}
\end{proof}

\subsection{Da Prato-Debussche trick}\label{sec:d-p}

In this section we introduce the effective force $\hat F_{\sUV,\sIR}[\varphi]$ which, as argued in Sec.~\ref{sec:intro_initial_value}, will be useful to deal with the effect of the initial condition. The definition $\hat F_{\sUV,\sIR}[\varphi]$ stated below is motivated by the so-called Da Prato-Debussche trick~\cite{daprato2003}.

\begin{dfn}\label{dfn:eff_force_hat}
Let $f^{i,m,a}_\sUV$, $(i,m,a)\in\frI_0$ be a boundary data, $F^{i,m}_{\sUV,\sIR}$, \mbox{$\sIR\in[0,1]$}, $i,m\in\bN_0$, be the corresponding effective force coefficients and $F_{\sUV,\sIR}[\varphi]$ be the corresponding effective force functional constructed as explained in Def.~\ref{dfn:force_ima}. We set $f^i_{\sUV,1}:=F^{i,0}_{\sUV,1}$ for $i\in\bN_0$ and $f^\triangleright_{\sUV}:=\sum_{i=0}^{i_\triangleright} \lambda^i f_{\sUV,1}^i$, $\varPhi_\sUV^\triangleright:=\hat G\ast f^\triangleright_{\sUV}$. The effective force coefficients $\hat F^{i,m}_{\sUV,\sIR}$, \mbox{$\sIR\in[0,1]$}, $i,m\in\bN_0$ are defined by the following relation between formal power series
\begin{equation}
 \langle \hat F_{\sUV,\sIR}[\varphi],\psi\rangle
 \equiv \sum_{i=0}^\infty \sum_{m=0}^\infty \lambda^i\,\langle \hat F^{i,m}_{\sUV,\sIR},\psi\otimes\varphi^{\otimes m}\rangle
 :=
 \langle F_{\sUV,\sIR}[\varphi+\hat G_\sIR\ast f^\triangleright_{\sUV}]- f^\triangleright_{\sUV},\psi\rangle
\end{equation}
where $\psi,\varphi\in C^\infty_\rc(\bM)$ are arbitrary. We omit $\sIR$ if $\sIR=0$. 
\end{dfn}

\begin{rem}
Recall that $i_\triangleright$ is the largest integer such that $\varrho(i_\triangleright)+\sigma\leq0$. This implies that $\varrho_\varepsilon(i)+\sigma<0$ for all $\varepsilon>0$ and $i\in\{0,\ldots,i_\triangleright\}$. 
\end{rem}

\begin{rem}
By Remark~\ref{rem:eff_force_shift} the functional $\hat F_{\sUV,\sIR}[\varphi]$ satisfies the flow equation~\eqref{eq:intro_flow_eq} in the sense of formal power series. Hence, the coefficients $\hat F^{i,m}_{\sUV,\sIR}$ satisfy the flow equation~\eqref{eq:flow_deterministic_i_m}. 
\end{rem}

\begin{rem}\label{rem:shift_identity}
Note that as argued in Remark~\ref{rem:intro_flow_solution} it holds $f_{\sUV,1} = F_{\sUV,\sIR}[\hat G_\sIR\ast f_{\sUV,1}]$ in the sense of formal power series, where $f_{\sUV,1}=\sum_{i=0}^\infty \lambda^i\,f_{\sUV,1}^i$. Consequently, $\hat F^{i,0}_{\sUV,\sIR}=0$ if $i\in\{0,\ldots,i_\triangleright\}$ and for any $i\in\bN_0$ we have
\begin{equation}\label{eq:recursive_f}
 \langle f^i_{\sUV,1},\psi\rangle 
 =
 \sum_{k\in\bN_0}
 \sum_{\substack{i_0,\ldots,i_m\in\bN_0\\i_0+\ldots+i_k=i}}
 \langle F^{i_0,k}_{\sUV,\sIR},\psi\otimes 
 \hat G_\sIR\ast f_{\sUV,1}^{i_1}\otimes\ldots\otimes \hat G_\sIR\ast f_{\sUV,1}^{i_k}\rangle.
\end{equation} 
As we will see below, the above identity allows to control $K_\sIR^{\ast\oo}\ast \mathtt{w} f^i_{\sUV,1}$ using the bounds established in Sec.~\ref{sec:stat_deterministic}. We also note that for $i\in\{1,\ldots,i_\triangleright\}$ and $m\in\bN_+$, or $i>i_\triangleright$ and $m\in\bN_0$ it holds
\begin{multline}
 \langle \hat F^{i,m}_{\sUV,\sIR},\psi\otimes\varphi^{\otimes m}\rangle 
 \\
 := \sum_{k\in\bN_0}\sum_{\substack{i_0,\ldots,i_k\in\bN_0\\i_0+\ldots+i_k=i\\i_1,\ldots,i_k\leq i_\triangleright}}\! \binom{m+k}{k} 
 \langle F^{i_0,k+m}_{\sUV,\sIR},\psi\otimes 
 \hat G_\sIR\ast f^{i_1}_{\sUV,1}\otimes\ldots\otimes \hat G_\sIR\ast f^{i_k}_{\sUV,1}\otimes\varphi^{\otimes m}\rangle,
\end{multline}
where $\psi,\varphi\in C^\infty_\rc(\bM)$ are arbitrary. The sums in the above formulas are finite because $F^{i,m}_{\sUV,\sIR}=0$ for all $m>im_\flat$. Note that $\hat F^{i,m}_{\sUV,\sIR}=0$ for all $m>im_\flat$ as well.
\end{rem}

For completeness we state without a proof the following simple qualitative lemma.

\begin{lem}\label{lem:hat_F_simple}
For every $\sIR\in[0,1]$ and $i,m\in\bN_0$ the effective force coefficients $\hat F^{i,m}_{\sUV,\sIR}\in\cD^{m;0}$ are uniquely defined. Moreover, for any $\mathtt{w}\in C^\infty_\rc(\bR)$ the function
\begin{equation}
 [0,1]\ni\sIR\mapsto (\delta_\bM\otimes J^{\otimes m})\ast \mathtt{w}\hat F^{i,m}_{\sUV,\sIR}\in\cV^m
\end{equation}
is continuous in the whole domain and differentiable in the domain $\sIR\in(0,1]$. 
\end{lem}

\begin{rem}\label{rem:R_rho_hat}
In the proof of the lemma and the theorem stated below we will use the following equalities: $R^{i,m} = R^{i_0,k+m} R^{i_1,0}\ldots R^{i_k,0}$ and
\begin{equation}\label{eq:varrho_equality_shift}
 \varrho_\varepsilon(i,m) = \varrho_\varepsilon(i_0,k+m) + \varrho_\varepsilon(i_1)+\sigma + \ldots + \varrho_\varepsilon(i_k)+\sigma
\end{equation}
valid for any $\varepsilon\in(0,1]$ and $i,m,k\in\bN_0$, $i_0,\ldots,i_k\in\bN_0$ such that $i_0+\ldots+i_k=i$. The first of the above identities follows from the definition $R^{i,m}=R^{1+i m_\flat-m}$. The second one is an immediate consequence of Def.~\ref{dfn:varrho} of $\varrho_\varepsilon(i,m)$.
\end{rem}

\begin{lem}\label{lem:deterministic_hat}
Under the assumptions of Theorem~\ref{thm:deterministic_taylor} there exists $\oo\in\bN_0$ and a universal family of distributions $f^i_{0,1}=F^{i,0}_{0,1}\in\cD$, $i\in\{0,\ldots i_\triangleright\}$, such that it holds
\begin{enumerate}
\item[(A${}_1$)]
$\|K_\sIR^{\ast\oo}\ast \mathtt{w} f^i_{0,1}\|_\cV\lesssim R^{i,0}\,[\sIR]^{\varrho_\varepsilon(i)}$,~~$\sIR\in(0,1]$,

\item[(B${}_1$)]
$\|K_\sIR^{\ast\oo}\ast \mathtt{w} f^i_{\sUV,1}\|_\cV\lesssim R^{i,0}\,[\sUV\vee\sIR]^{\varrho_\varepsilon(i)}$,~~$\sIR\in[0,1]$,

\item[(C${}_1$)]
$\|K_\sIR^{\ast\oo}\ast \mathtt{w}(f^i_{\sUV,1}-f^i_{0,1})\|_\cV\lesssim \delta R^{i,0}\,[\sIR]^{\varrho_\varepsilon(i)-\varepsilon}$,~~$\sIR\in(0,1]$,
\end{enumerate}
and
\begin{enumerate}
\item[(A${}_2$)] $\|\fP_\sIR^\oo\hat G_\sIR^{\phantom{\oo}}\ast \mathtt{w} f^i_{0,1}\|_\cV\lesssim R^{i,0}\, [\sIR]^{\varrho_\varepsilon(i)+\sigma}$,~~$\sIR\in(0,1]$,

\item[(B${}_2$)] $\|\fR_\sUV^{\phantom{\oo}}\fP_\sIR^\oo\hat G_\sIR^{\phantom{\oo}}\ast \mathtt{w} f^i_{\sUV,1}\|_\cV\lesssim R^{i,0}\, [\sUV\vee\sIR]^{\varrho_\varepsilon(i)+\sigma}$,~~$\sIR\in[0,1]$,

\item[(C${}_2$)] $\|\fR_\sUV^{\phantom{\oo}}\fP_\sIR^\oo\hat G_\sIR^{\phantom{\oo}}\ast \mathtt{w}(f^i_{\sUV,1}-f^i_{0,1})\|_\cV\lesssim \delta R^{i,0}\, [\sIR]^{\varrho_\varepsilon(i)+\sigma-\varepsilon}$,~~$\sIR\in(0,1]$,
\end{enumerate}
for all $i\in\{0,\ldots,i_\triangleright\}$ and all $\mathtt{w}\in C^\infty_\rc(\bR)$ satisfying the condition $\mathtt{u}=1$ on some neighborhood of \mbox{$\supp\,\mathtt{w}-[0,i(i+1)]$}. The constants of proportionality in the above bounds depend only on $i,m\in\bN_0$ and \mbox{$\mathtt{w},\mathtt{u}\in C^\infty_\rc(\bR)$} and are otherwise universal.
\end{lem}
\begin{proof}
We first note that by Lemma~\ref{lem:schauder}~(A) the bounds (A${}_1$), (B${}_1$), (C${}_1$) imply the bounds (A${}_2$), (B${}_2$), (C${}_2$). Next, we observe that $f^0_{\sUV,1}=F^{0,0,0}_{\sUV,\sIR}$ and define $f^0_{0,1}:=F^{0,0,0}_{0,\sIR}$. As a result, the bounds (A${}_1$), (B${}_1$), (C${}_1$) hold for $i=0$ by Theorem~\ref{thm:deterministic_taylor}. 

Now let us fix $i_\circ\in\bN_+$ and assume that the lemma holds for all $i<i_\circ$. We shall prove the lemma for $i=i_\circ$. To this end, we fix some $\mathtt{w}\in C^\infty_\rc(\bR)$ such that $\mathtt{u}=1$ on some neighborhood of \mbox{$\supp\,\mathtt{w}-[0,i(i+1)]$}. Note that by Remark~\ref{rem:shift_identity} and the support properties of the effective force coefficients stated in Lemma~\ref{lem:support} it holds
\begin{multline}
 \langle K_\sIR^{\ast\oo}\ast\mathtt{w} f^i_{\sUV,1},\psi\rangle 
 \\
 =
 \sum_{k\in\bN_0}
 \sum_{\substack{i_0,\ldots,i_k\in\bN_0\\i_0+\ldots+i_k=i}}
 \langle K_{\sUV,\sIR}^{k;\oo}\ast \mathtt{w} F^{i_0,k}_{\sUV,\sIR},\psi\otimes 
 \fR_\sUV^{\phantom{\oo}} \fP_\sIR^\oo  \hat G_\sIR\ast \mathtt{v} f_{\sUV,1}^{i_1}\otimes\ldots\otimes \fR_\sUV^{\phantom{\oo}} \fP_\sIR^\oo  \hat G_\sIR\ast \mathtt{v} f_{\sUV,1}^{i_k}\rangle,
\end{multline}
where $\mathtt{v}\in C^\infty_\rc(\bR)$ is such that $\mathtt{v}=1$ on some neighborhood of $\supp\,\mathtt{w}-[0,2i]$ and $\mathtt{u}=1$ on some neighborhood of $\supp\,\mathtt{v}-[0,(i-1)i]$. Hence, by Def.~\ref{dfn:cV} of the norm $\|\Cdot\|_{\cV^m}$ we obtain
\begin{multline}
 \|K_\sIR^{\ast\oo}\ast\mathtt{w} f^i_{\sUV,1}\|_\cV
 \leq
 \sum_{k\in\bN_0}
 \sum_{\substack{i_0,\ldots,i_k\in\bN_0\\i_0+\ldots+i_k=i}}
 \|K_{\sUV,\sIR}^{k;\oo}\ast \mathtt{w} F^{i_0,k}_{\sUV,\sIR}\|_{\cV^k}
 \\[1mm]
 \times
 \|\fR_\sUV^{\phantom{\oo}} \fP_\sIR^\oo  \hat G_\sIR\ast \mathtt{v} f_{\sUV,1}^{i_1}\|_\cV \ldots 
 \|\fR_\sUV^{\phantom{\oo}} \fP_\sIR^\oo  \hat G_\sIR\ast \mathtt{v} f_{\sUV,1}^{i_k}\|_\cV
\end{multline}
and
\begin{multline}
 \|K_\sIR^{\ast\oo}\ast\mathtt{w}(f^i_{\sUV,1}-f^i_{0,1})\|_\cV
 \leq
 \sum_{\substack{k,l\in\bN_0\\l\leq k}}
 \sum_{\substack{i_0,\ldots,i_k\in\bN_0\\i_0+\ldots+i_k=i}}
 \|K_{\sUV,\sIR}^{k;\oo}\ast \mathtt{w} F^{i_0,k}_{\sUV,\sIR}\|_{\cV^k}
 \\[1mm]
 \times
 \|\fR_\sUV^{\phantom{\oo}} \fP_\sIR^\oo  \hat G_\sIR\ast \mathtt{v} f_{\sUV,1}^{i_1}\|_\cV \ldots 
 \|\fR_\sUV^{\phantom{\oo}} \fP_\sIR^\oo  \hat G_\sIR\ast \mathtt{v} f_{\sUV,1}^{i_{l-1}}\|_\cV
 \\[2mm]
 \times
 \|\fR_\sUV^{\phantom{\oo}} \fP_\sIR^\oo  \hat G_\sIR\ast \mathtt{v} f_{\sUV,1}^{i_l}-\fP_\sIR^\oo  \hat G_\sIR\ast \mathtt{v} f_{0,1}^{i_l}\|_\cV\,
 \|\fP_\sIR^\oo  \hat G_\sIR\ast \mathtt{v} f_{0,1}^{i_{l+1}}\|_\cV \ldots 
 \|\fP_\sIR^\oo  \hat G_\sIR\ast \mathtt{v} f_{0,1}^{i_{k}}\|_\cV
 \\[4mm]
 +
 \sum_{k\in\bN_0}
 \sum_{\substack{i_0,\ldots,i_k\in\bN_0\\i_0+\ldots+i_k=i}}
 \|K_{\sUV,\sIR}^{k;\oo}\ast \mathtt{w}(F^{i_0,k}_{\sUV,\sIR}-F^{i_0,k}_{0,\sIR})\|_{\cV^k}\,
 \|\fP_\sIR^\oo  \hat G_\sIR\ast \mathtt{v} f_{0,1}^{i_1}\|_\cV \ldots 
 \|\fP_\sIR^\oo  \hat G_\sIR\ast \mathtt{v} f_{0,1}^{i_k}\|_\cV
\end{multline}
These bounds remain valid for $\sUV=0$. The bounds (A${}_1$), (B${}_1$), (C${}_1$) follow now from the induction hypothesis, Theorem~\ref{thm:deterministic_taylor} and Remark~\ref{rem:R_rho_hat} applied with $m=0$.
\end{proof}

\begin{rem}\label{rem:check_f_bounds}
Let $\check f^i_{\sUV,1}=\fQ G_1\ast f^i_{\sUV,1}$. In the next section we will need the bounds
\begin{enumerate}
 \item[(A)] $\|\mathtt{w} \check f^i_{0,1}\|_\cV\lesssim R^{i,0}$,
 \item[(B)] $\|\mathtt{w} \check f^i_{\sUV,1}\|_\cV\lesssim R^{i,0}$,
 \item[(C)] $\|\mathtt{w}(\check f^i_{\sUV,1}-\check f^i_{0,1})\|_\cV\lesssim \delta R^{i,0}$,
\end{enumerate}
which easily follow from the bounds (A${}_1$), (B${}_1$), (C${}_1$) stated in the above lemma and the fact that $\fQ G_1=-\dot G_1$ is smooth and $\|\fP^{\oo} \dot G_1\| \lesssim 1$. 
\end{rem}

\begin{rem}
Recall that $\alpha\equiv\alpha_\varepsilon:=-\dim(\varPhi)-\varepsilon$. Note that $\varrho_\varepsilon(i)+\sigma\geq \alpha_\varepsilon$ for any $i\in\bN_0$ and $\varepsilon\in(0,\varepsilon_\diamond)$.
\end{rem}

\begin{lem}\label{lem:singular_terms_main}
Under the assumptions of Theorem~\ref{thm:deterministic_taylor} there exists $\oo\in\bN_0$ and a universal distribution $\varPhi_0^\triangleright\in\cD$ such that
\begin{enumerate}
\item[(A)] $\|K_\sIR^{\ast\oo} \ast \mathtt{w} \varPhi^\triangleright_0\|_\cV\lesssim R^{i_\triangleright,0}\, [\sIR]^{\alpha_\varepsilon}$,  $\sIR\in(0,1]$,

\item[(B)] $\|K_\sIR^{\ast\oo} \ast \mathtt{w} \fR_\sUV^{\phantom{\oo}}\varPhi^\triangleright_\sUV\|_\cV\lesssim R^{i_\triangleright,0}\, [\sUV\vee\sIR]^{\alpha_\varepsilon}$, $\sIR\in[0,1]$,

\item[(C)] $\|K_\sIR^{\ast\oo}\ast \mathtt{w}(\varPhi^\triangleright_\sUV-\varPhi^\triangleright_0)\|_\cV\lesssim \delta R^{i_\triangleright,0}\, [\sIR]^{\alpha_\varepsilon-\varepsilon}$, $\sIR\in[\sUV,1]$,
\end{enumerate}
for all $\mathtt{w}\in C^\infty_\rc(\bR)$ satisfying the condition $\mathtt{u}=1$ on some neighborhood of \mbox{$\supp\,\mathtt{w}-[0,2i_\triangleright(i_\triangleright+1)+2]$}. The constants of proportionality in the above bounds depend only on \mbox{$\mathtt{w},\mathtt{u}\in C^\infty_\rc(\bR)$} and are otherwise universal.
\end{lem}
\begin{proof}
In the proof we use the results of the previous lemma. Recall that
\begin{equation}
 \varPhi_\sUV^\triangleright=\hat G\ast f_{\sUV,1}^\triangleright:=
 \sum_{i=0}^{i_\triangleright}\lambda^i \,\hat G\ast f_{\sUV,1}^i.
\end{equation}
For $\sUV=0$, this is the definition of the LHS. By the fact that $\hat G$ is supported in $[0,2]\times\bR^\rdim$ and Remark~\ref{rem:weight} we obtain
\begin{equation}
 \|K_\sIR^{\ast\oo} \ast \mathtt{w} \fR_\sUV^{\phantom{\oo}}\varPhi^\triangleright_\sUV\|_\cV
 =
 \|K_\sIR^{\ast\oo} \ast \mathtt{w} \fR_\sUV^{\phantom{\oo}}\hat G\ast \mathtt{v} f^\triangleright_\sUV\|_\cV
 \lesssim 
 \|K_\sIR^{\ast\oo} \ast \fR_\sUV^{\phantom{\oo}}\hat G\ast \mathtt{v} f^\triangleright_\sUV\|_\cV
\end{equation}
and
\begin{equation}
 \|K_\sIR^{\ast\oo} \ast \mathtt{w}(\varPhi^\triangleright_\sUV-\varPhi^\triangleright_0)\|_\cV
 =
 \|K_\sIR^{\ast\oo} \ast \mathtt{w}(\hat G\ast \mathtt{v} f^\triangleright_\sUV-\hat G\ast \mathtt{v} f^\triangleright_0)\|_\cV
 \lesssim 
 \|K_\sIR^{\ast\oo} \ast \hat G\ast \mathtt{v}(f^\triangleright_\sUV-f^\triangleright_0)\|_\cV,
\end{equation} 
where $\mathtt{v}\in C^\infty_\rc(\bR)$ is such that $\mathtt{v}=1$ on some neighborhood of $\supp\,\mathtt{w}-[0,2]$ and $\mathtt{u}=1$ on some neighbourhood of $\supp\,\mathtt{v}-[0,i_\triangleright(i_\triangleright+1)+2]$. The first of the above bounds is also valid for $\sUV=0$. The statement of the lemma follows now from Lemma~\ref{lem:schauder}~(B) and the bounds (A${}_1$), (B${}_1$), (C${}_1$) proved in the previous lemma.
\end{proof}

\begin{thm}\label{thm:F_hat_bound}
Under the assumptions of Theorem~\ref{thm:deterministic_taylor} there exists $\oo\in\bN_0$ and a universal family of differentiable functions 
\begin{equation}\label{eq:deterministic_hat_family}
 (0,1]\ni \sIR \mapsto \hat F^{i,m}_{0,\sIR}\in\cD^m, \qquad i,m\in\bN_0,
\end{equation}
satisfying the flow equation~\eqref{eq:flow_deterministic_i_m} such that it holds
\begin{enumerate}
\item[(A)] $\|K^{m;\oo}_{0,\sIR}\ast \mathtt{w} \hat F^{i,m}_{0,\sIR}\|_{\cV^m} \lesssim R^{i,m}\, [\sIR]^{\varrho_\varepsilon(i,m)}$,~~$\sIR\in(0,1]$, 

\item[(B)] $\|K^{m;\oo}_{\sUV,\sIR}\ast \mathtt{w} \hat F^{i,m}_{\sUV,\sIR}\|_{\cV^m} \lesssim R^{i,m}\, [\sUV\vee\sIR]^{\varrho_\varepsilon(i,m)}$,~~$\sIR\in[0,1]$, 

\item[(C)] $\|K^{m;\oo}_{\sUV,\sIR} \ast \mathtt{w}(\hat F^{i,m}_{\sUV,\sIR}-\hat F^{i,m}_{0,\sIR})\|_{\cV^m}
\lesssim \delta R^{i,m}\,[\sIR]^{\varrho_\varepsilon(i,m)-\varepsilon}$,~~$\sIR\in[\sUV,1]$,
\end{enumerate}
for all $i,m\in\bN_0$ and $\mathtt{w}\in C^\infty_\rc(\bR)$ satisfying the condition $\mathtt{u}=1$ on some neighborhood of $\supp\,\mathtt{w}-[0,i(i+1)]$. The constants of proportionality in the above bounds depend only on $i,m\in\bN_0$ and $\mathtt{w},\mathtt{u}\in C^\infty_\rc(\bR)$ and are otherwise universal.
\end{thm}
\begin{proof}
It follows from Remark~\ref{rem:shift_identity} and the support properties of the effective force coefficients stated in Lemma~\ref{lem:support} that
\begin{multline}
 \langle K_{\sUV,\sIR}^{m;\oo}\ast\mathtt{w}\hat F^{i,m}_{\sUV,\sIR},\psi\otimes\varphi^{\otimes m}\rangle 
 = \sum_{k\in\bN_0}\sum_{\substack{i_0,\ldots,i_k\in\bN_0\\i_0+\ldots+i_k=i\\i_1,\ldots,i_k\leq i_\triangleright}} \binom{m+k}{k} 
 \\
 \times
 \langle K_{\sUV,\sIR}^{k+m;\oo}\ast \mathtt{w} F^{i_0,k+m}_{\sUV,\sIR},\psi\otimes 
 \fR_\sUV^{\phantom{\oo}} \fP_\sIR^\oo \hat G_\sIR\ast \mathtt{v} f^{i_1}_{\sUV,1}\otimes\ldots\otimes \fR_\sUV^{\phantom{\oo}} \fP_\sIR^\oo \hat G_\sIR\ast \mathtt{v} f^{i_k}_{\sUV,1}\otimes\varphi^{\otimes m}\rangle
\end{multline}
for any $\mathtt{v}\in C^\infty_\rc(\bR)$ such that $\mathtt{v}=1$ one some neighborhood of $\supp\,\mathtt{w}-[0,2i]$ and $\mathtt{u}=1$ on some neighborhood of $\supp\,\mathtt{v}-[0,(i-1)i]$. Consequently, by Def.~\ref{dfn:cV} of the norm $\|\Cdot\|_{\cV^m}$ we obtain
\begin{multline}
 \|K_{\sUV,\sIR}^{m;\oo}\ast\mathtt{w}\hat F^{i,m}_{\sUV,\sIR}\|_{\cV^m}
 \lesssim \sum_{k\in\bN_0}\sum_{\substack{i_0,\ldots,i_k\in\bN_0\\i_0+\ldots+i_k=i\\i_1,\ldots,i_k\leq i_\triangleright}}
 \|K_{\sUV,\sIR}^{k+m;\oo}\ast \mathtt{w} F^{i_0,k+m}_{\sUV,\sIR}\|_{\cV^{k+m}}
 \\[1mm]
 \times
 \|\fR_\sUV^{\phantom{\oo}} \fP_\sIR^\oo \hat G_\sIR\ast \mathtt{v} f^{i_1}_{\sUV,1}\|_\cV \ldots
 \|\fR_\sUV^{\phantom{\oo}} \fP_\sIR^\oo \hat G_\sIR\ast \mathtt{v} f^{i_k}_{\sUV,1}\|_{\cV}
\end{multline}
and
\begin{multline}
 \|K_{\sUV,\sIR}^{m;\oo}\ast\mathtt{w}(\hat F^{i,m}_{\sUV,\sIR}-\hat F^{i,m}_{0,\sIR})\|_{\cV^m}
 \lesssim 
 \sum_{\substack{k,l\in\bN_0\\l\leq k}}\sum_{\substack{i_0,\ldots,i_k\in\bN_0\\i_0+\ldots+i_k=i\\i_1,\ldots,i_k\leq i_\triangleright}}
 \|K_{\sUV,\sIR}^{k+m;\oo}\ast \mathtt{w} F^{i_0,m+k}_{\sUV,\sIR}\|_{\cV^{k+m}}
 \\*[1mm]
 \times
 \|\fR_\sUV^{\phantom{\oo}} \fP_\sIR^\oo  \hat G_\sIR\ast \mathtt{v} f_{\sUV,1}^{i_1}\|_\cV \ldots 
 \|\fR_\sUV^{\phantom{\oo}} \fP_\sIR^\oo \hat G_\sIR\ast \mathtt{v} f_{\sUV,1}^{i_{l-1}}\|_\cV
 \\*[2mm]
 \times
 \|\fR_\sUV^{\phantom{\oo}} \fP_\sIR^\oo  \hat G_\sIR\ast \mathtt{v} f_{\sUV,1}^{i_l}-\fP_\sIR^\oo  \hat G_\sIR\ast \mathtt{v} f_{0,1}^{i_l}\|_\cV
 \,
 \|\fP_\sIR^\oo  \hat G_\sIR\ast \mathtt{v} f_{0,1}^{i_{l+1}}\|_\cV \ldots 
 \|\fP_\sIR^\oo  \hat G_\sIR\ast \mathtt{v} f_{0,1}^{i_{k}}\|_\cV
 \allowdisplaybreaks\\[4mm]
 +
 \sum_{k\in\bN_0}\sum_{\substack{i_0,\ldots,i_k\in\bN_0\\i_0+\ldots+i_k=i\\i_1,\ldots,i_k\leq i_\triangleright}}
 \|K_{\sUV,\sIR}^{m+k;\oo}\ast \mathtt{w}(F^{i_0,m+k}_{\sUV,\sIR}-F^{i_0,m+k}_{0,\sIR})\|_{\cV^{m+k}}
 \\*
 \times
 \|\fP_\sIR^\oo  \hat G_\sIR\ast \mathtt{v} f_{0,1}^{i_1}\|_\cV \ldots 
 \|\fP_\sIR^\oo  \hat G_\sIR\ast \mathtt{v} f_{0,1}^{i_k}\|_\cV.
\end{multline}
The bounds (A), (B), (C) follow now from Theorem~\ref{thm:deterministic_taylor}, the bounds (A${}_2$), (B${}_2$), (C${}_2$) stated in Lemma~\ref{lem:deterministic_hat} and Remark~\ref{rem:R_rho_hat}.
\end{proof}

\subsection{Effect of initial condition}\label{sec:effect_initial}

Let us fix some $t\in[0,\infty)$. The goal of most the remaining part of the deterministic analysis is the construction of $\breve \varPhi_\sUV$ and $\breve \varPhi_0$ as a function of the boundary data such that $\breve \varPhi_\sUV$ is the solution of the equation 
\begin{equation}
 \breve \varPhi_\sUV = G\ast (1_{(t,t+6i_\dagger)} (\hat F_\sUV[\breve \varPhi_\sUV] +\check f_\sUV^\triangleright)+ \delta_t\otimes\phi^\vartriangle_\sUV),
\end{equation}
in the time interval $[t,t+T]$ for sufficiently small $T\in(0,1)$ depending on the constant $R$ appearing in Assumptions~\ref{ass:deterministic} and~\ref{ass:deterministic_initial}. Note that $\breve \varPhi_\sUV(t,\Cdot)=\phi_\sUV^\vartriangle$. The time interval $(t,t+6i_\dagger)$, where $i_\flat\in\bN_+$ was introduced in Def.~\ref{dfn:im}, could in principle be replaced by any interval containing $(t,t+1)$. Moreover, recall that $\hat F_\sUV[\varphi]=F_\sUV[\varphi+\varPhi_\sUV^\triangleright]- f_\sUV^\triangleright$ is a local functional of $\varphi$ depending only on its spatial derivatives. We remind the reader that the strategy of the proof was described in Sec.~\ref{sec:intro_initial_value}. In Sec.~\ref{sec:maximal_solution} we show how to patch together the solutions defined in short time intervals and construct the maximal solution.

\begin{dfn}
We set $f_0^\vartriangle:=\delta_t\otimes\phi^\vartriangle_0$ and $f_\sUV^\vartriangle:=\delta_t\otimes\phi^\vartriangle_\sUV$, where $\delta_t\in\sS'(\bR)$ is the Dirac delta at $\mathring x=t$. We also set $\varPhi_0^\vartriangle:=\hat G\ast f_0^\vartriangle$ and $\varPhi_\sUV^\vartriangle:=\hat G\ast f_\sUV^\vartriangle$.
\end{dfn}

\begin{rem}
Note that $\varPhi^\vartriangle_\sUV$ is compactly supported in time since the kernel $\hat G=G_0-G_1$ is supported in $[0,2]\times \bR^\rdim$. Recall that $\beta\equiv\beta_\varepsilon=-\dim(\phi)-\varepsilon$ and $\dim(\phi)$ was introduced in Def.~\ref{dfn:dim_phi}.
\end{rem}

\begin{lem}\label{lem:initial_properties}
Under Assumption~\ref{ass:deterministic_initial} for $\sIR\in(0,1]$ and $\mathring x\in[t,t+T]$ it holds
\begin{enumerate}
\item[(A${}_1$)] $\|\bar K_\sIR^{\ast\oo}\ast \varPhi_0^\vartriangle(\mathring x,\Cdot)\|_{L^\infty(\bT)}\lesssim R\,[\sIR]^{\beta_\varepsilon}$,
\item[(B${}_1$)]
$\|\fR_\sUV^{\phantom{\oo}} \varPhi^\vartriangle_\sUV(\mathring x,\Cdot)\|_{L^\infty(\bT)}\lesssim \delta R\, [\sUV]^{\beta_\varepsilon}$,
\item[(C${}_1$)] $\|\bar K_\sIR^{\ast\oo}\ast (\varPhi_\sUV^\vartriangle(\mathring x,\Cdot)-\varPhi_0^\vartriangle(\mathring x,\Cdot))\|_{L^\infty(\bT)} \lesssim \delta R\,[\sIR]^{\beta_\varepsilon}$.
\end{enumerate}
Moreover, for $\sUV\in(0,\delta^{1/\varepsilon}]$ and $\varepsilon\in(0,\sigma/2)$ there exists $\oo\in\bN_0$ such that it holds
\begin{enumerate}
\item[(A${}_2$)] $\|K_\sIR^{\ast\oo}\ast f_0^\vartriangle\|_\cV\lesssim R\,[\sIR]^{\beta_\varepsilon-\sigma}$,~~$\sIR\in(0,1]$,

\item[(B${}_2$)] $\|K_\sIR^{\ast\oo}\ast \fR_\sUV^{\phantom{\oo}} f_\sUV^\vartriangle\|_\cV\lesssim R\,[\sIR]^{-\sigma}\,[\sUV\vee\sIR]^{\beta_\varepsilon}$,~~$\sIR\in(0,1]$,

\item[(C${}_2$)] $\|K_\sIR^{\ast\oo}\ast (\fR_\sUV^{\phantom{\oo}} f_\sUV^\vartriangle-f_0^\vartriangle)\|_\cV \lesssim \delta R\,[\sIR]^{\beta_\varepsilon-\sigma-\varepsilon}$,~~$\sIR\in[\sUV,1]$,
\end{enumerate}
and
\begin{enumerate}
\item[(A${}_3$)]$\|\fP_\sIR^\oo \hat G_\sIR^{\phantom{\oo}}\ast f_0^\vartriangle\|_\cV\lesssim R\, [\sIR]^{\beta_\varepsilon}$,~~$\sIR\in(0,1]$,

\item[(B${}_3$)]$\|\fP_\sIR^\oo \hat G_\sIR^{\phantom{\oo}}\ast \fR_\sUV^{\phantom{\oo}} f_\sUV^\vartriangle\|_\cV\lesssim R\, [\sUV\vee\sIR]^{\beta_\varepsilon}$,~~$\sIR\in[0,1]$,

\item[(C${}_3$)]
$\|\fP_\sIR^\oo \hat G_\sIR^{\phantom{\oo}}\ast (\fR_\sUV^{\phantom{\oo}} f_\sUV^\vartriangle-f_0^\vartriangle)\|_\cV \lesssim\delta R\, [\sIR]^{\beta_\varepsilon-\varepsilon}$,~~$\sIR\in[\sUV,1]$.
\end{enumerate}
The constants of proportionality in the above bounds are universal.
\end{lem}
\begin{proof}
The bounds (A${}_1$), (B${}_1$), (C${}_1$) follow immediately from Assumption~\ref{ass:deterministic_initial} and the fact that $\sup_{\mathring x\in\bR}\int_{\bR^\rdim} |\hat G(x)|\,\rd\bar x<\infty$. In order to prove the remaining bounds we first show that for $\ooo=\oo+\oooo$, where $\oo\in\bN_0$ is as in Assumption~\ref{ass:deterministic_initial} and $\oooo=\floor{\sigma/2}+1$, it holds
\begin{enumerate}
\item[(A${}_4$)] $\|\bar K_\sIR^{\ast\ooo}\ast \phi_0^\vartriangle\|_{L^\infty(\bT)}\lesssim R\,[\sIR]^{\beta_\varepsilon}$,~~$\sIR\in(0,1]$,

\item[(B${}_4$)] $\|\bar K_\sIR^{\ast\ooo}\ast \fR_\sUV^{\phantom{\ooo}} \phi_\sUV^\vartriangle\|_{L^\infty(\bT)}\lesssim R\,[\sUV\vee\sIR]^{\beta_\varepsilon}$,~~$\sIR\in[0,1]$,

\item[(C${}_4$)] $\|\bar K_\sIR^{\ast\ooo}\ast (\fR_\sUV^{\phantom{\ooo}} \phi_\sUV^\vartriangle-\phi_0^\vartriangle)\|_{L^\infty(\bT)} \lesssim \delta R\,[\sIR]^{\beta_\varepsilon-\varepsilon}$,~~$\sIR\in[\sUV,1]$.
\end{enumerate}
The bound (A${}_4$) is implied by the bound (A) in Assumption~\ref{ass:deterministic_initial}. For $\sIR\in[0,\sUV]$ the bound (B${}_4$) follows from Assumption~\ref{ass:deterministic_initial}~(B) and $\|\bar K_\sIR\|_{L^1(\bR^\rdim)}=1$ and for $\sIR\in(\sUV,1]$ it is a consequence of Assumption~\ref{ass:deterministic_initial}~(A),~(B) and the bound
\begin{multline}
 \|\bar K_\sIR^{\ast\ooo}\ast \fR_\sUV^{\phantom{\ooo}} \phi_\sUV^\vartriangle\|_{L^\infty(\bT)}\leq
 \|\fR_\sUV^{\phantom{\ooo}}\bar K_\sIR^{\ast\oooo}\ast \|_{L^1(\bR^\rdim)}\,
 \|\bar K_\sIR^{\ast\oo}\ast \fR_\sUV^{\phantom{\ooo}} \phi_\sUV^\vartriangle\|_{L^\infty(\bT)}
 \\\lesssim
 \|\bar K_\sIR^{\ast\oo}\ast \fR_\sUV^{\phantom{\ooo}} \phi_\sUV^\vartriangle\|_{L^\infty(\bT)},
\end{multline}
where we used Lemma~\ref{lem:kernel_simple_fact}~(C) to obtain the last estimate. To prove the bound~(C${}_4$) we use
\begin{multline}
 \|\bar K_\sIR^{\ast\ooo}\ast (\fR_\sUV^{\phantom{\oo}} \phi_\sUV^\vartriangle-\phi_0^\vartriangle)\|_{L^\infty(\bT)}
 \\
 \leq
 \|\bar K_\sIR^{\ast\ooo}\ast (\phi_\sUV^\vartriangle-\phi_0^\vartriangle)\|_{L^\infty(\bT)}
 +
 \|(\fR_\sUV^{\phantom{\oo}}-1)\bar K_\sIR^{\ast\ooo}\ast \phi_\sUV^\vartriangle\|_{L^\infty(\bT)},
\end{multline}
Assumption~\ref{ass:deterministic_initial} and the bound
\begin{equation}
 \|(\fR_\sUV^{\phantom{\oo}}-1)\bar K_\sIR^{\ast\oooo}\|_{L^1(\bR^\rdim)}
 \lesssim [\sUV]^{\sigma-\varepsilon}\,[\sIR]^{\varepsilon-\sigma}
 \leq \delta\, [\sUV]^{\sigma-2\varepsilon}\,[\sIR]^{\varepsilon-\sigma}
 \leq \delta [\sIR]^{-\varepsilon},\quad \sIR\in[\sUV,1],
\end{equation}
which is a consequence of Lemma~\ref{lem:kernel_fourier_transform}~(D).

The bounds (A${}_2$), (B${}_2$), (C${}_2$) follow from the equality $K_\sIR = \mathring K_\sIR\otimes\bar K_\sIR$ and the estimate $\sup_{\mathring x\in\bR} |\mathring K_\sIR(\mathring x)|\leq[\sIR]^{-\sigma}$. To prove the bounds (A${}_3$), (B${}_3$), (C${}_3$) we use the estimate $\sup_{\mathring x\in\bR}\int_{\bR^\rdim} |\bar\fP_\sIR^{\oo}\fP_\sIR^{\oo} \hat G_\sIR^{\phantom{\oo}}(x)|\,\rd\bar x\lesssim 1$, which follows from Remark~\ref{rem:dot_G_1}, and the equality $\bar\fP_\sIR^\oo\bar K_\sIR^{\ast\oo}=\delta_{\bR^\rdim}$. 
\end{proof}

\begin{dfn}
Let $\varepsilon\in(0,1]$. For $i,m\in\bN_+$ we define
\begin{equation}
 \tilde\varrho_\varepsilon(i,m):=\varrho_\varepsilon(i,m)\vee (3\varepsilon-\sigma),
\end{equation}
where $\varrho_\varepsilon(i,m)$ was introduced in Def.~\ref{dfn:varrho}. We also set
\begin{equation}
 \varepsilon_\triangleright :=((i_\triangleright+1)\dim(\lambda)-\dim(\varPhi))/(m_\flat+4),
\end{equation}
\begin{equation}
 \varepsilon_\vartriangle :=(\dim(\lambda)-\dim(\phi))/(m_\flat+4)\wedge\sigma/3,
\end{equation}
where $i_\triangleright$ is the largest integer such that $\varrho(i_\triangleright)+\sigma\leq 0$.
\end{dfn}

\begin{rem}\label{rem:varepsilon_bound_initial}
In what follows we assume that $\varepsilon\in(0,\varepsilon_\sigma\wedge\varepsilon_\diamond\wedge\varepsilon_\triangleright\wedge\varepsilon_\vartriangle)$. 
\end{rem}

\begin{rem}
We remind the reader that $\varepsilon_\sigma=\sigma+1-\ceil{\sigma}$ was introduced in Def.~\ref{dfn:varepsilon_sigma}, $\varepsilon_\diamond\leq \dim(\lambda)/(m_\flat+3)$ was introduced in Def.~\ref{dfn:varrho} and $\dim(\lambda)>\dim(\phi)$, $\dim(\varPhi)\geq\dim(\phi)$ by Def.~\ref{dfn:dim_phi}.
\end{rem}

\begin{rem}\label{rem:hat_rho}
For $\varepsilon\in(0,\varepsilon_\triangleright)$ and $i>i_\triangleright$ we have $\varrho_\varepsilon(i)\geq \varrho_\varepsilon(i_\triangleright+1)>3\varepsilon-\sigma$. Moreover, for any $\varepsilon\in(0,\varepsilon_\diamond)$ and any $i,m\in\bN_+$ it holds
\begin{equation}\label{eq:m+1_bound}
 \varrho_\varepsilon(i,m) \geq \varrho_\varepsilon(1,1)
 =
 \dim(\lambda)-\sigma-\varepsilon m_\flat \geq 3\varepsilon-\sigma. 
\end{equation}
Consequently, for $\varepsilon\in(0,\varepsilon_\diamond\wedge\varepsilon_\triangleright)$ we have $\tilde\varrho_\varepsilon(i,m)=\varrho_\varepsilon(i,m)$ unless $i\in\{0,\ldots,i_\triangleright\}$ and $m=0$. By construction $\hat F^{i,m}_{\sUV,\sIR}=0$ for all $i\in\{0,\ldots,i_\triangleright\}$ and $m=0$. As a result, one can replace $\varrho_\varepsilon(i,m)$ by $\tilde\varrho_\varepsilon(i,m)$ in Theorem~\ref{thm:F_hat_bound}.
\end{rem}

\begin{dfn}\label{dfn:eff_force_check}
Given a boundary data for $(\varPhi_\sUV^\triangleright,\breve\varPhi_\sUV)$ in the sense of Def.~\ref{dfn:boundary_data_initial} we define the effective force coefficients $\check F^{i,m}_{\sUV,\sIR}$, $i,m\in\bN_0$, $\sIR\in[0,1]$, by the relation
\begin{multline}
 \langle \check F_{\sUV,\sIR}[\varphi],\psi\rangle
 \equiv \sum_{i=0}^\infty \sum_{m=0}^\infty \lambda^i\,\langle \check F^{i,m}_{\sUV,\sIR},\psi\otimes\varphi^{\otimes m}\rangle
 \\
 :=
 \langle \hat F_{\sUV,\sIR}[\varphi+(G-G_\sIR)\ast \check f^\triangleright_{\sUV} + (G_\sIR-G_1)\ast f_\sUV^\vartriangle]+\check f^\triangleright_{\sUV}-f_\sUV^\vartriangle,\psi\rangle,
\end{multline}
where $\psi,\varphi\in C^\infty_\rc(\bM)$ are arbitrary, $\check f_\sUV^\triangleright:=\fQ G_1\ast f_\sUV^\triangleright$, $f_\sUV^\vartriangle=\delta_t\otimes\phi^\vartriangle_\sUV$ and $\hat F_{\sUV,\sIR}[\varphi]$, $f_\sUV^\triangleright$ were introduced in Def.~\ref{dfn:eff_force_hat}. We omit $\sIR$ if $\sIR=0$. 
\end{dfn}

\begin{rem}
By Remark~\ref{rem:eff_force_shift} the effective force $\check F_{\sUV,\sIR}[\varphi]$ satisfies the flow equation~\eqref{eq:intro_flow_eq} in the sense of formal power series. Hence, the coefficients $\check F^{i,m}_{\sUV,\sIR}$ satisfy the flow equation~\eqref{eq:flow_deterministic_i_m}. We have $\check F^{i,m}_{\sUV,\sIR}=0$ for all $m>im_\flat$. We note that an analog of Lemma~\ref{lem:hat_F_simple} holds true for $\check F^{i,m}_{\sUV,\sIR}$ as well. Observe also that
\begin{equation}
 \check F_{\sUV}[\varphi]=\check F_{\sUV,0}[\varphi] = \hat F_{\sUV}[\varphi+ \varPhi_\sUV^\vartriangle]+\check f^\triangleright_{\sUV}-f_\sUV^\vartriangle = 
 F_\sUV[\varphi+\varPhi^\triangleright_\sUV+ \varPhi_\sUV^\vartriangle]-\fQ\varPhi^\triangleright_\sUV +\delta_0\otimes\phi^\vartriangle_\sUV.
\end{equation}
\end{rem}

\begin{thm}\label{thm:F_check_bound}
Fix some $\mathtt{u}\in C^\infty_\rc(\bR)$, $R>1$, $\delta\in(0,1]$ and an admissible boundary data for $(\varPhi_0^\triangleright,\breve\varPhi_0)$. Let $\sUV\in(0,\delta^{1/\varepsilon}]$ and a boundary data for $(\varPhi_\sUV^\triangleright,\breve\varPhi_\sUV)$ be arbitrary such that all of the conditions specified in Assumptions~\ref{ass:deterministic} and~\ref{ass:deterministic_initial} are satisfied. There exists $\oo\in\bN_0$ and a universal family of differentiable functions 
\begin{equation}
 (0,1]\ni \sIR \mapsto \check F^{i,m}_{0,\sIR}\in\cD^m, \qquad i,m\in\bN_0,
\end{equation}
satisfying the flow equation~\eqref{eq:flow_deterministic_i_m} such that
\begin{enumerate}
\item[(A)] $\|K^{m;\oo}_{0,\sIR}\ast \mathtt{w}\check F^{i,m}_{0,\sIR}\|_{\cV^m} \lesssim R^{i,m}\,[\sIR]^{\tilde\varrho_\varepsilon(i,m)}$~~$\sIR\in(0,1]$,  

\item[(B)] $\|K^{m;\oo}_{\sUV,\sIR}\ast \mathtt{w}\check F^{i,m}_{\sUV,\sIR}\|_{\cV^m} \lesssim R^{i,m}\,[\sUV\vee\sIR]^{\tilde\varrho_\varepsilon(i,m)}$,~~$\sIR\in[0,1]$, 

\item[(C)] $\|K^{m;\oo}_{\sUV,\sIR}\ast \mathtt{w}(\check F^{i,m}_{\sUV,\sIR}-\check F^{i,m}_{0,\sIR})\|_{\cV^m}
\lesssim \delta R^{i,m}\,[\sIR]^{\tilde\varrho_\varepsilon(i,m)-\varepsilon}$,~~$\sIR\in[\sUV,1]$
\end{enumerate}
for all $i\in\bN_+$, $m\in\bN_0$ and $\mathtt{w}\in C^\infty_\rc(\bR)$ satisfying the condition $\mathtt{u}=1$ on some neighborhood of $\supp\,\mathtt{w}-[0,i(i+1)]$. The constants of proportionality in the above bounds depend only on $i,m\in\bN_0$ and $\mathtt{w},\mathtt{u}\in C^\infty_\rc(\bR)$ and are otherwise universal.
\end{thm}
\begin{rem}
The theorem is proved using the technique of the proof of Theorem~\ref{thm:F_hat_bound} with the use of the bounds proved in Lemma~\ref{lem:initial_properties} and Remark~\ref{rem:check_f_bounds} as well as the identity $R^{i,m} = R^{i,k+m} R^k$ and the bound
\begin{equation}\label{eq:varrho_bound_shift}
 \tilde\varrho_\varepsilon(i,m) \leq \tilde\varrho_\varepsilon(i,k+m) + k\beta_\varepsilon
\end{equation}
valid for $i\in\bN_+$ and $m,k\in\bN_0$. Let us prove the last bound. Without loss of generality we assume that $k\in\bN_+$. Recall that $\beta_\varepsilon=-\dim(\phi)-\varepsilon$. We have
\begin{equation}
 \varrho_\varepsilon(i,k+m) + k\beta_\varepsilon 
 =
 \varrho_\varepsilon(i,m) + k(\dim(\varPhi)+\varepsilon + \beta_\varepsilon) \geq \varrho_\varepsilon(i,m),
\end{equation}
where the last inequality is a consequence of the bound $\beta_\varepsilon\geq-\dim(\varPhi)-\varepsilon$, which follows from Def.~\ref{dfn:dim_phi}. Furthermore, for $\varepsilon\in(0,\varepsilon_\diamond\wedge\varepsilon_\vartriangle)$ we have
\begin{equation}
 \varrho_\varepsilon(i,k+m) + k\beta_\varepsilon 
 \geq
 \varrho_\varepsilon(1,1) + \beta_\varepsilon + (k-1)(\dim(\varPhi)+\varepsilon + \beta_\varepsilon)
 \geq 3\varepsilon-\sigma,
\end{equation}
where we used $\beta_\varepsilon\geq -\dim(\lambda)+\varepsilon(m_\flat+3)$ and $\varrho_\varepsilon(1,1)=\dim(\lambda)-\sigma-\varepsilon m_\flat$. This proves the bound~\eqref{eq:varrho_bound_shift}.
\end{rem}

\begin{dfn}\label{dfn:cW}
Let $m\in\bN_0$. For $V\in \cV^m$ we define
\begin{equation}
 \|V\|_{\cVtilde^m}
 :=
 \\
 \sup_{\bar x\in\bR^\rdim}\int_{\bR\times\bM^m}
 |V(x;y_1,\ldots,y_m)|\,\rd\mathring x\rd y_1\ldots\rd y_m.
\end{equation}
Recall that $x=(\mathring x,\bar x)\in\bM=\bR\times\bR^\rdim$. If $m=0$, then we write $\|V\|_{\cVtilde^m}=\|V\|_{\cVtilde}$.
\end{dfn}
\begin{lem}\label{lem:fB1_bound_tilde}
Let $m\in\bN_0$, $k\in\{0,\ldots,m\}$, $G\in\cG$, $W\in\cV^{k+1}$ and $U\in\cV^{m-k}$. We have
\begin{equation}\label{eq:fB1_bound_tilde}
 \|\fB(G,W,U)\|_{\tilde\cV^{m}}
 \leq
 \|G\|_\cK\,\|W\|_{\tilde\cV^{k+1}} \|U\|_{\cV^{m-k}}.
\end{equation}
\end{lem}
\begin{proof}
Recall that by Def.~\ref{dfn:maps_A_B}
\begin{multline}
 \fB(G,W,U)(x;\rd y_1\ldots\rd y_{m})
 \\
 :=
 \int_{\bM^2} W(x;\rd y\rd y_1\ldots\rd y_{k})
 \,G(y-z)\,
 U(z;\rd y_{k+1}\ldots\rd y_{m}) \,\rd z,
\end{multline}
We have
\begin{multline}
 \|\fB(G,W,U)\|_{\tilde\cV^{m}}
 \\
 \leq 
 \sup_{\bar x\in\bR^\rdim} \int_{\bR\times \bM^{m+2}}
 |W(x;\rd y\rd y_1\ldots\rd y_{k})|
 \,|G(y-z)|\,
 |U(z;\rd y_{k+1}\ldots\rd y_{m})| \,\rd\mathring x\rd z
 \\
 \leq\sup_{\bar x\in\bR^\rdim} \int_{\bR\times\bM^{k+1}} |W(x;\rd y\rd y_1\ldots\rd y_{k})|
 \,\rd\mathring x
 \\\times
 \int_\bM |G(z)|\,\rd z~
 \sup_{x\in\bM}\int_{\bM^{m-k}}|U(x;\rd y_{k+1}\ldots\rd y_{m})|.
\end{multline}
This coincides with the RHS of the bound~\eqref{eq:fB1_bound_tilde}.
\end{proof}

\begin{dfn}\label{dfn:eff_force_tilde}
Given a boundary data for $(\varPhi_\sUV^\triangleright,\breve\varPhi_\sUV)$ in the sense of Def.~\ref{dfn:boundary_data_initial} the effective force coefficients $\tilde F^{i,m}_{\sUV,\sIR}$, $\sIR\in[0,1]$, $i,m\in\bN_0$, are defined recursively by the flow equation~\eqref{eq:flow_deterministic_i_m} together with the conditions 
\begin{equation}
 \tilde F^{i,m}_{\sUV,\sIR}=0,~~ m>i m_\flat,
 \qquad
 \tilde F^{i,m}_{\sUV,0}=\tilde F^{i,m}_\sUV
 =1_{(t,t+6i_\dagger)} \check F^{i,m}_\sUV,~~i,m\in\bN_0,
\end{equation}
where the coefficients $\check F^{i,m}_{\sUV\phantom{0}}=\check F^{i,m}_{\sUV,0}$ are constructed as explained in Def.~\ref{dfn:eff_force_check}. We also define $\tilde F_{\sUV,\sIR}[\varphi]$ and $\tilde F_{\sUV}[\varphi]=\tilde F_{\sUV,0}[\varphi]$ by an analog of Eq.~\eqref{eq:effective_force_determ}.
\end{dfn}

\begin{rem}
Since $F^{i,m}_{\sUV}\in\cD^m$ is supported on the diagonal the same is true for $\hat F^{i,m}_\sUV\in\cD^m$ and $\tilde F^{i,m}_\sUV\in\cD^m$ by Def.~\ref{dfn:eff_force_hat} and Def.~\ref{dfn:eff_force_tilde}. Furthermore, it holds
\begin{equation}
 \tilde F_{\sUV}[\varphi] = 1_{(t,t+6i_\dagger)} (\hat F_\sUV[\varphi+\varPhi^\vartriangle_{\sUV}]+\check f_\sUV^\triangleright),
 \qquad
\hat F_\sUV[\varphi] = F_\sUV[\varphi+\varPhi^\triangleright_{\sUV}]-f^\triangleright_{\sUV},
\end{equation}
where $f^\triangleright_{\sUV}$ was introduced in Def.~\ref{dfn:eff_force_hat}, $\varPhi^\triangleright_{\sUV}=\hat G\ast f^\triangleright_{\sUV}$ and $\check f^\triangleright_{\sUV}=\fQ G_1\ast f^\triangleright_{\sUV}$. We used the fact that the measure $1_{(t,t+6i_\dagger)}\delta_t=0$ vanishes.
\end{rem}

\begin{rem}\label{rem:force_hat_i_m}
Recall that $F^{i,m}_{\sUV}=F^{i,m}_{\sUV,0}$ vanishes unless $i\leq i_\flat$ and $m\leq m_\flat$. In consequence, by Remark~\ref{rem:shift_identity} $\hat F^{i,m}_{\sUV}=\hat F^{i,m}_{\sUV,0}$ vanishes unless $i\leq i_\dagger$ and $m\leq m_\flat$, where $i_\dagger=i_\flat+i_\triangleright m_\flat$. It follows that $\check F^{i,m}_{\sUV}=\check F^{i,m}_{\sUV,0}$ and $\tilde F^{i,m}_{\sUV}=\tilde F^{i,m}_{\sUV,0}$ vanish unless $i\leq i_\dagger$ and $m\leq m_\flat$ as well.
\end{rem}

\begin{lem}
For $i,m\in\bN_0$ the effective force coefficients $\tilde F^{i,m}_{\sUV,\sIR}\in\cD^{m;0}$ are uniquely defined. Moreover, the function 
\begin{equation}
 [0,1]\ni\sIR\mapsto (\delta_\bM\otimes J^{\otimes m})\ast \tilde F^{i,m}_{\sUV,\sIR}\in\cV^m
\end{equation} 
is continuous in the whole domain and differentiable in the domain $\sIR\in(0,1]$.
\end{lem}
\begin{rem}
This simple qualitative lemma can be proved using the strategy of the proof of Lemma~\ref{lem:eff_force_basic_properties}; details are omitted.
\end{rem}

\begin{lem}\label{lem:tilde_F_hat_F}
Let $\sIR\in[0,1]$ and $i,m\in\bN_0$ be arbitrary. The following equality $\tilde F^{i,m}_{\sUV,\sIR} = \check F^{i,m}_{\sUV,\sIR}$ is satisfied in the domain
\begin{equation}\label{eq:tilde_domain}
 \{ (x,x_1,\ldots,x_m)\in\bM^{1+m}\,|\,\mathring x\in (t+2i\sIR,t+6i_\dagger) \}.
\end{equation}
\end{lem}
\begin{proof}
We first note that the statement of the lemma is true if $i=0$ or $\sIR=0$ or $m>i m_\flat$. Next, we fix some $i_\circ\in\bN_+$ and $m_\circ\in\bN_0$ and assume that the statement is true for all $i,m\in\bN_0$ such that either $i<i_\circ$, or $i=i_\circ$ and $m>m_\circ$. We shall prove the statement for $i=i_\circ$ and $m=m_\circ$. Since
\begin{equation}
 \partial_\sIR \tilde F^{i,m}_{\sUV,\sIR} = \tilde F^{i,m}_{\sUV} + \int_0^\sIR \partial_\uIR \tilde F^{i,m}_{\sUV,\uIR}\,\rd\uIR,
 \qquad
 \partial_\sIR \check F^{i,m}_{\sUV,\sIR} = \check F^{i,m}_{\sUV} + \int_0^\sIR \partial_\uIR \check F^{i,m}_{\sUV,\uIR}\,\rd\uIR
\end{equation}
it is enough to show that $\partial_\sIR \tilde F^{i,m}_{\sUV,\sIR} = \partial_\sIR \check F^{i,m}_{\sUV,\sIR}$ in the domain~\eqref{eq:tilde_domain}. We first note that by the flow equation~\eqref{eq:flow_deterministic_i_m} and the induction hypothesis it holds
\begin{equation}
 \partial_\sIR \tilde F^{i,m}_{\sUV,\sIR}
 =
 -\frac{1}{m!}
 \sum_{\pi\in\cP_m}\sum_{j=1}^i\sum_{k=0}^m
 (k+1)\,\fY_\pi\fB\big(\dot G_\sIR, \check F^{j,k+1}_{\sUV,\sIR}, \tilde F^{i-j,m-k}_{\sUV,\sIR}\big)
\end{equation}
in the domain~\eqref{eq:tilde_domain}. Next, we observe that for any $i,m\in\bN_0$ it holds
\begin{equation}
 \supp\,\check F^{i,m}_{\sUV,\sIR}\subset\{ (x,x_1,\ldots,x_m)\in\bM^{1+m}\,|\, \mathring x_1,\ldots,\mathring x_m \in [\mathring x-2(i-1)\sIR,\mathring x]\}.
\end{equation}
The proof of this fact is the same as the proof of Lemma~\ref{lem:support}. Using the above support property, the fact that $\supp\,\dot G_\sIR\subset [0,2\sIR]\times\bR^\rdim$ and the induction hypothesis we obtain
\begin{equation}
 \partial_\sIR \tilde F^{i,m}_{\sUV,\sIR}
 =
 -\frac{1}{m!}
 \sum_{\pi\in\cP_m}\sum_{j=1}^i\sum_{k=0}^m
 (k+1)\,\fY_\pi\fB\big(\dot G_\sIR, \check F^{j,k+1}_{\sUV,\sIR}, \check F^{i-j,m-k}_{\sUV,\sIR}\big)
\end{equation}
in the domain~\eqref{eq:tilde_domain}. By the flow equation~\eqref{eq:flow_deterministic_i_m} the RHS of the above equation coincides with $\partial_\sIR \check F^{i,m}_{\sUV,\sIR}$. This completes the proof of the inductive step.
\end{proof}

\begin{lem}\label{lem:supp_tilde}
For every $\sIR\in[0,1]$ and $i,m\in\bN_0$ it holds
\begin{equation}
 \supp\, \tilde F^{i,m}_{\sUV,\sIR} \subset ([t,t+6i_\dagger]\times\bR^\rdim)^{1+m}.
\end{equation}
\end{lem}
\begin{rem}
This can be proved using the strategy of the proof of the previous lemma.
\end{rem}

\begin{thm}\label{thm:F_tilde_bound}
Let $\mathtt{u}\in C^\infty_\rc(\bR)$ be such that $\mathtt{u}=1$ on some neighborhood of $[t-i_\dagger(i_\dagger+1),t+6i_\dagger]$. Fix some $R>1$, $\delta\in(0,1]$ and an admissible boundary data for $(\varPhi_0^\triangleright,\breve\varPhi_0)$. Let $\sUV\in(0,\delta^{1/\varepsilon}]$ and a boundary data for $(\varPhi_\sUV^\triangleright,\breve\varPhi_\sUV)$ be arbitrary such that all of the conditions specified in Assumptions~\ref{ass:deterministic} and~\ref{ass:deterministic_initial} are satisfied. There exists $\oo\in\bN_0$ and a universal family of differentiable functions 
\begin{equation}
 (0,1]\ni \sIR \mapsto \tilde F^{i,m}_{0,\sIR}\in\cD^m, \qquad i,m\in\bN_0,
\end{equation}
satisfying the flow equation~\eqref{eq:flow_deterministic_i_m} such that for all $i,m\in\bN_0$ it holds
\begin{enumerate}
\item[(A)] $\|K^{m;\oo}_{0,\sIR} \ast \tilde F^{i,m}_{0,\sIR}\|_{\cV^m}
 \lesssim R^{i,m}\, [\sIR]^{\tilde\varrho_\varepsilon(i,m)}$,~~$\sIR\in(0,1]$,

\item[(B)] $\|K^{m;\oo}_{\sUV,\sIR} \ast \tilde F^{i,m}_{\sUV,\sIR}\|_{\cV^m}
 \lesssim R^{i,m}\, [\sUV\vee\sIR]^{\tilde\varrho_\varepsilon(i,m)}$,~~$\sIR\in[0,1]$,
 
\item[(C)] $\|K^{m;\oo}_{\sUV,\sIR}\ast (\tilde F^{i,m}_{\sUV,\sIR}-\tilde F^{i,m}_{0,\sIR})\|_{\cV^m} 
 \lesssim \delta R^{i,m}\, [\sIR]^{\tilde\varrho_\varepsilon(i,m)-\varepsilon}$,~~$\sIR\in[\sUV,1]$,
\end{enumerate}
The constants of proportionality in the above bounds depend only on $i,m\in\bN_0$ and $\mathtt{u}\in C^\infty_\rc(\bR)$ and are otherwise universal.
\end{thm}
\begin{rem}
The proof relies crucially on the results established in Theorem~\ref{thm:F_check_bound} and the fact that the force coefficients $\tilde F^{i,m}_{\sUV,\sIR}$ coincide with $\check F^{i,m}_{\sUV,\sIR}$ in the domain specified in Lemma~\ref{lem:tilde_F_hat_F} and vanish outside the domain specified in Lemma~\ref{lem:supp_tilde}. We will also use the equality $R^{i,m} = R^{j,k+1}\,R^{i-j,m-k}$ and the bound
\begin{equation}\label{eq:varrho_bound_initial}
 \tilde\varrho_\varepsilon(i,m)\leq 
 \tilde\varrho_\varepsilon(j,k+1)
 +
 \tilde\varrho_\varepsilon(i-j,m-k)+\sigma,
\end{equation}
valid for $i\in\bN_+$, $m\in\bN_0$, $j\in\{1,\ldots,i\}$, $k\in\{0,\ldots,m\}$, and which follows from the definition of $\tilde\varrho_\varepsilon$ and Eq.~\eqref{eq:varrho_equality_flow}.
\end{rem}
\begin{proof}
We first note that the theorem is trivially true if $m>i m_\diamond$ since in this case $\tilde F^{i,m}_{\sUV,\sIR}=0$ and $\tilde F^{i,m}_{0,\sIR}:=0$. Next, we fix some $i_\circ,m_\circ\in\bN_0$ and assume that the theorem is true for all $i,m\in\bN_0$ such that either $i<i_\circ$, or $i=i_\circ$ and $m>m_\circ$. We shall prove the theorem for $i=i_\circ$ and $m=m_\circ$. We note that the strategy of the proof of the induction step is different in the cases $i\in\{0,\ldots,i_\dagger\}$ and $i>i_\dagger$. 

Let $\mathtt{v}\in C^\infty(\bR)$ be such that $\mathtt{v}(\mathring x)=1$ for $\mathring x\in(-\infty,t+2i_\dagger]$ and $\mathtt{v}(\mathring x)=0$ for $\mathring x\in[t+3i_\dagger,\infty)$. For $\sIR\in(0,1]$ we define $\mathtt{v}_\sIR,\mathtt{w}_\sIR\in C^\infty_\rc(\bR)$ by 
\begin{equation}
\begin{gathered}
 \mathtt{v}_\sIR(\mathring x):=1-\mathtt{v}(\mathring x/\sIR)-\mathtt{v}((6i_\dagger-\mathring x)/\sIR),
 \\
 \mathtt{w}_\sIR(\mathring x):=\mathtt{v}(\mathring x/\sIR)-\mathtt{v}(-\mathring x/\sIR)+\mathtt{v}((6i_\dagger-\mathring x)/\sIR)-\mathtt{v}((\mathring x-6i_\dagger)/\sIR).
\end{gathered} 
\end{equation}
Note that $\partial_\sIR \mathtt{w}_\sIR =-\partial_\sIR \mathtt{v}_\sIR$ on some neighborhood of $[t,t+6i_\dagger]$ and 
\begin{equation}
\begin{gathered}
 \supp\,\mathtt{v}_\sIR\subset [t+2i_\dagger\sIR,t+6i_\dagger-2i_\dagger\sIR],
 \\
 \supp\,\mathtt{w}_\sIR\subset [t-3i_\dagger\sIR,t+3i_\dagger\sIR]\cup [t+6i_\dagger-3i_\dagger\sIR,t+6i_\dagger+3i_\dagger\sIR].
\end{gathered} 
\end{equation}
Hence, by Lemma~\ref{lem:tilde_F_hat_F} and Lemma~\ref{lem:supp_tilde} we obtain
\begin{equation}\label{eq:initial_decomposition}
 \tilde F^{i,m}_{\sUV,\sIR} 
 = \mathtt{v}_\sIR \tilde F^{i,m}_{\sUV,\sIR}
 + (1-\mathtt{v}_\sIR) \tilde F^{i,m}_{\sUV,\sIR}
 = \mathtt{v}_\sIR \check F^{i,m}_{\sUV,\sIR}
 + \mathtt{w}_\sIR \tilde F^{i,m}_{\sUV,\sIR}
\end{equation}
and
\begin{equation}\label{eq:initial_decomposition_d}
 \partial_\sIR (\mathtt{w}_\sIR \tilde F^{i,m}_{\sUV,\sIR})
 =
 \mathtt{w}_\sIR (\partial_\sIR^{\phantom{i}} \tilde F^{i,m}_{\sUV,\sIR})
 -
 (\partial_\sIR \mathtt{v}_\sIR) \tilde F^{i,m}_{\sUV,\sIR}
 =
 \mathtt{w}_\sIR (\partial_\sIR^{\phantom{i}} \tilde F^{i,m}_{\sUV,\sIR})
 -
 (\partial_\sIR \mathtt{v}_\sIR) \check F^{i,m}_{\sUV,\sIR}
\end{equation}
for any $i\in\{0,\ldots,i_\dagger\}$ and $m\in\bN_0$.

We first observe that by Lemma~\ref{lem:time_estimates}~(A) applied with $\mathtt{v}_\sIR/\mathtt{w}$ we have
\begin{enumerate}
\item[$(\check{\mathrm A}_1)$] $\|K^{m;\oo}_{0,\sIR}\ast \mathtt{v}_\sIR \check F^{i,m}_{0,\sIR}\|_{\cV^m} \!\lesssim\! 
\|K^{m;\oo}_{0,\sIR}\ast \mathtt{w} \check F^{i,m}_{0,\sIR}\|_{\cV^m}$,

\item[$(\check{\mathrm B}_1)$] $\|K^{m;\oo}_{\sUV,\sIR}\ast \mathtt{v}_\sIR \check F^{i,m}_{\sUV,\sIR}\|_{\cV^m} \!\lesssim\! 
\|K^{m;\oo}_{\sUV,\sIR}\ast \mathtt{w} \check F^{i,m}_{\sUV,\sIR}\|_{\cV^m}$,

\item[$(\check{\mathrm C}_1)$] $\|K^{m;\oo}_{\sUV,\sIR}\ast \mathtt{v}_\sIR (\check F^{i,m}_{\sUV,\sIR}-\check F^{i,m}_{0,\sIR})\|_{\cV^m}
\lesssim
\|K^{m;\oo}_{\sUV,\sIR}\ast \mathtt{w}(\check F^{i,m}_{\sUV,\sIR}-\check F^{i,m}_{0,\sIR})\|_{\cV^m}$,
\end{enumerate}
where $\mathtt{w}\in C^\infty_\rc(\bR)$ is arbitrary such that $\mathtt{w}=1$ on some neighborhood of $[t,t+6i_\dagger]$. Next, we note that it holds $\partial_\sIR \mathtt{v}_\sIR =[\sIR]^{-\sigma} \dot{\mathtt{v}}_\sIR$, where 
\begin{equation}
 \dot{\mathtt{v}}_\sIR(\mathring x):= \dot{\mathtt{v}}(\mathring x/\sIR) - \dot{\mathtt{v}}((6i_\dagger-\mathring x)/\sIR),
 \quad
 \dot{\mathtt{v}}(\mathring x):=-\mathring x\, \partial_{\mathring x}\mathtt{v}(\mathring x).
\end{equation}
Moreover, we have
\begin{equation}
 \supp\,\dot{\mathtt{v}}\subset [t+2i_\dagger \sIR,t+3i_\dagger\sIR]\cup [t+6i_\dagger -3i_\dagger \sIR,t+6i_\dagger-2i_\dagger\sIR].
\end{equation}
Hence, the application of Lemma~\ref{lem:time_estimates}~(B) with $\mathtt{u}_\sIR=\dot{\mathtt{v}}_\sIR/\mathtt{w}$ allows to conclude that it holds
\begin{enumerate}
\item[$(\check{\mathrm A}_2)$] $\|K^{m;\oo}_{0,\sIR}\ast (\partial_\sIR \mathtt{v}_\sIR) \check F^{i,m}_{0,\sIR}\|_{\cVtilde^m} \lesssim
\|K^{m;\oo}_{0,\sIR}\ast \mathtt{w} \check F^{i,m}_{0,\sIR}\|_{\cV^m}$,

\item[$(\check{\mathrm B}_2)$] $\|K^{m;\oo}_{\sUV,\sIR}\ast (\partial_\sIR \mathtt{v}_\sIR) \check F^{i,m}_{\sUV,\sIR}\|_{\cVtilde^m} \lesssim
\|K^{m;\oo}_{\sUV,\sIR}\ast \mathtt{w} \check F^{i,m}_{\sUV,\sIR}\|_{\cV^m}$,

\item[$(\check{\mathrm C}_2)$] $\|K^{m;\oo}_{\sUV,\sIR}\ast (\partial_\sIR \mathtt{v}_\sIR) (\check F^{i,m}_{\sUV,\sIR}-\check F^{i,m}_{0,\sIR})\|_{\cVtilde^m}
\lesssim
\|K^{m;\oo}_{\sUV,\sIR}\ast \mathtt{w}(\check F^{i,m}_{\sUV,\sIR}-\check F^{i,m}_{0,\sIR})\|_{\cV^m}$,
\end{enumerate}
where $\mathtt{w}\in C^\infty_\rc(\bR)$ is again arbitrary such that $\mathtt{w}=1$ on some neighborhood of $[t,t+6i_\dagger]$. Recall also that by Theorem~\ref{thm:F_check_bound} it holds
\begin{enumerate}
\item[$(\check{\mathrm A})$] $\|K^{m;\oo}_{0,\sIR}\ast \mathtt{w} \check F^{i,m}_{0,\sIR}\|_{\cV^m} \lesssim R^{i,m} [\sIR]^{\tilde\varrho_\varepsilon(i,m)}$,~~$\sIR\in(0,1]$,

\item[$(\check{\mathrm B})$] $\|K^{m;\oo}_{\sUV,\sIR}\ast \mathtt{w} \check F^{i,m}_{\sUV,\sIR}\|_{\cV^m} \lesssim R^{i,m} [\sUV\vee\sIR]^{\tilde\varrho_\varepsilon(i,m)}$,~~$\sIR\in[0,1]$,

\item[$(\check{\mathrm C})$] $\|K^{m;\oo}_{\sUV,\sIR}\ast \mathtt{w}(\check F^{i,m}_{\sUV,\sIR}-\check F^{i,m}_{0,\sIR})\|_{\cV^m}
\lesssim \delta R^{i,m}[\sIR]^{\tilde\varrho_\varepsilon(i,m)-\varepsilon}$,~~$\sIR\in[\sUV,1]$,
\end{enumerate}
for $i\in\{0,\ldots,i_\dagger\}$ provided $\mathtt{w}\in C^\infty_\rc(\bR)$ is chosen in such a way that in addition $\mathtt{u}=1$ on some neighborhood of $\supp\,\mathtt{w}-[0,i_\dagger(i_\dagger+1)]$, which is always possible by our assumption on $\mathtt{u}$. By the flow equation~\eqref{eq:flow_deterministic_i_m} we obtain
\begin{equation}
 \mathtt{w}_\sIR \partial_\sIR  \tilde F^{i,m}_{\sUV,\sIR}
 =
 -\frac{1}{m!}
 \sum_{\pi\in\cP_m}\sum_{j=1}^i\sum_{k=0}^m (k+1)\,
 \fY_\pi\fB\big(\dot G_\sIR,\mathtt{w}_\sIR \tilde F^{j,k+1}_{\sUV,\sIR}, \tilde F^{i-j,m-k}_{\sUV,\sIR}\big).
\end{equation}
For $\sUV=0$ the flow equation~\eqref{eq:flow_deterministic_i_m} is the definition of $\partial_\sIR  \tilde F^{i,m}_{0,\sIR}$. Consequently, this equality remains valid for $\sUV=0$. Note that for $i=0$ the RHS of the flow equation vanishes identically. The above flow equation together with Lemma~\ref{lem:fB1_bound_tilde} imply that
\begin{multline}\label{eq:tilde_thm_bound}
 \|K_{\sUV,\sIR}^{m;\oo}\ast \mathtt{w}_\sIR \partial_\sIR \tilde F_{\sUV,\sIR}^{i,m}\|_{\tilde\cV^m} 
 \leq 
 \sum_{j=1}^{i}\sum_{k=0}^m
 (k+1)\,\|\fR_\sUV^{\phantom{\oo}}\fP_\sIR^{2\oo}\dot G_\sIR^{\phantom{\oo}}\|_\cK\\
 \times
 \|K_{\sUV,\sIR}^{k+1;\oo}\ast \mathtt{w}_\sIR \tilde F_{\sUV,\sIR}^{j,k+1}\|_{\tilde\cV^{k+1}} 
 \|K_{\sUV,\sIR}^{m-k;\oo}\ast \tilde F_{\sUV,\sIR}^{i-j,m-k}\|_{\cV^{m-k}}
\end{multline}
and
\begin{multline}\label{eq:tilde_thm_bound_diff}
 \|K_{\sUV,\sIR}^{m;\oo}\ast \mathtt{w}_\sIR \partial_\sIR (\tilde F_{\sUV,\sIR}^{i,m}-\tilde F_{0,\sIR}^{i,m})\|_{\tilde\cV^m} 
 \leq 
 \sum_{j=1}^{i}\sum_{k=0}^m
 (k+1)\,\|\fR_\sUV^{\phantom{\oo}}\fP_\sIR^{2\oo}\dot G_\sIR^{\phantom{\oo}}\|_\cK\,
 \\*
 \times\Big(\|K_{\sUV,\sIR}^{k+1;\oo}\ast \mathtt{w}_\sIR (\tilde F_{\sUV,\sIR}^{j,k+1}-\tilde F_{0,\sIR}^{j,k+1})\|_{\tilde\cV^{k+1}}
 \,
 \|K_{\sUV,\sIR}^{m-k;\oo}\ast \tilde F_{\sUV,\sIR}^{i-j,m-k}\|_{\cV^{m-k}}
 \\*[1mm]
 +
 \|K_{0,\sIR}^{k+1;\oo}\ast \mathtt{w}_\sIR \tilde F_{0,\sIR}^{j,k+1}\|_{\tilde\cV^{k+1}} 
 \,
 \|K_{\sUV,\sIR}^{m-k;\oo}\ast (\tilde F_{\sUV,\sIR}^{i-j,m-k}-\tilde F_{0,\sIR}^{i-j,m-k})\|_{\cV^{m-k}}\Big).
\end{multline}
These bounds are also valid for $\sUV=0$. Let us recall that by Lemma~\ref{lem:kernel_dot_G} it holds
\begin{equation}
 \|\fR_\sUV^{\phantom{\oo}}\fP_\sIR^{2\oo}\dot G_\sIR\|_\cK \lesssim [\sUV\vee\sIR]^{\sigma-\varepsilon}[\sIR]^{\varepsilon-\sigma}
\end{equation}
uniformly in $\sUV\in[0,1]$ and $\sIR\in(0,1]$. As a result, by the induction hypothesis and Lemma~\ref{lem:time_estimates}~(A) we obtain 
\begin{enumerate}
\item[$(\tilde{\mathrm A}_1)$] $\|K^{m;\oo}_{0,\sIR}\ast \mathtt{w}_\sIR \partial_\sIR\tilde F^{i,m}_{0,\sIR}\|_{\cVtilde^m} 
\lesssim R^{i,m}\, [\sIR]^{\tilde\varrho_\varepsilon(i,m)}$,~~$\sIR\in(0,1]$,

\item[$(\tilde{\mathrm B}_1)$] $\|K^{m;\oo}_{\sUV,\sIR}\ast \mathtt{w}_\sIR \partial_\sIR\tilde F^{i,m}_{\sUV,\sIR}\|_{\cVtilde^m} 
\lesssim R^{i,m}\, [\sUV\vee\sIR]^{\tilde\varrho_\varepsilon(i,m)}$,~~$\sIR\in(0,1]$,

\item[$(\tilde{\mathrm C}_1)$] $\|K^{m;\oo}_{\sUV,\sIR}\ast \mathtt{w}_\sIR \partial_\sIR(\tilde F^{i,m}_{\sUV,\sIR}-\tilde F^{i,m}_{0,\sIR})\|_{\cVtilde^m}
\lesssim \delta R^{i,m}\,[\sIR]^{\tilde\varrho_\varepsilon(i,m)-\varepsilon}$,~~$\sIR\in[\sUV,1]$,
\end{enumerate}
Hence, by Eq.~\eqref{eq:initial_decomposition_d} and the results established earlier we arrive at
\begin{enumerate}
\item[$(\tilde{\mathrm A}_2)$] $\|K^{m;\oo}_{0,\sIR}\ast \partial_\sIR(\mathtt{w}_\sIR \tilde F^{i,m}_{0,\sIR})\|_{\cVtilde^m} 
\lesssim R^{i,m}\, [\sIR]^{\tilde\varrho_\varepsilon(i,m)}$,~~$\sIR\in(0,1]$,

\item[$(\tilde{\mathrm B}_2)$] $\|K^{m;\oo}_{\sUV,\sIR}\ast \partial_\sIR(\mathtt{w}_\sIR \tilde F^{i,m}_{\sUV,\sIR})\|_{\cVtilde^m} 
\lesssim R^{i,m}\, [\sUV\vee\sIR]^{\tilde\varrho_\varepsilon(i,m)}$,~~$\sIR\in(0,1]$,

\item[$(\tilde{\mathrm C}_2)$] $\|K^{m;\oo}_{\sUV,\sIR}\ast \partial_\sIR( \mathtt{w}_\sIR  \tilde F^{i,m}_{\sUV,\sIR}-\mathtt{w}_\sIR \tilde F^{i,m}_{0,\sIR})\|_{\cVtilde^m}
\lesssim \delta R^{i,m}\,[\sIR]^{\tilde\varrho_\varepsilon(i,m)-\varepsilon}$,~~$\sIR\in[\sUV,1]$,
\end{enumerate}
We observe that by the flow equation~\eqref{eq:flow_deterministic_i_m} and Remark~\ref{rem:fB1_bound} the bounds~\eqref{eq:tilde_thm_bound} and~\eqref{eq:tilde_thm_bound_diff} remain true when $\mathtt{w}_\sIR$ is replaced with $1$ and $\cVtilde^m$ is replaced with $\cV^m$. As a result, 
\begin{enumerate}
\item[$(\tilde{\mathrm A}_3)$] $\|K^{m;\oo}_{0,\sIR}\ast \partial_\sIR \tilde F^{i,m}_{0,\sIR}\|_{\cV^m} 
\lesssim R^{i,m}\,[\sIR]^{\tilde\varrho_\varepsilon(i,m)-\sigma}$,~~$\sIR\in(0,1]$,

\item[$(\tilde{\mathrm B}_3)$] $\|K^{m;\oo}_{\sUV,\sIR}\ast \partial_\sIR \tilde F^{i,m}_{\sUV,\sIR}\|_{\cV^m} 
\lesssim R^{i,m}\, [\sIR]^{\varepsilon-\sigma}\,[\sUV\vee\sIR]^{\tilde\varrho_\varepsilon(i,m)-\varepsilon}$,~~$\sIR\in(0,1]$,

\item[$(\tilde{\mathrm C}_3)$] $\|K^{m;\oo}_{\sUV,\sIR}\ast \partial_\sIR(\tilde F^{i,m}_{\sUV,\sIR}-\tilde F^{i,m}_{0,\sIR})\|_{\cV^m}
\lesssim \delta R^{i,m}\,[\sIR]^{\tilde\varrho_\varepsilon(i,m)-\sigma-\varepsilon}$,~~$\sIR\in[\sUV,1]$.
\end{enumerate}

Now we show how to use the results gathered above in order to conclude the proof of the induction step. We first consider the case $i>i_\dagger$ or $\sIR\in[0,\sUV)$.  We note that
\begin{equation}
 \tilde F^{i,m}_{\sUV,\sIR}
 =
 \tilde F^{i,m}_{\sUV}+
 \int_0^\sIR \partial_\uIR \tilde F^{i,m}_{\sUV,\uIR}
 \,\rd \uIR.
\end{equation}
For $\sUV=0$ and $i>i_\dagger$ we define $\tilde F^{i,m}_{0} = 0$ and use the above equation with $\sUV=0$ to define $\tilde F^{i,m}_{0,\sIR}$, $\sIR\in(0,1]$. Using $\tilde F^{i,m}_{\sUV} = 1_{(t,t+6i_\dagger)} \check F_{\sUV}^{i,m}$ and the bound~$(\check{\mathrm B})$ aplied with $\mu=0$ we obtain
\begin{equation}
 \|K^{m;\oo}_{\sUV,0}\ast \tilde F^{i,m}_{\sUV}\|_{\cV^m}
 \leq
 \|K^{m;\oo}_{\sUV,0}\ast \mathtt{w}\check F^{i,m}_{\sUV}\|_{\cV^m}
\lesssim R^{i,m}\, [\sUV]^{\tilde\varrho_\varepsilon(i,m)-\varepsilon}
\end{equation}
for any $\mathtt{w}\in C^\infty_\rc(\bR)$ such that $\mathtt{w}=1$ on some neighborhood of $[t,t+6i_\diamond]$. We note that $\tilde\varrho_\varepsilon(i,m)-\varepsilon\geq 0$ for $i>i_\dagger$ by Remark~\ref{rem:epsilon}. As a result, for $i>i_\dagger$ the bounds $(\tilde{\mathrm A}_3)$, $(\tilde{\mathrm B}_3)$, $(\tilde{\mathrm C}_3)$ imply
\begin{enumerate}
\item[$(\tilde{\mathrm C}'_3)$] $\|K^{m;\oo}_{\sUV,\sIR}\ast \partial_\sIR(\tilde F^{i,m}_{\sUV,\sIR}-\tilde F^{i,m}_{0,\sIR})\|_{\cV^m}
\lesssim \delta R^{i,m}\,[\sIR]^{\varepsilon-\sigma}\,[\sUV\vee\sIR]^{\tilde\varrho_\varepsilon(i,m)-2\varepsilon}$,~~$\sIR\in(0,1]$.
\end{enumerate}
Hence, for $i>i_\dagger$ or $\sIR\in[0,\sUV)$ the theorem follows from Lemma~\ref{lem:integration} and the bounds~$(\tilde{\mathrm A}_3)$, $(\tilde{\mathrm B}_3)$, $(\tilde{\mathrm C}'_3)$. 

It remains to investigate the case $i\in\{1,\ldots,i_\dagger\}$. We note that by the previous paragraph and Lemma~\ref{lem:time_estimates}~(B) applied with $\mathtt{u}_\sIR=\mathtt{w}_\sIR$ we have
\begin{equation}
 \|K^{m;\oo}_{\sUV,\sIR}\ast \mathtt{w}_\sIR\tilde F^{i,m}_{\sUV,\sIR}\|_{\cVtilde^m} 
 \lesssim 
 [\sIR]^\sigma\,\|K^{m;\oo}_{\sUV,\sIR}\ast\tilde F^{i,m}_{\sUV,\sIR}\|_{\cV^m} 
 \lesssim 
 [\sIR]^\sigma\,[\sUV]^{\tilde\varrho_\varepsilon(i,m)}
\end{equation}
for any $\sIR\in[0,\sUV)$. We have
\begin{equation}
 \mathtt{w}_\sIR \tilde F^{i,m}_{\sUV,\sIR}
 =
 \int_0^\sIR \partial_\uIR (\mathtt{w}_\uIR \tilde F^{i,m}_{\sUV,\uIR})
 \,\rd \uIR,
 \qquad
 \tilde F^{i,m}_{\sUV,\sIR} 
 = \mathtt{v}_\sIR \check F^{i,m}_{\sUV,\sIR}
 + \mathtt{w}_\sIR \tilde F^{i,m}_{\sUV,\sIR}.
\end{equation}
For $i\in\{1,\ldots,i_\dagger\}$ and $m\in\bN_0$ we define
\begin{equation}
 (\mathtt{w}_\sIR \tilde F^{i,m}_{0,\sIR})
 :=
 \int_0^\sIR \partial_\uIR (\mathtt{w}_\uIR \tilde F^{i,m}_{0,\uIR})
 \,\rd \uIR,
 \qquad
 \tilde F^{i,m}_{0,\sIR} 
 := \mathtt{v}_\sIR \check F^{i,m}_{0,\sIR}
 + (\mathtt{w}_\sIR \tilde F^{i,m}_{0,\sIR}).
\end{equation}
Consequently, by Lemma~\ref{lem:kernel_u_v} it holds
\begin{equation}\label{eq:proof_int_tilde}
 \|K^{m;\oo}_{\sUV,\sIR}\ast \mathtt{w}_\sIR \tilde F^{i,m}_{\sUV,\sIR}\|_{\cVtilde^m}
 \leq
 \int_0^\sIR \|K^{m;\oo}_{\sUV,\uIR}\ast \partial_\uIR (\mathtt{w}_\uIR \tilde F^{i,m}_{\sUV,\uIR})\|_{\cVtilde^m}
 \,\rd \uIR
\end{equation}
and
\begin{multline}\label{eq:proof_int_tilde2}
 \|K^{m;\oo}_{\sUV,\sIR}\ast (\mathtt{w}_\sIR \tilde F^{i,m}_{\sUV,\sIR}-\mathtt{w}_\sIR \tilde F^{i,m}_{0,\sIR})\|_{\cVtilde^m}
 \\\leq
 \int_0^\sIR \|K^{m;\oo}_{\sUV,\uIR}\ast (\partial_\uIR( \mathtt{w}_\uIR \tilde F^{i,m}_{\sUV,\uIR})
 -\partial_\uIR( \mathtt{w}_\uIR \tilde F^{i,m}_{0,\uIR}))\|_{\cVtilde^m}
 \,\rd \uIR.
\end{multline}
The estimate~\eqref{eq:proof_int_tilde} remains valid for $\sUV=0$. Since $\tilde\varrho_\varepsilon(i,m)\geq3\varepsilon-\sigma$ by the bounds $(\tilde{\mathrm A}_2)$, $(\tilde{\mathrm B}_2)$, $(\tilde{\mathrm C}_2)$ we have
\begin{enumerate}
\item[$(\tilde{\mathrm C}'_2)$] \!$\|K^{m;\oo}_{\sUV,\sIR}\ast \partial_\sIR( \mathtt{w}_\sIR  \tilde F^{i,m}_{\sUV,\sIR}-\mathtt{w}_\sIR \tilde F^{i,m}_{0,\sIR})\|_{\cVtilde^m}
\!\lesssim\! \delta R^{i,m} [\sIR]^{\varepsilon-\sigma}[\sUV\vee\sIR]^{\tilde\varrho_\varepsilon(i,m)+\sigma-2\varepsilon}$,~~$\sIR\in(0,1]$. 
\end{enumerate}
Using the bounds~$(\tilde{\mathrm A}_2)$, $(\tilde{\mathrm B}_2)$, $(\tilde{\mathrm C}'_2)$ as well as the estimates~\eqref{eq:proof_int_tilde} and~\eqref{eq:proof_int_tilde2} we obtain 
\begin{enumerate}
\item[$(\tilde{\mathrm A}_4)$] $\|K^{m;\oo}_{0,\sIR}\ast \mathtt{w}_\sIR\tilde F^{i,m}_{0,\sIR}\|_{\cVtilde^m} 
\lesssim R^{i,m}\, [\sIR]^{\tilde\varrho_\varepsilon(i,m)+\sigma}$,~~$\sIR\in(0,1]$,

\item[$(\tilde{\mathrm B}_4)$] $\|K^{m;\oo}_{\sUV,\sIR}\ast \mathtt{w}_\sIR\tilde F^{i,m}_{\sUV,\sIR}\|_{\cVtilde^m} 
\lesssim R^{i,m}\, [\sIR]^{\sigma}\,[\sUV\vee\sIR]^{\tilde\varrho_\varepsilon(i,m)}$,~~$\sIR\in[0,1]$,

\item[$(\tilde{\mathrm C}_4)$] $\|K^{m;\oo}_{\sUV,\sIR}\ast \mathtt{w}_\sIR(\tilde F^{i,m}_{\sUV,\sIR}-\tilde F^{i,m}_{0,\sIR})\|_{\cVtilde^m}
\lesssim \delta R^{i,m}\,[\sIR]^{\tilde\varrho_\varepsilon(i,m)-\varepsilon+\sigma}$,~~$\sIR\in[\sUV,1]$.
\end{enumerate}
As a result, by Lemma~\ref{lem:time_estimates}~(C) it holds
\begin{enumerate}
\item[$(\tilde{\mathrm A}_5)$] $\|K^{m;\oo}_{0,\sIR}\ast \mathtt{w}_\sIR\tilde F^{i,m}_{0,\sIR}\|_{\cV^m} 
\lesssim R^{i,m}\, [\sIR]^{\tilde\varrho_\varepsilon(i,m)}$,~~$\sIR\in(0,1]$,

\item[$(\tilde{\mathrm B}_5)$] $\|K^{m;\oo}_{\sUV,\sIR}\ast \mathtt{w}_\sIR\tilde F^{i,m}_{\sUV,\sIR}\|_{\cV^m} 
\lesssim R^{i,m}\, [\sUV\vee\sIR]^{\tilde\varrho_\varepsilon(i,m)}$,~~$\sIR\in[0,1]$,

\item[$(\tilde{\mathrm C}_5)$] $\|K^{m;\oo}_{\sUV,\sIR}\ast \mathtt{w}_\sIR(\tilde F^{i,m}_{\sUV,\sIR}-\tilde F^{i,m}_{0,\sIR})\|_{\cV^m}
\lesssim \delta R^{i,m}\,[\sIR]^{\tilde\varrho_\varepsilon(i,m)-\varepsilon}$,~~$\sIR\in[\sUV,1]$.
\end{enumerate}
To finish the proof of the inductive step in the case $i\in\{1,\ldots,i_\dagger\}$ it is now enough to use Eq.~\eqref{eq:initial_decomposition}, the above bounds as well as the bounds~$(\check{\mathrm A}_1)$, $(\check{\mathrm B}_1)$, $(\check{\mathrm C}_1)$ and $(\check{\mathrm A})$, $(\check{\mathrm B})$, $(\check{\mathrm C})$.
\end{proof}

\begin{lem}\label{lem:time_estimates}
Let $\oo\in\bN_0$ and $\mathtt{u}_\sIR\in C^\infty_\rc(\bR)$, $\sIR\in(0,1]$.
\begin{enumerate}
 \item[(A)] $\|K^{m;\oo}_{\sUV,\sIR}\ast \mathtt{u}_\sIR V\|_{\cV^m} 
 \lesssim \|K^{m;\oo}_{\sUV,\sIR}\ast V\|_{\cV^m}$  if $\forall_{\ooo\in\{0,\ldots,\oo\}}\|\partial_{\mathring x}^\ooo \mathtt{u}_\sIR\|_{L^\infty(\bR)}\lesssim \sIR^{-\ooo}$.
 \item[(B)] $\|K^{m;\oo}_{\sUV,\sIR}\ast \mathtt{u}_\sIR V\big\|_{\cVtilde^m} 
 \lesssim  \|K^{m;\oo}_{\sUV,\sIR}\ast V\|_{\cV^m}\,[\sIR]^\sigma$ if $\forall_{\ooo\in\{0,\ldots,\oo\}}\|\partial_{\mathring x}^\ooo \mathtt{u}_\sIR\|_{L^1(\bR)}\lesssim \sIR^{1-\ooo}$.
 \item[(C)] $\|K^{m;\oo+1}_{\sUV,\sIR}\ast V\|_{\cV^m}
 \lesssim  \|K^{m;\oo}_{\sUV,\sIR}\ast V\|_{\cVtilde^m}\,[\sIR]^{-\sigma}$.
\end{enumerate}
All of the above bounds are uniform in $\sUV\in[0,1]$, $\sIR\in(0,1]$ and $V\in\cV^m$.
\end{lem}

\begin{proof}
We consider below only the case $m=0$. The general case can be proved using exactly the same technique, but the notation is more complicated. For $m=0$ we have $K^{m;\oo}_{\sUV,\sIR}=K_\sIR^{\ast\oo}$ and $\cV^m=\cV$, $\cVtilde^m=\cVtilde$. Let us set $V=v$. By Lemma~\ref{lem:time_estimates_prep} it holds
\begin{equation}
 \|K_\sIR^{\ast\oo}\ast \mathtt{u}_\sIR v\|_\cV \lesssim
 \sum_{\ooo=0}^\oo \mu^\ooo\,
 \|(\partial_{\mathring x}^\ooo \mathtt{u}_\sIR) (K_\sIR^{\ast\oo}\ast v)\|_\cV.
\end{equation}
The inspection of the proof of Lemma~\ref{lem:time_estimates_prep} shows that the same bound holds true for $\cV$ replaced with $\cVtilde$. As a result, we obtain
\begin{equation}
 \|K_\sIR^{\ast\oo}\ast \mathtt{u}_\sIR v\|_\cV \lesssim
 \sum_{\ooo=0}^\oo \mu^\ooo\, \|\partial_{\mathring x}^\ooo \mathtt{u}_\sIR\|_{L^\infty(\bR)}\,
 \|K_\sIR^{\ast\oo}\ast v\|_\cV
\end{equation}
and
\begin{equation}
 \|K_\sIR^{\ast\oo}\ast \mathtt{u}_\sIR v\|_\cVtilde \lesssim
 \sum_{\ooo=0}^\oo \mu^\ooo\, \|\partial_{\mathring x}^\ooo \mathtt{u}_\sIR\|_{L^1(\bR)}\,
 \|K_\sIR^{\ast\oo}\ast v\|_\cV.
\end{equation}
These estimates together with the assumed bounds for $\mathtt{u}_\sIR$ imply Parts~(A) and~(B). Part~(C) follows from
\begin{equation}
 \|K_\sIR^{\ast(\oo+1)}\ast v\|_{\cV}
 =
 \|K^{\ast\oo}_\sIR\ast v\|_{\cVtilde}~
 \sup_{x\in\bM}\sup_{\mathring z\in\bR}\int_{\bR^\rdim} K_\sIR(x-z) \,\rd \bar z
 \lesssim 
 \|K^{\ast\oo}_\sIR\ast v\|_{\cVtilde}
 \,[\sIR]^{-\sigma}.
\end{equation}
This completes the proof.
\end{proof}

\subsection{Convergence of perturbative expansion}\label{sec:convergence_PT}

\begin{dfn}
For $\tilde R>0$ and $i,m\in\bN_0$ we set
\begin{equation}
 \tilde R^{i,m}:=\frac{\tilde R^{1+2(im_\flat-m)}}{4(i+1)^2\,4(m+1)^2}.
\end{equation}
\end{dfn}

\begin{thm}\label{thm:deterministic_convergence}
There exists a constant $\tilde c\geq 1$ and $\oo\in\bN_0$ such that under assumptions of Theorem~\ref{thm:F_tilde_bound} it holds:
\begin{enumerate}
\item[(A)] $\|K^{m;\oo}_{0,\sIR}\ast \tilde F^{i,m}_{0,\sIR}\|_{\cV^m}
\leq \tilde R^{i,m}
\,[\sIR]^{\tilde\varrho_\varepsilon(i,m)}$,~~$\sIR\in(0,1]$,

\item[(B)] $\|K^{m;\oo}_{\sUV,\sIR}\ast \tilde F^{i,m}_{\sUV,\sIR}\|_{\cV^m}
 \leq \tilde R^{i,m}
 \,[\sIR]^{\varepsilon (i/i_\dagger-1)\vee0}\,
 [\sUV\vee\sIR]^{\tilde\varrho_\varepsilon(i,m)-\varepsilon (i/i_\dagger-1)\vee0}$,~~$\sIR\in(0,1]$,

\item[(C)] $\|K^{m;\oo}_{\sUV,\sIR}\ast (\tilde F^{i,m}_{\sUV,\sIR}-\tilde F^{i,m}_{0,\sIR})\|_{\cV^m}
 \leq 2\delta \tilde R^{i,m}
 \,[\sIR]^{\tilde\varrho_\varepsilon(i,m)-\varepsilon}$,~~$\sIR\in[\sUV,1]$,
\end{enumerate}
where $\tilde R=\tilde c\,R$.
\end{thm}
\begin{rem}
Recall that by Remark~\ref{rem:force_hat_i_m} $\tilde F^{i,m}_{\sUV,0}$ vanishes unless $i\in\{0,\ldots,i_\dagger\}$. The same is true for $\tilde F^{i,m}_{0,0}$ since $\tilde F^{i,m}_{0,0}$ vanishes unless $\varrho(i,m)\leq 0$.
\end{rem}

\begin{proof}
Without loss of generality we can restrict attention to the case $m\leq i m_\flat$ since otherwise $\tilde F^{i,m}_{\sUV,\sIR}$ is identically zero. Theorem~\ref{thm:F_tilde_bound} and Lemma~\ref{lem:kernel_dot_G} imply that there exist $\oo\in\bN_0$ and $\tilde c\geq 1$ such that the theorem holds for all $i\leq i_\dagger$ and $m\in\bN_0$. Moreover we choose $\tilde c\geq 1$ in such a way  that the bound
\begin{equation}\label{eq:thm_conv_bound_dot_G}
 \|\fR_\sUV^{\phantom{\oo}} \fP^{2\oo}_\sIR \dot G_\sIR^{\phantom{\oo}}\|_\cK \leq C^{-1}\,\tilde R\,
 [\sIR]^{\varepsilon-\sigma}\,
 [\sUV\vee\sIR]^{\sigma-\varepsilon}
\end{equation}
is satisfied for all $\sUV\in[0,1]$, $\sIR\in(0,1]$, where $\tilde R=\tilde c\,R$ and
\begin{equation}
 C:=(1/\varrho_\varepsilon^\diamond+1/(\dim(\varPhi)+\varepsilon))\sigma
 +3 i_\dagger m_\flat\,\sigma/\varepsilon,
 \qquad
 \varrho_\varepsilon^\diamond:=\varrho_\varepsilon(i_\dagger+1)-\varepsilon>0.
\end{equation}
The existence of such $\tilde c$ follows from Lemma~\ref{lem:kernel_dot_G}. We note that
\begin{equation}\label{eq:thm_conv_C}
\frac{(m+1)\sigma}{\varrho_\varepsilon(i,m)-\varepsilon} 
+
\frac{(m+1) \sigma}{\varepsilon (i/i_\dagger-1)} 
\leq 
\frac{(m+1)\sigma}{\varrho_\varepsilon^\diamond+(\dim(\varPhi)+\varepsilon)m}
+
\frac{(i m_\flat+1) \sigma}{\varepsilon (i/i_\dagger-1)}
\leq C,
\end{equation}
for $i>i_\dagger$ and $m\leq i m_\flat$. Now let us fix some $i_\circ\in\bN_+$, $i_\circ>i_\dagger$, and $m_\circ\in\bN_0$. Assume that theorem is true for all $i,m\in\bN_0$ such that either $i<i_\circ$, or $i=i_\circ$ and $m>m_\circ$. We shall prove the theorem for all $i=i_\circ$ and $m=m_\circ$. By the flow equation~\eqref{eq:flow_deterministic_i_m},  Remark~\ref{rem:fB1_bound} we have
\begin{multline}
 \|K_{\sUV,\sIR}^{m;\oo}\ast \partial_\sIR \tilde F_{\sUV,\sIR}^{i,m}\|_{\cV^m} 
 \leq 
 \sum_{j=1}^i\sum_{k=0}^m
 (k+1)\,\|\fR_\sUV^{\phantom{\oo}}\fP_\sIR^{2\oo}\dot G_\sIR^{\phantom{\oo}}\|_\cK 
 \\
 \times
 \|K_{\sUV,\sIR}^{k+1;\oo}\ast \tilde F_{\sUV,\sIR}^{j,k+1}\|_{\cV^{k+1}} 
 \,
 \|K_{\sUV,\sIR}^{m-k;\oo}\ast \tilde F_{\sUV,\sIR}^{i-j,m-k}\|_{\cV^{m-k}}
\end{multline}
and
\begin{multline}
 \|K_{\sUV,\sIR}^{m;\oo}\ast \partial_\sIR (\tilde F_{\sUV,\sIR}^{i,m}-\tilde F_{0,\sIR}^{i,m})\|_{\cV^m} 
 \leq 
 \sum_{j=1}^i\sum_{k=0}^m
 (k+1)\,\|\fR_\sUV^{\phantom{\oo}}\fP_\sIR^{2\oo}\dot G_\sIR^{\phantom{\oo}}\|_\cK 
 \\[1mm]
 \times
 \Big( \|K_{\sUV,\sIR}^{k+1;\oo}\ast (\tilde F_{\sUV,\sIR}^{j,k+1}-\tilde F_{0,\sIR}^{j,k+1})\|_{\cV^{k+1}} 
 \,
 \|K_{\sUV,\sIR}^{m-k;\oo}\ast \tilde F_{\sUV,\sIR}^{i-j,m-k}\|_{\cV^{m-k}}
 \\[2mm]
 +
 \|K_{0,\sIR}^{k+1;\oo}\ast \tilde F_{0,\sIR}^{j,k+1}\|_{\cV^{k+1}} 
 \,
 \|K_{\sUV,\sIR}^{m-k;\oo}\ast (\tilde F_{\sUV,\sIR}^{i-j,m-k}-\tilde F_{0,\sIR}^{i-j,m-k})\|_{\cV^{m-k}}\Big).
\end{multline}
These bounds remain valid for $\sUV=0$. Using the inequality
\begin{equation}
 \sum_{j=0}^i\frac{1}{4(1+j)^2\,4(1+i-j)^2}\leq \frac{1}{4\,(1+i)^2}
\end{equation}
we obtain
\begin{equation}
 \tilde R~ \sum_{j=1}^i\sum_{k=0}^m \tilde R^{j,k+1} \tilde R^{i-j,m-k} \leq \tilde R^{i,m}.
\end{equation}
The bounds stated above, the bound~\eqref{eq:varrho_bound_initial} and the induction hypothesis imply
\begin{equation}
 \|K^{m;\oo}_{0,\sIR}\ast\partial_\sIR \tilde F^{i,m}_{0,\sIR}\|_{\cV^m}
 \leq C^{-1}\,(m+1)\,\tilde R^{i,m}\, [\sIR]^{\varrho_\varepsilon(i,m)-\sigma},
\end{equation}
\begin{multline}
 \|K^{m;\oo}_{\sUV,\sIR}\ast\partial_\sIR \tilde F^{i,m}_{\sUV,\sIR}\|_{\cV^m}
 \\
 \leq C^{-1}\,(m+1)\,\tilde R
 \sum_{j=1}^i\sum_{k=0}^m\tilde R^{j,k+1} \tilde R^{i-j,m-k}\,
 [\sIR]^{\varepsilon (i/i_\dagger-1)-\sigma}\, [\sUV\vee\sIR]^{\varrho_\varepsilon(i,m)-\varepsilon (i/i_\dagger-1)}
 \\
 \leq C^{-1}\,(m+1)\,\tilde R^{i,m}\,
 [\sIR]^{\varepsilon (i/i_\dagger-1)-\sigma}\, [\sUV\vee\sIR]^{\varrho_\varepsilon(i,m)-\varepsilon (i/i_\dagger-1)},
\end{multline}
\begin{multline}
 \|K^{m;\oo}_{\sUV,\sIR}\ast\partial_\sIR (\tilde F^{i,m}_{\sUV,\sIR}-\tilde F^{i,m}_{0,\sIR})\|_{\cV^m}
 \\
 \leq 2\delta C^{-1}\,(m+1)\,\tilde R
 \sum_{j=1}^i\sum_{k=0}^m\tilde R^{j,k+1} \tilde R^{i-j,m-k}\,
 [\sIR]^{\varepsilon (i/i_\dagger-1)-\sigma}\, [\sUV\vee\sIR]^{\varrho_\varepsilon(i,m)-\varepsilon i/i_\dagger}
 \\
 \leq 2\delta C^{-1}\, (m+1)\,\tilde R^{i,m}\,
 [\sIR]^{\varepsilon (i/i_\dagger-1)-\sigma}\, [\sUV\vee\sIR]^{\varrho_\varepsilon(i,m)-\varepsilon i/i_\dagger}.
\end{multline}
These bounds are valid for all $\sIR\in(0,1]$. For $i>i_\dagger$ it holds $\tilde F^{i,m}_{\sUV,0}=0$ and
\begin{equation}
 \tilde F^{i,m}_{\sUV,\sIR} 
 =
 \int_0^\sIR\partial_\uIR\tilde F^{i,m}_{\sUV,\uIR}\,\rd\uIR,
 \qquad
 \tilde F^{i,m}_{\sUV,\sIR}-\tilde F^{i,m}_{0,\sIR}
 =
 \int_0^\sIR\partial_\uIR(\tilde F^{i,m}_{\sUV,\uIR}-\tilde F^{i,m}_{0,\uIR})\,\rd\uIR.
\end{equation}
As a result, the induction step is a consequence of the bounds~\eqref{eq:thm_conv_C} and Lemma~\ref{lem:integration} applied with $n=1$, $\epsilon=\varepsilon\,(i/i_\dagger-1)>0$. This completes the proof.
\end{proof}

\begin{dfn}
Let $\oo\in\bN_0$ be as in Theorem~\ref{thm:deterministic_convergence}. 
\end{dfn}

\begin{dfn}
The families of sets $\cV_{\sUV,\sIR}$, $\cB_{\sUV,\sIR}$, $(\sUV,\sIR)\in[0,1]^2\setminus\{(0,0)\}$, are defined as follows
\begin{equation}
\begin{aligned}
 \cV_{\sUV,\sIR}&:=\{J_\sUV\ast K^{\ast\oo}_\sIR\ast\varphi\in\cV\,|\,\varphi\in\cV\},
 \\
 \cB_{\sUV,\sIR}&:=\{J_\sUV\ast K^{\ast\oo}_\sIR\ast\varphi\in\cV\,|\,\varphi\in\cV,\,\|\varphi\|_\cV< 3/4\}.
\end{aligned} 
\end{equation}
For sufficiently small $T$ and $(\sUV,\sIR)\in[0,T]^2\setminus\{(0,0)\}$ we define the functionals
\begin{equation}
 \tilde F_{\sUV,\sIR}\,:\, \cB_{\sUV,\sIR}\to \cD,
 \qquad
 \rD\tilde F_{\sUV,\sIR}\,:\, \cB_{\sUV,\sIR}\times \cV_{\sUV,\sIR} \to \cD
\end{equation}
by the following formulas
\begin{equation}\label{eq:deterministic_eff_force_series}
 \langle \tilde F_{\sUV,\sIR}[\varphi],\psi\rangle
 :=\sum_{i=0}^\infty \sum_{m=0}^\infty \lambda^i\,\langle \tilde F^{i,m}_{\sUV,\sIR},\psi\otimes\varphi^{\otimes m}\rangle,
\end{equation}
\begin{equation}\label{eq:deterministic_D_eff_force_series}
 \langle \rD\tilde F_{\sUV,\sIR}[\varphi,\zeta],\psi\rangle
 :=\sum_{i=0}^\infty \sum_{m=0}^\infty \lambda^i\,(m+1)\,
 \langle \tilde F^{i,m+1}_{\sUV,\sIR},\psi\otimes\zeta\otimes\varphi^{\otimes m}\rangle,
\end{equation}
where $\psi\in C^\infty_\rc(\bM)$, $\varphi\in\cB_{\sUV,\sIR}$ and $\zeta\in\cV_{\sUV,\sIR}$.
\end{dfn}

\begin{rem}
By Lemma~\ref{lem:kernel_u_v} if $\varphi\in\cB_{\sUV,\sIR}$, then $\varphi\in\cB_{\sUV,\uIR}$ for any $\uIR$ in a small enough neighbourhood of $\sIR$.
\end{rem}

\begin{thm}\label{thm:eff_force_convergence}
The bounds established in Theorem~\ref{thm:deterministic_convergence} imply that for $T$ such that
\begin{equation}\label{eq:T_bound}
0<T\leq (|\lambda|\vee\tilde R)^{-4m_\flat (i_\triangleright+1)/\varepsilon}\leq  1
\end{equation}
the functionals $\tilde F_{\sUV,\sIR}$ and $\rD\tilde F_{\sUV,\sIR}$ are well defined for all $(\sUV,\sIR)\in[0,T]^2\setminus\{(0,0)\}$ and the following statements are true.
\begin{enumerate}
\item[(A)]
For any $\sIR\in[0,T]$, $\varphi\in\cB_{\sUV,\sIR}$ and $\zeta\in\cV_{\sUV,\sIR}$ it holds
\begin{equation}
\begin{gathered}
 \|K^{\ast\oo}_\sIR\ast\tilde F_{\sUV,\sIR}[\varphi]\|_{\cV} \leq [\sUV\vee\sIR]^{2\varepsilon-\sigma},
 \\
 \|K^{\ast\oo}_\sIR\ast\rD\tilde F_{\sUV,\sIR}[\varphi,\zeta]\|_{\cV} \leq [\sUV\vee\sIR]^{2\varepsilon-\sigma}\,
 \|\fR_\sUV\fP_\sIR^\oo\zeta\|_\cV.
\end{gathered} 
\end{equation}
These bounds are also valid for $\sUV=0$ and $\sIR\in(0,T]$.
\item[(B)]
For any $\sIR\in[0,T]$, $\varphi_1,\varphi_2\in\cB_{\sUV,\sIR}$ it holds
\begin{equation}
 \|K^{\ast\oo}_\sIR\ast(\tilde F_{\sUV,\sIR}[\varphi_1]-\tilde F_{\sUV,\sIR}[\varphi_2])\|_{\cV} \leq [\sUV\vee\sIR]^{2\varepsilon-\sigma}\,\|\fR_\sUV\fP_\sIR^\oo(\varphi_1-\varphi_2)\|_\cV.
\end{equation}
This bound is also valid for $\sUV=0$ and $\sIR\in(0,T]$.
\item[(C)] For any $\sUV,\sIR\in(0,T]^2$ and $\varphi\in\cB_{\sUV,\sIR}$ it holds
\begin{equation}
 \|K^{\ast\oo}_\sIR\ast(\tilde F_{\sUV,\sIR}[\varphi]-\tilde F_{0,\sIR}[\varphi])\|_{\cV} \leq 2\delta\,[\sIR]^{\varepsilon-\sigma}.
\end{equation}
\item[(D)]
Given $\sIR\in(0,T]$ and $\varphi\in\cB_{\sUV,\sIR}$ the flow equation 
\begin{equation}
 \partial_\uIR \tilde F_{\sUV,\uIR}[\varphi]
 =
 -\rD \tilde F_{\sUV,\uIR}[\varphi,\dot G_\uIR\ast \tilde F_{\sUV,\uIR}[\varphi]]
\end{equation}
is satisfied for all $\uIR$ in some sufficiently small neighbourhood of $\sIR$. This remains valid for $\sUV=0$.
\end{enumerate}
\end{thm}
\begin{proof}
We first prove the following bound
\begin{equation}
 \tilde\varrho_\varepsilon(i,m)- \varepsilon i/(i_\triangleright+1) \geq 2\varepsilon-\sigma
\end{equation}
for all $i,m\in\bN_0$. Since $\tilde\varrho_\varepsilon(i,m)=\varrho_\varepsilon(i,m)\vee(3\varepsilon-\sigma)$ the bound holds trivially for all $i\in\{1,\ldots,i_\triangleright+1\}$. For $i>i_\triangleright+1$ we use
\begin{equation}
 \varrho_\varepsilon(i,m)- \varepsilon i/(i_\triangleright+1) \geq \varrho_\varepsilon(i_\triangleright+1)-\varepsilon + (i-i_\triangleright-1)(\dim(\lambda)-\varepsilon m_\flat-\varepsilon) 
\end{equation}
and $\varrho_\varepsilon(i_\triangleright+1)>3\varepsilon-\sigma$ as well as $\dim(\lambda)>\varepsilon(m_\flat+3)$. As a result, we obtain the following estimate
\begin{multline}
 |\lambda|^i\, \tilde R^{1+2(i m_\flat-m)}
 [\sUV\vee\sIR]^{\varrho_\varepsilon(i,m)\vee(3\varepsilon-\sigma)} 
 \\
 \leq 
 |\lambda|^i \,
 \tilde R^{1+2(i m_\flat-m)}\,T^{\varepsilon i/(i_\triangleright+1)}
 [\sUV\vee\sIR]^{2\varepsilon-\sigma}
 \leq [\sUV\vee\sIR]^{2\varepsilon-\sigma}
\end{multline}
for all $i,m\in\bN_0$, $\sUV,\sIR\in[0,T]^2\setminus\{(0,0)\}$ and $\lambda\in\bR$, $\tilde R>1$ such that the bound~\eqref{eq:T_bound} is satisfied. Consequently, for $T$ satisfying the condition~\eqref{eq:T_bound} by Theorem~\ref{thm:deterministic_convergence} we have
\begin{enumerate}
\item[(A${}_0$)] $|\lambda|^i\,\|K^{m;\oo}_{0,\sIR}\ast \tilde F^{i,m}_{0,\sIR}\|_{\cV^m}
\leq 
[\sIR]^{2\varepsilon-\sigma}/(4(1+i)^2\,4(1+m)^2)$,
~~$\sIR\in(0,T]$,

\item[(B${}_0$)]
$|\lambda|^i\,\|K^{m;\oo}_{\sUV,\sIR}\ast \tilde F^{i,m}_{\sUV,\sIR}\|_{\cV^m}
\leq 
[\sUV\vee\sIR]^{2\varepsilon-\sigma}/(4(1+i)^2\,4(1+m)^2)$,
 ~~$\sIR\in[0,T]$,
\item[(C${}_0$)]
$|\lambda|^i\,\|K^{m;\oo}_{\sUV,\sIR}\ast (\tilde F^{i,m}_{\sUV,\sIR}-\tilde F^{i,m}_{0,\sIR})\|_{\cV^m}
\leq
2\delta\,[\sIR]^{\varepsilon-\sigma}/(4(1+i)^2\,4(1+m)^2)$,
~~$\sIR\in(0,T]$.
\end{enumerate}
Note that for $\sIR\in(0,\sUV]$ the third bound is a consequence of the first two.

Let $\varphi,\varphi_1,\varphi_2,\zeta\in\cV$ be such that $\|\varphi\|_\cV,\|\varphi_1\|_\cV,\|\varphi_2\|_\cV<3/4$. It follows from Def.~\ref{dfn:cV} of the norm $\|\Cdot\|_{\cV^m}$ that the following bounds are satisfied:
\begin{enumerate}
\item[(A)] 
$
 \|K^{\ast\oo}_\sIR\ast\tilde F_{\sUV,\sIR}[J_\sUV\ast K^{\ast\oo}_\sIR\ast\varphi]\|_{\cV} 
 \leq \sum_{i=0}^\infty\sum_{m=0}^\infty |\lambda|^i\,\|K^{m;\oo}_{\sUV,\sIR}\ast \tilde F^{i,m}_{\sUV,\sIR}\|_{\cV^m} \,\|\varphi\|_\cV^m,
$

$
 \|K^{\ast\oo}_\sIR\ast\rD\tilde F_{\sUV,\sIR}[J_\sUV\ast K^{\ast\oo}_\sIR\ast\varphi,J_\sUV\ast K^{\ast\oo}_\sIR\ast\zeta]\|_{\cV} 
 \\
 \leq
 \sum_{i=0}^\infty\sum_{m=0}^\infty (m+1)\,|\lambda|^i\,\|K^{m+1;\oo}_{\sUV,\sIR}\ast \tilde F^{i,m+1}_{\sUV,\sIR}\|_{\cV^{m+1}} \,\|\zeta\|_\cV\, \|\varphi\|_\cV^m,
$
 
\item[(B)]
$
 \|K^{\ast\oo}_\sIR\ast(\tilde F_{\sUV,\sIR}[J_\sUV\ast K^{\ast\oo}_\sIR\ast\varphi_1]-\tilde F_{\sUV,\sIR}[J_\sUV\ast K^{\ast\oo}_\sIR\ast\varphi_2])\|_{\cV}
 \\ 
 \leq \sum_{i=0}^\infty\sum_{m=0}^\infty\sum_{k=0}^m
 |\lambda|^i\,\|K^{m+1;\oo}_{\sUV,\sIR}\ast \tilde F^{i,m+1}_{\sUV,\sIR}\|_{\cV^{m+1}}\, 
 \|\varphi_1\|_\cV^k\,\|\varphi_2\|_\cV^{m-k}\, 
 \|\varphi_1-\varphi_2\|_\cV,
$

\item[(C)]
$
 \|K^{\ast\oo}_\sIR\ast(\tilde F_{\sUV,\sIR}[J_\sUV\ast K^{\ast\oo}_\sIR\ast\varphi]-\tilde F_{0,\sIR}[J_\sUV\ast K^{\ast\oo}_\sIR\ast\varphi])\|_{\cV}
 \\ 
 \leq \sum_{i=0}^\infty\sum_{m=0}^\infty |\lambda|^i\,
 \|K^{m;\oo}_{\sUV,\sIR}\ast (\tilde F^{i,m}_{\sUV,\sIR}-\tilde F^{i,m}_{0,\sIR})\|_{\cV^m}\, 
 \|\varphi\|_\cV^m.
$
\end{enumerate}
In order to show Parts (A), (B) and (C) of the theorem it is enough use the above bounds and the estimates
\begin{equation}
 \sum_{i=0}^\infty\sum_{m=0}^\infty \frac{1}{4(1+i)^2}\frac{(3/4)^m}{4(1+m)^2}\leq \sum_{i=0}^\infty\sum_{m=0}^\infty \frac{1}{4(1+i)^2}\frac{(3/4)^m}{4(1+m)} \leq 1.
\end{equation}
Part (D) follows from the fact that the coefficients $F^{i,m}_{\sUV,\sIR}$ satisfy the flow equation~\eqref{eq:flow_deterministic_i_m}, the absolute convergence of the series defining the functionals $\tilde F_{\sUV,\sIR}$ and $\rD\tilde F_{\sUV,\sIR}$ and the bound $\|\fR_\sUV^{\phantom{\oo}}\fP_\sIR^{2\oo}\dot G_\sIR^{\phantom{\oo}}\|_\cK \lesssim [\sUV\vee\sIR]^{\sigma-\varepsilon}[\sIR]^{\varepsilon-\sigma}$.
\end{proof}

\subsection{Local solution theory}\label{sec:local_solution}

\begin{lem}\label{lem:solution_aux}
Let $T\in(0,1]$ be as in Theorem~\ref{thm:eff_force_convergence}. Set $\tilde f_{0,T}:=\tilde F_{0,T}[0]$ and $\tilde f_{\sUV,T}:=\tilde F_{\sUV,T}[0]$. We have:
\begin{enumerate}
\item[(A${}_1$)] 
$\|K^{\ast\oo}_\sIR\ast \tilde f_{0,T}\|_{\cV}\leq [\sIR]^{2\varepsilon-\sigma}$, $\sIR\in(0,T]$,
\item[(B${}_1$)] 
$\|K^{\ast\oo}_\sIR\ast \tilde f_{\sUV,T}\|_{\cV}\leq [\sUV\vee\sIR]^{2\varepsilon-\sigma}$, $\sIR\in[0,T]$,
\item[(C${}_1$)] 
$\|K^{\ast\oo}_\sIR\ast (\tilde f_{\sUV,T}-\tilde f_{0,T})\|_{\cV}\leq 6\delta\, [\sIR]^{\varepsilon-\sigma}$, $\sIR\in(0,T]$
\end{enumerate}
and
\begin{enumerate}
\item[(A${}_2$)] 
 $\|\fP_\sIR^{\oo}(G_\sIR-G_T)\ast\tilde f_{0,T}\|_{\cV} \leq 1/3$, $\sIR\in[0,T]$,
\item[(B${}_2$)] 
 $\|\fR_\sUV^{\phantom{\oo}}\fP_\sIR^{\oo}(G_\sIR-G_T)\ast\tilde f_{\sUV,T}\|_{\cV}\leq 1/3$, $\sIR\in[0,T]$,
\item[(C${}_2$)] 
 $\|\fP_\sIR^{\oo}(G_\sIR-G_T)\ast(\tilde f_{\sUV,T}-\tilde f_{0,T})\|_{\cV}\leq \delta$, $\sIR\in[0,T]$.
\end{enumerate}
Moreover, $\tilde f_{\sUV,T} = \tilde F_{\sUV,\sIR}[(G_\sIR-G_T) \ast \tilde f_{\sUV,T}]$ for $\sIR\in[0,T]$. 
\end{lem}
\begin{proof}
We first note that the lemma is true for $\sIR=T$ by Theorem~\ref{thm:eff_force_convergence}~(A), (C). Next, choose $\uv\in[0,T]$ and assume that the bounds (A${}_1$), (B${}_1$), (C${}_1$) hold true for all $\sIR\in[\uv,T]$. Then by Lemma~\ref{lem:kernel_u_v} we have
\begin{equation}
\begin{gathered}
  \|K^{\ast\oo}_{\uIR}\ast \tilde f_{0,T}\|_{\cV}
 \leq
 \tilde R\, \|K^{\ast\oo}_{\sIR}\ast \tilde f_{0,T}\|_{\cV}
 \leq
 \tilde R\,[\sIR]^{2\varepsilon-\sigma}
 \leq
 \tilde R\,[\uIR]^{2\varepsilon-\sigma},
 \\
 \|K^{\ast\oo}_{\uIR}\ast \tilde f_{\sUV,T}\|_{\cV}
 \leq
 \tilde R\, \|K^{\ast\oo}_{\sIR}\ast \tilde f_{\sUV,T}\|_{\cV}
 \leq
 \tilde R\,[\sUV\vee\sIR]^{2\varepsilon-\sigma}
 \leq
 \tilde R\,[\sUV\vee\uIR]^{2\varepsilon-\sigma},
 \\
 \|K^{\ast\oo}_{\uIR}\ast (\tilde f_{\sUV,T}-\tilde f_{0,T})\|_{\cV}
 \leq 
 \tilde R\,\|K^{\ast\oo}_{\sIR}\ast (\tilde f_{\sUV,T}-\tilde f_{\sUV,0})\|_{\cV}
 \leq 
 6\delta \tilde R\, [\sIR]^{\varepsilon-\sigma}
 \leq
 6\delta \tilde R\,[\uIR]^{\varepsilon-\sigma}
\end{gathered} 
\end{equation}
for all $\uIR\in[\tau\sIR,\sIR]$, $\sIR\in[\uv,T]$ and some $\tau\in(0,1)$ that depends only on $\tilde R$. Recall the bound~\eqref{eq:thm_conv_bound_dot_G}
\begin{equation}
 \|\fR_\sUV^{\phantom{\oo}} \fP^{2\oo}_\sIR \dot G_\sIR^{\phantom{\oo}}\|_\cK \leq C^{-1}\,\tilde R~
 [\sIR]^{\varepsilon-\sigma}\,
 [\sUV\vee\sIR]^{\sigma-\varepsilon},
 \qquad C\leq 3\sigma/\varepsilon,
\end{equation}
and note that the bound~\eqref{eq:T_bound} implies $\tilde R^2\,T^\varepsilon\leq 1$. Consequently, we obtain
\begin{equation}
\begin{gathered}
 \|\fP_\sIR^{\oo}(G_\sIR-G_T)\ast\tilde f_{0,T}\|_{\cV}\leq 
 \int_\sIR^T \|\fP_\uIR^{2\oo}\dot G_\uIR^{\phantom{\oo}}\|_\cK\,
 \|K^{\ast\oo}_\uIR\ast \tilde f_{0,T}\|_{\cV}\,\rd \uIR \leq 1/3,
 \\
 \|\fR_\sUV^{\phantom{\oo}}\fP_\sIR^{\oo}(G_\sIR-G_T)\ast\tilde f_{\sUV,T}\|_{\cV}\leq 
 \int_\sIR^T \|\fR_\sUV^{\phantom{\oo}}\fP_\uIR^{2\oo}\dot G_\uIR^{\phantom{\oo}}\|_\cK\,
 \|K^{\ast\oo}_\uIR\ast \tilde f_{\sUV,T}\|_{\cV}\,\rd \uIR \leq 1/3,
 \\
 \|\fP_\sIR^{\oo}(G_\sIR-G_T)\ast(\tilde f_{\sUV,T}-\tilde f_{0,T})\|_{\cV}\leq 
 \int_\sIR^T\|\fP_\uIR^{2\oo}\dot G_\uIR^{\phantom{\oo}}\|_\cK\,
 \|K^{\ast\oo}_\uIR\ast (\tilde f_{\sUV,T}-\tilde f_{0,T})\|_{\cV}\,\rd \uIR \leq 2\delta
\end{gathered} 
\end{equation}
for all $\sIR\in[\tau\uv,T]$. This proves the bounds (A${}_2$), (B${}_2$), (C${}_2$) and shows that 
\begin{equation}
 (G_\sIR-G_T)\ast\tilde f_{0,T}\in\cB_{0,\sIR},
 \quad
 (G_\sIR-G_T)\ast\tilde f_{\sUV,T}\in\cB_{\sUV,\sIR},
 \quad
 (G_\sIR-G_T)\ast(\tilde f_{\sUV,T}-\tilde f_{0,T})\in\cV_{0,\sIR}
\end{equation}
for all $\sIR\in[\tau\uv,T]$. Using Theorem~\ref{thm:eff_force_convergence} and Remark~\ref{rem:intro_flow_solution} we prove the identity $\tilde f_{\sUV,T} = \tilde F_{\sUV,\sIR}[(G_\sIR-G_T) \ast \tilde f_{\sUV,T}]$ for all $\sIR\in[\tau\uv,T]$. By the last identity we obtain
\begin{equation}
 \|K^{\ast\oo}_\sIR\ast \tilde f_{\sUV,T}\|_{\cV} = \|K^{\ast\oo}_\sIR\ast \tilde F_{\sUV,\sIR}[(G_\sIR-G_T) \ast \tilde f_{\sUV,T}]\|_\cV
\end{equation}
and
\begin{multline}
 \|K^{\ast\oo}_\sIR\ast (\tilde f_{\sUV,T}-\tilde f_{0,T})\|_{\cV} \leq \|K^{\ast\oo}_\sIR\ast (\tilde F_{\sUV,\sIR}[(G_\sIR-G_T) \ast \tilde f_{\sUV,T}]-\tilde F_{0,\sIR}[(G_\sIR-G_T) \ast \tilde f_{\sUV,T}])\|_\cV
 \\
 +\|K^{\ast\oo}_\sIR\ast (\tilde F_{0,\sIR}[(G_\sIR-G_T) \ast \tilde f_{\sUV,T}]-\tilde F_{0,\sIR}[(G_\sIR-G_T) \ast \tilde f_{0,T}])\|_\cV.
\end{multline}
As a result, the bounds (A${}_1$), (B${}_1$), (C${}_1$) for all $\sIR\in[\tau\uv,T]$ follow from Theorem~\ref{thm:eff_force_convergence}~(A), (B), (C). To complete the proof of the theorem it is enough to apply the above reasoning recursively with $\uv=\tau^n T$, $n\in\bN_0$.
\end{proof}

The above lemma applied with $\sIR=0$ implies immediately the following theorem.
\begin{thm}\label{thm:solution_local_aux}
Set $\tilde \varPhi_{0}:=(G-G_T)\ast \tilde f_{0,T}$ and $\tilde \varPhi_{\sUV}:=(G-G_T)\ast \tilde f_{\sUV,T}$. It holds
\begin{equation}\label{eq:thm_sol_bound}
 \|\tilde \varPhi_0\|_{\cV}\leq 1/3,
 \qquad
 \|\fR_\sUV \tilde \varPhi_\sUV\|_{\cV}\leq 1/3,
 \qquad
 \|\tilde \varPhi_\sUV-\tilde\varPhi_0\|_{\cV}\leq 2\delta.
\end{equation}
for any $\sUV\in[0,T]$. Moreover, the following equation is satisfied
\begin{equation}\label{eq:thm_sol_identity}
 \tilde \varPhi_\sUV = (G-G_T)\ast \tilde F_\sUV[\tilde \varPhi_\sUV].
\end{equation}
\end{thm}

\begin{rem}
Recall that a boundary data and an enhanced data were introduced in Def.~\ref{dfn:boundary_data_initial} and $\beta\equiv\beta_\varepsilon = -\dim(\phi)-\varepsilon$.
\end{rem}

\begin{thm}\label{thm:solution_local}
Let $t\in[0,\infty)$ and $\mathtt{u}\in C^\infty_\rc(\bR)$ be such that $\mathtt{u}=1$ on some neighborhood of $[t-i_\dagger(i_\dagger+1),t+6i_\dagger]$. There exists a constant $\breve c\geq 1$ such that the following is true. Fix some $R>1$, $\delta\in(0,1]$ and an admissible boundary data for $(\varPhi_0^\triangleright,\breve\varPhi_0)$. Let $\sUV\in(0,\delta^{1/\varepsilon}]$ and a boundary data for $(\varPhi_\sUV^\triangleright,\breve\varPhi_\sUV)$ be arbitrary such that all of the conditions specified in Assumptions~\ref{ass:deterministic} and~\ref{ass:deterministic_initial} are satisfied. There exists $T\in(0,1)$ depending only on the value of the constant $R$ such that for all $\mathring x\in[t,t+T]$ and all $\sIR\in(0,1]$ it holds:
\begin{enumerate}
\item[(A)] $\|\bar K_\sIR^{\ast\oo}\ast\breve \varPhi_0(\mathring x,\Cdot)\|_{L^\infty(\bT)}\leq \breve c\, R\,[\sIR]^{\beta_\varepsilon}$,

\item[(B)] $\|\fR_\sUV \breve \varPhi_\sUV(\mathring x,\Cdot)\|_{L^\infty(\bT)}\leq \breve c\,\delta R\,[\sUV]^{\beta_\varepsilon}$,

\item[(C)] $\|\bar K_\sIR^{\ast\oo}\ast(\breve \varPhi_\sUV(\mathring x,\Cdot)-\breve\varPhi_0(\mathring x,\Cdot))\|_{L^\infty(\bT)}\leq \breve c\, \delta R\,[\sIR]^{\beta_\varepsilon}$,
\end{enumerate}
where 
\begin{equation}
 \breve\varPhi_\sUV:=\varPhi_\sUV^\vartriangle+\tilde\varPhi_\sUV\in C([t,t+T],\cC^\gamma(\bT)),
 \quad
 \breve\varPhi_0:=\varPhi_0^\vartriangle+\tilde\varPhi_0\in C([t,t+T],\sC^\beta(\bT)).
\end{equation}
Moreover, $\breve\varPhi_\sUV$ satisfies the following equation
\begin{equation}\label{eq:equation_breve_varPhi_loc}
 \breve \varPhi_\sUV = G\ast (1_{(t,\infty)} (\hat F_\sUV[\breve \varPhi_\sUV]+ \check f_\sUV^\triangleright)+ \delta_t\otimes\phi^\vartriangle_\sUV).
\end{equation}
\end{thm}
\begin{proof}
The bounds (A), (B), (C) follow immediately from the previous theorem and Lemma~\ref{lem:initial_properties}. In particular, bound~(B) implies that $\breve\varPhi_\sUV\in C([t,t+T],\cC^\gamma(\bT))$. Recall that $\supp\,G_T\subset [T,\infty)\times\bR^\rdim$. In order to prove that $\breve\varPhi_\sUV$ satisfies Eq.~\eqref{eq:equation_breve_varPhi_loc} in the domain $[t,t+T]\times\bR^\rdim$ it is enough to note that in this domain the equation $\tilde \varPhi_\sUV = (G-G_T)\ast \tilde F_\sUV[\tilde \varPhi_\sUV]$ is equivalent to the equation $\tilde \varPhi_\sUV = G\ast \tilde F_\sUV[\tilde \varPhi_\sUV]$ and it holds $\tilde F_\sUV[\varphi] = 1_{(t,\infty)}(\hat F_\sUV[\varphi+\varPhi_\sUV^\vartriangle]+\check f_\sUV^\triangleright)$ and $\varPhi_\sUV^\vartriangle=G\ast(\delta_t\otimes\phi_\sUV^\vartriangle)$ in the domain $[t,t+T]\times\bR^\rdim$. This finishes the proof.
\end{proof}

\subsection{Construction of maximal solution}\label{sec:maximal_solution}

In order to construct the universal function $\breve\varPhi_0\in C([0,\breve T_0),\sC^\beta(\bT))$ with maximal life time and conclude the construction we first discuss briefly the classical solution theory for the equation $\breve\varPhi_\sUV = G\ast 1_{(0,\infty)} (\hat F_\sUV[\breve\varPhi_\sUV]+ \check f_\sUV^\triangleright)+ \delta_0\otimes\phi^\vartriangle_\sUV$ in the space $C([0,T],\cC^\gamma(\bT))$. The classical existence and uniqueness of result is stated in Theorem~\ref{thm:regular_eq}. We use this result to prove the uniqueness of the maximal function $\breve\varPhi_0$ constructed in Theorem~\ref{thm:solution_max}.

\begin{rem}
The spaces $\cC^\gamma(\bT)$ and $C([t,t+T],\cC^\gamma(\bT))$, where $\gamma\equiv\gamma_\varepsilon=\sigma-\varepsilon$, were introduced in Def.~\ref{dfn:C_gamma} and Def.~\ref{dfn:C_gamma_time}. Recall that
\begin{equation}
 \|\phi\|_{\cC^\gamma(\bT)}=\|\fR\phi\|_{L^\infty(\bT)},
 \qquad
 \|\varphi\|_{C([t,t+T],\cC^\gamma(\bT))} = \|\fR\varphi\|_{L^\infty([t,t+T]\times\bT)}.
\end{equation}
By $C_R([t,t+T],\cC^\gamma(\bT))$ we denote the ball in the space $C([t,t+T],\cC^\gamma(\bT))$ of radius $R>0$ centered at zero.
\end{rem}

\begin{rem}
Note that $G_T=0$ in the domain $[0,T]\times\bR^\rdim$. By Lemma~\ref{lem:kernel_dot_G} and Remark~\ref{rem:dot_G_1} there exists $R\geq 1$ such that
\begin{equation}
 \|\fR(G-G_T)\|_{L^1(\bM)} \leq  R\, [T]^\varepsilon,
 \quad
 \sup_{\mathring x\in\bR}\|(G-G_T)(\mathring x,\Cdot)\|_{L^1(\bR^\rdim)} \leq R.
\end{equation}
\end{rem}

\begin{rem}\label{rem:hat_force}
Recall that a boundary data $f^{0,0,0}_\sUV\equiv \varXi_\sUV \in C(\bH)$, $f^{i,m,a}_\sUV\in\bR$, $(i,m,a)\in\frI$, is such that $f^{i,m,a}_\sUV=0$ unless $i\leq i_\flat$ and $m\leq m_\flat$ and $a\in\bar\frM^m$, $|a|<\sigma$. It follows that the force coefficients $\hat F^{i,m}_{\sUV}=\hat F^{i,m}_{\sUV,0}\in \cD^{m;0}$ introduced in Def.~\ref{dfn:eff_force_hat} vanish unless $i\leq i_\dagger=i_\flat+i_\triangleright m_\flat$ and $m\leq m_\flat$ and the following bound
\begin{equation}
 \|(\delta_\bM\otimes J^{\otimes m})\ast \hat F^{i,m}_{\sUV}\|_{L^\infty([t,t+T]\times\bT)} \leq R
\end{equation}
is true for all $i,m\in\bN_0$, where the constant $R>1$ depends on the choice of the boundary data and the time interval $[t,t+T]$. Recall that $\check f_\sUV^\triangleright=\sum_{i=0}^{i_\triangleright}\lambda^i \fQ G_1\ast f_\sUV^i$ and the force $\hat F_\sUV[\varphi]$ is defined by the equality
\begin{equation}
 \langle\hat F_\sUV[\varphi],\psi\rangle:=\sum_{i=0}^\infty\sum_{m=0}^\infty \langle\hat F^{i,m}_\sUV,\psi\otimes\varphi^{\otimes m}\rangle,
\end{equation}
where $\psi,\varphi\in C^\infty_\rc(\bM)$. As a result, for any $\varphi,\varphi_1,\varphi_2\in C_{\breve R}([t,t+T],\cC^\gamma(\bT))$ it holds
\begin{equation}
\begin{gathered}
\|\hat F_\sUV[\varphi]\|_{L^\infty([t,t+T]\times\bT)}\leq \sum_{i=0}^{i_\dagger}\sum_{m=0}^{m_\flat}\lambda^i R\, \breve R^m,
\qquad
\|\check f_\sUV^\triangleright\|_{L^\infty([t,t+T]\times\bT)} \leq \sum_{i=0}^{i_\triangleright} \lambda^i R,
 \\
 \|\hat F_\sUV[\varphi_1]-\hat F_\sUV[\varphi_2]\|_{L^\infty([t,t+T]\times\bT)}\leq \sum_{i=0}^{i_\dagger}\sum_{m=0}^{m_\flat}m\,\lambda^i R\, \breve R^{m-1}\,\|\varphi_1-\varphi_2\|_{C([t,t+T],\cC^\gamma(\bT))}.
\end{gathered} 
\end{equation}
\end{rem}

\begin{rem}
Let $\phi_\sUV^\vartriangle\in \cC^\gamma(\bT)$ be such that $\|\phi_\sUV^\vartriangle\|_{\cC^\gamma(\bT)}\leq R$ for some $R>1$.  Then for any $T>0$ it holds $$G\ast(\delta_{t}\otimes\phi_\sUV^\vartriangle)=(G-G_T)\ast(\delta_{t}\otimes\phi_\sUV^\vartriangle)\in C([t,t+T],\cC^\gamma(\bT))$$ and
\begin{equation}
 \|G\ast(\delta_{t}\otimes\phi_\sUV^\vartriangle)\|_{C([t,t+T],\cC^\gamma(\bT))}\leq R\,
 \sup_{\mathring x\in\bR}\|(G-G_T)(\mathring x,\Cdot)\|_{L^1(\bar\bM)}
 \leq R^2.
\end{equation}
\end{rem}

\begin{lem}\label{lem:regular_eq}
Let $t\in[0,\infty)$, $R>1$ and $\mathtt{u}\in C^\infty_\rc(\bR)$ be such that
\begin{equation}
 \|\phi_\sUV^\vartriangle\|_{\cC^\gamma(\bT)}\leq R,
 \qquad
 \|\mathtt{u}\varXi_\sUV\|_{L^\infty(\bH)}\leq R,
 \qquad
 |f^{i,m,a}_{\sUV}|\leq R,\quad (i,m,a)\in\frI,
\end{equation}
and $\mathtt{u}=1$ on some neighbourhood of $[t-i_\triangleright(i_\triangleright+1)-2,t+1]$. There exists $T\in(0,1)$ depending only on $R$ such that the equation 
\begin{equation}\label{eq:equation_breve_lemma_global}
 \breve\varPhi_\sUV = G\ast (1_{(t,\infty)} (\hat F_\sUV[\breve\varPhi_\sUV]+ \check f_\sUV^\triangleright) +\delta_{t}\otimes\phi_\sUV^\vartriangle)
\end{equation}
has a unique solution $\breve\varPhi_\sUV$ in the space $C([t,t+T],\cC^\gamma(\bT))$.
\end{lem}
\begin{rem}
Since $\supp\,G_T\subset [T,\infty)\times\bR^\rdim$ for any $t\in[0,\infty)$ Eq.~\eqref{eq:equation_breve_lemma_global} posed in the space $C([t,t+T],\cC^\gamma(\bT))$ is equivalent to the same equation with $G$ replaced by $G-G_T$.  
\end{rem}
\begin{proof}
The proof is very elementary and we give only its outline. Let $\breve R=2R^2$. It follows from the above remarks that for sufficiently small $T>0$ depending on $R$ the map $S_\sUV$ defined by the equality
\begin{equation}
 S_\sUV[\varphi]:=(G-G_T)\ast (1_{(t,\infty)} (\hat F_\sUV[\varphi]+ \check f_\sUV^\triangleright) +\delta_{t}\otimes\phi_\sUV^\vartriangle)
\end{equation}
leaves invariant the ball $C_{\breve R}([t,t+T],\cC^\gamma(\bT))$ and has the unique fixed point $\varPhi_\sUV\in C_{\breve R}([t,t+T],\cC^\gamma(\bT))$.

To prove the uniqueness of the solution assume that there is another fixed point $\ubar{\varPhi}_\sUV$ in some larger ball $C_{\breve{\mathrm R}}([t,t+T],\cC^\gamma(\bT))$. Then there exists some time \mbox{$\mathrm T\in(0,T]$} such that the map $S_\sUV$ leaves invariant the above ball and has a unique fixed point there. This fixed point has to coincide with the restriction of both $\varPhi_\sUV$ and $\ubar{\varPhi}_\sUV$ to the time interval $[t,t+\mathrm T]$. In order to prove that these functions coincide in the whole time interval $[t,t+T]$ we iterate the above argument using the lemma below.
\end{proof}

\begin{lem}\label{lem:patching}
Suppose that for some $0\leq t_1< t_2 \leq s_1< s_2<\infty$ the function $\breve\varPhi^{(1)}_\sUV\in C([t_1,s_1],\cC^\gamma(\bT))$ solves the equation
\begin{equation}\label{eq:patching1}
 \breve\varPhi_\sUV = G\ast (1_{(t_1,\infty)} (\hat F_\sUV[\breve\varPhi_\sUV]+ \check f_\sUV^\triangleright) +\delta_{t_1}\otimes\phi_\sUV^\vartriangle)
\end{equation}
and the function $\breve\varPhi^{(2)}_\sUV\in C([t_2,s_2],\cC^\gamma(\bT))$ solves the equation
\begin{equation}\label{eq:patching2}
 \breve\varPhi_\sUV = G\ast (1_{(t_2,\infty)} (\hat F_\sUV[\breve\varPhi_\sUV]+ \check f_\sUV^\triangleright) +\delta_{t_2}\otimes\breve\varPhi^{(1)}_\sUV(t_2,\Cdot)).
\end{equation}
Then the function $\breve\varPhi^{(12)}_\sUV\in C([t_2,s_1],\cC^\gamma(\bT))$ obtained by restricting $\breve\varPhi^{(1)}_\sUV$ to the time interval $[t_2,s_1]$ solves Eq.~\eqref{eq:patching2} and the function \mbox{$\breve\varPhi_\sUV^{[12]}\in C([t_1,s_2],\cC^\gamma(\bT))$} defined by
\begin{equation}
 \breve\varPhi^{[12]}_\sUV(x)=
 \begin{cases}
  \breve\varPhi^{(1)}_\sUV(x)~~\textrm{for}~~x\in[t_1,s_1]\times\bT,
  \\
  \breve\varPhi^{(2)}_\sUV(x)~~\textrm{for}~~x\in(s_1,s_2]\times\bT
 \end{cases}
\end{equation}
solves Eq.~\eqref{eq:patching1}.
\end{lem}
\begin{rem}
In order to prove the above lemma it is enough to use the semigroup property of the heat kernel $G$ of the fractional Laplacian $(-\Delta)^{\sigma/2}$.
\end{rem}

Using Lemma~\ref{lem:regular_eq} and Lemma~\ref{lem:patching} stated above one proves the following theorem.

\begin{thm}\label{thm:regular_eq}
Let $T>0$, $R>1$ and $\mathtt{u}\in C^\infty_\rc(\bR)$ be such that
\begin{equation}
 \|\phi_\sUV^\vartriangle\|_{\cC^\gamma(\bT)}\leq R,
 \qquad
 \|\mathtt{u}\varXi_\sUV\|_{L^\infty(\bH)}\leq R,
 \qquad
 |f^{i,m,a}_{\sUV}|\leq R,\quad (i,m,a)\in\frI,
\end{equation}
and $\mathtt{u}=1$ on some neighbourhood of $[-i_\triangleright(i_\triangleright+1)-2,T]$. There exists $\breve T_\sUV\in(0,T]$ and a function $\breve \varPhi_\sUV\in C([0,\breve T_\sUV],\cC^\gamma(\bT))$ such that: 
\begin{enumerate}
\item[(A)] $\breve \varPhi_\sUV$ satisfies Eq.~\eqref{eq:statement_rem_breve}, $\breve\varPhi_\sUV= G\ast (1_{(0,\infty)} (\hat F_\sUV[\breve \varPhi_\sUV]+ \check f_\sUV^\triangleright)+ \delta_0\otimes\phi^\vartriangle_\sUV)$,

\item[(B)] $\|\breve\varPhi_\sUV(\mathring x,\Cdot)\|_{\cC^\gamma(\bT)}<R$ for all $\mathring x\in[0,\breve T_0)$ and, in addition, if $\breve T_\sUV<\Tmax$, then $\|\breve\varPhi_0(\mathring x,\Cdot)\|_{\cC^\gamma(\bT)}=R$ for $\mathring x=\breve T_\sUV$.
\end{enumerate}
The pair $(\breve T_\sUV,\breve \varPhi_\sUV)$ with the above properties is unique.
\end{thm}

\begin{rem}\label{rem:maximal_solution}
By definition $\breve T_\sUV(R,\Tmax)$ is a non-decreasing function of $R$ and $\Tmax$ and for $\mathrm R\geq R$ and $\mathrm\Tmax\geq\Tmax$ the first of the following functions
\begin{equation}
 \breve\varPhi_\sUV(R,\Tmax)\in C([0,\breve T_\sUV],\sC^\beta(\bT)),
 \qquad
 \breve\varPhi_\sUV(\mathrm R,\mathrm\Tmax)\in C([0,\breve{\mathrm T}_\sUV],\sC^\beta(\bT))
\end{equation}
coincides with the restriction of the second function to the interval $[0,\breve T_\sUV]$, where $\breve T_\sUV=\breve T_\sUV(R,\Tmax)$ and $\breve{\mathrm T}_\sUV=\breve T_\sUV(\mathrm R,\mathrm\Tmax)$. Let 
\begin{equation}
 \breve T_\sUV^{\mathrm{max}}:=\lim_{R\nearrow\infty}\lim_{\Tmax\nearrow\infty} \breve T_\sUV(R,\Tmax)\in(0,\infty].
\end{equation}
There is a unique function $\breve\varPhi_\sUV^{\mathrm{max}}\in C([0,\breve T_\sUV^{\mathrm{max}}),\sC^\beta(\bT))$ such that the restriction of $\breve\varPhi_\sUV^{\mathrm{max}}$ to $[0,\breve T_\sUV(R,\Tmax)]$ coincides with $\breve\varPhi_\sUV(R,\Tmax)$.
\end{rem}

\begin{thm}\label{thm:solution_max}
Let $\Tmax>0$ be some finite time horizon and let $\mathtt{u}\in C^\infty_\rc(\bR)$ be such that $\mathtt{u}=1$ on some neighborhood of $[-i_\dagger(i_\dagger+1),6i_\dagger+\Tmax]$. Fix some $R>1$, $\delta\in(0,1]$ and an admissible boundary data for $(\varPhi_0^\triangleright,\breve\varPhi_0)$. Let $\sUV\in(0,\delta^{1/\varepsilon}]$ and a boundary data for $(\varPhi_\sUV^\triangleright,\breve\varPhi_\sUV)$ be arbitrary such that all of the conditions specified in Assumptions~\ref{ass:deterministic} and~\ref{ass:deterministic_initial} are satisfied. Let $\breve \varPhi_\sUV\in C([0,\breve T_\sUV),\cC^\gamma(\bT))$ be the maximal classical solution of Eq.~\eqref{eq:statement_rem_breve}
\begin{equation}
 \breve\varPhi_\sUV= G\ast (1_{(0,\infty)} (\hat F_\sUV[\breve \varPhi_\sUV]+ \check f_\sUV^\triangleright)+ \delta_0\otimes\phi^\vartriangle_\sUV)
\end{equation}
constructed in Theorem~\ref{thm:regular_eq} and Remark~\ref{rem:maximal_solution}. There exists a universal time \mbox{$\breve T_0\in(0,\Tmax]$} and a universal function $\breve\varPhi_0\in C([0,\breve T_0],\sC^\beta(\bT))$ such that:
\begin{enumerate}
\item[(A)] $\|\breve\varPhi_0(\mathring x,\Cdot)\|_{\sC^\beta(\bT)}<R$ for all $\mathring x\in[0,\breve T_0)$ and, in addition, if $\breve T_0<\Tmax$, then $\|\breve\varPhi_0(\mathring x,\Cdot)\|_{\sC^\beta(\bT)}=R$ for $\mathring x=\breve T_0$,
\item[(B)] for all sufficiently small $\sUV\in(0,1]$ it holds $\breve T_\sUV\geq\breve T_0$ and there exits a constant $c>0$ such that $\|\breve\varPhi_\sUV-\breve\varPhi_0\|_{C([0,\breve T_0],\sC^\beta(\bT))}\leq c\,\delta$.
\end{enumerate}
\end{thm}
\begin{rem}
The pair $(\breve\varPhi_0,\breve T_0)=(\breve\varPhi_0(R,\Tmax),\breve T_0(R,\Tmax))$ depends only on the choice of the initial data $\phi^\vartriangle_0$, the choice of the renormalization parameters $\mathfrak{f}^{i,m,a}_0$, $(i,m,a)\in\frI^-$, and the choice of the constants $\Tmax>0$ and $R>1$.
\end{rem}

\begin{rem}
We define the maximal solution $(\breve T_0^{\mathrm{max}},\breve\varPhi_0^{\mathrm{max}})$ as in Remark~\ref{rem:maximal_solution}.
\end{rem}

\begin{proof}
Let $t=0$. By Theorem~\ref{thm:solution_local} for all sufficiently small $\mathrm T>0$ depending only on $R$ there exists a function $\breve\varPhi_0\in C([t,t+\mathrm T],\sC^\beta(\bT))$ such that \mbox{$\breve\varPhi_0=\lim_{\sUV\searrow0}\breve\varPhi_\sUV$} in the domain $[t,t+\mathrm T]\times\bT\subset\bH$. If $\|\breve\varPhi_0(\mathring x,\Cdot)\|_{\sC^\beta(\bT)}<R$ for all $\mathring x\in[t,t+\mathrm T]$, and $t+\mathrm T<T$, then we repeat the above procedure with $t$ replaced by $t+\mathrm T$. Otherwise the procedure terminates. The function $\breve\varPhi_0\in C([0,\breve T_0],\sC^\beta(\bT))$ with all the desired properties is constructed by patching together the functions \mbox{$\breve\varPhi_0\in C([t,t+\mathrm T],\sC^\beta(\bT))$}.

In the construction described above we implicitly used the following two facts. (1) For every finite $R$ and $T$ the above procedure always terminates after finitely many of steps. (2) Suppose that the conditions stated in Assumption~\ref{ass:deterministic_initial} hold true for $\phi_0^\vartriangle=\breve\varPhi_0(t,\Cdot)$ and $\phi_\sUV^\vartriangle=\breve\varPhi_\sUV(t,\Cdot)$ with $R>1$ and $\delta=\delta_0\in(0,1/\breve c)$, where the constant $\breve c\geq 1$ was introduced in Theorem~\ref{thm:solution_local}. Then by the bound \mbox{$\|\breve\varPhi_0(t+\mathrm T,\Cdot)\|_{\sC^\beta(\bT)}<R$} and Theorem~\ref{thm:solution_local} (B), (C)
the conditions in Assumption~\ref{ass:deterministic_initial} hold true also for $\phi_0^\vartriangle=\breve\varPhi_0(t+\mathrm T,\Cdot)$ and $\phi_\sUV^\vartriangle=\breve\varPhi_\sUV(t+\mathrm T,\Cdot)$ with the same $R>1$ and $\delta=\breve c\,\delta_0$. 
\end{proof}

\section*{Acknowledgments}

I would like to thank Markus Fr{\"o}b, Stefan Hollands, Pablo Linares, Felix Otto, Markus Tempelmayr and Pavlos Tsatsoulis for discussions. I am grateful to the Max-Planck Society for supporting the collaboration between MPI-MiS and Leipzig U., grant Proj. Bez. M.FE.A.MATN0003.

\section*{List of symbols}\label{sec:app}

(The rightmost column indicates the page where a symbol is introduced.)

\begin{center}
\renewcommand{\arraystretch}{1}
\begin{longtable}{llr}\hline
Symbol & Meaning & \hspace{-3mm}Page\\
\hline
$\rdim$ & Dimension of space & \pageref{itemize:intro}\\
$\sigma$ & Order of fractional Laplacian &  \pageref{itemize:intro}\\
$\rDim$ & Effective dimension of spacetime $\sigma+\rdim$ & \pageref{dfn:dimensions}\\
$\lambda$ & Prefactor of the non-linear term in the SPDE &  \pageref{itemize:intro}\\
$\llangle X\rrangle $ & Expected value of a random variable $X$ & \pageref{dfn:expected_value} \\
$\llangle X_1,\ldots,X_n\rrangle $ & Joint cumulant of the random variables $X_1,\ldots,X_n$ & \pageref{dfn:cumulants} \\
$[\sUV]$ & $|\sUV|^{1/\sigma}$ &\pageref{dfn:square_bracket}\\
$\bM$ & Spacetime $\bR\times\bR^\rdim$ &  \pageref{itemize:intro}\\
$\bT$ & Torus $(\bR/2\pi\bZ)^\rdim$ & \pageref{dfn:spacetime_distance}\\
$\bH$ & Spacetime $\bR\times\bT^\rdim$ & \pageref{dfn:spacetime_distance}\\
$\microXi_\sUV$ & Microscopic noise & \pageref{eq:force_micro}\\
$\varXi_\sUV$ & Macroscopic noise & \pageref{eq:noise_rescaled}\\
$\varXi_{\uv;\sUV}$ & Mollified macroscopic noise & \pageref{dfn:noise_two_parameters}\\
$\varXi_0$ & White noise & \pageref{eq:intro_singular}\\
$\microF_\sUV$ & Microscopic force & \pageref{eq:force_micro}\\
$F_\sUV$ & Macroscopic force & \pageref{eq:force_macro}\\
$\microf_\sUV^{i,m,a}$ & Microscopic force coefficients & \pageref{eq:force_micro}\\
$f_\sUV^{i,m,a}$ & Macroscopic force coefficients & \pageref{eq:f_micro_macro}\\
$F_{\sUV,\sIR}$ & Stationary effective force& \pageref{eq:intro_effective_force}\\
$F_{\uv;\sUV,\sIR}$ & Generalized stationary effective force& \pageref{dfn:eff_force_uv}\\
$\hat F_{\sUV,\sIR}$ & Modified stationary effective force& \pageref{eq:eff_force_hat}\\
$\tilde F_{\sUV,\sIR}$ & Non-stationary effective force& \pageref{eq:intro_ansatz_F_tilde}\\
$\check F_{\sUV,\sIR}$ & Modified non-stationary effective force& \pageref{eq:eff_force_check}\\
$F_{\sUV,\sIR}^{i,m},F_{\sUV,\sIR}^{i,m,a}$ & Effective force coefficients & \pageref{eq:intro_eff_force_coeff}\\
$f_{\sUV,\sIR}^{i,m,a}$ & Local effective force coefficients & \pageref{eq:intro_eff_force_coeff}\\
$\mathfrak{f}_{\sUV}^{i,m,a}$ & Renormalization conditions & \pageref{eq:intro_ren_conditions}\\
$i_\diamond,i_\triangleright,i_\flat,i_\dagger,m_\diamond,m_\flat$\hspace{-2mm} & Useful parameters & \pageref{dfn:im}\\
$\fQ$ & Parabolic operator $\partial_{\mathring x}+(-\Delta_{\bar x})^{\sigma/2}$ & \pageref{eq:intro_spde}\\
$\fP_\sIR,\mathring\fP_\sIR,\bar\fP_\sIR,\fR_\sIR$&  Pseudo-differential operators & \pageref{dfn:diff_op}\\
$\fF$& Fourier transform & \pageref{sec:kernels}\\
$\fS_\sIR$& Rescaling map & \pageref{dfn:scaling_S}\\
$\fT$ & Periodization map & \pageref{dfn:periodization}\\
$\fA,\fB$ & Linear and quadratic maps in the flow equation & \pageref{dfn:maps_A_B}\\
$\fX^a_\ro$ & Taylor expansion of order $\ro-1$ & \pageref{dfn:map_X}\\
$\fI$ & Map extracting local part of a distribution & \pageref{dfn:map_I}\\
$\fL^m$ & Local extension map & \pageref{dfn:map_L_m}\\
$\fY_\pi,\fY^\omega$ & Permutation maps & \pageref{dfn:map_Y}\\
$\cX^{m,a}$ & Polynomial in relative coordinates & \pageref{eq:polynomial_relative} \\
$\frM,\bar\frM,\frM_\sigma,\bar\frM_\sigma$ & Spacetime/spatial multi-index sets & \pageref{dfn:st_multi}\\
$|a|,[a]$ & Norms for spacetime multi-indices & \pageref{dfn:st_multi}\\
$\bar\frI$ & Set of indices $(i,m,a)$ of force coefficients & \pageref{dfn:bar_frI}\\
$\bar\frI^+,\bar\frI^-$ & Subsets of $\bar\frI$ consisting of relevant/irrelevant indices & \pageref{dfn:bar_frI}\\
$\frI$ & Set of indices $(i,m,a)$ of effective force coefficients & \pageref{dfn:relevant_irrelevant}\\
$\frI^+,\frI^-$ & Subsets of $\frI$ consisting of relevant/irrelevant indices & \pageref{dfn:relevant_irrelevant}\\
$\phi^\vartriangle_\sUV,\phi^\triangleright_\sUV$ & Regular and irregular part of the initial data & \pageref{par:initial_data}\\
$G, \dot G_\sIR$ & Fractional heat kernel and its scale decomposition & \pageref{dfn:propagator_G}\\
$K_\sIR,\mathring K_\sIR,\bar K_\sIR,J_\sIR$ & Elementary regularizing kernels & \pageref{dfn:kernels}\\
$K_{\sUV,\sIR}^{m;\oo},K_{\sUV,\sIR}^{n,m;\oo}$ & Regularizing kernels & \pageref{dfn:kernels_K_n_m}\\
$\cK,\cK^n$ & Spaces of sign measures & \pageref{dfn:cK}\\
$\sC^\alpha(\bT),\sC^\alpha(\bH)$ & Besov-type spaces & \pageref{dfn:besov_T}\\
$\cD^m,\cD^{\mathsf{m}}$ & Spaces for effective force coefficients and cumulants & \pageref{dfn:cD}\\
$\cV^m,\cV^{\mathsf{m}}$ & Subspaces of $\cD^m,\cD^{\mathsf{m}}$ & \pageref{dfn:cV}\\
\hline
\end{longtable}
\end{center}

\end{document}